\documentclass[openright]{memo-l}

\usepackage[T1]{fontenc}
\usepackage[utf8]{inputenc}
\usepackage{color}
\usepackage[colorlinks=true]{hyperref}
\hypersetup{urlcolor=blue, linkcolor=blue, citecolor=darkgreen, anchorcolor=blue]}
\usepackage{bbm}
\usepackage{graphicx}
\usepackage{mathrsfs}
\usepackage{imakeidx}
\theoremstyle{plain}
\newtheorem{theorem}{Theorem}[chapter]
\newtheorem{lemma}[theorem]{Lemma}
\newtheorem{proposition}[theorem]{Proposition}
\newtheorem{corollary}[theorem]{Corollary}
\theoremstyle{definition}

\theoremstyle{remark}
\newtheorem{remark}[theorem]{Remark}
\numberwithin{section}{chapter}
\numberwithin{equation}{chapter}
\newtheorem{TheoremIntro}{Theorem}

\newcommand{\R}{{\mathbb R}}
\newcommand{\N}{{\mathbb N}}

\newcommand{\ird}[1]{\int_{\R^d}{#1}\,\dx}
\newcommand{\nrm}[2]{\left\|{#1}\right\|_{#2}}
\newcommand{\X}{\mathcal{X}_m}
\newcommand{\dx}{\,{\rm d}x}
\newcommand{\dy}{\,{\rm d}y}
\newcommand{\dz}{\,{\rm d}z}

\newcommand{\dt}{\,{\rm d}t}
\newcommand{\ds}{\,{\rm d}s}

\newcommand{\be}[1]{\begin{equation}\label{#1}}
\newcommand{\ee}{\end{equation}}
\renewcommand{\(}{\left(}
\renewcommand{\)}{\right)}

\definecolor{darkblue}{rgb}{0.05,.05,.65}
\definecolor{darkgreen}{rgb}{0.1,.65,.1}
\definecolor{darkred}{rgb}{0.8,0,0}

\newcommand{\idx}[1]{\index{#1}{#1}}

\newcommand{\lrangle}[1]{\langle #1\rangle}
\newcommand{\DD}{\mathsf D}
\newcommand{\nrmm}[2]{\left\|{#1}\right\|_{#2}}
\newcommand{\izx}[1]{\iint_{\R^{d+1}_+}#1\,z^{2\,\nu-1}\,\dx\,\dz}
\newcommand{\irdsph}[2]{\int_0^\infty\kern-5pt\int_{\mathbb S^{d-1}}#2\,{#1}^d\,\frac{{\mathrm d}{#1}}{{#1}}\,\mathrm d\omega}
\newcommand{\irdmu}[1]{\int_{\R^{d+1}}#1\,\mathrm d\mu}
\renewcommand{\P}{\mathsf P}
\newcommand{\alphaa}{a}

\newcommand{\mB}{\mathcal B}
\newcommand{\Mstar}{\mathcal M}
\newcommand{\pc}{p_m}

\newcommand{\lambdaBarenblatt}{\mathtt{b}}
\newcommand{\taustarA}{{\mathsf{\overline{c}}_\star}}
\newcommand{\taustar}{{\mathsf{c}_\star}}
\newcommand{\h}{h}
\newcommand{\muscal}{\lambda_{\bullet}}
\newcommand{\relativePressure}{\mathcal P}

\newcommand{\pcc}{{\mathsf p}}
\newcommand\supp{\mathrm{supp}}
\newcommand{\NN}{{\mathbb N}}
\newcommand{\RR}{{\mathbb R}}
\newcommand{\rd}{{\rm d}}
\newcommand{\ka}{\overline{\kappa}}
\newcommand{\kb}{\underline{\kappa}}
\newcommand{\OO}{\mathtt{Q}}

\newcounter{step}
\newenvironment{steps}{\setcounter{step}{0}}{}
\newcommand{\stepitem}{\addtocounter{step}1\par\smallskip\noindent\emph{Step \thestep.} }
\newcommand{\cc}{\mathsf{c}_3}

\newcommand{\sigmalambda}{\lambda}
\newcommand{\lambdasigma}{\sigma}

\newcommand{\J}{\mathcal K_\star}
\newcommand{\E}{\mathcal S_\star}
\newcommand{\Rr}{\mathfrak R}
\newcommand{\aA}{\mbox{\sc a}}
\newcommand{\bB}{\mbox{\sc b}}

\makeatletter
\@namedef{subjclassname@2020}{%
 \textup{2020} Mathematics Subject Classification}
\makeatother

\makeindex
\begin{document}
\frontmatter

%%%%%%%%%%%%%%%%%%%%%%%%%%%%%%%%%%%%%%%%%%%%%%%%%%%%%%%%%%%%%%%%%%%%%%%%
%%%%%%%%%%%%%%%%%%%%%%%%%%%%%%%%%%%%%%%%%%%%%%%%%%%%%%%%%%%%%%%%%%%%%%%%
\title{Stability in Gagliardo-Nirenberg-Sobolev inequalities\\
Flows, regularity and the entropy method}

\author[M.~Bonforte]{Matteo Bonforte}
\address{\hspace*{-24pt} M.~Bonforte: Departamento de Matem\'{a}ticas, Universidad Aut\'{o}noma de Madrid, and ICMAT, Campus de Cantoblanco, 28049 Madrid, Spain.}
\email{matteo.bonforte@uam.es}

\author[J.~Dolbeault]{Jean Dolbeault}
\address{\hspace*{-24pt} J.~Dolbeault: Ceremade, UMR CNRS n$^\circ$~7534, Universit\'e Paris-Dau\-phine, PSL Research University, Place de Lattre de Tassigny, 75775 Paris Cedex~16, France.}
\email{dolbeaul@ceremade.dauphine.fr}

\author[B.~Nazaret]{Bruno Nazaret}
\address{\hspace*{-24pt} B.~Nazaret: SAMM (EA 4543), FP2M (FR CNRS 2036), Universit\'e Paris 1, 90, rue de Tolbiac, 75634 Paris Cedex~13, and Mokaplan, Inria Paris, France.}
\email{bruno.nazaret@univ-paris1.fr}

\author[N.~Simonov]{Nikita Simonov}
\address{\hspace*{-24pt} N.~Simonov: LJLL (CNRS UMR n$^\circ$~7598), Sorbonne Universit\'e, Universit\'e Paris Cit\'e, CNRS, INRIA, F-75005 Paris, France.}
\email{nikita.simonov@sorbonne-universite.fr}

\date{\today}


\date{\today}

\begin{abstract}
We start by discussing stability results in Gagliardo-Nirenberg-Sobolev inequalities from a variational point of view. Using a non scale invariant form of the inequalities, which is equivalent to entropy - entropy production inequalities arising in the study of large time asymptotics of solutions to fast diffusion equations, we first establish non constructive estimates where the distance to the manifold of optimal functions is measured by a relative Fisher information.
When the tails of the initial data have a certain decay, solutions to the fast diffusion equation converge to self-similar Barenblatt functions in the strong topology of uniform convergence in relative error after some finite time. This threshold time plays a fundamental role in obtaining a constructive stability result.
Up to the threshold time, that is, in the initial time layer, the carr\'e du champ method provides improved decay rates of the relative entropy. After the threshold time, in the asymptotic time layer, improved rates of decay can be deduced from improved spectral gap estimates in the linearized problem, under appropriate orthogonality conditions. In the subcritical regime, these orthogonality conditions follow from an appropriate choice of the coordinates which amount to fix the center of mass at the origin. In the critical case, that is for Sobolev's inequality, scale invariance has to be taken into account. This can be rephrased as a strategy for computing the relative entropy with respect to a notion of best matching self-similar Barenblatt functions in place of the standard approach where entropy is defined relatively to a fixed family of self-similar solutions. Best matching is adapted to nonlinear evolution equations and degenerate in the asymptotic regime into more standard orthogonality conditions.
With this method, we provide fully constructive stability estimates, to the price of a small restriction of the functional space which is inherent to the method.
\end{abstract}

\maketitle

%%%%%%%%%%%%%%%%%%%%%%%%%%%%%%%%%%%%%%%%%%%%%%%%%%%%%%%%%%%%%%%%%%%%%%%%
\newpage\vspace*{2cm}
\section*{Foreword}~
\vspace*{1cm}

This Memoir is devoted to constructive stability results in a class of Gagliardo-Nirenberg-Sobolev inequalities. Beyond optimal constants and optimal functions, it is a classical question in the study of functional inequalities to ask how to measure the distance to the set of optimal functions in terms of the deficit functional, that is, the difference of the two terms in the inequality with the optimal constant.

A famous example is provided by Sobolev's inequalities. In 1991, G.~Bian\-chi and H.~Egnell proved that the difference of the two terms in Sobolev's inequalities is bounded from below by a distance to the manifold of the Aubin-Talenti functions. They argued by contradiction and gave a very elegant although not constructive proof. Estimating the stability constant and giving a constructive proof has been a challenge before this memoir and another more recent result.

Entropy methods and nonlinear flows for various functional inequalities are popular in the context of mass transportation and abstract gradient flow theories. They also relate optimal constants in functional inequalities with rates of decay and convergence of the solutions of the flow to self-similar Barenblatt solutions. Here we focus on Gagliardo-Nirenberg-Sobolev inequalities on the Euclidean space associated with the fast diffusion flow, which have Sobolev, Onofri and logarithmic Sobolev inequalities as limit cases.

Proving stability amounts to establishing, under constraints compatible with the nonlinear flow, a new version of the entropy~-- entropy production inequality with an improved constant. This is a refined version of the nonlinear \emph{carr\'e du champ} method based on some ideas of D.~Bakry and M.~Emery with a few additional ingredients. During an \emph{initial time layer}, we obtain a nonlinear improvement of the convergence rate to the Barenblatt solutions based on a backward in time estimate. A constructive Harnack inequality based on J.~Moser's methods allows us to prove a fully quantitative \emph{global Harnack Principle} for the nonlinear flow and quantify the \emph{threshold time} after which the solution of the evolution problem enters a neighborhood of the optimal functions in a relative uniform norm. From there on, we have an \emph{asymptotic time layer} improvement of the rates as a consequence of an improved \emph{Hardy-Poincar\'e} inequality based on spectral analysis. Properly rewritten so that Barenblatt solutions are transformed into Aubin-Talenti type functions, the improved entropy~-- entropy production inequality which measures the rate of convergence to equilibrium becomes a stability estimate where the distance to the manifold of optimal functions is measured either by a relative entropy or by a nonlinear relative Fisher information.

The whole method relies on entropies, which suppose a finite second moment, and on a \emph{global Harnack Principle}, which holds if and only if the tails of the solutions have a certain decay. These limitations are the price we have to pay in order to get a constructive stability estimate with an \emph{explicit constant}. All functions with finite entropy and sufficient decay, including of course all smooth compactly supported functions, are covered by our assumptions. This is the first result in such a large function space in which explicit stability results with a strong notion of distance and constructive constants are obtained.

Beyond sharp functional inequalities, the issue of stability is the next step. The purpose of analysis in mathematics is to compare quantities which, in the context of partial differential equations, are based on functions or their derivatives. Without explicit estimates on the constants, it is to some extend useless, particularly in view of applications for scientific computing or for predictions on models used in other areas of science, for instance in physics or biology. From a purely mathematical point of view, providing explicit estimates, especially in the perspective of stability questions, gives a far more better picture of the variational structure of functional inequalities than any argument by contradiction or based on a compactness method. This also puts in evidence phenomena that cooperate, like the role of modes associated with higher order eigenvalues in asymptotic regimes or purely nonlinear properties in improved entropy - entropy production inequalities away from optimal functions of the standard inequalities. As far as decay rates are concerned, it is of course crucial for applications to have explicit estimates.

\setcounter{tocdepth}{3}
\tableofcontents

\mainmatter

%%%%%%%%%%%%%%%%%%%%%%%%%%%%%%%%%%%%%%%%%%%%%%%%%%%%%%%%%%%%%%%%%%%%%%%%
%%%%%%%%%%%%%%%%%%%%%%%%%%%%%%%%%%%%%%%%%%%%%%%%%%%%%%%%%%%%%%%%%%%%%%%%
\chapter*{Introduction}\label{Chapter-0}

The purpose of this memoir is to establish a quantitative and constructive stability result for a class of Gagliardo-Nirenberg-Sobolev inequalities. We develop a new strategy in which the flow of the fast diffusion equation is used as a tool: a stability result in the inequality is equivalent to an improved rate of convergence to equilibrium for the flow. In both cases, the tail behaviour of the functions plays a key role. The regularizing effect of the parabolic flow allow us to connect an improved \idx{entropy - entropy production inequality} during the \idx{initial time layer} to spectral properties of a suitably linearized problem which is relevant for the \idx{asymptotic time layer}. The key issue is to determine a \idx{threshold time} between the two regimes as a consequence of a \index{Global Harnack Principle}{global Harnack Principle} and a uniform relative convergence towards \idx{Barenblatt self-similar solutions} of the flow. Altogether, the stability in the inequalities is measured by a \idx{deficit functional} which controls in strong norms the distance to the manifold of optimal functions.

The fast diffusion equation determines the family of inequalities that we consider, which includes Sobolev's inequality as an endpoint, and dictates the functional setting and the strategy, including for results based on variational methods. We extend the \index{Bianchi-Egnell result}{Bianchi-Egnell} \idx{stability} result in the subcritical range, but even in the critical case of \idx{Sobolev's inequality}, we provide a new stability result, with a strong norm that differs from the usual ones and arises from the \idx{entropy methods}. The main advantage compared to pure variational approaches is that we have a completely constructive method, with elementary estimates of the constants. This comes to the price of a slight restriction on the functional space (uniform integral condition on the tail behaviour of the functions) which is intrinsic to the \index{Global Harnack Principle}{global Harnack Principle}. The critical case of Sobolev's inequality is handled as a limit case, which requires an additional control of the convergence needed to control dilations.

Apart from stability results which are entirely new, some of the other results have already appeared, but we give new or simplified proofs for most of our statements. Except for the P\'olya–Szeg\H o principle, we rely only on elementary tools and present a large and consistent picture of the \idx{Gagliardo-Nirenberg-Sobolev inequalities} in connection with entropy methods and \idx{fast diffusion equation}s.

%%%%%%%%%%%%%%%%%%%%%%%%%%%%%%%%%%%%%%%%%%%%%%%%%%%%%%%%%%%%%%%%%%%%%%%%
\medskip Let us give a general overview of our strategy for proving \idx{stability}, outline of the contents of the chapters and state some key results. We consider the \emph{Gagliardo-Nirenberg-Sobolev} interpolation inequalities
\[
\nrm{\nabla f}2^\theta\,\nrm f{\mathsf s}^{1-\theta}\ge\mathscr C\,\nrm f{\mathsf t}
\]
where $\nrm fq$ denotes the $\mathrm L^q$ norm of $f$ with respect to Lebesgue's measure on the Euclidean space $\R^d$, $d\ge1$. The exponents $\mathsf s$ and $\mathsf t$ are such that $1<\mathsf s<\mathsf t$, and $\mathsf t\le 2\,d/(d-2)$ if $d\ge3$. By scaling invariance, the exponent $\theta$ is uniquely determined, such that $(d-2)\,\theta/2+d\,(1-\theta)/\mathsf s=d/\mathsf t$. We assume that $\mathscr C$ is the best possible constant in the inequality, for all smooth functions with compact support, and by a standard completion argument, in the natural Sobolev space. Existence of optimal functions, \emph{i.e.}, of functions which realize the equality case in the inequality, is a classical result of the Calculus of Variations. Homogeneity, invariance under translations and dilations mean that there is a whole manifold~$\mathfrak M$ of optimal functions. It is a well known issue to prove \emph{\idx{stability} results} like
\[
\nrm{\nabla f}2^{2\,\theta}\,\nrm f{\mathsf s}^{2\,(1-\theta)}-\mathscr C\,\nrm f{\mathsf t}^2\ge\kappa\,\inf_{g\in\mathfrak M}\mathsf d^2(f,g)
\]
where $\kappa$ is a positive constant and $\mathsf d(f,g)$ is a notion of distance. For instance, in the \emph{critical} case of \idx{Sobolev's inequality} corresponding to $\mathsf t=2\,d/(d-2)$, with $d\ge3$ and $\theta=1$ (so that $\nrm f{\mathsf s}$ plays no role), G.~Bianchi and H.~Egnell proved in the celebrated paper~\cite{MR1124290} that the \idx{stability} inequality holds with $\mathsf d(f,g)=\nrm{\nabla f-\nabla g}2$ while $\mathfrak M$ is the manifold of the \idx{Aubin-Talenti functions} generated from the profile $\mathsf g(x)=\big(1+|x|^2\big)^{-(d-2)/2}$ by the invariances of the inequality. Such a result is said \emph{quantitative} because the dependence on $\mathsf d$ is explicit, but not \emph{constructive} as the constant $\kappa$ is obtained by a compactness argument and no estimate of $\kappa$ can be deduced in this approach. Our purpose is to give constructive estimates, which requires a completely different approach. We use \emph{\idx{entropy methods}} and the \emph{\idx{fast diffusion equation}} for this purpose. 

\medskip A first restriction is that the exponents $\mathsf s$ and $\mathsf t$ have to satisfy the condition $\mathsf t=2\,(\mathsf s-1)=2\,p$ for some $p>1$, with $p\le d/(d-2)$ if $d\ge3$. Under this condition, it has been shown in~\cite{DelPino2002} that $\mathfrak M$ is $(d+2)$-dimensional manifold generated by
\[
\mathsf g(x)=\(1+|x|^2\)^{-\frac1{p-1}}\quad\forall\,x\in\R^d
\]
using multiplications, translations and scalings. The \index{Gagliardo-Nirenberg-Sobolev inequalities}{\emph{Gagliardo-Nirenberg-Sobolev inequality}} becomes
\be{GNS-Introduction}\tag{GNS}
\nrm{\nabla f}2^\theta\,\nrm f{p+1}^{1-\theta}\ge\mathcal C_{\mathrm{GNS}}\,\nrm f{2\,p}\,.
\ee
It is equivalent to the \emph{\idx{entropy - entropy production inequality}}
\[
\mathcal J[f|g]\ge4\,\mathcal E[f|g]
\]
for nonnegative functions $f$, where
\[
\mathcal E[f|g]:=\frac{2\,p}{1-p}\ird{\(f^{p+1}-g^{p+1}-\tfrac{1+p}{2\,p}\,g^{1-p}\(f^{2p}-g^{2p}\)\)}
\]
is the \emph{\idx{relative entropy}} of $f$ with respect to $g\in\mathfrak M$, and
\[
\mathcal J[f|g]:=\frac{p+1}{p-1}\ird{\left|(p-1)\,\nabla f+f^p\,\nabla g^{1-p}\right|^2}
\]
is a nonlinear \emph{\idx{relative Fisher information}}. There are various ways to measure the distance which separates an arbitrary function $f$ from $\mathfrak M$ and it turns out to be particularly convenient to use $\mathcal E[f|g_f]$ where $g_f$ is determined by the moment condition 
\be{MC}\tag{MC}
\ird{\big(1,x\big)\,f^{2p}}=\ird{\big(1,x\big)\,g^{2p}}
\ee
and, although nonlinear, $\mathcal E[f|g_f]$ is a good notion of distance as, for instance, it controls $\nrm{|f|^{2p}-|g_f|^{2p}}1$ by the \idx{Csisz\'ar-Kullback inequality}. With this specific choice of $g=g_f$, we aim at proving an \emph{improved \idx{entropy - entropy production inequality}}
\be{ImprovedEEP-Intro}\tag{E}
\mathcal J[f|g]-4\,(1+\zeta)\,\mathcal E[f|g]\ge0
\ee
for some $\zeta>0$, which is equivalent to
\[
\mathcal J[f|g]-4\,\mathcal E[f|g]\ge\kappa\,\mathcal J[f|g]
\]
with $\kappa=\zeta/(1+\zeta)$. Also it is not an $\mathrm H^1_0(\R^d)$ measure of the distance to $\mathfrak M$ as in~\cite{MR1124290}, the right-hand side of this new inequality provides us with a \idx{stability} result, which has a counterpart for~\eqref{GNS-Introduction}. In Chapter~\ref{Chapter-1}, we review existing related \idx{stability} results and give a proof of existence of $\kappa>0$ by variational (concentration-compactness) methods. Measuring \idx{stability} by the \idx{relative Fisher information} $\mathcal J[f|g]$ is an entirely new approach. However, at this stage, we have no estimate of~$\kappa$.

\medskip So far, we did not make use of any nonlinear flow nor of entropy methods. Let us explain how \emph{\idx{fast diffusion equation}s} enter into play. In \idx{self-similar variables}, the \idx{fast diffusion equation}, posed on $\R^d$, $d\ge2$, with exponent $m\in[m_1,1)$ and $m_1:=1-1/d$, is
\be{FDr-Intro}\tag{FDE}
\frac{\partial v}{\partial t}+\nabla\cdot\(v\,\nabla v^{m-1}\)=2\,\nabla\cdot(x\,v)\,,\quad v(t=0,\cdot)=v_0\,.
\ee
By applying this flow to the \emph{\idx{relative entropy}}
\[
\mathcal F[v]:=\frac1{m-1}\ird{\(v^m-\mathcal B^m-m\,\mathcal B^{m-1}\,(v-\mathcal B)\)}
\]
where $\mathcal B$ is the \idx{Barenblatt function}
\[
\mathcal B(x):=\big(1+|x|^2\big)^\frac1{m-1}\quad\forall\,x\in\R^d\,,
\]
we obtain by a standard computation that a solution $v$ of~\eqref{FDr-Intro} satisfies
\[
\frac \rd{\dt}\mathcal F[v(t,\cdot)]=-\,\mathcal I[v(t,\cdot)]
\]
for the \emph{\idx{relative Fisher information}} functional $\mathcal I$ defined by
\[
\mathcal I[v]:=\frac m{1-m}\ird{v\,\big|\nabla v^{m-1}-\nabla \mathcal B^{m-1}\big|^2}\,.
\]
It is a key step to recognise that we are dealing with the same quantities as in the variational approach. With
\[
p=\frac1{2\,m-1}\quad\Longleftrightarrow\quad m=\frac{p+1}{2\,p}\,,\quad v=f^{2p}\,,\quad\mathcal B=\mathsf g^{2p}
\]
and in particular with the condition $1<p\le d/(d-2)$, $d\ge3$, which is equivalent to $m_1\le m<1$, it turns out that
\[
\mathcal F[v]=\mathcal E[f|\mathsf g]\quad\mbox{and}\quad\mathcal I[v]=\mathcal J[f|\mathsf g]\,.
\]
As observed in~\cite{DelPino2002},~\eqref{GNS-Introduction} with sharp constant is equivalent to the decay estimate
\[
\mathcal F[v(t,\cdot)]\le\mathcal F[v_0]\,e^{-4t}\quad\forall\,t\ge0
\]
if $v$ solves~\eqref{FDr-Intro}. Our overall strategy is now to consider $\mathcal E[f|g_f]$ with~$g_f$ as in~\eqref{MC}, which is equivalent to specify that $f$ is such that $g_f=\mathsf g$ at least at $t=0$ and prove that $\mathcal F[v(t,\cdot)]$ decays with the rate
\be{FKt}\tag{R}
\mathcal F[v(t,\cdot)]\le\mathcal F[v_0]\,e^{-(4+\zeta)t}\quad\forall\,t\ge0
\ee
using the properties of~\eqref{FDr-Intro}. In a word, we look for improved decay rates of the entropy in order to establish an improved \idx{entropy - entropy production inequality}. Details are given in Chapter~\ref{Chapter-2}.

\medskip Why is it that we can expect to obtain an improved decay rate of $\mathcal F[v(t,\cdot)]$~? Let us start with the \index{asymptotic time layer}{asymptotic} regime as $t\to+\infty$. It is of standard knowledge, see for instance~\cite{Vazquez2003}, that solutions to~\eqref{FDr-Intro} converge to $\mathcal B$ in strong topologies. Hence, it makes sense to consider the Taylor expansions of the entropy and the Fisher information around $\mathcal B$. Let us consider the two quadratic forms
\[
\mathsf F[h]=\lim_{\varepsilon\to0}\varepsilon^{-2}\,\mathcal F\big[\mB+\varepsilon\,\mB^{2-m}\,h\big]\quad\mbox{and}\quad\mathsf I[h]=\lim_{\varepsilon\to0}\varepsilon^{-2}\,\mathcal I\big[\mB+\varepsilon\,\mB^{2-m}\,h\big]\,.
\]
By a \emph{\idx{Hardy-Poincar\'e inequality}} detailed in Chapter~\ref{Chapter-2}, we have
\[
\mathsf I[h]\ge\Lambda\,\mathsf F[h]
\]
with $\Lambda=4$ if $\ird{h\,\mB^{2-m}}=0$ and $\Lambda=4\,\big(1+d\,(m-m_1)\big)$ if, additionally, we assume that $\ird{x\,h\,\mB^{2-m}}=0$. In other words, the optimal decay rate of $\mathcal F[v(t,\cdot)]$ is characterized in the \index{asymptotic time layer}{asymptotic} regime as $t\to+\infty$ by the spectral gap $\Lambda=4$. Under the additional moment condition~\eqref{MC} (on the center of mass), we obtain $\zeta=\Lambda-4>0$ if $m>m_1$. Recall that $m>m_1$ means $p<d/(d-2)$ and covers the whole subcritical range in~\eqref{GNS-Introduction}. Altogether, we have an improved decay rate on an \emph{\index{asymptotic time layer}{asymptotic} time layer} $[T_\star,+\infty)$, that has been explored in~\cite{Blanchet2009} and subsequent papers. An important feature is that the estimates on $\Lambda$ are explicit but require strong regularity conditions. 

\medskip Under the additional moment condition~\eqref{MC}, the nonlinear generalization of the \emph{\idx{carr\'e du champ} method} of D.~Bakry and M.~Emery shows that an improved \idx{entropy - entropy production inequality} holds whenever $\mathcal F[v(t,\cdot)]$ is bounded away from $0$, as shown in~\cite{Dolbeault2013917}, and provides an explicit estimate, which is also given in Chapter~\ref{Chapter-2}. This means that there is an improved decay rate on the \emph{\idx{initial time layer}} $[0,T_\star]$, which is also explicit. However, the precise value of the improvement $\zeta$ depends on the \emph{\idx{threshold time}}~$T_\star$, which has to be carefully estimated in terms of the initial datum $v_0$. In Chapter~\ref{Chapter-3}, we follow J.~\index{Moser iteration}{Moser}'s original ideas in~\cite{Moser1964,Moser1971} for proving regularity estimates for solutions of linear parabolic equations and establish explicit constants which are not available from the literature, with new expressions and simplified proofs of the \idx{Harnack inequality} and the corresponding \idx{H\"older continuity} estimates. In Chapter~\ref{Chapter-4}, these results are applied to the \idx{fast diffusion equation}
\[
\frac{\partial u}{\partial t}=\Delta u^m\,,\quad u(t=0,\cdot)=u_0\ge0\,,
\]
which is equivalent to~\eqref{FDr-Intro} up to a change of variables. By a \emph{\index{Global Harnack Principle}{global Harnack Principle}}, we obtain the \emph{\idx{uniform convergence in relative error}} of the solutions to~\eqref{FDr-Intro}, with an explicit rate:
%-----------------------------------------------------------------------
\begin{TheoremIntro}\label{uniform.convergence-Intro}\emph{
Assume that $d\ge2$, $m\in[m_1,1)$. If $v$ solves~\eqref{FDr-Intro} for some initial datum $v_0\ge0$ such that $\ird{v_0}=\ird{\mathcal B}$, and
\[
\|v_0\|_{\mathcal{X}_m}:=\sup_{R>0}R^\frac{1+d\,(m-m_1)}{1-m}\,\int_{|x|>R}v_0(x)\,\dx<\infty\,,
\]
then there exists an explicit $\varepsilon_\star$ such that, for any $\varepsilon\in(0,\varepsilon_\star)$,
\[
\sup_{x\in\R^d}\Big|\frac{v(t,x)}{B(t,x)}-1\Big|\le \varepsilon\,\quad \forall\,t\ge T_\star:=\mathscr{C}_\star\,\varepsilon^{-\mathrm{a}}\,.
\]
}\end{TheoremIntro}
%-----------------------------------------------------------------------
\noindent Here $\varepsilon_\star$ and $\mathrm{a}>0$ are numerical constants which depend only on $d$ and $m$ while~$\mathscr{C}_\star$ depends also on $\|v_0\|_{\mathcal{X}_m}$. We refer to Theorem~\ref{Thm:RelativeUniform} and Proposition~\ref{Prop:TT} for a more detailed statement which also covers the case $d=1$. Theorem~\ref{uniform.convergence-Intro} is a fully constructive version of various earlier qualitative results that can be found in~\cite{Bonforte2006,bonforte2020fine}. Notice that we have a simplified form using Corollary~\ref{Cor:FA}, which allows us to control $\mathcal F[v_0]$ in terms of~$\|v_0\|_{\mathcal{X}_m}$. Chapters~\ref{Chapter-4} and~\ref{Chapter-5} are at the core of our method: we use \emph{regularization} properties of the \idx{fast diffusion equation}, made fully quantitative, to control the \emph{\idx{threshold time}}~$T_\star$. This induces the limitation $\|v_0\|_{\mathcal{X}_m}<\infty$, which is discussed in Chapter~\ref{Chapter-7} and is the main limitation of our approach.

\medskip Summarizing, the combination of the improved decay rates in the \emph{\idx{initial time layer}} and in the \emph{\index{asymptotic time layer}{asymptotic} time layer} based on the explicit estimate of the \emph{\idx{threshold time}} $T_\star$ of Theorem~\ref{uniform.convergence-Intro} establishes~\eqref{FKt} for some explicit $\zeta>0$. Since~\eqref{FKt} holds for any $t\ge0$ and in particular at $t=0$, this proves an \emph{improved \idx{entropy - entropy production inequality}}~\eqref{ImprovedEEP-Intro} in the subcritical range. In Chapter~\ref{Chapter-5}, we establish the following result (see Theorem~\ref{Thm:ImprovedE-EPinequality} for a refined statement).
%---------------------------------------------------------------------
\begin{TheoremIntro}\label{Thm:ImprovedE-EPinequality-Intro}\emph{
Let $m\in(m_1,1)$ if $d\ge2$, $m\in(1/2,1)$ if $d=1$. Then
\[
\mathcal I[v]\ge(4+\zeta)\,\mathcal F[v]
\]
for any nonnegative function $v\in\mathrm L^1(\R^d)$ such that $\|v\|_{\mathcal{X}_m}<\infty$, $\ird v=\ird{\mathcal B}$, $\ird{x\,v}=0$, for some $\zeta$ depending only on $\|v\|_{\mathcal{X}_m}$, $m$ and $d$.}\end{TheoremIntro}
%---------------------------------------------------------------------
When $d\ge3$, in the critical case $m=m_1$ corresponding to $p=d/(d-2)$, that is, when~\eqref{GNS-Introduction} is \idx{Sobolev's inequality}, our result fails because we have no more an improved spectral gap in the \emph{\idx{Hardy-Poincar\'e inequality}} if we assume only that $\ird{(1,x)\,h\,\mB^{2-m}}=(0,0)$. This is because \hbox{$\Lambda=4\,\big(1+d\,(m-m_1)\big)=4$}. To restore an improved spectral gap, one has to impose the additional constraint that \hbox{$\ird{|x|^2\,h\,\mB^{2-m}}=0$}. This is not as easy as for the lower order moments, because the second moment is not conserved by~\eqref{FDr-Intro}: if $v$ is a solution, then
\[
\frac d{dt}\ird{|x|^2\,v(t,x)}=2\,d\,\frac{1-m}m\ird{v^m(t,x)}-4\ird{|x|^2\,v(t,x)}\,.
\]
This differential equation does not have a simple, explicit expression. However, we overcome this problem by using a different \idx{relative entropy} $\mathcal E[f|g_f]$, where $g_f$ is the \emph{\idx{best matching}} function in $\mathfrak M$ in the sense of \idx{relative entropy}, that is,
\[
\mathcal E[f|g_f]=\min_{g\in\mathfrak M}\mathcal E[f|g]\,,
\]
where $g_f$ is uniquely determined by the condition 
\[
\ird{\big(1,x,|x|^2\big)\,f^{2p}}=\ird{\big(1,x,|x|^2\big)\,g^{2p}}\,.
\] 
Notice that we have added a condition on the second moment of $f^{2p}$. In terms of quantities related to the flow, we have $\min_{g\in\mathfrak M}\mathcal E[f|g]=\mathcal{F}_\lambda[v]$ where $\mathcal{F}_\lambda [v]:=\frac1{m-1}\ird{\(v^m-\mathcal{B}_\lambda^m-m\,\mathcal{B}_\lambda^{m-1}\,(v-\mathcal{B}_\lambda)\)}$ is the \idx{relative entropy} for the \idx{fast diffusion equation} taken with respect to the \emph{\idx{best matching} Barenblatt} function $\mathcal B_\lambda(x)=\lambda^{-d/2}\,\mathcal B\big(x/\sqrt\lambda\big)$, which is such that
\[
\ird{|x|^2\,B_\lambda(x)}=\ird{|x|^2\,v(t,x)}\,,
\]
and it involves a scaling parameter $\lambda=\lambda(t)$ which is not explicit. However, in Chapter~\ref{Chapter-6}, we are able to find estimates on $\lambda(t)$ such that a refined version of the \idx{relative entropy}, namely the \idx{relative entropy} of $v$ with respect to \idx{best matching} \idx{Barenblatt function} \emph{at any time $t\ge0$}, can be used. In the \index{asymptotic time layer}{asymptotic} regime as $t\to+\infty$, we gain an improved \emph{\idx{Hardy-Poincar\'e inequality}}, which guarantees an \emph{improved \index{entropy - entropy production inequality}{entropy~- entropy production inequality}} and allows us to extend the result of Theorem~\ref{Thm:ImprovedE-EPinequality-Intro} to the critical case.

\medskip The improvements in the entropy formulation can be recast into the framework of the functional inequality~\eqref{GNS-Introduction}. Our main \idx{stability} results of Chapters~\ref{Chapter-5} and~\ref{Chapter-6} can be summarized as follows. Let us define the \emph{\idx{deficit functional}} by
\[\label{deficit}
\delta[f]:=(p-1)^2\,\nrm{\nabla f}2^2+4\,\tfrac{d-p\,(d-2)}{p+1}\,\nrm f{p+1}^{p+1}-\mathcal K_{\mathrm{GNS}}\,\nrm f{2\,p}^{2\,p\,\gamma}
\]
with $\gamma=\frac{d+2-p\,(d-2)}{d-p\,(d-4)}$ and $\mathcal K_{\mathrm{GNS}}$ chosen so that $\delta[\mathsf g]=0$. Up to a scaling, $\delta[f]\ge0$ is equivalent to~\eqref{GNS-Introduction} and $\mathcal K_{\mathrm{GNS}}$ can be computed in terms of $\mathcal C_{\mathrm{GNS}}$.
%-----------------------------------------------------------------------
\begin{TheoremIntro}\label{Thm:stabilityBDNS}\emph{
Let $d\ge3$ and assume that $1<p\le d/(d-2)$, or $d=1$,~$2$ and $p\in(1,+\infty)$. For any $f\in\mathrm L^{2p}(\mathbb R^d)$ with $\nabla f\in\mathrm L^2(\mathbb R^d)$ such that
\[
A:=\sup_{r>0}r^\frac{d-p\,(d-4)}{p-1}\int_{|x|>r}|f|^{2p}\,\dx<\infty\,,
\]
we have the estimate
\[\label{Ineq:StabFisher}
\delta[f]\ge\kappa\,\inf_{\varphi\in\mathfrak M}\ird{\big|(p-1)\,\nabla f+f^p\,\nabla\varphi^{1-p}\big|^2}
\]
for some explicit positive constant $\kappa$ which depends only on $d$, $p$, $\nrm f{2p}$,~$A$, and takes positive values on $\mathfrak M$.}\end{TheoremIntro}
%-----------------------------------------------------------------------
\noindent In the right-hand side of the stability estimate, the exponent is optimal, as will be discussed in Section~\ref{Sec:Conclusion}, point (ii). More detailed statements can be found in Corollary~\ref{Cor:DefFisher} and Theorem~\ref{Thm:Main}. Notice that the critical case $p=d/(d-2)$ corresponding to \idx{Sobolev's inequality} is covered, thus providing a \idx{stability} estimate with an explicit \idx{stability} constant. The parabolic equation~\eqref{FDr-Intro} is finally no more than a technical tool which allows us to relate an initial datum $v_0=|f|^{2p}$ to an \index{asymptotic time layer}{asymptotic} solution $v(t,\cdot)$ satisfying a spectral gap property as $t\to+\infty$. Controlling the \idx{threshold time} $T_\star$ however requires the property that $\|v_0\|_{\mathcal{X}_m}$ is finite. While this is clearly a restriction due to the method, we emphasize that a \idx{stability} result based on the \idx{relative entropy} for measuring the distance to $\mathfrak{M}$ cannot be true without an assumption of uniform boundedness of the second moment. This restriction and the limitations of our method are discussed in Chapter~\ref{Chapter-7}.

%%%%%%%%%%%%%%%%%%%%%%%%%%%%%%%%%%%%%%%%%%%%%%%%%%%%%%%%%%%%%%%%%%%%%%%%
\newpage\begin{center}
{\bf Acknowledgements}
\end{center}

\indent The first author was partially supported by the Projects\,\,MTM2017-85757-P,  PID2020-113596GB-I00 and PID2023-150166NB-I00 (Ministry of Science and Innovation, Spain) and the Spanish Ministry of Science and Innovation, through the ``Severo Ochoa Programme for Centres of Excellence in R\&D'' (CEX2019-000904-S and CEX2023-001347-S) and by the E.U.~H2020 MSCA programme, grant agreement 777822. The second author was partially supported by the Project EFI (ANR-17-CE40-0030) of the French National Research Agency (ANR). The third author was partially supported by INRIA Paris, Inria Mokaplan team. The fourth author was partially supported by the Spanish Ministry of Science and Innovation, through the  FPI-grant BES-2015-072962, associated to the project MTM2014-52240-P (Spain) and  by the E.U. H2020 MSCA programme, grant agreement 777822, by the Project EFI (ANR-17-CE40-0030) of the French National Research Agency (ANR), and by the DIM Math-Innov of the Region \^Ile-de-France.
\\[4pt]
The authors thank two referees for relevant comments and especially one of them for his careful reading and detailed suggestions which helped a lot to improve the presentation of the results. The authors also want to thank M.~Fathi for pointing out the reference~\cite{eldan2020stability}.
\\[4pt]
\noindent{\sl\small\copyright~2022 by the authors. This paper may be reproduced, in its entirety, for non-commercial purposes.}

%%%%%%%%%%%%%%%%%%%%%%%%%%%%%%%%%%%%%%%%%%%%%%%%%%%%%%%%%%%%%%%%%%%%%%%%
%%%%%%%%%%%%%%%%%%%%%%%%%%%%%%%%%%%%%%%%%%%%%%%%%%%%%%%%%%%%%%%%%%%%%%%%
\chapter{Gagliardo-Nirenberg-Sobolev inequalities by variational methods}\label{Chapter-1}

This chapter is devoted to the study of a family of \idx{Gagliardo-Nirenberg-Sobolev inequalities} which contains the classical Sobolev inequality as an endpoint. We give a self-contained presentation of various results based on the identification of all \idx{optimal function}s. Our goal is to consistently collect and expose those results, for which we claim no originality, but with some new proofs. We also include some considerations on \idx{stability}, which motivate the whole memoir. In this chapter, we rely only on classical methods of the Calculus of Variations and tools of the \idx{concentration-compactness method}. Except for standard symmetrization and regularity results, we put an effort in using only elementary techniques and keep the proofs as self-contained as possible. Sources and references to further results or alternative methods are collected at the end of the chapter.

%%%%%%%%%%%%%%%%%%%%%%%%%%%%%%%%%%%%%%%%%%%%%%%%%%%%%%%%%%%%%%%%%%%%%%%%
%%%%%%%%%%%%%%%%%%%%%%%%%%%%%%%%%%%%%%%%%%%%%%%%%%%%%%%%%%%%%%%%%%%%%%%%
\section{Gagliardo-Nirenberg-Sobolev inequalities}

%%%%%%%%%%%%%%%%%%%%%%%%%%%%%%%%%%%%%%%%%%%%%%%%%%%%%%%%%%%%%%%%%%%%%%%%
\subsection{A one-parameter family of inequalities}

In this memoir, we consider the family of \idx{Gagliardo-Nirenberg-Sobolev inequalities} given by
\be{GNS}
\nrm{\nabla f}2^\theta\,\nrm f{p+1}^{1-\theta}\ge\mathcal C_{\mathrm{GNS}}(p)\,\nrm f{2p}\quad\forall\,f\in\mathcal H_p(\R^d)\,.
\ee
The invariance of~\eqref{GNS} under dilations determines the exponent
\be{Ch1:theta}
\theta=\frac{d\,(p-1)}{\big(d+2-p\,(d-2)\big)\,p}\,,
\ee
and the space $\mathcal H_p(\R^d)$ is defined as the completion of the space of infinitely differentiable functions on $\R^d$ with compact support, with respect to the norm
\[
f\mapsto(1-\theta)\,\nrm f{p+1}+\theta\,\nrm{\nabla f}{2}\,.
\]
Norms are defined by $\nrm fq=\(\ird{|f|^q}\)^{1/q}$ for any $q>1$ and $\nrm f\infty$ denotes the $\mathrm L^\infty(\R^d)$ norm. We shall say that the exponent $p$ is \emph{\index{admissible exponent}{admissible}} if
\[
p\in(1,+\infty)\mbox{ if }d=1\mbox{ or }2\,,\quad p\in(1,p^\star]\mbox{ if }d\ge3\quad\mbox{with}\quad p^\star:=\tfrac d{d-2}\,.
\]
In the limit case where $p=p^\star$, $d\ge3$ for which $\theta=1$, we are left with the Beppo-Levi space (see~\cite{AIF_1954_5_305_0})
\[
\mathcal H_{p^\star}(\R^d):=\left\{f\in\mathrm L^{2\,p^\star}(\R^d)\,:\,|\nabla f|\in\mathrm L^2(\R^d)\right\}\,.
\]

%%%%%%%%%%%%%%%%%%%%%%%%%%%%%%%%%%%%%%%%%%%%%%%%%%%%%%%%%%%%%%%%%%%%%%%%
\subsection{Optimality}

Let us consider the manifold of the \emph{\idx{Aubin-Talenti functions}}
\[
\mathfrak M:=\left\{g_{\lambda,\mu,y}\,:\,(\lambda,\mu,y)\in(0,+\infty)\times\R\times\R^d\right\}
\]
where $g_{\lambda,\mu,y}(x):=\lambda^\frac d{2p}\,\mu^\frac1{2p}\,\mathsf g\big(\lambda\,(x-y)\big)$ with the convention $\mu^q=|\mu|^{q-1}\,\mu$ if $\mu<0$ and
\be{Aubin.Talenti}
\mathsf g(x)=\(1+|x|^2\)^{-\frac1{p-1}}\quad\forall\,x\in\R^d\,.
\ee
Strictly speaking, \idx{Aubin-Talenti functions} correspond only to the critical case $p=p^\star\!$, but we shall use this denomination also  in the subcritical case $p<p^\star\!$. The complete statement on \idx{Gagliardo-Nirenberg-Sobolev inequalities}~\eqref{GNS} goes as follows.
%-----------------------------------------------------------------------
\begin{theorem}\label{Thm:GNS} Assume that $d\ge1$ is an integer and let $p$ be an admissible exponent. Then equality case in~\eqref{GNS} is achieved if and only if $f\in\mathfrak M$. As a consequence, the \idx{optimal constant} is
\[
\mathcal C_{\mathrm{GNS}}(p)=\tfrac{\big(\tfrac{4\,d}{p-1}\,\pi\big)^\frac\theta2\(2\,(p+1)\)^\frac{1-\theta}{p+1}}{(d+2-p\,(d-2))^\frac{d-p\,(d-4)}{2\,p\,(d+2-p\,(d-2))}}\,\Gamma\(\tfrac{2\,p}{p-1}\)^{-\frac\theta d}\,\Gamma\(\tfrac{2\,p}{p-1}-\tfrac d2\)^\frac\theta d\,.
\]
\end{theorem}
%-----------------------------------------------------------------------

%%%%%%%%%%%%%%%%%%%%%%%%%%%%%%%%%%%%%%%%%%%%%%%%%%%%%%%%%%%%%%%%%%%%%%%%
\subsection{Related inequalities}

Inequalities~\eqref{GNS} have various interesting limits. First of all, if $d\ge3$, the case $p=p^\star$ is simply \emph{\idx{Sobolev's inequality}}
\be{SobolevRd}
\nrm{\nabla f}2\ge\mathsf S_d\,\nrm f{2\,p^\star}\quad\forall\,f\in\mathcal H_{p^\star}(\R^d)\,,
\ee
where
\[
\mathsf S_d=\mathcal C_{\mathrm{GNS}}(p^\star)=\sqrt{\pi\,d\,(d-2)}\(\frac{\Gamma\(\tfrac d2\)}{\Gamma\(d\)}\)^{\!\frac1d}
\]
is the \idx{optimal constant} in \idx{Sobolev's inequality}. The other classical expression 
\[
\mathsf S_d^2=\frac14\,d\,(d-2)\,\frac{2^\frac2d\,\pi^{1+\frac1d}}{\Gamma\(\tfrac{d+1}2\)^\frac2d}=\frac14\,d\,(d-2)\,|\mathbb S^d|^\frac2d
\]
can be  easily recovered using the \idx{duplication formula} $\Gamma(\tfrac d2)\,\Gamma\(\tfrac{d+1}2\)=2^{1-d}\,\sqrt\pi\,\Gamma(d)$.

If $d=2$, another interesting limit endpoint is the \emph{Euclidean \idx{Onofri inequality}}
\[
\int_{\R^2}e^{h-\bar h}\,\tfrac{\dx}{\pi\,\(1+|x|^2\)^2}\le e^{\frac1{16\,\pi}\int_{\R^2}|\nabla h|^2\,\dx}\quad\mbox{where}\quad\bar h=\int_{\R^2}h(x)\,\tfrac{\dx}{\pi\,\(1+|x|^2\)^2}\,,
\]
that can be recovered by taking the limit in~\eqref{GNS} as $p\to+\infty$ of 
\[
f_p(x):=\mathsf g(x)\(1+\tfrac1{2\,p}\,(h(x)-\bar h)\)\,.
\]

The last remarkable endpoint is the limit as $p\to1$. We obtain the \emph{Euclidean \idx{logarithmic Sobolev inequality} in scale invariant form}
\[\label{Ineq:LogSobEuclideanWeissler}
\frac d2\,\log\(\frac2{\pi\,d\,e}\ird{|\nabla f|^2}\)\ge\ird{|f|^2\,\log|f|^2}
\]
for any function $f\in\mathrm H^1(\R^d,\dx)$ such that $\ird{|f|^2}=1$, which is probably better known in the non-scale invariant form, or \emph{Euclidean \idx{logarithmic Sobolev inequality}},
\[
\ird{|\nabla f|^2}\ge\frac12\ird{|f|^2\,\log\(\tfrac{|f|^2}{\nrm f2^2}\)}+\frac d4\,\log\big(2\,\pi\,e^2\big)\ird{|f|^2}
\]
for any function $f\in\mathrm H^1(\R^d,\dx)\setminus\{0\}$.

%%%%%%%%%%%%%%%%%%%%%%%%%%%%%%%%%%%%%%%%%%%%%%%%%%%%%%%%%%%%%%%%%%%%%%%%
%%%%%%%%%%%%%%%%%%%%%%%%%%%%%%%%%%%%%%%%%%%%%%%%%%%%%%%%%%%%%%%%%%%%%%%%
\section{An elementary proof of the inequalities}

%%%%%%%%%%%%%%%%%%%%%%%%%%%%%%%%%%%%%%%%%%%%%%%%%%%%%%%%%%%%%%%%%%%%%%%%
\subsection{Sobolev's inequality in a nutshell}

Many proofs of \idx{Sobolev's inequality} \eqref{SobolevRd} can be found in the literature. Here is a sketch of a proof which is of particular interest for the general strategy of this memoir. Some details on regularity and the justification of the boundary conditions are omitted, which can be recovered by considering smooth and compactly supported functions and then by arguing by density.

%.......................................................................
\medskip\subsubsection{Schwarz symmetrization}\label{Sec:Schwarz}

The \idx{Schwarz symmetrization} of a nonnegative measurable function $f$ on $\R^d$ such that $|\{f>\mu\}|<\infty$ for all $\mu>0$ is defined as
\[
f^*(x)=\int_0^\infty\mathbbm 1_{\{f>\mu\}^*}(x)\,\mathrm d\mu\,,
\]
where
\[
 \{ f>\mu\}^*=\left\{x\in\R^d;\ \Omega_d|x|^d<|\{f>\mu\}|\right\},
\]
$\Omega_d$ being the volume of the unit ball.
The function $f^*$ is nonincreasing, radial and it has the property that
\[
\ird{{|f^*|}^q}=\ird{|f|^q}
\]
for all $q>0$. The \idx{P\'olya–Szeg\H o principle} asserts that, for any nonnegative measurable function $f$ on $\R^d$ such that $|\{f>\mu\}|<\infty$ for all $\mu>0$, if $\ird{|\nabla f|^2}<\infty$, then
\[
\ird{|\nabla f^*|^2}\le\ird{|\nabla f|^2}\,.
\]
Using $\ird{|\nabla f|^2}=\ird{\big|\nabla|f|\big|^2}$, it is then clear that the proof of \eqref{SobolevRd} can be reduced to nonnegative radial nonincreasing functions on $\R^d$. By density, it is enough to prove that
\[
\int_0^\infty|f'(r)|^2\,r^{d-1}\,\rd r\ge\mathsf S_d^2\,{|\mathbb S^{d-1}|}^{-\frac2d}\(\int_0^\infty|f(r)|^\frac{2\,d}{d-2}\,r^{d-1}\,\rd r\)^\frac{d-2}d
\]
for nonincreasing functions $f\in C^1(0,+\infty)$.

%.......................................................................
\medskip\subsubsection{Equivalent formulations}

With the \idx{Emden-Fowler transformation}
\[
f(r)=r^{-\frac{d-2}2}\,g(\log r)\,,
\]
the problem is reduced to
\be{EF-Var}
\int_{\R}|g'(s)|^2\,\ds+\tfrac{(d-2)^2}4\int_{\R}|g(s)|^2\,\ds\ge\mathsf S_d^2\,{|\mathbb S^{d-1}|}^{-\frac2d}\(\int_{\R}|g(s)|^\frac{2\,d}{d-2}\,\ds\)^\frac{d-2}d
\ee
for a $ C^1$ function $g$ of $s\in(-\infty,+\infty)$ which vanishes at infinity. By homogeneity, there is no restriction in assuming that $\int_{\R}|g(s)|^2\,\ds=1$. With the change of variables and unknown function
\[
z(t)=\int_{-\infty}^t|g(s)|^2\,\ds\quad\mbox{and}\quad h(z(t))=|g(t)|^2\,,
\]
the new variable $z$ is now defined on the interval $[0,1]$ with $h(0)=0$ and $h(1)=0$ and the inequality is reduced to
\be{SobolevInterval}
\int_0^1|h'(z)|^2\,\dz+(d-2)^2\ge4\,\mathsf S_d^2\,{|\mathbb S^{d-1}|}^{-\frac2d}\(\int_0^1|h(z)|^\frac2{d-2}\,\dz\)^{\!\frac{d-2}d}\,.
\ee
This last inequality, with a non \idx{optimal constant}, follows from
\[
0\le h(z)\le\int_0^z|h'(t)|\,\dt\le\sqrt z\(\int_0^1|h'(z)|^2\,\dz\)^{\!\frac12}
\]
so that
\[
\(\int_0^1|h(z)|^\frac2{d-2}\,\dz\)^{\!\frac{d-2}d}\le\big(\tfrac{d-2}{d-1}\big)^{\!\frac{d-2}d}\(\int_0^1|h'(z)|^2\,\dz\)^{\!\frac1d}\,.
\]

%.......................................................................
\medskip\subsubsection{Existence and uniqueness of an optimal function}

The existence of an \idx{optimal function} in~\eqref{SobolevInterval} is a consequence of the compactness in the Arzel\`a-Ascoli theorem (see for instance~\cite[p.~394]{MR1157815}), since any bounded sequence in $H^1_0(0,1)$ is equicontinuous. Coming back to the equivalent formulation~\eqref{EF-Var}, we also know that there is an optimal nonnegative function $g\in C^1\cap\mathrm H^1(\R)$, which solves the Euler-Lagrange equation
\be{ELline}
-\,g''+\tfrac{(d-2)^2}4\,g=g^\frac{d+2}{d-2}
\ee
where, using the homogeneity of~\eqref{EF-Var} again, we assume now that
\[
\mathsf S_d^2\,\big|\mathbb S^{d-1}\big|^{-\frac2d}\(\int_{\R}|g(s)|^\frac{2\,d}{d-2}\,\ds\)^{\!-\frac2d}=1\,.
\]
This problem is of course invariant under translation, but we know from the condition that the $\mathrm H^1(\R)$-norm is finite that $\lim_{s\to\pm\infty}g(s)=0$ and $\lim_{s\to\pm\infty}g'(s)=0$, so that there is a maximum point for some $s_0\in\R$. At such a point, we have $g(s_0)=g_0>0$ and $g'(s_0)=0$. By multiplying~\eqref{ELline} by $g'(s)$ and integrating from~$-\infty$ to $s$ we find that
\[
-|g'(s)|^2+\tfrac{(d-2)^2}4\,|g(s)|^2-\tfrac{d-2}d\,|g(s)|^\frac{2\,d}{d-2}=C\quad\forall\,s\in\R\,.
\]
for some constant $C$. After taking into account the limit as $s\to-\infty$, we learn that $C=0$.
As a consequence, this determines $g(s_0)=g_0$ as the unique positive root, \emph{i.e.},
\[
g_0=\(\tfrac14\,d\,(d-2)\)^\frac{d-2}4\,.
\]
In order to complete the proof, we notice that~\eqref{ELline} has a unique solution with maximum point at $s=s_0$. On the other hand, it is easy to check that for a unique choice of the positive parameters $\mathsf a$ and $\mathsf b$, the function
\[
g(s)=\mathsf a\,\cosh(\mathsf b\,s)^{-\frac2{d-2}}\quad\forall\,s\in\R
\]
solves~\eqref{ELline} with $s_0=0$. This completes this elementary proof of \idx{Sobolev's inequality}~\eqref{SobolevRd} and the computation of the \idx{optimal constant}. The only  non-elementary point of the proof is of course the \idx{P\'olya–Szeg\H o principle}.

The above scheme is interesting in the case of radial functions, but also applies to functions on the half-line in presence of weights. After symmetrization, the fact that $d$ is an integer plays no role. As a consequence, for any real number $n\in(2,+\infty)$, we can write that the inequality (with \idx{optimal constant})
\[\label{RadialSobolev}
\int_0^\infty|f'(r)|^2\,r^{n-1}\,\rd r\ge\frac14\,n\,(n-2)\(\tfrac{\pi\,\Gamma\big(\tfrac n2\big)}{\Gamma\big(\tfrac{n+1}2\big)}\)^\frac2n\(\int_0^\infty|f(r)|^\frac{2\,n}{n-2}\,r^{n-1}\,\rd r\)^\frac{n-2}n
\]
holds for any $f\in C^1(\R^+)$. Moreover, equality is achieved if and only if $f(r)=(1+r^2)^{-(n-2)/2}$, up to a multiplication by a constant and a dilation. 

%%%%%%%%%%%%%%%%%%%%%%%%%%%%%%%%%%%%%%%%%%%%%%%%%%%%%%%%%%%%%%%%%%%%%%%%
\medskip\subsection{Existence of an optimal function in the subcritical range}~

%.......................................................................
\medskip\subsubsection{Non-scale invariant \idx{Gagliardo-Nirenberg-Sobolev inequalities}}\label{Sec:GNscaling}

Inequality~\eqref{GNS} can also be written in non-scale invariant form as
\be{GNS-Intro}
(p-1)^2\,\nrm{\nabla f}2^2+4\,\frac{d-p\,(d-2)}{p+1}\,\nrm f{p+1}^{p+1}-\mathcal K_{\mathrm{GNS}}\,\nrm f{2p}^{2p\gamma}\ge0\quad\forall\,f\in\mathcal H_p(\R^d)
\ee
where
\be{Ch1:gamma}
\gamma=\frac{d+2-p\,(d-2)}{d-p\,(d-4)}\,.
\ee
According to~\cite[Section~4.1]{MR3493423}, the best constant $\mathcal K_{\mathrm{GNS}}$ in~\eqref{GNS-Intro} is related with the \idx{optimal constant} $\mathcal C_{\mathrm{GNS}}$ in~\eqref{GNS} as follows.
%-----------------------------------------------------------------------
\begin{lemma}\label{Lem:GNscaling} Assume that $d\ge1$ is an integer and let $p$ be an admissible exponent. Then~\eqref{GNS-Intro} holds with \idx{optimal constant}
\be{CGN-KGN}
\mathcal K_{\mathrm{GNS}}=C(p,d)\,\mathcal C_{\mathrm{GNS}}^{2\,p\,\gamma}
\ee
where $\gamma$ is given by~\eqref{Ch1:gamma} and $C(p,d)$ is an explicit positive constant given by
\be{Ch1:Cpd}
C(p,d)=c(p,d)\((p-1)^\theta\(4\,\tfrac{d-p\,(d-2)}{p+1}\)^\frac{1-\theta}{p+1}\)^{2\,p\,\gamma}
\ee
where $c(p,d)=\(\frac d2\,\frac{p-1}{d-p\,(d-2)}\)^{2\,\frac{d-p\,(d-2)}{d-p\,(d-4)}}+\(\frac2d\,\frac{d-p\,(d-2)}{p-1}\)^{\frac{d\,(p-1)}{d-p\,(d-4)}}$.\end{lemma}
%-----------------------------------------------------------------------
Let us define the \emph{\idx{deficit functional}}
\be{Deficit}
\delta[f]:=(p-1)^2\,\nrm{\nabla f}2^2+4\,\frac{d-p\,(d-2)}{p+1}\,\nrm f{p+1}^{p+1}-\mathcal K_{\mathrm{GNS}}\,\nrm f{2p}^{2p\gamma}\,.
\ee
\begin{proof} With $a=2-d+d/p$ and $b=d\,(p-1)/(2\,p)$, the optimization with respect to $\lambdasigma>0$ of $h(\lambdasigma):=\lambdasigma^a\,X+\lambdasigma^{-b}\,Y$ shows that
\[
h(\lambdasigma)\ge c(p,d)\,X^\frac b{a+b}\,Y^\frac a{a+b}\quad\mbox{where}\quad c(p,d):=\(\tfrac ba\)^\frac a{a+b}+\(\tfrac ab\)^\frac b{a+b}\,,
\]
with equality if and only if $\lambdasigma=\big(\frac{b\,Y}{a\,X}\big)^{1/(a+b)}$, and we can check that
\[
\tfrac{2\,b}{a+b}+\tfrac{(p+1)\,a}{a+b}=2\,p\,\gamma\,,
\]
where $\gamma$ is defined by~\eqref{Ch1:gamma}. We apply this optimization to $h(\lambdasigma)=\delta[f_\lambdasigma]$, where $f_\lambdasigma$ is given by the scaling
\be{scaling2}
f_\lambdasigma(x)=\lambdasigma^\frac d{2\,p}\,f(\lambdasigma\,x)\quad\forall\,x\in\R^d\,.
\ee
We have $\nrm{f_\lambdasigma}{2p}=\nrm f{2p}$ and, for the optimal choice of $\lambdasigma$, that is,
\[\label{Lminimizing}
\lambdasigma=\(\frac{b\,Y}{a\,X}\)^\frac1{a+b}\quad\mbox{where}\quad X=(p-1)^2\,\nrm{\nabla f}2^2\quad\mbox{and}\quad Y=4\,\tfrac{d-p\,(d-2)}{p+1}\,\nrm f{p+1}^{p+1}\,,
\]
we have the inequality
\begin{multline*}
\frac{(p-1)^2\,\nrm{\nabla f}2^2+4\,\tfrac{d-p\,(d-2)}{p+1}\,\nrm{f}{p+1}^{p+1}}{\nrm f{2p}^{2p\gamma}}\geq \frac{(p-1)^2\,\nrm{\nabla f_\lambdasigma}2^2+4\,\tfrac{d-p\,(d-2)}{p+1}\,\nrm{f_\lambdasigma}{p+1}^{p+1}}{\nrm{f_\lambdasigma}{2p}^{2p\gamma}}\\
=C(p,d)\left(\frac{\nrm{\nabla f}2^\theta\,\nrm{f}{p+1}^{1-\theta}}{{\nrm{f}{2p}}}\right)^{2p\gamma},
\end{multline*}
with $\theta$ as in~\eqref{Ch1:theta} and $C(p,d)$ given by~\eqref{Ch1:Cpd}.
Inequalities~\eqref{GNS} and~\eqref{GNS-Intro} written with \idx{optimal constant}s are therefore equivalent and the proof of~\eqref{CGN-KGN} is completed.\end{proof}

%.......................................................................
\medskip\subsubsection{Restriction to nonnegative functions}

Since $\delta[f]=\delta\big[|f|\big]$ for any function $f\in\mathcal H_p(\R^d)$, we shall from now on assume that all functions are nonnegative unless it is explicitly specified. Stability results for \idx{sign-changing functions} will be discussed in Corollary~\ref{Cor:StabSubCriticalNormalized}. 

%.......................................................................
\medskip\subsubsection{A variational problem}\label{Sec:IM}

The optimality in~\eqref{GNS-Intro} can be reformulated as
\[
\textstyle\mathcal K_{\mathrm{GNS}}=\inf\left\{(p-1)^2\,\nrm{\nabla f}2^2+4\,\frac{d-p\,(d-2)}{p+1}\,\nrm f{p+1}^{p+1}\,:\,f\in\mathcal H_p(\R^d)\,,\quad\nrm f{2p}^{2p}=1\right\}\,.
\]
As a preliminary remark, let us observe that the minimization problem
\[
\textstyle I_M=\inf\left\{(p-1)^2\,\nrm{\nabla f}2^2+4\,\frac{d-p\,(d-2)}{p+1}\,\nrm f{p+1}^{p+1}\,:\,f\in\mathcal H_p(\R^d)\,,\quad\nrm f{2p}^{2p}=M\right\}
\]
enters in the framework of the \idx{concentration-compactness method} for any $M>0$. As a consequence, we have $I_1=\,\mathcal K_{\mathrm{GNS}}$ and
\[
I_M=I_1\,M^\gamma\quad\forall\,M>0
\]
according to Lemma~\ref{Lem:GNscaling} and its proof.
%-----------------------------------------------------------------------
\begin{lemma}\label{Lem:Concavity} Assume that $d\ge1$ is an integer and let $p$ be an admissible exponent. With the above notations, we have
\[
I_{M_1+M_2}<I_{M_1}+I_{M_2}\quad\forall\,M_1,\,M_2>0\,.
\]\end{lemma}
%-----------------------------------------------------------------------
\begin{proof} For an admissible $p$, it is elementary to check that $\gamma$ defined by~\eqref{Ch1:gamma} has a monotone dependance in $p$, with value $1-2/d$ at $p=2^\star$ and limit $1$ as $p$ goes to $1$, which implies that $\gamma\in(0,1)$. Hence Lemma~\ref{Lem:Concavity} follows with $\mathsf z=M_1/(M_1+M_2)$ from the inequality $\mathsf z^\gamma+(1-\mathsf z)^\gamma>1$ for any $\mathsf z\in(0,1)$, which is a consequence of the concavity of $\mathsf z\mapsto \mathsf z^\gamma+(1-\mathsf z)^\gamma$.\end{proof}

%.......................................................................
\medskip\subsubsection{Concentration-compactness and the subcritical inequalities}

Lemma~\ref{Lem:Concavity} shows what is the reason for compactness, but we do not need the full \idx{concentration-compactness method} in the \emph{subcritical range}. Let us simply recall a special case of~\cite[Lemma I.1, p. 231]{MR778974} which is of direct interest for our purpose.
%-----------------------------------------------------------------------
\begin{lemma}\label{Lem:CW} Assume that $d\ge1$ is an integer and let $p$ be an admissible exponent with $p<p^\star$ if $d\ge3$. If $(f_n)_n$ is bounded in $\mathcal H_p(\R^d)$ and if
\[
\limsup_{n\to+\infty}\sup_{y\in\R^d}\int_{B_1(y)}|f_n|^{p+1}\,\dx=0\,,
\]
then $\lim_{n\to\infty}\nrm{f_n}{2p}=0$.\end{lemma}
%-----------------------------------------------------------------------
With this lemma, we are now in position to prove the existence of an \idx{optimal function} for~\eqref{GNS-Intro} \emph{in the subcritical range}.
%-----------------------------------------------------------------------
\begin{theorem}\label{Thm:SubcriticalExistence} Assume that $d\ge1$ is an integer and let $p$ be an admissible exponent with $p<p^\star$ if $d\ge3$. Then there is an \idx{optimal function} $f\in\mathcal H_p(\R^d)$ for~\eqref{GNS} and~\eqref{GNS-Intro}.\end{theorem}
%-----------------------------------------------------------------------
\begin{proof} Let us consider a sequence $(f_n)_{n\in\N}$ of functions in $\mathcal H_p(\R^d)$ such that $\nrm{f_n}{2p}=1$ for any $n\in\N$ and
\[
\lim_{n\to+\infty}\((p-1)^2\,\nrm{\nabla f_n}2^2+4\,\frac{d-p\,(d-2)}{p+1}\,\nrm{f_n}{p+1}^{p+1}\)=\mathcal K_{\mathrm{GNS}}\,.
\]
For any $n\in\N$, we can find $y_n\in\R^d$ such that
\[
\sup_{y\in\R^d}\int_{B_1(y)}|f_n|^{p+1}\,\dx=\int_{|y-y_n|<1}|f_n|^{p+1}\,\dx
\]
and consider the translated functions $f_n(\cdot-y_n)$ which have the same properties as~$f_n$. As a consequence, we can assume that $y_n=0$ for any $n\in\N$, without loss of generality. Because of the boundedness in $\mathcal H_p(\R^d)$, up to the extraction of a subsequence, we know that there exists some function $f\in\mathcal H_p(\R^d)$ such that
\[
f_n\rightharpoonup f\quad\mbox{in}\quad\mathrm L^{p+1}(\R^d)\quad\mbox{and}\quad\nabla f_n\rightharpoonup\nabla f\quad\mbox{in}\quad\mathrm L^2(\R^d)\,,
\]
and also $f_n\to f$ in $\mathrm L^{p+1}_{\rm loc}(\R^d)$ and a.e. By applying Lemma~\ref{Lem:CW}, we know that $f$ is non-trivial because $\nrm{f_n}{2p}=1$ for any $n\in\N$ is incompatible with \hbox{$\displaystyle\lim_{n\to\infty}\nrm{f_n}{2p}=0$}.

By the Brezis-Lieb Lemma (see~\cite[Theorem 1]{MR699419}), we know that
\begin{align*}
&1=\nrm{f_n}{2p}^{2p}=\nrm f{2p}^{2p}+\lim_{n\to\infty}\nrm{f_n-f}{2p}^{2p}\,,\\
&\lim_{n\to\infty}\(\nrm{f_n}{p+1}^{p+1}-\nrm f{p+1}^{p+1}-\nrm{f_n-f}{p+1}^{p+1}\)=0\,,\\
&\lim_{n\to\infty}\(\nrm{\nabla f_n}2^2-\nrm{\nabla f}2^2-\nrm{\nabla f_n-\nabla f}2^2\)=0\,.
\end{align*}
Using~\eqref{GNS-Intro} applied to $f$ and to $f_n-f$, we find that
\[
\mathcal K_{\mathrm{GNS}}\ge\mathcal K_{\mathrm{GNS}}\(\mathsf z^\gamma+(1-\mathsf z)^\gamma\)
\]
with $\mathsf z=\nrm f{2p}^{2p}\in(0,1]$ and $\gamma$ defined by~\eqref{Ch1:gamma}, so that $\mathsf z=1$. This proves that $f$ is an \idx{optimal function} because, by semi-continuity, we already know that
\[
(p-1)^2\,\nrm{\nabla f}2^2+4\,\frac{d-p\,(d-2)}{p+1}\,\nrm f{p+1}^{p+1}\le\mathcal K_{\mathrm{GNS}}\,.
\]
\end{proof}

%%%%%%%%%%%%%%%%%%%%%%%%%%%%%%%%%%%%%%%%%%%%%%%%%%%%%%%%%%%%%%%%%%%%%%%%
%\medskip
\newpage\subsection{Relative entropy, uncertainty principle and functional setting}\label{Sec:UncertaintyPrinciple}

%.......................................................................
~\medskip\subsubsection{Definitions and functional settings}

The \emph{free energy} or \emph{\idx{relative entropy} functional} of a \emph{nonnegative} function $f$ with respect to a nonnegative function $g\in\mathfrak M$ is defined by
\be{RelativeEntropy}
\mathcal E[f|g]:=\frac{2\,p}{1-p}\ird{\(f^{p+1}-g^{p+1}-\tfrac{1+p}{2\,p}\,g^{1-p}\(f^{2p}-g^{2p}\)\)}\,.
\ee
If we assume that, for some function $g\in\mathfrak M$,
\begin{multline}\label{Hyp0}
\ird{f^{2p}}=\ird{g^{2p}}\,,\quad\ird{x\,f^{2p}}=\ird{x\,g^{2p}}\\
\mbox{and}\quad\ird{|x|^2\,f^{2p}}=\ird{|x|^2\,g^{2p}}\,,
\end{multline}
then $\mathcal E[f|g]$ is simplified and can be written as
\be{RelativeEntropySimplified}
\mathcal E[f|g]=\frac{2\,p}{1-p}\ird{\(f^{p+1}-g^{p+1}\)}\,.
\ee
The \emph{\idx{relative Fisher information}} is defined by
\be{RelativeFisher}
\mathcal J[f|g]:=\frac{p+1}{p-1}\ird{\left|(p-1)\,\nabla f+f^p\,\nabla g^{1-p}\right|^2}\,.
\ee
When $g=\mathsf g$, we obtain the inequality
\[
\frac d{p+1}\ird{f^{p+1}}\le\frac{p-1}4\ird{|\nabla f|^2}+\frac1{p-1}\ird{|x|^2\,f^{2p}}
\]
after expanding the square and integrating by parts the cross term. An optimization under scaling proves that
\be{Heisenberg}
\(\frac d{p+1}\ird{f^{p+1}}\)^2\le\ird{|\nabla f|^2}\ird{|x|^2\,f^{2p}}\,.
\ee
This is, for $p>1$, a nonlinear extension of the \emph{\idx{Heisenberg uncertainty principle}} whose standard form corresponds to $p=1$. Such an extension is known, including in the presence of weights, see for instance~\cite{MR2350131}. Because of~\eqref{Heisenberg}, the space
\[
\mathcal W_p(\R^d):=\left\{f\in\mathcal H_p(\R^d)\,:\,\lrangle x\,|f|^p\in\mathrm L^2(\R^d)\right\}
\]
where $\lrangle x:=\sqrt{1+|x|^2}$, is a natural space for \idx{stability} properties in Gagliardo-Nirenberg inequalities when the distance to the manifold of the \idx{Aubin-Talenti functions} is measured by the \idx{relative Fisher information}.

%.......................................................................
\medskip\subsubsection{Basic properties of relative entropies}\label{Sec:EntropyBasic}

Let us recall some known results on relative entropies.
%-----------------------------------------------------------------------
\begin{lemma}\label{Lem:ConvexEntropy} Assume that $d\ge1$ and $p>1$. For any nonnegative $f\in\mathcal W_p(\R^d)$ and any $g\in\mathfrak M$, we have $\mathcal E[f|g]\ge0$ with equality if and only if $f=g$.\end{lemma}
%-----------------------------------------------------------------------
\begin{proof} Let $m=(p+1)/(2\,p)\in(0,1)$ and consider the strictly convex function $\varphi(s)=s^m/(m-1)$. Since
\be{Ephi}
\mathcal E[f|g]=\ird{\Big(\varphi\(f^{2p}\)-\varphi\(g^{2p}\)-\varphi'\(g^{2p}\)\(f^{2p}-g^{2p}\)\Big)}\,,
\ee
the result easily follows.\end{proof}

A standard improvement of Lemma~\ref{Lem:ConvexEntropy} is obtained using a Taylor expansion. By the \emph{\idx{Csisz\'ar-Kullback inequality}}, $\mathcal E[f|g]$ controls the $\mathrm L^1(\R^d)$ distance between $f^{2p}$ and $g^{2p}$. The precise statement goes as follows.
%-----------------------------------------------------------------------
\begin{lemma}\label{Lem:CK} Let $d\ge1$ and $p>1$. Then the inequality 
\[
\nrm{f^{2p}-g^{2p}}1^2\le\frac{8\,p}{p+1}\(\ird{g^{3p-1}}\)\mathcal E[f|g]
\]
holds for any nonnegative functions $f\in\mathcal W_p(\R^d)$, $g\in\mathcal W_p(\R^d)\cap\mathrm L^{3p-1}(\R^d)$ such that $\nrm g{2p}=\nrm f{2p}$.
\end{lemma}
%-----------------------------------------------------------------------
\begin{proof} Let $\varphi$ be as in the proof of Lemma~\ref{Lem:ConvexEntropy} and notice that $\varphi(t)-\varphi(s)-\varphi'(s)\,(t-s)\ge\tfrac m2\,s^{m-2}\,(s-t)^2$ for any $s$ and $t$ such that $0\le t\le s$. Applied with $s=g^{2p}$ and $t=f^{2p}$, we deduce that
\[
\frac14\,\nrm{f^{2p}-g^{2p}}1^2=\(\int_{f\le g}\left|f^{2p}-g^{2p}\right|\,\dx\)^2\le\frac2m\(\ird{g^{2\,p\,(2-m)}}\)\mathcal E[f|g]
\]
by the Cauchy-Schwarz inequality and the conclusion follows from $m=(p+1)/(2\,p)$. See Section~\ref{Appendix:CK} for more details in a special case.\end{proof}

Since $|t-1|^{2p}<\left|t^{2p}-1\right|$ for any $t\in(0,1)\cup(1,+\infty)$, we can notice that
\[
\nrm{f-g}{2p}^{2p}\le\nrm{f^{2p}-g^{2p}}1\le\sqrt{\textstyle\frac{8\,p}{p+1}\(\ird{g^{3p-1}}\)\mathcal E[f|g]}
\]
so that, if $\lim_{n\to+\infty}\mathcal E[f_n|g]=0$, then the sequence $(f_n)_{n\in\N}$ converges to $g$ in $\mathrm L^{2p}(\R^d)$.

\medskip Let $\mathsf g$ be as in~\eqref{Aubin.Talenti}. For any nonnegative $f\in\mathcal W_p(\R^d)\setminus\{0\}$, let us define
\begin{multline}\label{lambdamuy}
\mu[f]:=\frac{\nrm f{2p}^{2p}}{\nrm{\mathsf g}{2p}^{2p}}\,,\quad y[f]:=\frac{\ird{x\,f^{2p}}}{\ird{f^{2p}}}\\
\mbox{and}\quad\lambda[f]:=\sqrt{\tfrac{d\,(p-1)}{d+2-p\,(d-2)}\,\frac{\ird{f^{2p}}}{\ird{\big|x-y[f]\big|^2\,f^{2p}}}}\,.
\end{multline}
The numerical coefficients in the definitions of $\mu[f]$ and $\lambda[f]$ are such that $\mu[\mathsf g]=1$ and $\lambda[\mathsf g]=1$. We recall that $g_{\lambda,\mu,y}(x)=\lambda^\frac d{2p}\,\mu^\frac1{2p}\,\mathsf g\big(\lambda\,(x-y)\big)$ where~$\mathsf g(x)$ is as in~\eqref{Aubin.Talenti} and define
\be{gf}
\mathsf g_f:=g_{\lambda[f],\mu[f],y[f]}\,.
\ee
When there is no ambiguity, we also use the shorter notation $x_f:=y[f]$. For any nonnegative $f\in\mathcal W_p(\R^d)$, we have that 
\[
\ird{\(1,x,|x-x_f|^2\)f^{2p}}=\ird{\(1,x,|x-x_f|^2\)\mathsf g_f^{2p}}\,.
\]
The unique $g\in\mathfrak M$ such that $f$ satisfies~\eqref{Hyp0} is $g=\mathsf g_f$ 
because
\begin{align*}
\ird{|x|^2\,f^{2p}}&=\ird{\(|x-x_f|^2+2\,(x-x_f)\cdot x_f+|x_f|^2\)f^{2p}}\\
&=\ird{|x-x_f|^2\,f^{2p}}+|x_f|^2\ird{f^{2p}}\\
&=\ird{|x-x_f|^2\,\mathsf g_f^{2p}}+|x_f|^2\ird{\mathsf g_f^{2p}}=\ird{|x|^2\,\mathsf g_f^{2p}}
\end{align*}
Optimizing the relative entropy of a given nonnegative function in $f\in\mathcal W_p(\R^d)$ with respect to the set of the \idx{Aubin-Talenti functions} is the main reason for introducing~\eqref{lambdamuy} and~\eqref{gf}.
%-----------------------------------------------------------------------
\begin{lemma}\label{Lem:BestMatchAubinTalenti} Assume that $d\ge1$ and $p>1$. For any nonnegative and non-trivial function $f\in\mathcal W_p(\R^d)$, we have
\[
\mathcal E[f|\mathsf g_f]=\inf_{g\in\mathfrak M}\mathcal E[f|g]\,.
\]
\end{lemma}
%-----------------------------------------------------------------------
In other words, among all functions $g\in\mathfrak M$, $\mathsf g_f$ is the \emph{best matching} \index{Aubin-Talenti functions}{Aubin-Talenti function} in the sense of relative entropy.
\begin{proof} From the two identities
\begin{align*}
&\ird{\mathsf g^{p+1}}=-\frac1d\ird{x\cdot\nabla\mathsf g^{p+1}}=\frac2d\,\frac{p+1}{p-1}\ird{|x|^2\,\mathsf g^{2p}}\,,\\
&\ird{\mathsf g^{p+1}}=\ird{\mathsf g^{1-p}\,\mathsf g^{2p}}=\ird{\(1+|x|^2\)\mathsf g^{2p}}=\nrm{\mathsf g}{2p}^{2p}+\ird{|x|^2\,\mathsf g^{2p}}\,,
\end{align*}
we obtain that
\begin{multline*}
\ird{\mathsf g^{p+1}}=\frac{d\,(p-1)}{d+2-p\,(d-2)}\,\nrm{\mathsf g}{2p}^{2p}\\
\mbox{and}\quad\ird{|x|^2\,\mathsf g^{2p}}=\frac{2\,(p+1)}{d+2-p\,(d-2)}\,\nrm{\mathsf g}{2p}^{2p}\,.
\end{multline*}
For a given nonnegative $f\in\mathcal W_p(\R^d)$, let us consider the relative entropy with respect to $g_{\lambda,\mu,y}$, that is, of
\begin{align*}
\mathcal E[f|g_{\lambda,\mu,y}]=\,&\frac{2\,p}{1-p}\ird{f^{p+1}}+\mu^\frac{p+1}{2\,p}\,\lambda^{-d\,\frac{p-1}{2\,p}}\ird{\mathsf g^{p+1}}\\
&+\mu^\frac{p+1}{2\,p}\,\lambda^{-d\,\frac{p-1}{2\,p}}\,\frac{p+1}{p-1}\ird{\(1+\lambda^{-2}\,|x-y|^2\)\(f^{2p}-\mathsf g^{2p}\)}\,.
\end{align*}
An optimization on $(\lambda,\mu,y)\in(0,+\infty)\times(0,+\infty)\times\R^d$ shows that the minimum of $(\lambda,\mu,y)\mapsto\mathcal E[f|g_{\lambda,\mu,y}]$ is achieved if $(\lambda,\mu,y)=(\lambda[f],\mu[f],y[f])$, which concludes the proof using~\eqref{gf}. \end{proof}

As we shall see later, $\mathsf g_f$ plays an important role in stability results. Since the whole manifold $\mathfrak M$ is generated by mutiplications by a constant, translations and scalings, we can alternatively choose an arbitrary given function in $\mathfrak M$, for simplicity $\mathsf g$, and restrict our study to the functions such that $\mathsf g_f=\mathsf g$, without loss of generality. Such a choice simply amounts to require that $f$ satisfies~\eqref{Hyp0}.

%%%%%%%%%%%%%%%%%%%%%%%%%%%%%%%%%%%%%%%%%%%%%%%%%%%%%%%%%%%%%%%%%%%%%%%%
%\medskip
\newpage\subsection{Optimal functions and optimal constant}~

%.......................................................................
\smallskip\subsubsection{Radial optimal functions}

{}From the Schwarz symmetrization and Theorem~\ref{Thm:SubcriticalExistence}, we know that there is an optimal, nonnegative function $f\in\mathcal W_p(\R^d)$ which is radial and non-increasing, and moreover solves the equation
\be{ELrad}
-2\,(p-1)^2\,\Delta f+4\,\big(d-p\,(d-2)\big)\,f^p-C\,f^{2\,p-1}=0
\ee
for some constant $C>0$ to be determined. Notice that $f=\mathsf g$ solves~\eqref{ELrad} if $C=8\,p$. Let us prove that this is indeed the case. A first step is to establish some decay estimates.
%-----------------------------------------------------------------------
\begin{lemma}\label{Lem:Strauss} Assume that $f\in\mathcal W_p(\R^d)$ is nonnegative and radial. Then for any $x\in\R^d\setminus\{0\}$, we have the estimate
\[
|f(x)|\le\(|\mathbb S^{d-1}|\,|x|^d\)^{-\frac1{p+1}}\,\((p+1)\,\nrm{\nabla f}2\(\ird{|x|^2\,f^{2p}}\)^{1/2}+d\,\nrm f{p+1}^{p+1}\)^{\frac1{p+1}}\,.
\]
\end{lemma}
%-----------------------------------------------------------------------
\begin{proof} This estimate is inspired by the \index{Strauss lemma}{Lemma of W.~Strauss} in~\cite[p.~155]{MR454365}. If we abusively write that $f$ is a function of $r=|x|$, then we can compute
\[
\frac{\rd}{\rd r}\(r^d\,|f(r)|^{p+1}\)=r^\frac{d-1}2\,f'(r)\cdot r^\frac{d+1}2\,|f(r)|^{p-1}\,f(r)+d\,r^{d-1}\,|f(r)|^{p+1}\,.
\]
If $f$ has compact support, an integration on the interval $(r,+\infty)$ and a Cauchy-Schwarz inequality for the first term of the right-hand side gives the desired inequality. A density argument completes the proof.\end{proof}

%.......................................................................
\medskip\subsubsection{A rigidity result}
Let us define the function
\be{Eqn:h-BE}
\P=\frac{p+1}{p-1}\,f^{1-p}\,.
\ee
With these notations, a lengthy but elementary computation which exactly follows the steps of \cite[Sections~3.1 and~3.2]{Dolbeault_2016} shows the following \emph{\idx{rigidity}} identity. We shall come back on a flow interpretation of this computation in Chapter~\ref{Chapter-2}.
%-----------------------------------------------------------------------
\begin{lemma}\label{Lem:Carre} Assume that $d\ge1$ is an integer and let $p$ be an admissible exponent with $p<p^\star$ if $d\ge3$. If $f\in\mathcal W_p(\R^d)$ is an \idx{optimal function} for~\eqref{GNS-Intro} which solves~\eqref{ELrad} and if $f$ is smooth and sufficiently decreasing, then we have
\begin{multline}\label{SquareElliptic}
\big(d-p\,(d-2)\big)\ird{f^{p+1}\,\left|\Delta\P-(p+1)^2\,\frac{\ird{|\nabla f|^2}}{\ird{f^{p+1}}}\right|^2}\\
+2\,d\,p\ird{f^{p+1}\,\left\|\,\mathrm D^2\P-\frac 1d\,\Delta\P\,\mathrm{Id}\,\right\|^2}=0\,.
\end{multline}
\end{lemma}
%-----------------------------------------------------------------------
Here $\|\mathsf m\|^2$ denotes the sum of the square of the elements of the matrix $\mathsf m$ and $\mathrm D^2\P$ denotes the Hessian matrix of $\P$. In view of the general strategy, it is interesting to give an \emph{elliptic proof} of Lemma~\ref{Lem:Carre}, which makes the link with the \emph{\idx{carr\'e du champ}} method. From here on, we shall assume that $f$ is smooth and \emph{sufficiently decreasing} so that we can perform all necessary integrations by parts on $\R^d$ without taking into account asymptotic boundary terms. Here \emph{sufficiently decreasing} means that the function has the same decay as~$\mathsf g$ as well as its derivatives up to order $2$. We already know that these properties hold true for the radial and non-increasing \idx{optimal function} of Lemma~\ref{Lem:Carre}, but the computations that we present below require no \emph{a priori} symmetry.

\begin{proof}[Proof of Lemma~\ref{Lem:Carre}] Let us consider a solution $f$ of~\eqref{ELrad}. If we test this equation by $f$ and by $x\cdot\nabla f$, we find that
\begin{align*}
&2\,(p-1)^2\,\nrm{\nabla f}2^2+4\,\big(d-p\,(d-2)\big)\,\nrm f{p+1}^{p+1}-C\,\nrm f{2p}^{2p}=0\,,\\
&(d-2)\,(p-1)^2\,\nrm{\nabla f}2^2+4\,d\,\frac{d-p\,(d-2)}{p+1}\,\nrm f{p+1}^{p+1}-\frac{d\,C}{2\,p}\,\nrm f{2p}^{2p}=0\,,
\end{align*}
where the second identity is the standard \index{Poho\v zaev identity}{Poho\v zaev} computation. As a consequence, we have
\[
\nrm f{p+1}^{p+1}=\frac{p^2-1}{2\,d}\,\nrm{\nabla f}2^2\,,
\]
\be{Eqn:C8pi}
C\,\nrm f{2p}^{2p}=p\,\frac{d+2-2-p\,(d-2)}{d\,(p-1)}\,\nrm{\nabla f}2^2\,.
\ee
In particular, with this observation,~\eqref{SquareElliptic} takes the form
\begin{multline*}
\big(d-p\,(d-2)\big)\ird{f^{p+1}\,\left|\Delta\P-2\,d\,\frac{p+1}{p-1}\right|^2}\\
+2\,d\,p\ird{f^{p+1}\,\left\|\,\mathrm D^2\P-\frac 1d\,\Delta\P\,\mathrm{Id}\,\right\|^2}=0\,.
\end{multline*}

After these preliminaries, let us rewrite~\eqref{ELrad} as
\[\label{ELrad1}
-\Delta f+2\,\frac{d-p\,(d-2)}{(p-1)^2}\,f^p-\frac C{2\,(p-1)^2}\,f^{2\,p-1}=0\,.
\]
By testing this equation with $-f^{1-p}\,\Delta f$ and $p\,f^{-p}\,|\nabla f|^2$, we obtain
\[
\ird{f^{1-p}\,(\Delta f)^2}+2\,\frac{d-p\,(d-2)}{(p-1)^2}\ird{|\nabla f|^2}=\frac{p\,C}{2\,(p-1)^2}\ird{f^{p-1}\,|\nabla f|^2}
\]
and
\begin{multline*}
-p\ird{f^{-p}\,\Delta f\,|\nabla f|^2}+2\,p\,\frac{d-p\,(d-2)}{(p-1)^2}\ird{|\nabla f|^2}\\
=\frac{p\,C}{2\,(p-1)^2}\ird{f^{p-1}\,|\nabla f|^2}\,,
\end{multline*}
so that, by subtraction
\be{BE-Identity:f}
\ird{f^{1-p}\,(\Delta f)^2}+p\ird{f^{-p}\,\Delta f\,|\nabla f|^2}-\,2\,\frac{d-p\,(d-2)}{p-1}\ird{|\nabla f|^2}=0\,.
\ee
To exploit this identity, we need two elementary formulae.\\
1) Let $\P$ be given by~\eqref{Eqn:h-BE} and expand
\begin{multline*}
\ird{f^{p+1}\,(\Delta\P-\lambda)^2}=(p+1)^2\ird{f^{1-p}\(\Delta f-p\,\frac{|\nabla f|^2}f\)^2}\\
+\lambda^2\ird{f^{p+1}}-\,2\,\lambda\,(p+1)^2\ird{|\nabla f|^2}\,.
\end{multline*}
With the choice
\[
\lambda=(p+1)^2\,\frac{\ird{|\nabla f|^2}}{\ird{f^{p+1}}}=2\,d\,\frac{p+1}{p-1}\,,
\]
we obtain
\begin{multline*}\label{Id:Dirichlet}
\frac1{(p+1)^2}\ird{f^{p+1}\,(\Delta\P-\lambda)^2}\\
=\ird{f^{1-p}\(\Delta f-p\,\frac{|\nabla f|^2}f\)^2}-\,2\,d\,\frac{p+1}{p-1}\ird{|\nabla f|^2}\,.
\end{multline*}
2) Using the elementary identity
\[
\frac12\,\Delta\,|\nabla\P|^2=\left\|\,\mathrm D^2\P-\tfrac1d\,\Delta\P\,\mathrm{Id}\,\right\|^2+\tfrac1d\,(\Delta\P)^2+\nabla\P\cdot\nabla\Delta\P
\]
applied with $\P$ given by~\eqref{Eqn:h-BE}, we obtain after multiplication by $f^{p+1}$ and some integrations by parts that
\begin{multline*}
\frac12\,(p+1)\ird{f^{-p}\(\Delta f+p\,\frac{|\nabla f|^2}f\)|\nabla f|^2}\\
=\frac1{(p+1)^2}\ird{f^{p+1}\,\left\|\,\mathrm D^2\P-\tfrac1d\,\Delta\P\,\mathrm{Id}\,\right\|^2}+\frac1d\ird{f^{1-p}\(\Delta f-p\,\frac{|\nabla f|^2}f\)^2}\\
-\ird{f^{1-p}\(\Delta f+\frac{|\nabla f|^2}f\)\(\Delta f-p\,\frac{|\nabla f|^2}f\)}\,,
\end{multline*}
so that
\begin{multline*}
\frac1{(p+1)^2}\ird{f^{p+1}\,\left\|\,\mathrm D^2\P-\tfrac1d\,\Delta\P\,\mathrm{Id}\,\right\|^2}\\
=\(1-\frac1d\)\ird{f^{1-p}\,(\Delta f)^2}-\frac{p\,(d-4)-3\,d}{2\,d}\ird{f^{-p}\,\Delta f\,|\nabla f|^2}\\
-p\,\frac{d-p\,(d-2)}{2\,d}\ird{f^{-(p+1)}\,|\nabla f|^4}\,.
\end{multline*}
Collecting terms, we find that~\eqref{BE-Identity:f} is equivalent to~\eqref{SquareElliptic} because
\begin{align*}
&\hspace*{-12pt}\frac{d-p\,(d-2)}{d\,(p+1)^3}\ird{f^{p+1}\,\left|\Delta\P-2\,d\,\frac{p+1}{p-1}\right|^2}\\
&\hspace*{1cm}+\frac{2\,d\,p}{d\,(p+1)^3}\ird{f^{p+1}\,\left\|\,\mathrm D^2\P-\frac 1d\,\Delta\P\,\mathrm{Id}\,\right\|^2}\\
&=\ird{f^{1-p}\,(\Delta f)^2}+p\ird{f^{-p}\,\Delta f\,|\nabla f|^2}-\,2\,\frac{d-p\,(d-2)}{p-1}\ird{|\nabla f|^2}\,.
\end{align*}
This concludes the proof of Lemma~\ref{Lem:Carre}.
\end{proof}
%-----------------------------------------------------------------------
\begin{corollary}\label{Cor:Aubin-Talenti function} Assume that $d\ge1$ is an integer and let $p$ be an admissible exponent with $p<p^\star$ if $d\ge3$. If $f\in\mathcal W_p(\R^d)$ is a radial and non-increasing \idx{optimal function} for~\eqref{GNS-Intro} which \index{Aubin-Talenti functions}{solves}~\eqref{ELrad} such that $\nrm f{2p}=\nrm{\mathsf g}{2p}$, then $f=\mathsf g$ and the constant in~\eqref{ELrad} is $C=8\,p$.
\end{corollary}
%-----------------------------------------------------------------------
\begin{proof} Since $p < p^\star$, nonnegative solutions of~\eqref{ELrad} turn out to be bounded and H\"older continuous, hence classical and even $C^\infty$ by standard elliptic bootstrap. By standard tail-analysis of radial solutions, it is possible to show that $f(r)$ decays like $\mathsf g(r) \sim r^{-2/(p-1)}$ as $r\to+\infty$,  and the same property holds for all derivatives.

We can therefore apply Lemma~\ref{Lem:Carre}, and $\Delta\P$ is constant, which means that $\P(r)=\mathsf a\,r^2+\mathsf b$ for some positive constants $\mathsf a$ and $\mathsf b$. A simple algebraic computation taking the constraint $\nrm f{2p}=\nrm{\mathsf g}{2p}$ and~\eqref{ELrad} into account shows that $\mathsf a=\mathsf b=1$ and $C=8\,p$ using~\eqref{Eqn:C8pi}.\end{proof}

At this point, Inequality~\eqref{GNS-Intro} is established with optimal constant $\mathcal K_{\mathrm{GNS}}$ determined by the equality case when $f=\mathsf g$. Using Lemma~\ref{Lem:GNscaling}, this also establishes Inequality~\eqref{GNS} with optimal constant $\mathcal C_{\mathrm{GNS}}$ and the equality case if $f=\mathsf g$. In order to complete the proof of Theorem~\ref{Thm:GNS}, it remains to characterize \emph{all} optimal functions and compute the value of $\mathcal C_{\mathrm{GNS}}$.

%.......................................................................
\medskip\subsubsection{A first \idx{stability} result}

Let us consider the \emph{\idx{deficit functional}} $\delta$ as defined in~\eqref{Deficit}. We recall that $\mathcal E$, $\mathcal J$ and $(\lambda,\mu)$ are defined respectively by~\eqref{RelativeEntropy},~\eqref{RelativeFisher} and~\eqref{lambdamuy}.
%-----------------------------------------------------------------------
\begin{lemma}\label{Lem:BasicEntropyProp} Assume that $d\ge1$ is an integer and let $p$ be an admissible exponent. Then the following properties hold:
\begin{enumerate}
\item[(i)] We have the inequality
\[
\delta[f]\ge0\quad\forall\,f\in\mathcal W_p(\R^d)
\]
with equality if and only if \index{Aubin-Talenti functions}{$f\in\mathfrak M$} and
\be{ATsubmanifold}
\lambda[f]\,\mu[f]^\frac{p-1}{d-p\,(d-4)}=\sqrt{\tfrac{d\,(p-1)}{d+2-p\,(d-2)}}\,.
\ee
\item[(ii)] For any $f\in\mathcal W_p(\R^d)$ and $g\in\mathfrak M$ satisfying~\eqref{ATsubmanifold}, if $\nrm g{2p}=\nrm f{2p}$, then we have the identity
\be{deficit-entropy-production}
\frac{p+1}{p-1}\,\delta[f]=\mathcal J[f|g]-4\,\mathcal E[f|g]\,.
\ee
\item[(iii)] If $d\le2$, or $d\ge3$ and $p<p^\star$, and if $f$ satisfies~\eqref{Hyp0} with $g=\mathsf g$, then
\be{FirstStabilityEstimate}
\delta[f]\ge\,c\,\Big(\mathcal E[f|\mathsf g]\Big)^2
\ee
for some positive constant $c$.
\end{enumerate}
\end{lemma}
%-----------------------------------------------------------------------
\begin{proof} The inequality $\delta[f]\ge0$ is a simple consequence of Corollary~\ref{Cor:Aubin-Talenti function} and of the definitions \eqref{Deficit},~\eqref{RelativeEntropy} and~\eqref{RelativeFisher} of $\delta[f]$, $\mathcal E[f|\mathsf g]$ and $\mathcal J[f|\mathsf g]$. The proof of~(iii) is given in~\cite[Theorem~7]{MR3493423}. For completeness, let us give the main idea. With the notations of the proof of Lemma~\ref{Lem:GNscaling}, the result can be recovered as a consequence of 
\[
\delta[f]=\inf_{\lambda>0}\delta[f_\lambda]+\(\delta[f]-\inf_{\lambda>0}\delta[f_\lambda]\)\ge\delta[f]-\inf_{\lambda>0}\delta[f_\lambda]\ge c\,(\mathcal E[f|\mathsf g])^2\,,
\]
where $f_\lambda$ is as in~\eqref{scaling2} and using~\eqref{RelativeEntropySimplified}. As a consequence of Lemma~\ref{Lem:ConvexEntropy}, we read that $\delta[f]=0$ if and only if $f=\mathsf g$, under the conditions of~(iii). 

Next, let us consider the equality case in~(i). With the change of variables
\[
f(x)=\mu^\frac1{2p}\,\lambda^\frac d{2p}\,\tilde f\big(\lambda\,(x-x_f)\big)
\]
and the specific choice
\[\label{mulambda}
\mu=\mu[f]\quad\mbox{and}\quad\lambda=\mu[f]^{-\frac{p-1}{d-p\,(d-4)}}\,,
\]
which amounts to~\eqref{ATsubmanifold}, we find that
\[
\delta[f]=\mu[f]^{-\gamma}\,\delta\big[\tilde f\,\big]
\]
with $\gamma$ as in~\eqref{Ch1:gamma}. Since $\tilde f$ satisfies~\eqref{Hyp0} with $g=\mathsf g$, $\delta[f]=0$ means that $\tilde f=\mathsf g$, \emph{i.e.}, \hbox{$f=g_{\lambda,\mu,y}$}. This completes the proof of~(i).

Concerning~(ii), taking $g=g_{\lambda,\mu,y}$ with $\nrm g{2p}=\nrm f{2p}$ means $\mu[f]=\mu[g]$ while Condition~\eqref{ATsubmanifold} amounts to
\[
g_{\lambda,\mu,y}(x)=\(\mu[f]^{-\frac{2\,(p-1)}{d-p\,(d-4)}}+|x-y|^2\)^{-\frac1{p-1}}\quad\forall\,x\in\R^d\,.
\]
In the computation of $\mathcal J[f|g]-4\,\mathcal E[f|g]$, it is then clear that the terms involving $|x-y|^2$ cancel while the term with $x-y$ disappears after an integration by parts. This completes the proof.\end{proof}

Lemma~\ref{Lem:BasicEntropyProp} deserves some comments. In~(i), changing~\eqref{GNS} into~\eqref{GNS-Intro} in Lemma~\ref{Lem:GNscaling} determines a scale. The equality case puts this scale in evidence, as shown by~\eqref{ATsubmanifold}. The inequality \hbox{$\delta[f]\ge0$}, \emph{i.e.}, Inequality~\eqref{GNS-Intro}, is a simple consequence of Corollary~\ref{Cor:Aubin-Talenti function} but the optimal functions are identified as a consequence of~(iii). For a given mass $\nrm f{2p}^{2p}$,  only $y\in\R^d$ is a free parameter, which reflects the fact that~\eqref{GNS-Intro} is translation invariant. Inequality~\eqref{FirstStabilityEstimate} is a first stability estimate, which is however weaker than the one we look for, as we aim at removing the square in the right-hand side. Finally, in~(iii), let us notice that we assume: $\mathsf g_f=\mathsf g$, so that $\lambda[f]=1$, $\mu[f]=1$, $y[f]=0$ and $\delta[f]$ as defined in~\eqref{Deficit} can be rewritten using $\delta[\mathsf g]=0$ and $\nrm{\mathsf g}{2p}=\nrm f{2p}$ as
\be{DeltaNorm}
\delta[f]=(p-1)^2\(\nrm{\nabla f}2^2-\nrm{\nabla\mathsf g}2^2\)+4\,\frac{d-p\,(d-2)}{p+1}\(\nrm f{p+1}^{p+1}-\nrm{\mathsf g}{p+1}^{p+1}\)\,.
\ee

%.......................................................................
\medskip\subsubsection{Uniqueness of the optimal functions up to the invariances}

A straightforward but very important consequence of Lemma~\ref{Lem:ConvexEntropy} is the following result, which characterizes \emph{all} \idx{optimal function}s of~\eqref{GNS} \emph{in the subcritical range}.
%-----------------------------------------------------------------------
\begin{corollary}\label{Cor:Uniqueness} Assume that $d\ge1$ is an integer and let $p$ be an admissible exponent with $p<p^\star$ if $d\ge3$. There is equality in~\eqref{GNS} if and only if $f=\mathsf g_f$ and in particular there is a unique nonnegative \idx{optimal function} $f\in\mathcal W_p(\R^d)$ such that~\eqref{Hyp0} holds with $g=\mathsf g$.\end{corollary}
%-----------------------------------------------------------------------
In other words, we have identified all \idx{optimal function}s as the \idx{Aubin-Talenti functions} of $\mathfrak M$, which are obtained from $\mathsf g$ by the multiplications by a real constant, the dilations and the translations, that is, the transformations associated with the natural invariances of~\eqref{GNS}. Up to these transformations, Corollary~\ref{Cor:Uniqueness} is a \emph{\idx{uniqueness} result} of the minimizers. Also notice that the result of Lemma~\ref{Lem:BasicEntropyProp},~(iii), is a \idx{stability} result, with the major drawback that  $\delta[f]$ controls $(\mathcal E[f|\mathsf g])^2$ while we expect that it controls $\mathcal E[f|\mathsf g]$ under the condition that~\hbox{$\mathsf g_f=\mathsf g$}.

%.......................................................................
\medskip\subsubsection{Proof of Theorem \texorpdfstring{\ref{Thm:GNS}}{1.1}}

With $\mathsf g$ given by~\eqref{Aubin.Talenti}, we may notice that
\[
\frac{\nrm{\nabla\mathsf g}2^2}{\nrm{\mathsf g}{2p}^{2p}}=\frac{4\,d}{(d+2-p\,(d-2))\,(p-1)}\quad\mbox{and}\quad\frac{\nrm{\mathsf g}{p+1}^{p+1}}{\nrm{\mathsf g}{2p}^{2p}}=\frac{2\,(p+1)}{d+2-p\,(d-2)}
\]
using integrations by parts, so that after taking into account
\[
|\mathbb S^{d-1}|=\frac{2\,\pi^\frac d2}{\Gamma(d/2)}\quad\mbox{and}\quad\int_0^{+\infty}\(1+r^2\)^{-\frac{2p}{p-1}}\,r^{d-1}\,\rd r=\frac d{p^2-1}\,\mathrm B\(\tfrac d2,\tfrac{p+1}{p-1}-\tfrac d2\)
\]
where $\mathrm B$ is the \emph{Euler Beta function}
\be{Betafct}
\mathrm B(a,b):=\int_0^1t^{a-1}\,(1-t)^{b-1}\,\dt=\frac{\Gamma(a)\,\Gamma(b)}{\Gamma(a+b)}\,,
\ee
so that
\[
\nrm{\mathsf g}{2p}^{2p}=\pi^\frac d2\,\frac{\Gamma\(\tfrac{2\,p}{p-1}-\tfrac d2\)}{\Gamma\(\tfrac{2\,p}{p-1}\)}\,.
\]
Using the fact that $\frac\theta2+\frac\theta{p+1}-\frac1{2p}=\frac\theta d$ and collecting the above estimates, we obtain
\[
\mathcal C_{\mathrm{GNS}}(p)=\frac{\nrm{\nabla\mathsf g}2^\theta\,\nrm{\mathsf g}{p+1}^{1-\theta}}{\nrm{\mathsf g}{2p}}=\tfrac{\big(\tfrac{4\,d}{p-1}\big)^\frac\theta2\(2\,(p+1)\)^\frac{1-\theta}{p+1}}{(d+2-p\,(d-2))^\frac{d-p\,(d-4)}{2\,p\,(d+2-p\,(d-2))}}\(\pi^\frac d2\,\frac{\Gamma\(\tfrac{2\,p}{p-1}-\tfrac d2\)}{\Gamma\(\tfrac{2\,p}{p-1}\)}\)^{\!\frac\theta d}.
\]

%.......................................................................
\subsubsection{An important remark}

The symmetry of the \idx{optimal function}s arises from the \idx{uniqueness}, up to multiplications by a real constant, dilations and translations and the fact that, by Lemma~\ref{Lem:BasicEntropyProp}, (iii), \idx{optimal function}s are minimizers of the entropy $\mathcal E[f|\mathsf g_f]$.

%%%%%%%%%%%%%%%%%%%%%%%%%%%%%%%%%%%%%%%%%%%%%%%%%%%%%%%%%%%%%%%%%%%%%%%%
%%%%%%%%%%%%%%%%%%%%%%%%%%%%%%%%%%%%%%%%%%%%%%%%%%%%%%%%%%%%%%%%%%%%%%%%
\section{Stability results by variational methods}

\emph{Stability} is a step beyond the identification of the \idx{optimal function}s and the \idx{optimal constant} in an inequality. The goal is to introduce the \emph{deficit} as the difference of the two terms of the inequality, in our case $\delta[f]$, and bound it from below by a distance to the set of \idx{optimal function}s, in our case the manifold $\mathfrak M$ of the \idx{Aubin-Talenti functions}.

Such a question has been answered in a celebrated paper of G.~Bianchi and H.~Egnell in the critical case (\idx{Sobolev's inequality}) and more recently in the subcritical case by F.~Seuffert using a clever reformulation of~\eqref{GNS} which appeared in the book of D.~Bakry, I.~Gentil and M.~Ledoux. This is a very interesting strategy and we will give a few steps of the proof here, although the complete proof is rather difficult and technical.

In Section~\ref{Sec:Abstract}, we state and prove a new \idx{stability} result by a direct variational method. Like all results of this chapter, we obtain a \emph{quantitative} result, in the sense that we prove the existence of a positive constant and obtain a standard notion of distance, but the method is not \emph{constructive} because the value of the constant is unknown. The purpose of the next chapters is to remedy this issue.

%%%%%%%%%%%%%%%%%%%%%%%%%%%%%%%%%%%%%%%%%%%%%%%%%%%%%%%%%%%%%%%%%%%%%%%%
\medskip\subsection{Stability results based on the Bianchi-Egnell result}~

After briefly recalling the result of~\cite{MR1124290} in the critical case, we explain with some details how to reformulate the subcritical inequality~\eqref{GNS} as a critical case using the strategy of~\cite{MR3155209} and state the result of~\cite{MR3695890}. This section is given only for providing a complete picture of the \idx{stability} \index{Bianchi-Egnell result}{results} associated with~\eqref{GNS} and can be skipped at first reading.

%.......................................................................
\medskip\subsubsection{The Bianchi-Egnell result in the critical case}

If $p=p^\star$, $d\ge3$, G.~Bian\-chi and H.~Egnell proved in~\cite{MR1124290} the \index{Bianchi-Egnell result}{existence of a positive constant} $\mathcal C$ such that
\be{Bianchi-Egnell}
\nrm{\nabla f}2^2-\mathsf S_d^2\,\nrm f{2^*}^2\ge\mathcal C\,\inf\nrm{\nabla f-\nabla g}2^2\,,
\ee
where $\mathsf S_d^2$ is the \idx{optimal constant} in \idx{Sobolev's inequality} and the infimum is taken over the $(d+2)$-dimensional manifold of the \idx{Aubin-Talenti functions} $g_{\lambda,\mu,y}$. This result was immediately recognized as a major breakthrough, with the irritating drawback that the constant $\mathcal C$ was unknown, because the existence of $\mathcal C$ is obtained by \idx{concentration-compactness method}s and a contradiction argument. It is only recently that the proof has been made constructive, in~\cite{DEFFL2023}.
%.......................................................................
\medskip\subsubsection{Equivalence with a critical \idx{Sobolev's inequality} in higher dimension}\label{Sec:BGLtransformation}

The subcritical \idx{Gagliardo-Nirenberg-Sobolev inequalities}~\eqref{GNS} can be reinterpreted as critical inequalities.

Let us start by collecting some useful identities. The reader interested in further details is invited to refer to~\cite{AS}. Let us consider the function $(\nu,\beta)\mapsto\mathrm I(\nu,\beta)$ defined by
\[
\mathrm I(\nu,\beta):=\int_0^\infty\frac{z^{2\,\nu-1}}{\(1+z^2\)^\beta}\,\dz=\tfrac12\,\mathrm B(\beta-\nu,\nu)
\]
for any $\nu<\beta$, where $\mathrm B$ is the Beta function defined in~\eqref{Betafct}. If we write that $\(1+z^2\)^{-\beta}=\(1+z^2\)\(1+z^2\)^{-(\beta+1)}$, we find that
\[
\mathrm I(\nu,\beta)=\mathrm I(\nu,\beta+1)+\mathrm I(\nu+1,\beta+1)\,.
\]
On the other hand, if we write $z\(1+z^2\)^{-(\beta+1)}=-\,\frac1{2\,\beta}\,\frac d{\dz}\(1+z^2\)^{-\beta}$, using one integration by parts we find that
\[\label{A1}
\mathrm I(\nu+1,\beta+1)=\frac\nu\beta\,\mathrm I(\nu,\beta)\,.
\]
This also shows that
\[\label{A2}
\mathrm I(\nu,\beta+1)=\frac{\beta-\nu}\beta\,\mathrm I(\nu,\beta)\,.
\]
{}From these two identities, we get that
\be{A3}
\mathrm I(\nu+1,\beta)=\frac\beta{\beta-\nu-1}\,\mathrm I(\nu+1,\beta+1)=\frac\nu{\beta-\nu-1}\,\mathrm I(\nu,\beta)\,.
\ee

\bigskip The following results are inspired from~\cite[Section~6.10]{MR3155209}. The goal is to relate the \idx{Gagliardo-Nirenberg-Sobolev inequalities}~\eqref{GNS-Intro} with exponent $p$ to a \index{Sobolev's inequality}{Sobolev type inequality} in \index{weights and non-integer dimensions}{\emph{non-integer higher} dimension} $N$, by taking advantage of the fact that the exponent $(N-2)/2$ of the \idx{Aubin-Talenti functions} is increasing with~$N$ so that there is a chance that
\[
\mathsf g(x)=\(1+|x|^2\)^{-\frac1{p-1}}\,,
\]
which is optimal for~\eqref{GNS-Intro} in the subcritical range corresponding to an exponent $p<p^\star$, could be seen as an \idx{optimal function} for a critical \index{Sobolev's inequality}{Sobolev inequality} in a higher dimension, \emph{i.e.}, as one of the \idx{Aubin-Talenti functions} in the usual sense. Since $p$ is a real number, we cannot expect that this \emph{dimension} is an integer, except for an at most countable number of values of $p$. However, the notion of dimension appears only for scaling properties, so that it is possible to introduce weights in order to account for \index{weights and non-integer dimensions}{non-integer values}.

On $\R^{d+1}_+:=\R^d\times\R^+\ni(x,z)$, we consider the measure $\mathrm d\mu:=z^{2\,\nu-1}\,\dx\,\dz$ for some $\nu>0$ and functions $w$ which take the form
\be{PartialSymmetry}
w(x,z)=\(h(x)+z^2\)^{1-\beta}\quad\forall\,(x,z)\in\R^d\times\R^+
\ee
for some nonnegative function $h$ defined on $\R^d$ and some $\beta>1$. Let us introduce the notations
\[
\DD w:=\(\nabla w,\tfrac{\partial w}{\partial z}\)\quad\mbox{and}\quad\nrmm w{q,N}:=\(\iint_{\R^{d+1}}|w|^q\,z^{N-d-1}\,\dx\,\dz\)^{1/q}
\]
with $N:=d+2\,\nu$. When $N>d$ is an integer, the function $w$ can be considered as a function of $(x,y)\in\R^d\times\R^{N-d}$ which is radially symmetric in $y$ and such that $z=|y|$, which corresponds to functions with \emph{cylindrical symmetry}, but we shall also consider the case $N\in\R\setminus\N$ of a \index{weights and non-integer dimensions}{\emph{non-integer dimensions}}. A simple change of variables shows that
\[
\nrmm w{q,N}^q=\izx{w^q}=\mathrm I\(\nu,(\beta-1)\,q\)\,\irdsph r{h^{\nu-(\beta-1)\,q}}\,,
\]
where we use spherical coordinates $(r,\omega)\in\R^+\times\mathbb S^{d-1}$ with $r=|x|$ and $\omega=x/r$ for any $r>0$. Since
\[
|\DD w|^2=|\nabla w|^2+\left|\tfrac{\partial w}{\partial z}\right|^2=(1-\beta)^2\(|\nabla h|^2+4\,z^2\)\(h(x)+z^2\)^{-2\beta}\,,
\]
we also obtain that
\begin{multline*}
\nrmm{\DD w}{2,N}^2=\izx{\(|\nabla w|^2+\left|\tfrac{\partial w}{\partial z}\right|^2\)}\\
=(1-\beta)^2\,\mathrm I(\nu,2\,\beta)\,\irdsph r{|\nabla h|^2\,h^{\nu-2\,\beta}}\\
+4\,(1-\beta)^2\,\mathrm I(\nu+1,2\,\beta)\,\irdsph r{h^{\nu-2\,\beta+1}}\,.
\end{multline*}
According to~\eqref{A3} the identity $\mathrm I(\nu+1,2\,\beta)=\frac\nu{2\,\beta-\nu-1}\,\mathrm I(\nu,2\,\beta)$ holds, so that we can write
\begin{multline*}
\nrmm{\DD w}{2,N}^2=(1-\beta)^2\,\mathrm I(\nu,2\,\beta)\(\irdsph r{|\nabla h|^2\,h^{\nu-2\,\beta}}\right.\\
\left.+\frac{4\,\nu}{2\,\beta-\nu-1}\,\irdsph r{h^{\nu-2\,\beta+1}}\)\,.
\end{multline*}
We identify the various terms of the change of variables by imposing that
\[
f=h^{\frac\nu2-\beta+1}\,,\quad f^{p+1}=h^{\nu-2\,\beta+1}\quad\mbox{and}\quad f^{2p}=h^{\nu-(\beta-1)\,q}\,.
\]
This provides two equations,
\[
(p+1)\(\frac\nu2-\beta+1\)=\nu-2\,\beta+1\quad\mbox{and}\quad2\,p\(\frac\nu2-\beta+1\)=\nu-(\beta-1)\,q\,,
\]
which can be combined with
\[
N=d+2\,\nu\quad\mbox{and}\quad q=\frac{2\,N}{N-2}
\]
to provide the relations
\[
N=d+2\,\nu\,,\quad\nu=\frac{d-p\,(d-2)}{p-1}\,,\quad q=\frac{2\,N}{N-2}\quad\mbox{and}\quad\beta=\frac N2\,.
\]
As a consequence we have the identities $(\beta-1)\,q=N=2\,\beta$, $(\nu-2\,\beta+1)=\frac{p+1}{2\,p}\,(\nu-(\beta-1)\,q)=-\,\frac{p+1}{p-1}$ and $\(\frac\nu2-\beta+1\)=\frac1{2\,p}\,(\nu-(\beta-1)\,q)=-\,\frac1{p-1}$. Hence, if we let
\[\label{PartialSymmetry2}
f=h^{\frac\nu2-\beta+1}=h^{\frac{\nu-2\,\beta+1}{p+1}}=h^{\frac{\nu-(\beta-1)\,q}{2\,p}}\,,
\]
then
\[
\nrmm w{q,N}^q=\mathrm I(\nu,N)\,\nrm f{2p}^{2p}
\]
and
\[
\nrmm{\DD w}{2,N}^2=\(\tfrac{N-2}2\)^2\,\mathrm I(\nu,N)\((p-1)^2\,\nrm{\nabla f}{2}^2+\frac{4\,\nu}{2\,\beta-\nu-1}\,\nrm f{p+1}^{p+1}\)\,.
\]
Upon noting that
\[
\frac{4\,\nu}{2\,\beta-\nu-1}=4\,\frac{d-p\,(d-2)}{p+1}\,,
\]
we obtain that~\eqref{GNS-Intro} can be rewritten as
\[
\frac{\nrmm{\DD w}{2,N}^2}{\(\tfrac{N-2}2\)^2\,\mathrm I(\nu,N)}\ge\mathcal K_{\rm GNS}(p,d)\,\(\frac{\nrmm w{q,N}^q}{\mathrm I(\nu,N)}\)^\frac2q\quad\mbox{with}\quad q=\frac{2\,N}{N-2}
\]
for all $w$ as in~\eqref{PartialSymmetry}, because $\frac 2q=\frac{d+2-p\,(d-2)}{d-p\,(d-4)}=\gamma$ with $\gamma$ as in~\eqref{Ch1:gamma}. Here we recognize a \index{Sobolev's inequality}{Sobolev inequality} on the space of dimension $N$, which suggests to consider the radial \index{Aubin-Talenti functions}{Aubin-Talenti function}
\[
w_\star(x,z):=\(1+|x|^2+z^2\)^{-\frac{N-2}2}\quad\forall\,(x,z)\in\R^{d+1}_+\,.
\]
With $p=\frac{N+d}{N+d-4}$, let us define
\[
\mathsf S_{d,N}:=\frac{2^{2+\frac2N}\,\mathrm I(\nu,N)^{-\frac{N+2}N}}{(N-2)^2\,\mathcal K_{\rm GNS}(p,d)}\,.
\]
%-----------------------------------------------------------------------
\begin{theorem}\label{Thm:Sob}\cite{MR3155209} Assume that $d\ge1$ is an integer, $N>\min\{2,d\}$ is a real number. Then the inequality
\be{SobBGL}
\(\irdmu{w^\frac{2\,N}{N-2}}\)^{1-\frac2N}\le\mathsf S_{d,N}\,\irdmu{|\DD w|^2}
\ee
holds for any smooth compactly supported function $w$ on $\R^d$, with \idx{optimal constant}~$\mathsf S_{d,N}$ and equality
is achieved, up to a multiplication by a constant, a translation in the $s$ variable and a dilation, if and only if $w=w_\star$.\end{theorem}
%-----------------------------------------------------------------------

%.......................................................................
\medskip\subsubsection{A \idx{stability} result in the sub-critical case}

A result of \idx{stability} \emph{\`a la} \index{Bianchi-Egnell result}{Bianchi-Egnell} for~\eqref{SobBGL} can be found in~\cite{seuffert2016stability}. The \idx{stability} of the manifold~$\mathfrak M$ of \idx{optimal function}s for~\eqref{GNS}, for any $p>1$ if $d=1$ or $d=2$, or for any $p\in(1,p^\star)$ if $d\ge3$ then follows from the computations of Section~\ref{Sec:BGLtransformation}. It has been proved in~\cite[Theorem~1.10]{MR3695890} that there is a constant $\kappa>0$ such that
\[
\nrm{\nabla f}2^\theta\,\nrm f{p+1}^{1-\theta}-\mathcal C_{\mathrm{GNS}}(p)\,\nrm f{2p}\ge\kappa\inf_g\nrm{f^{2p}-g^{2p}}1\quad\forall\,f\in\mathcal H_p(\R^d)\,,
\]
where the infimum is taken on the submanifold of the functions $g\in\mathfrak M$ such that $\nrm f{2p}=\nrm g{2p}$. Estimates with other distances have been obtained in~\cite{MR3989143}. No estimate of the constant $\kappa$ can be given in this approach.

Our next purpose is to give a direct proof of the \idx{stability} in~\eqref{GNS} without relying on Theorem~\ref{Thm:Sob}, still by variational methods and without estimating the constant. We shall also use the \emph{\idx{relative entropy}} $\mathcal E[f|g]$ defined by~\eqref{RelativeEntropy} as a measure of the distance to~$\mathfrak M$, in preparation for the next chapters which are devoted to constructive estimates of the constant $\kappa$.

%%%%%%%%%%%%%%%%%%%%%%%%%%%%%%%%%%%%%%%%%%%%%%%%%%%%%%%%%%%%%%%%%%%%%%%%
\subsection{A new abstract \idx{stability} result}\label{Sec:Abstract}

After taking into account Theorem~\ref{Thm:GNS} and the preliminary \idx{stability} results of Lemma~\ref{Lem:BasicEntropyProp}, we can state a result in the subcritical case.
%-----------------------------------------------------------------------
\begin{theorem}\label{Thm:StabSubCriticalNormalized} If $d\le2$, or $d\ge3$ and $p<p^\star$, and if $f\in\mathcal W_p(\R^d)$ is a nonnegative function which satisfies~\eqref{Hyp0} with $g=\mathsf g$, that is, if
\[
\ird{\(1,x,|x|^2\)f^{2p}}=\ird{\(1,x,|x|^2\)\mathsf g^{2p}}\,,
\]
then
\[
\delta[f]\ge\mathcal C\,\mathcal E[f|\mathsf g]
\]
for some positive constant $\mathcal C$.
\end{theorem}
%-----------------------------------------------------------------------
For the convenience of the reader, we split the proof of this result in several cases and observations which are detailed below. The general strategy is based on the \idx{concentration-compactness method} as in~\cite{MR1124290} and we argue by contradiction, so that, in this chapter, we obtain no estimate on $\mathcal C$. The exponent in Theorem~\ref{Thm:StabSubCriticalNormalized} is optimal, as shown by an expansion around the Aubin-Talenti functions. Combined with Lemma~\ref{Lem:CK} (Csisz\'ar-Kullback inequality), Theorem~\ref{Thm:StabSubCriticalNormalized} gives a stability result in terms of $\inf_{g\in\mathfrak M}\|f^{2p}-g^{2p}\|_{\mathrm L^1(\R^d)}^2$ which has been used for instance in~\cite{Carlen2013,Dolbeault2013917}. Stronger notions of stability will be given later in Section~\ref{Comparison:Bianchi-Egnell} (Corollary~\ref{Cor:StabSubCriticalNormalized}) as well as their counterparts for Gagliardo-Nirenberg-Sobolev inequalities (Theorems~\ref{Thm:stabilityDraft2bis} and~\ref{Thm:Main}). The assumption on the boundedness of the second moment is further discussed in Section~\ref{ssec:secondmomentdiscussion}.

%.......................................................................
\medskip\subsubsection{A variational approach: minimizing sequences and possible limits}\label{Sec:Abstract1}

Let us consider a minimizing sequence $(f_n)_{n\in\N}\in\mathcal W_p(\R^d)$ for $f\mapsto\delta[f]/\mathcal E[f|\mathsf g]$, for which we assume that $g_{f_n}=\mathsf g$ holds for any $n\in\N$. It follows from Lemma~\ref{Lem:BasicEntropyProp},~(iii), that $(\mathcal E[f_n|g])_{n\in\N}$ is bounded and, up to the extraction of a subsequence, has a nonnegative limit $\ell=\lim_{n\to\infty}\mathcal E[f_n|g]$.

If $\ell>0$, we have thanks to Lemma~\ref{Lem:BasicEntropyProp}, that
\[
\liminf_{n\to+\infty}\frac{\delta[f_n]}{\mathcal E[f_n|\mathsf g]}\geq c\,\ell>0.
\]
The difficult case is $\ell=0$, which means that we have simultaneously
\[
\lim_{n\to+\infty}\mathcal E[f_n|\mathsf g]=0\quad\mbox{and}\quad\lim_{n\to+\infty}\delta[f_n]=0\,,
\]
so that we know that $f_n$ converges as $n\to+\infty$ to $\mathsf g$ in $\mathrm L^{p+1}(\R^d)\cap\mathrm L^{2p}(\R^d,\lrangle x^2\,\dx)$ and also $\nabla f_n\to\nabla\mathsf g$ in $\mathrm L^2(\R^d)$. Let us write
\be{minimizingstab}
f_n=\mathsf g+\eta_n\,h_n\quad\mbox{with}\quad\eta_n:=\nrm{\nabla f_n-\nabla \mathsf g}2+\(\ird{|f_n-\mathsf g|^{2p}\,\lrangle x^2}\)^\frac1{2p}\,.
\ee
Collecting what we already know, we obtain
\[
\lim_{n\to+\infty}\eta_n=0
\]
while $(h_n)_{n\in\N}$ is bounded in $\mathcal W_p(\R^d)$, so that there is a function $h\in\mathcal W_p(\R^d)$ such that, up to the extraction of a subsequence,
\begin{multline*}
h_n\rightharpoonup h\quad\mbox{in}\quad\mathrm L^{p+1}(\R^d)\,,\quad h_n\rightharpoonup h\quad\mbox{in}\quad\mathrm L^{2p}(\R^d,\lrangle x^2\,\dx)\\
\mbox{and}\quad\nabla h_n\rightharpoonup\nabla h\quad\mbox{in}\quad\mathrm L^2(\R^d)\,,
\end{multline*}
and also $h_n\to h$ in $\mathrm L^{p+1}_{\rm loc}(\R^d)$ and a.e. We also know from~\eqref{Hyp0} that
\begin{multline*}
\ird{|\mathsf g+\eta_n\,h_n|^{2p}}=\ird{|\mathsf g|^{2p}}\,,\quad\ird{x\,|\mathsf g+\eta_n\,h_n|^{2p}}=\ird{x\,|\mathsf g|^{2p}}\\
\mbox{and}\quad\ird{|x|^2\,|\mathsf g+\eta_n\,h_n|^{2p}}=\ird{|x|^2\,|\mathsf g|^{2p}}\,.
\end{multline*}
Our goal is now to compute a lower bound for $\delta[f_n]/\mathcal E[f_n|\mathsf g]$ by linearizing around~$\mathsf g$.
%.......................................................................
\medskip\subsubsection{The linearized problem}~

The analysis of the \idx{stability} relies on a proper linearization of the problem. For sake of completenesss, let us collect here the results (without proofs) that are needed and for which we claim no originality. The following results are taken mostly from~\cite{Dolbeault2011a}, with minor adaptations.

On~$\R^d$, let us consider the measure $\mathrm d\mu_\alphaa:=\mu_\alphaa\,\dx$, where $\mu_\alphaa(x):= (1+|x|^2)^{\alphaa}$, and the operator $\mathcal L_{\alphaa,d}$ on $\mathrm L^2(\R^d,\mathrm d\mu_{\alphaa-1})$ such that
\[
\mathcal L_{\alphaa,d}\,u=-\,\mu_{1-\alphaa}\,\mathrm{div}\left[\,\mu_\alphaa\,\nabla u\,\right]\,.
\]
The fundamental property of $\mathcal L_{\alphaa,d}$ is $\int_{\R^d}u\,(\mathcal L_{\alphaa,d}\,u)\,\mathrm d\mu_{\alphaa-1}=\int_{\R^d}|\nabla u|^2\,\mathrm d\mu_\alphaa$. The measure $\mathrm d\mu_\alphaa$ is a bounded positive measure if $\alphaa\in(-\infty,-d/2)$ and $\int_{\R^d}|x|^2\,\mathrm d\mu_\alphaa$ is also finite if $\alphaa<-(d+2)/2$.
%-----------------------------------------------------------------------
\begin{proposition}[\cite{Dolbeault2011a}]\label{Prop:Spectrum} The \idx{essential spectrum} of $\mathcal L_{\alphaa,d}$ is $[\Lambda_{\mathrm{ess}},+\infty)$ where
\[
\Lambda_{\mathrm{ess}}:=\(\alphaa+\tfrac12\,(d-2)\)^2\,.
\]
The operator $\mathcal L_{\alphaa,d}$ has \idx{discrete spectrum} only for $\alphaa<-(d-2)/2$. For $d\ge 2$, it is made of the eigenvalues
\[
\lambda_{\ell k}=-\,2\,\alphaa\,\(\ell+2\,k\)-4\,k\(k+\ell+\tfrac d2-1\)
\]
with $\ell$, $k=0$, $1$, \ldots provided $(\ell,k)\neq(0,0)$ and $\ell+2\,k-1<-(d+2\,\alphaa)/2$. If $d=1$,
the \idx{discrete spectrum} is made of $\lambda_k=k\,(1-2\,\alphaa-k)$ with $k\in{\mathbb N}\cap[1,1/2-\alphaa]$. The eigenspaces corresponding to $\lambda_{0,0}$, $\lambda_{1,0}$ and $\lambda_{0,1}$ are generated respectively by $x\mapsto1$, $x\mapsto x_i$ with $i=1$, $2$,\ldots $d$, and $x\mapsto|x|^2+d/(2\,\alphaa)$.\end{proposition}
%-----------------------------------------------------------------------

Let us consider the case $\alphaa=2\,p/(1-p)$ for some admissible $p$. In that case, the spectral gap inequality associated with $\mathcal L_{\alphaa,d}$ is the \emph{\idx{Hardy-Poincar\'e inequality}}
\be{Hardy-Poincare}
\ird{|\nabla u|^2\,\mathsf g^{2p}}\ge\Lambda\ird{|u|^2\,\mathsf g^{3\,p-1}}\,,
\ee
which holds for any function $u$ such that $\ird{u\,\mathsf g^{3\,p-1}}=0$, with \idx{optimal constant} $\Lambda=4\,p/(p-1)$, for any $d\ge1$.

Under the constraint that
\be{TaylorConstraint}
\ird{\(1,x,|x|^2\)u(x)\,\mathsf g(x)^{3\,p-1}}=0\,,
\ee
we obtain the \emph{improved \idx{Hardy-Poincar\'e inequality}}
\be{Improved-Hardy-Poincare}
\ird{|\nabla u|^2\,\mathsf g^{2p}}\ge\Lambda_\star\ird{|u|^2\,\mathsf g^{3\,p-1}}
\ee
where the \idx{optimal constant} $\Lambda_\star>4\,p/(p-1)$ depends on the dimension:
\begin{enumerate}
\item[(i)] if $d=1$ and $1<p\le1+\frac 2d$, then $\Lambda_\star=\frac{6\,p+6}{p-1}$,
\item[(ii)] if either $d=2$, or $d\ge3$ and $1<p\le1+\frac 2d$, then $\Lambda_\star=\frac{8\,p}{p-1}$,
\item[(iii)] if $d\ge3$ and $1+\frac 2d\le p\le\min\left\{1+\frac4{d+2},p^\star\right\}$, then $\Lambda_\star=\frac{16\,p}{p-1}-4\,(d+2)$,
\item[(iv)] if $3\le d\le 5$ and $1+\frac4{d+2}<p\le p^\star$, then $\Lambda=\big(\frac{d-2}2-\frac{2\,p}{p+1}\big)^2$.
\end{enumerate}
In dimension $d=1$, we have $\Lambda_\star=\lambda_3$, while for $d=2$, $\Lambda_\star$ is determined by $\min\{\lambda_{0,2},\lambda_{2,0},\Lambda_{\mathrm{ess}}\}$. See~\cite[p.~710]{Dolbeault2011a} for details.

\medskip The \emph{\idx{Hardy-Poincar\'e inequality}}~\eqref{Hardy-Poincare} can be rewritten for $h=u\,\mathsf g^p$ as
\[
\ird{\left|\nabla h+\frac{2\,p}{p-1}\,\frac{x\,h}{1+|x|^2}\right|^2}\ge\Lambda\ird{\frac{|h|^2}{1+|x|^2}}\,.
\]
By expanding the square, we obtain the inequality
\begin{multline*}\label{HardyPoicareExpanded}
\ird{|\nabla h|^2}+\Big(2\,p\,\frac{d-2-p\,(d-4)}{(p-1)^2}-\Lambda\Big)\ird{\frac{|h|^2}{1+|x|^2}}\\
\ge4\,p\,\frac{2\,p-1}{(p-1)^2}\ird{\frac{|h|^2}{\(1+|x|^2\)^2}}\,.
\end{multline*}
Under the condition $\ird{\frac h{1+|x|^2}}=0$, the \idx{optimal constant} in~\eqref{Hardy-Poincare} is $\Lambda=\frac{4\,p}{p-1}$, and it is straightforward to check that
\[
2\,p\,\frac{d-p\,(d-2)}{(p-1)^2}=2\,p\,\frac{d-2-p\,(d-4)}{(p-1)^2}-\Lambda\,.
\]
As a consequence, Inequality~\eqref{Hardy-Poincare} means that the quadratic form
\be{Q}
\mathsf Q[h]:=\!\ird{\!|\nabla h|^2}+2\,\frac{d-p\,(d-2)}{(p-1)^2}\!\ird{\!\frac{|h|^2}{1+|x|^2}}-4\,p\,\frac{2\,p-1}{(p-1)^2}\!\ird{\!\frac{|h|^2}{\(1+|x|^2\)^2}}
\ee
is nonnegative under the condition $\ird{\frac h{1+|x|^2}}=0$. Furthermore, under Condition~\eqref{TaylorConstraint}, that is,
\be{TaylorConstraint:h}
\ird{\frac{\(1,x,|x|^2\)h(x)}{1+|x|^2}}=0\,,
\ee
the \emph{improved \idx{Hardy-Poincar\'e inequality}} amounts to
\be{Ineq:Q}
\mathsf Q[h]\ge(\Lambda_\star-\Lambda)\ird{\frac{|h|^2}{1+|x|^2}}
\ee

\medskip Concerning the \idx{essential spectrum} of $\mathcal L_{\alpha,d}$, according to a characterization due to A.~Persson, we have
\[
\Lambda_{\mathrm{ess}}=\lim_{R\to+\infty}\inf_u\frac{\ird{|\nabla u|^2\,\mathsf g^{2p}}}{\ird{|u|^2\,\mathsf g^{3\,p-1}}}
\]
where the infimum is taken on all functions $u$ which are smooth with compact support in $\R^d\setminus B_R$. Because $1+|x|^2$ can be replaced by $|x|^2$ in the expression of $\mathsf g$ in the above Rayleigh quotient, after a dilation,
\be{LambdaEss}
\Lambda_{\mathrm{ess}}=\frac14\(d-2-\frac{4\,p}{p+1}\)^2
\ee
is the \idx{optimal constant} in the \emph{weighted Hardy inequality}
\be{WeightedHardy}
\ird{|\nabla u|^2\,|x|^{-\frac{4\,p}{p-1}}}\ge\Lambda_{\mathrm{ess}}\ird{|u|^2\,|x|^{-\frac{4\,p}{p-1}-2}}\,.
\ee
The proof of such an inequality is easy, as it is enough to expand the square and perform one integration by part in the trivial inequality
\[
\ird{\left|\nabla u+\sqrt{\Lambda_{\mathrm{ess}}}\,\frac x{|x|^2}\,u\right|^2}\ge0
\]
for a smooth compactly supported function with support in $\R^d\setminus\{0\}$, and argue by density. For optimality, it is enough to construct a smooth approximation of the function $x\mapsto|x|^{-\sqrt{\Lambda_{\mathrm{ess}}}}$.

As a consequence, we can also write that~\eqref{Ineq:Q} holds for $h=u\,\mathsf g^p$ with $\Lambda_*=\Lambda+\frac12\(\Lambda_{\mathrm{ess}}-\Lambda\)$ if $h$ is supported in $B_{R_\star}^c$ for some $R_\star$ large enough.\medskip

%.......................................................................
\subsubsection{A heuristic point of view}\label{Sec:Heuristics}

Let $f$ be a nonnegative function and consider the \emph{\idx{relative entropy}} $\mathcal E[f|g]$ and the \emph{\idx{relative Fisher information}} $\mathcal J[f|g]$ respectively defined by~\eqref{RelativeEntropy} and~\eqref{RelativeFisher}. If we specialize to $g=\mathsf g$, notice that $\mathsf g(x)^{1-p}=1+|x|^2$, so that
\[
\mathcal E[f|\mathsf g]=\frac{2\,p}{1-p}\(\nrm f{p+1}^{p+1}-\nrm{\mathsf g}{p+1}^{p+1}\)-\frac{p+1}{p-1}\ird{\(1+|x|^2\)\(f^{2p}-\mathsf g^{2p}\)}\,,
\]
and we can also expand $\mathcal J[f|\mathsf g]$~as
\begin{multline*}
\mathcal J[f|\mathsf g]=\(p^2-1\)\ird{|\nabla f|^2}+4\ird{x\cdot\nabla\(f^{p+1}\)}\\
+4\,\frac{p+1}{p-1}\ird{|x|^2\(f^{2p}-\mathsf g^{2p}\)}\,.
\end{multline*}
Using one integration by parts and $|x|^2=\(1+|x|^2\)-1$, we obtain
\begin{multline*}
\mathcal J[f|\mathsf g]=\(p^2-1\)\(\nrm{\nabla f}2^2-\nrm{\nabla\mathsf g}2^2\)-4\,d\(\nrm f{p+1}^{p+1}-\nrm{\mathsf g}{p+1}^{p+1}\)\\
+4\,\frac{p+1}{p-1}\ird{\(1+|x|^2\)\(f^{2p}-\mathsf g^{2p}\)}-4\,\frac{p+1}{p-1}\(\nrm f{2p}^{2p}-\nrm{\mathsf g}{2p}^{2p}\)\,.
\end{multline*}
As a consequence, we can measure the deficit as
\begin{align*}
\frac{p+1}{p-1}\,\delta[f|\mathsf g]&=\mathcal J\big[f|\mathsf g\big]-4\,\mathcal E\big[f|\mathsf g\big]\\
&=\(p^2-1\)\(\nrm{\nabla f}2^2-\nrm{\nabla\mathsf g}2^2\)+4\,\frac{d-p\,(d-2)}{p-1}\(\nrm f{p+1}^{p+1}-\nrm{\mathsf g}{p+1}^{p+1}\)\\
&\hspace*{1cm}-4\,\frac{p+1}{p-1}\(\nrm f{2p}^{2p}-\nrm{\mathsf g}{2p}^{2p}\)\,.
\end{align*}

Let $f_\eta=g+\eta\,h$ for some smooth and sufficiently decaying function $h$. Using the fact that $\mathsf g(x)^{p-1}=\big(1+|x|^2)^{-1}$, $\mathsf g(x)^{2p-2}=\big(1+|x|^2)^{-2}$ and a Taylor expansion as $\eta\to0_+$, we obtain
\begin{multline*}
\delta[f_\eta|\mathsf g]=2\,\eta\ird{\(-(p-1)^2\,\Delta\mathsf g+2\,\big(d-p\,(d-2)\big)\,\mathsf g^p-4\,p\,\mathsf g^{2\,p-1}\)h}\\
+(p-1)^2\,\mathsf Q[h]\,\eta^2+o\(\eta^2\)
\end{multline*}
where $\mathsf Q$ is defined by~\eqref{Q}.
Hence $\delta[f_\eta|\mathsf g]=(p-1)^2\,\mathsf Q[h]\,\eta^2+o\(\eta^2\)$ because
\[
-(p-1)^2\,\Delta\mathsf g+2\,\big(d-p\,(d-2)\big)\,\mathsf g^p-4\,p\,\mathsf g^{2\,p-1}=0\,.
\]

Similarly a simple Taylor expansion shows that
\[
\mathcal E[f_\eta|g]=p\,(p+1)\,\eta^2\ird{g^{2\,(p-1)}\,h^2}+o(\eta^2)\,.
\]
As a consequence, under Condition~\eqref{TaylorConstraint:h}, we have
\[
\lim_{\eta\to0}\frac{\delta[f_\eta]}{\mathcal E[f_\eta|\mathsf g]}\ge\frac{(p-1)^2\,}{p\,(p+1)}\,(\Lambda_\star-\Lambda)
\]
with $\Lambda=\frac{4\,p}{p-1}$ and $\Lambda_\star$ as in~\eqref{Improved-Hardy-Poincare}.
A similar improvement is obtained if, instead of Condition~\eqref{TaylorConstraint:h}, we assume that~$h$ is supported in $B_{R_\star}^c$.\medskip

%.......................................................................
\subsubsection{A Taylor expansion}\label{Sec:Abstract3}

In order to prove Theorem~\ref{Thm:StabSubCriticalNormalized}, we have to adapt the formal Taylor expansion to the minimizing sequence $(f_n)_{n\in\N}$ introduced in \eqref{minimizingstab}. The constraint that~$f_n$ satisfies~\eqref{Hyp0} with $g_{f_n}=\mathsf g$ means that~\eqref{TaylorConstraint:h} can be expected only for the linearized problem in the limit as $n\to+\infty$.

Let us introduce the function $\Psi:\eta\mapsto \mathcal E\big[\mathsf g+\eta\,h_n|\mathsf g\big]$. By a standard Taylor expansion, there exists some $\zeta_n\in[0,\eta_n]$ such that
\[
\Psi(\eta_n)=\Psi(0)+\Psi'(0)\,\eta_n+\frac12\,\Psi''(\zeta_n)\,\eta_n^2\,.
\]
{}From Lemma~\ref{Lem:ConvexEntropy}, we learn that $\Psi(0)=\Psi'(0)=0$. With $m=2\,p/(p+1)\in(0,1)$ and using the strictly convex function $\varphi(s)=s^m/(m-1)$ as in~\eqref{Ephi}, we obtain
\begin{multline*}
\frac{1-p}{2\,p}\,\mathcal E[f_n|\mathsf g]=\nrm{f_n}{p+1}^{p+1}-\nrm{\mathsf g}{p+1}^{p+1}-m\ird{|\mathsf g^{2p}|^{m-1}\(f_n^{2p}-\mathsf g^{2p}\)}\\
=-\,\frac12\,(p^2-1)\,\eta_n^2\ird{\left(\mathsf g^{p-1}+G_n\right)\,h_n^2}
\end{multline*}
where
\[
G_n=\(1+\frac{2\,p-1}{p-1}\,\frac{g_n^{p-1}}{\mathsf g^{p-1}}\)\(g_n^{p-1}-\mathsf g^{p-1}\)\quad\mbox{and}\quad g_n=\mathsf g+\zeta_n\,h_n\,.
\]
Notice that $g_n$ takes values in the interval $\(\min\{f_n,\mathsf g\},\max\{f_n,\mathsf g\}\)$. Since $f_n$ strongly converges in $\mathcal W_p(\R^d)$ to $\mathsf g$ as $n\to+\infty$, we know that $g_n$ strongly converges in $\mathcal W_p(\R^d)$ to $\mathsf g$. Since $\mathsf g(x)^{1-p}=\lrangle x^2=1+|x|^2$, the key point is to observe that $\left|\ird{G_n\,h_n^2}\right|$ can be estimated using H\"older's inequality first with exponents $(p+1)/(p-1)$ and $(p+1)/2$, and then with exponents $2\,p/(p-1)$, $p$ and $2\,p/(p-1)$ and the measure $\mathsf g(x)^{1-p}\,\dx=\lrangle x^2\,\dx$, to get the bounds
\begin{align*}
&\textstyle\left|\ird{\(g_n^{p-1}-\mathsf g^{p-1}\)h_n^2}\right|\le\(\ird{|g_n^{p-1}-\mathsf g^{p-1}|^\frac{p+1}{p-1}}\)^\frac{p-1}{p+1}\(\ird{h_n^{p+1}}\)^\frac2{p+1}\,,\\[4pt]
&\textstyle\left|\ird{g_n^{p-1}\(g_n^{p-1}-\mathsf g^{p-1}\)h_n^2\,\lrangle x^2}\right|\\
&\hspace*{12pt}\textstyle\le\(\ird{g_n^{2p}\,\lrangle x^2}\)^\frac{p-1}{2\,p}\(\ird{h_n^{2\,p}\,\lrangle x^2}\)^\frac1p\(\ird{|g_n^{p-1}-\mathsf g^{p-1}|^\frac{2\,p}{p-1}\,\lrangle x^2}\)^\frac{p-1}{2\,p}\,.
\end{align*}
$\bullet$ If $1<p\le2$, using the estimate $|g_n^{p-1}-\mathsf g^{p-1}|\le|g_n-\mathsf g|^{p-1}$ and the convergence of $g_n$ to $\mathsf g$ in $\mathrm L^{p+1}(\R^d,\dx)\cap\mathrm L^{2p}(\R^d,\lrangle x^2\,\dx)$ as $n\to+\infty$, we conclude that $\lim_{n\to+\infty}\left|\ird{G_n\,h_n^2}\right|=0$.\\
$\bullet$ If $p>2$, we have the estimate $|g_n^{p-1}-\mathsf g^{p-1}|\le(p-1)\,|g_n-\mathsf g|\(g_n+\mathsf g\)^{p-2}$ and notice that
\begin{align*}
&\textstyle\ird{\(|g_n-\mathsf g|\(g_n+\mathsf g\)^{p-2}\)^\frac{p+1}{p-1}}\\
&\textstyle\hspace*{3cm}\le\(\ird{|g_n-\mathsf g|^{p+1}}\)^\frac1{p-1}\(\ird{\(g_n+\mathsf g\)^{p+1}}\)^\frac{p-2}{p-1}\,,\\
&\textstyle\ird{\(|g_n-\mathsf g|\(g_n+\mathsf g\)^{p-2}\)^\frac{2\,p}{p-1}\,\lrangle x^2}\\
&\textstyle\hspace*{3cm}\le\(\ird{|g_n-\mathsf g|^{2p}\,\lrangle x^2}\)^\frac1{p-1}\(\ird{\(g_n+\mathsf g\)^{2p}\,\lrangle x^2}\)^\frac{p-2}{p-1}\,,
\end{align*}
using H\"older's inequality twice with exponents $p-1$ and $(p-1)/(p-2)$. Again we conclude that $\lim_{n\to+\infty}\left|\ird{G_n\,h_n^2}\right|=0$.\\
Altogether, this proves that
\be{ApproxEnt}
\frac1{\eta_n^2}\,\mathcal E[f_n|\mathsf g]=p\,(p+1)\,\ird{\mathsf g^{p-1}\,h_n^2}+o(1)\quad\mbox{as}\quad n\to+\infty\,.
\ee
By H\"older's inequality,
\[
\int_{B_R^c}\mathsf g^{p-1}\,h_n^2\,\dx\le\(\int_{B_R^c}\mathsf g^{p+1}\,\dx\)^\frac{p-1}{p+1}\(\int_{B_R^c}h_n^{p+1}\,\dx\)^\frac2{p+1}
\]
can be made small uniformly in $n$, for $R>0$ large enough. Inside the ball $B_R=B_R(0)$, by writing
\[
\int_{B_R}\mathsf g^{p-1}\,\left|h_n^2-h^2\right|\,\dx\le\nrm{\mathsf g}{p+1}^{p-1}\(\int_{B_R}|h_n^2-h^2|^\frac{p+1}2\,\dx\)^\frac2{p+1}
\]
and using the fact that $h_n\to h$ strongly in $\mathrm L^{p+1}_{\text{loc}}(\R^d)$, eventually after extracting a subsequence, we obtain
\be{ApproxEntLim}
\lim_{n\to+\infty}\frac1{\eta_n^2}\,\mathcal E[f_n|\mathsf g]=p\,(p+1)\ird{\mathsf g^{p-1}\,h^2}=p\,(p+1)\ird{\frac{h^2}{1+|x|^2}}\,.
\ee
By applying~\eqref{DeltaNorm} to $f_n$, we also know that
\begin{multline*}
\frac{p+1}{p-1}\,\delta[f_n]=(p^2-1)\(\nrm{\nabla f_n}2^2-\nrm{\nabla\mathsf g}2^2\)+4\,\frac{d-p\,(d-2)}{p-1}\(\nrm{f_n}{p+1}^{p+1}-\nrm{\mathsf g}{p+1}^{p+1}\) \\
-\,4\,\frac{p+1}{p-1}\(\nrm{f_n}{2p}^{2p}-\nrm{\mathsf g}{2p}^{2p}\)\,.
\end{multline*}
Using again a Taylor expansion, we find some $g_n$ trapped between $\mathsf g$ and $f_n$ hence strongly converging to $\mathsf g$ in $\mathcal W_p$ such that
\begin{multline*}
\frac1{\eta_n^2}\,\delta[f_n]=(p-1)^2\,\mathsf Q[h_n]+2\,p\,\big(d-p\,(d-2)\big)\(\ird{g_n^{p-1}\,h_n^2}-\ird{\mathsf g^{p-1}\,h_n^2}\)\\
-4\,p\,(2\,p-1)\(\ird{g_n^{2\,p-2}\,h_n^2}-\ird{\mathsf g^{2\,p-2}\,h_n^2}\)\,.
\end{multline*}
By arguing as above, we obtain
\be{ApproxDeficit}
\frac1{\eta_n^2}\,\delta[f_n]=(p-1)^2\,\mathsf Q[h_n]+o(1)\quad\mbox{as}\quad n\to+\infty
\ee
and, because $h_n\to h$ strongly in $\mathrm L^{2p}_{\text{loc}}(\R^d)$, eventually after extracting a subsequence, and $\nabla h_n\to\nabla h$ weakly in $\mathrm L^2(\R^d)$,
\be{ApproxDeficitLim}
\lim_{n\to+\infty}\frac1{\eta_n^2}\,\delta[f_n]\ge(p-1)^2\,\mathsf Q[h]\,.
\ee

%.......................................................................
\medskip\subsubsection{The locally compact case}\label{Sec:Abstract3b}~

Assume now that $h\not=0$. We deduce from~\eqref{ApproxEntLim} and~\eqref{ApproxDeficitLim} that
\[
\lim_{n\to+\infty}\frac{\delta[f_n]}{\mathcal E[f_n|\mathsf g]}\ge\frac{(p-1)^2\,\mathsf Q[h]}{p\,(p+1)\ird{\frac{h^2}{1+|x|^2}}}\,.
\]
It remains to prove that $h$ satisfies~\eqref{TaylorConstraint}. We can first notice that
\begin{multline*}
0=\frac1{\eta_n}\,\ird{|x|^2\(f_n^{2p}-\mathsf g^{2p}\)}=2\,p\,\ird{\mathsf |x|^2\,\mathsf g^{2\,p-1}\,h_n}\\
+p\,(2\,p-1)\,\eta_n\,\ird{|x|^2\,g_n^{2\,(p-1)}\,h_n^2}
\end{multline*}
for some (new) sequence of functions $g_n$, which strongly converge in $\mathcal W_p$ to $\mathsf g$ as $n\to+\infty$. Passing to the limit then leads to
\[
\ird{|x|^2\,\mathsf g^{2p-1}h}=0\,.
\]
We repeat exactly the same argument for $1$ and $x$ to get~\eqref{TaylorConstraint}. Finally, the improved Hardy-Poincar\'e inequality~\eqref{Improved-Hardy-Poincare} provides the contradiction that
\[
0=\ell=\lim_{n\to+\infty}\frac{\delta[f_n]}{\mathcal E[f_n|\mathsf g]}\ge\frac{(p-1)^2\,}{p\,(p+1)}\,(\Lambda_\star-\Lambda)\,.
\]

%.......................................................................
\subsubsection{The vanishing case}\label{Sec:Abstract4}

To complete the proof of Theorem~\ref{Thm:StabSubCriticalNormalized}, we have to discard the only remaining case: \hbox{$h=0$}. In that case, we know that
\[
\lim_{n\to+\infty}\frac1{\eta_n^2}\,\mathcal E[f_n|\mathsf g]=0
\]
and, as a consequence, $\lim_{n\to+\infty}\delta[f_n]/\eta_n^2=0$ and $\lim_{n\to+\infty}\nrm{\nabla h_n}2=0$. Taking into account the fact that $\lim_{n\to+\infty}\nrm{h_n}{2p}=0$ and, by definition of $\eta_n$,
\[
\lim_{n\to+\infty}\ird{\lrangle x^2\,h_n^{2p}}=1\,,
\]
\emph{vanishing} (in the language of \index{concentration-compactness method}{concentration-compactness}) occurs, which means that the same type of estimates as the one from the \idx{Hardy-Poincar\'e inequality} holds, except that $\Lambda$ in~\eqref{Hardy-Poincare} has to be replaced by the gap associated with the \idx{essential spectrum}, $\Lambda_{\mathrm{ess}}$ given by~\eqref{LambdaEss}. Making this idea rigorous requires some explanations.

Let us introduce the IMS decomposition (IMS stands for for Ismagilov, Morgan, Morgan-Simon and Sigal: see \cite{10.2307/24714089,Simon1983}) $h_n^2=h_{n,R,1}^2+h_{n,R,2}^2$ where $R>0$, $h_{n,R,1}$ (resp. $h_{n,R,2}$) is supported in $B_{2R}$ (resp. $B_R^c$) and, for some constant $C>0$,
\[
\Big||\nabla h|^2-|\nabla h_{n,R,1}|^2-|\nabla h_{n,R,2}|^2\Big|\le C\,\frac{h_n^2}{R^2}\,\mathbbm 1_{\{R\le|x|\le2\,R\}}\,.
\]
We also have that
\[
\lim_{n\to+\infty}\int_{R\leq|x|\le2R}h_n^2 \dx\le(\omega_d\,R)^{\frac{p-1}{p+1}}\nrm{h_n}{p+1}^{p+1}=0
\]
where $\omega_d:=|\mathbb S^{d-1}|$. Let us first fix some $R>0$ such that~\eqref{Improved-Hardy-Poincare} is valid for any function supported in $B_{R}^c$ with constant $\Lambda_\star=(\Lambda_{\mathrm{ess}}+\Lambda)/2$. We can then decompose the estimates~\eqref{ApproxEnt} and~\eqref{ApproxDeficit} into a part on $B_{2R}$ that converges to $0$ and a part for which we use the improved Hardy-Poincar\'e inequality on $B_R^c$. As a consequence, we again obtain a contradiction because
\[
0=\ell=\lim_{n\to+\infty}\frac{\delta[f_n]}{\mathcal E[f_n|\mathsf g]}\ge\frac{(p-1)^2\,}{2\,p\,(p+1)}\,(\Lambda_{\mathrm{ess}}-\Lambda)>0\,.
\]

%%%%%%%%%%%%%%%%%%%%%%%%%%%%%%%%%%%%%%%%%%%%%%%%%%%%%%%%%%%%%%%%%%%%%%%%
\subsection{Comparison with the Bianchi-Egnell result for the critical case}\label{Comparison:Bianchi-Egnell}

In~\eqref{Bianchi-Egnell}, the distance to the manifold of the \idx{Aubin-Talenti functions} is measured in the $\mathrm H^1_0(\R^d)$ norm, while the result of Theorem~\ref{Thm:StabSubCriticalNormalized} is a essentially a \idx{stability} result in a Lebesgue space, according to Lemma~\ref{Lem:CK}. However, it is possible to \index{Bianchi-Egnell result}{reformulate} the \index{stability}{result} also in the stronger sense of, essentially, a Sobolev norm.

According to Lemma~\ref{Lem:GNscaling}, the \idx{Gagliardo-Nirenberg-Sobolev inequalities}~\eqref{GNS} in the subcritical range are equivalent to the inequality $\delta[f]\ge0$. According to Theorem~\ref{Thm:StabSubCriticalNormalized}, under the condition~\eqref{Hyp0} with $g=\mathsf g$, then with $K=\frac14\,\frac{p+1}{p-1}\,\mathcal C$, the inequality $\delta[f]\ge\mathcal C\,\mathcal E[f|\mathsf g]$ can be rewritten as
\[
\mathcal J[f|\mathsf g]-4\,(1+K)\,\mathcal E[f|\mathsf g]\ge0\,,
\]
which is equivalent to
\[
\mathcal J[f|\mathsf g]-4\,\mathcal E[f|\mathsf g]\ge\(\frac1{1+K}\,\mathcal J[f|\mathsf g]-4\,\mathcal E[f|\mathsf g]\)+\frac K{1+K}\,\mathcal J[f|\mathsf g]\ge\frac K{1+K}\,\mathcal J[f|\mathsf g]\,.
\]
Hence we have
\[\label{StrongStabSubCritical}
\delta[f]\ge\frac{p-1}{p+1}\,\frac K{1+K}\,\mathcal J[f|\mathsf g]\,.
\]
In the right-hand side of the inequality, we have a measure of the distance to the \index{Aubin-Talenti functions}{Aubin-Talenti function} $\mathsf g$ in a stronger sense than in the result of Theorem~\ref{Thm:StabSubCriticalNormalized}.

\medskip Theorem~\ref{Thm:StabSubCriticalNormalized} is limited to nonnegative functions. In case of a \index{sign-changing functions}{sign-changing function} $f\in\mathcal W_p(\R^d)$, let us extend the definition~\eqref{RelativeFisher} of the \emph{relative Fisher information} by
\be{Fisher-Sign}
\mathcal J[f|g]:=\frac{p+1}{p-1}\ird{\left|(p-1)\,\nabla f+|f|^{p-1}\,f\,\nabla g^{1-p}\right|^2}\,.
\ee
%-----------------------------------------------------------------------
\begin{corollary}\label{Cor:StabSubCriticalNormalized} If $d\le2$, or $d\ge3$ and $p<p^\star$, there is a positive constant $\kappa$ such that, for any $f\in\mathcal W_p(\R^d)$ satisfying
\[
\ird{\(1,x,|x|^2\)|f|^{2p}}=\ird{\(1,x,|x|^2\)\mathsf g^{2p}}\,,
\]
we have
\[
\delta[f]\ge\kappa\,\mathcal J[f|\mathsf g]\,.
\]
\end{corollary}
%-----------------------------------------------------------------------
\begin{proof} It follows from the above considerations that $\kappa=\frac{p-1}{p+1}\,\frac K{1+K}$ for nonnegative functions. In the general case, it is straightforward to notice that
\[
\mathcal J[f|g]=\mathcal J\big[|f|\,|\,g\big]\,,
\]
which concludes the proof.\end{proof}

The \idx{stability} estimates of Theorem~\ref{Thm:StabSubCriticalNormalized} and Corollary~\ref{Cor:StabSubCriticalNormalized} differ in nature from~\eqref{Bianchi-Egnell} because the $\mathrm L^{p+1}$ norm plays a key role, which is not present in the \index{Bianchi-Egnell result}{Bianchi-Egnell} result, and also because we work in the slightly smaller space $\mathcal W_p(\R^d)$, in order to properly define $\mathcal J[f|\mathsf g]$. In~\cite{MR1124290}, it is only required that the functions are in $\mathrm L^{2\,p^\star}(\R^d)$ with gradient in $\mathrm L^2(\R^d)$. Here we use estimates based on the \idx{relative entropy} functional, which explains why our function space $\mathcal W_p(\R^d)$ involves $\mathrm L^{2p}(\R^d,\lrangle x^2\,\dx)$ and why our estimates are done either for nonnegative functions~$f$ or for~$|f|$.

%%%%%%%%%%%%%%%%%%%%%%%%%%%%%%%%%%%%%%%%%%%%%%%%%%%%%%%%%%%%%%%%%%%%%%%%
%%%%%%%%%%%%%%%%%%%%%%%%%%%%%%%%%%%%%%%%%%%%%%%%%%%%%%%%%%%%%%%%%%%%%%%%
\section{Bibliographical comments}~

The characterization of \emph{\idx{optimal function}s} in functional inequalities is a standard problem in nonlinear analysis and in the calculus of variations. Various proofs of the optimality of the \idx{Aubin-Talenti functions} defined by~\eqref{Aubin.Talenti} are known: by direct variational methods, symmetry and ODE techniques in~\cite{DelPino2002}; using the \emph{\idx{carr\'e du champ} method} at formal level in~\cite{Carrillo2000,Carrillo2003} and then with a complete proof in~\cite{Carrillo2001} (also see the simpler presentation of~\cite{MR3497125}); by \idx{mass transport} in~\cite{MR2032031}; by a continuous dimension argument in~\cite{MR3155209,Carlen2013,MR3695890}. We will not review here the methods based on the \emph{\idx{carr\'e du champ}}, except to quote the large overview provided by the book~\cite{MR3155209}. More details will be given in the next chapters.

The issue of the \emph{\idx{stability}} of the \idx{optimal function}s started with the study of solitary waves obtained by minimization methods as in~\cite{MR677997,MR901236,MR691044}. In the case of \index{Sobolev's inequality}{Sobolev} and Gagliardo-Nirenberg inequalities, some pioneering results were obtained in bounded domains in~\cite{MR790771,MR998512}, but the breakthrough came with the result~\cite{MR1124290} of G.~Bianchi and H.~Egnell finally made constructive in~\cite{DEFFL2023,DEFFL2024}. In recent years, the problem of finding \emph{\idx{stability} results} for various sharp inequalities in analysis and geometry, such as the isoperimetric inequality, the Brunn–Min\-kowski inequality, the \index{Sobolev's inequality}{Sobolev inequality}, the \idx{logarithmic Sobolev inequality}, \emph{etc.}, has been intensively studied. See for instance~\cite{Fusco2008,MR2672283,MR3055988,MR3023863,MR3658711}. The \idx{stability} of the \index{Sobolev's inequality}{Sobolev's inequalities} in the case of an $\mathrm L^q$ norm of the gradient with $q\neq2$ was proven by A.~Cianchi, N.~Fusco, F.~Maggi and A.~Pratelli in~\cite{Cianchi2009} and improved recently by A.~Figalli, R.~Neumayer and Y.~R.-L.~Zhang in~\cite{Figalli_2018,MR4048334,figalli2020sharp} by a different method. Also see~\cite{Figalli_2020} for a related result for $q=2$. For the Gagliardo-Nirenberg inequality~\eqref{GNS}, some \idx{stability} results have been obtained by A.~Figalli and E.~Carlen in~\cite{Carlen2013}, F.~Seuffert in~\cite{MR3695890} and V.H.~Nguyen in~\cite{MR3989143}, with non-constructive constants. In~\cite{MR3179693}, remainder terms are obtained in the fractional Sobolev inequality, with a nice treatment of the linearization based on the inequality written on the sphere using the inverse stereographic projection. We can refer to~\cite{MR3469147} for a general introduction to \idx{stability} issues and some consequences of the known results, and~\cite{Carlen_2014,MR3223578} for \idx{stability} results for some Gagliardo-Nirenberg-Sobolev inequalities. All these results are \emph{quantitative}, in the sense that a precise notion of distance is controlled by the deficit of the inequality, but the proportionality constant is achieved by compactness, or through a contradiction argument, and no estimate on the constant is given. In this sense, except in~\cite{MR2672283},~\cite[Theorem~1.1, p.~83]{MR3055988} and~\cite{DEFFL2023}, these methods are \emph{not constructive}. This is due to the fact that a typical strategy is to reduce the problem to a \index{Sobolev's inequality}{critical Sobolev inequality} and adapt the method of~\cite{MR1124290}. In~\cite{MR2672283,MR3055988}, the authors obtain constructive stability estimates for $1$-Sobolev inequalities but no consequences are known so far in our setting. However, by duality (see~\cite{MR717827}) and flows, quantitative and constructive results were obtained in~\cite{MR2915466,MR3227280}, although with a weaker norm than in~\cite{MR1124290}.

In the subcritical regime, the situation is slightly better than for the critical case of \index{Sobolev's inequality}{Sobolev inequalities}. \emph{Constructive} results have been obtained in~\cite{Bonforte2010c} and can also be deduced from~\cite{Blanchet2009} in a very restricted neighborhood of the manifold of the \idx{Aubin-Talenti functions} (as it is measured in the uniform norm associated with the relative error). The global result of~\cite{Dolbeault2013917} is explicit but clearly sub-optimal as the remainder term is of the order of the square of the entropy while one would expect a linear dependence. This result has been rephrased in~\cite{MR3493423} using scaling considerations, and we use it here to prove the \idx{uniqueness} up to the trivial invariances of the optimizers in~\eqref{GNS}. Progresses have been achieved in subcritical interpolation inequalities on the sphere in~\cite{1504,Dolbeault_2020}, with a constant which is obtained through an explicit, standard minimization problem (whose value is not known), except in the limit case of the \idx{logarithmic Sobolev inequality} or if additional symmetry assumptions are imposed. Also see~\cite{Frank-2021,2210,2302,2303,Koenig2023,Koenig_stab} for recent advances.

\emph{\index{Gagliardo-Nirenberg-Sobolev inequalities}{Gagliardo-Nirenberg inequalities}} go back to~\cite{Gagliardo1958,Nirenberg1959}. An interesting account is given in~\cite{Nirenberg-Preface2020} of the first meeting of E.~Gagliardo and L.~Nirenberg in a special issue \emph{in memoriam} of E.~Gagliardo, and how these inequalities became known as the Gagliardo-Nirenberg inequalities. The special class of inequalities~\eqref{GNS} appears in~\cite{Gunson91,DelPino2002} and the link with fast diffusion equations was made clear in~\cite{DelPino2002}.

\medskip Now, let us enter in the details of the results of this chapter. Although classical, the \emph{concentration-compactness method} has not been, as far as we know, applied to the \idx{Gagliardo-Nirenberg-Sobolev inequalities}~\eqref{GNS}, except in~\cite[page~280]{MR778974} at a rather abstract level. Methods are however standard and we primarily refer to~\cite{MR778970,MR778974,MR834360,MR850686}. For the analysis of the behaviour of minimizing sequences from the point of view of \index{concentration-compactness method}{concentration-compactness}, we shall refer to~\cite{MR1632171}. Also see~\cite[Chapter~I, Section~4]{MR2431434} for a concise introduction to \index{concentration-compactness method}{concentration-compactness}. In~\cite{DelPino2002}, the existence of an \idx{optimal function} is obtained first on a centred ball of radius $R$, for which it is known that \idx{optimal function}s are radial by the moving plane method on the ball, see~\cite{MR544879}, as well as on the whole space according to~\cite{MR634248}. The solutions are unique by~\cite{MR1647924,MR1803216}. Using barrier functions, the minimizer on the whole space is obtained by taking the limit as $R\to+\infty$. Other proofs of the existence of \idx{optimal function}s in~\eqref{GNS} have been obtained by \idx{mass transport} in~\cite{MR2032031} or by reducing the problem of \idx{Gagliardo-Nirenberg-Sobolev inequalities} to a \index{Sobolev's inequality}{Sobolev inequality} in a higher dimension, see~\cite{MR3155209,Carlen2013,MR3695890}. There is also a proof by \idx{entropy methods} which is given in~\cite{Carrillo2001} but whose ideas go back to~\cite{Carrillo2000,Carrillo2003}. Last but not least, in~\cite{Gunson91}, J.~Gunson noticed the role of \idx{Aubin-Talenti functions}, however with only a sketch of a proof, and without making the link with fast diffusion equations as, for instance, in~\cite{DelPino2002}. The \idx{concentration-compactness method} has also been used in the \idx{Caffarelli-Kohn-Nirenberg inequalities}, see~\cite{Caffarelli1984}, which are  similar to the \idx{Gagliardo-Nirenberg-Sobolev inequalities}, except that they involve weights: see for instance~\cite{Catrina2001,1005}. Because of the \idx{Emden-Fowler transformation}, the two families of inequalities are closely related, as critical \idx{Caffarelli-Kohn-Nirenberg inequalities} can be rewritten as Gagliardo-Nirenberg-Sobolev interpolation inequalities without weights on cylinders.

\smallskip Concerning our presentation of Sobolev's inequalities, here are a few additional details. Symmetry obtained by \emph{\idx{spherically symmetric rearrangement}s} is used in this chaper. We refer to~\cite{MR1322313,MR810619,MR1961519,MR2238193} of an overview of symmetrization methods. The \idx{P\'olya–Szeg\H o principle} goes back to~\cite{MR0043486}. A simpler proof based on the Riesz lemma and monotonicity properties of the heat kernel, due to E.~Lieb in~\cite{MR471785}, can be found in~\cite{MR1817225}. Equality cases for the Dirichlet integral in \idx{spherically symmetric rearrangement} are known to be radial up to translation (see for instance~\cite[p.~154]{MR929981}) but this is by no way elementary. In order to establish the \idx{uniqueness} among radial functions, we use the \emph{\idx{Emden-Fowler transformation}}. This transformation was known for a long time in astrophysics, at least in dimension $d=3$ for the critical exponent $2^*=6$, as it is discussed at length in~\cite[Chapter IV, on \emph{Polytropic and isothermal gas spheres}]{MR0092663}, with important contributions (for the equations of polytropic gases) going back at least to H.~Lane in 1870 in~\cite{Lane_1870}. After the \idx{Emden-Fowler transformation}, the problem becomes one-dimensional: see~\cite{MR3263963} for detailed considerations concerning \emph{Gagliardo-Nirenberg inequalities on the line}. Details on a \idx{mass transport} proof can also be found there, which are of course connected with the result of~\cite{MR2032031} (also see~\cite{MR2053603}) and~\cite{MR2427077}. Notice that the knowledge of the \idx{optimal function}s points in the direction of a duality approach using \idx{mass transport} methods: see~\cite[Section~2]{MR3263963} for an application to~\eqref{EF-Var}. The change of variable $z(s)$ in Section~\ref{Sec:Schwarz} is also reminiscent of \idx{mass transport}. This change of variables appears in~\cite[Sections~2 and~4]{Benguria_2004}. As far as we know, it has not been used so far to obtain a compactness result.

On optimality in \emph{Sobolev's inequalities}, we refer to the classical papers of T.~Aubin and G.~Talenti~\cite{Aubin-76,Talenti-76} for the optimal \index{Sobolev's inequality}{\emph{Sobolev inequality}}, with earlier contributions by G.~Rosen in~\cite{MR0289739}, and G.~Bliss in~\cite{bliss1930integral}. E.~Rodemich's seminar~\cite{rodemich1966sobolev} is probably the first complete proof (see~\cite[p.~158]{MR1814364} for a quote) although this reference is not widely known and difficult to find. The \emph{\idx{Onofri inequality}} can be seen as a limit case of Sobolev's inequalities in dimension \hbox{$d=2$}. The inequality was established in~\cite{MR677001} with \idx{optimal constant}, but already appears in~\cite{MR0301504}. See~\cite{Dolbeault:2015ly} for a collection of various methods of proof,~\cite{Carlen_1992,Beckner_1993} and~\cite{del_Pino_2012} for extensions to dimensions $d>2$. The \emph{Euclidean \index{logarithmic Sobolev inequality}{logarithmic Sobolev}} seen as a limit case of~\eqref{GNS} appears in~\cite{DelPino2002}, long after the well known papers~\cite{Gross75,Federbush} in the Euclidean case,~\cite[Theorem~2]{MR479373} for the scale invariant form of the inequality, and also~\cite[Inequality~(2.3)]{MR0109101}. The optimality case is characterized in~\cite[Inequality~(26)]{MR1132315} and we refer to~\cite{2014arXiv1408.2115B,Fathi_2016,MR3493423,MR4088499,eldan2020stability,2303,DEFFL2023,DEFFL2024} for some recent \idx{stability} results.

\smallskip\emph{Dilations and homogeneity} play an essential role in this chapter. Property (i) of Lemma~\ref{Lem:BasicEntropyProp} is a consequence of scalings and homogeneity, as in Lemma~\ref{Lem:GNscaling}: for details, see~\cite{Dolbeault2013917}. Property (ii) of Lemma~\ref{Lem:BasicEntropyProp} arises from a simple computation which is detailed in~\cite{DelPino2002}. Concerning the \idx{stability} result of Lemma~\ref{Lem:BasicEntropyProp},~(iii), an expression of $\mathcal C$ is given in~\cite[Theorem~6]{MR3493423}, based on the comparison of the non-scale invariant inequality $\delta[f]\ge0$ with~\eqref{GNS}, using dilations. An earlier proof based on the \emph{\idx{carr\'e du champ}} method appeared in~\cite{Dolbeault2013917}. This \idx{stability} result is rather weak but has anyway important consequences, starting with \idx{uniqueness} issues.

In this chapter, in order to identify the \idx{optimal function}s, we proceed in two steps. First we identify the \emph{radial \idx{optimal function}s}: the result of Lemma~\ref{Lem:Carre} is exactly the computation for R\'enyi entropies of~\cite{Dolbeault_2016} except that we do it only for a critical point, with sufficient decay properties and regularity. It is actually a \emph{\idx{rigidity} result}, in the sense that we do not use the optimality of the solution, but only that it is a critical point. See~\cite{MR615628,MR1134481,MR1412446,Dolbeault20141338,1703} for related results on manifolds. The strategy of Lemma~\ref{Lem:Carre} is very similar to the strategy of~\cite{Dolbeault2016} and not limited to radial functions. A major advantage of critical points is that, by elliptic regularity, computations are much easier to justify than, for instance, along the flow of a parabolic equation: see~\cite{DEL-JEPE} for a discussion. The price to pay is that the strategy for ordering the computations in the proof of Lemma~\ref{Lem:Carre} is by no way obvious. On the contrary, the parabolic version of the method offers a clearer strategy, to the price of computations which are more difficult to justify, as we shall see in the next chapter.

\smallskip The \emph{\idx{uniqueness} result} (up to trivial invariances) of Corollary~\ref{Cor:Uniqueness} is a remarkable result, which is a direct consequence of the \idx{stability} result of Lemma~\ref{Lem:BasicEntropyProp},~(iii) and arises from the nonlinearity rather than from the geometry, as there is no curvature involved in the case of the Euclidean space. This observation goes back to~\cite{MR3493423}, although a parabolic proof appeared earlier in~\cite{Dolbeault2013917}. There are a number of alternative proofs which have appeared in the literature. Equality cases for the Dirichlet integral in \idx{spherically symmetric rearrangement} are known to be radial up to translation (see for instance~\cite[p.~154]{MR929981}) but this is by no way elementary. It is also possible to work directly on the Euler-Lagrange equation: up to an analysis of the regularity and the decay properties of the solutions as $|x|\to+\infty$, the moving plane methods can be applied and the result of~\cite{MR634248} shows that \idx{optimal function}s are radial up to translation. This is the approach of~\cite{DelPino2002} and again \idx{uniqueness} is by no way elementary. Knowing that the solution is radial is not enough to guarantee \idx{uniqueness} and one has then to rely on delicate ODE arguments that can be found in~\cite{MR1647924,MR1803216}. The \idx{uniqueness} (up to trivial invariances) is also a consequence of the proof by \idx{mass transport} in~\cite{MR2032031} and ultimately of the underlying gradient flow structure with respect to the Wasserstein distance (see for instance~\cite{MR2053603}). The \emph{\idx{carr\'e du champ}} method on which we will rely in the next chapters is even stronger, as it proves that not only the minimizers but also all nonnegative critical points are in $\mathfrak M$, a much stronger result which has to do with a convexity property along the flow and the improved inequality of Lemma~\ref{Lem:BasicEntropyProp},~(iii). The computation of Lemma~\ref{Lem:Carre}, is a close analogue of the results of M.-F.~Bidaut-V\'eron and L.~V\'eron in~\cite{MR1134481}, adapted to the Euclidean case. The counterpart of the \emph{\idx{carr\'e du champ}} method in elliptic theory goes back to~\cite{MR615628} and has been exploited in~\cite{MR1134481}: the link has been made clear for instance in~\cite{Dolbeault20141338} in the case of compact manifolds. The Euclidean case is more delicate than the case of compact manifolds: see~\cite{DEL-JEPE} for a  discussion. The key point is then to obtain regularity and decay estimates for non-radial solutions of~\eqref{ELrad} and their derivatives, which involves for instance a Moser iteration scheme. To our knowledge, this has not been done yet but presents no real difficulties. With these precautions, Lemma~\ref{Lem:Carre} applies to all critical points, and not only to smooth and sufficiently decreasing functions.

In Section~\ref{Sec:EntropyBasic}, there is a number of observations on the \idx{relative entropy} which are borrowed from \idx{entropy methods}, that is, from the topic of the next chapter. Let us emphasize that the elliptic point of view of Chapter~\ref{Chapter-1} can be considered as a special case of the \idx{entropy methods} in the parabolic case applied either to self-similar situations in original variables, or to stationary solutions in self-similar variables. For instance, the idea of using \emph{best matching \idx{Aubin-Talenti functions}} in Lemma~\ref{Lem:BestMatchAubinTalenti} refers to the notion of \emph{best matching Barenblatt functions} which will be exploited in the next chapters and was introduced in~\cite{Dolbeault2013917,Dolbeault_2016} or~\cite[Section~4]{1751-8121-48-6-065206}. Similar remarks apply to the \emph{\idx{Csisz\'ar-Kullback inequality}}, which is a classical tool of \idx{entropy methods} and will be further discussed in~\ref{Appendix:CK}.

\smallskip \emph{Linearization and the spectral analysis} of the corresponding operators are essential in the \idx{stability} analysis of~\eqref{GNS}, as we need to linearize around an \index{Aubin-Talenti functions}{Aubin-Talenti function}. Spectral gap properties in connection with \idx{Gagliardo-Nirenberg-Sobolev inequalities} (and fast diffusion flows) go back to~\cite{Scheffer01,MR1982656,Denzler2005,Blanchet2007}. Here are a few entry points in the literature of the \emph{weighted Hardy inequality} and the \emph{\idx{Hardy-Poincar\'e inequality}}. The proof of~\eqref{WeightedHardy} is classical: see for instance~\cite{Blanchet2007}. Persson's characterization of the bottom of the \idx{essential spectrum}, $\Lambda_{\mathrm{ess}}$, is well known in spectral theory and one of the classical paper on the topic is~\cite{MR0133586}. For results concerning more specifically the measure~$\mathrm d\mu_\alphaa$, the \idx{discrete spectrum} of the operator $\mathcal L_{\alphaa,d}$ and the computation of $\Lambda$ and~$\Lambda_\star$ in~\eqref{Hardy-Poincare} and~\eqref{Improved-Hardy-Poincare}, we refer to~\cite[Corollaries~1 and~2]{Dolbeault2011a} and to~\cite{MR1982656,Denzler2005,Blanchet2007,Blanchet2009,Bonforte2010c} for earlier partial results. The role of~\eqref{WeightedHardy} and the analysis of the \idx{essential spectrum}, in connection with~\cite{MR0133586}, is clarified in~\cite[Appendix~A.1]{BDLS2020}. In Section~\ref{Sec:Abstract} (proof of Theorem~\ref{Thm:StabSubCriticalNormalized}), we perform a Taylor expansion similar to the one of the slightly more complicated framework of \idx{Caffarelli-Kohn-Nirenberg inequalities}, in~\cite{Bonforte2017a}.

\smallskip Inequality~\eqref{SobBGL} in Theorem~\ref{Thm:Sob} is of \index{Caffarelli-Kohn-Nirenberg inequalities}{Caffarelli-Kohn-Nirenberg} type. The computation of Section~\ref{Sec:BGLtransformation} can be found in~\cite[Section~6.10]{MR3155209} and a sketch of a proof of Theorem~\ref{Thm:Sob} is also provided there when the dimension~$N$ is an integer, based on the \emph{\idx{curvature-dimension criterion}}. The case of a \index{weights and non-integer dimensions}{non-integer dimension} is more subtle because the property that ``the \idx{carr\'e du champ} operator~$\Gamma$ [is] defined on a suitable algebra $\mathcal A$ of functions in the $\mathrm L^2(\mathbb S^d, \mathrm d\mu_n)$-domain'' of the diffusion operator is to some extent formal and cannot be assumed as in~\cite[page~71]{MR3155209} (notations have been adapted). The \index{Sobolev's inequality}{Sobolev inequality}~\eqref{SobBGL} with a \index{weights and non-integer dimensions}{non-integer dimension}~$N$ can indeed be seen as a weighted inequality on a standard Euclidean space, eventually with some \emph{cylindrical symmetry}, but weights are a serious source of trouble even for the simplest versions of the \emph{\idx{carr\'e du champ} method}: see for instance~\cite{Dolbeault2017a,Dolbeault2016,Dolbeault2017}, and~\cite{DEL-JEPE} for the discussion of the regularity which is needed in the proofs. In any case, V.H.~Nguyen in~\cite{Nguyen_2015} gave an alternative proof of Theorem~\ref{Thm:Sob} based on an adaptation of the \idx{mass transport} method of~\cite{MR2032031}, which also covers various other cases. Recently F.~Seuffert extended the first result of~\cite{Carlen2013} and generalized to Inequality~\eqref{SobBGL} the result of \index{Bianchi-Egnell result}{Bianchi-Egnell} in~\cite[Theorem~1.10]{seuffert2016stability}, while~\cite{MR3989143} contains additional \idx{stability} results for~\eqref{GNS}, with various notions of distance to the manifold $\mathfrak M$. All these quantitative results are based on the \emph{equivalence with \idx{Sobolev's inequality}~\eqref{SobBGL} with a \index{weights and non-integer dimensions}{non-integer dimension}}.

%%%%%%%%%%%%%%%%%%%%%%%%%%%%%%%%%%%%%%%%%%%%%%%%%%%%%%%%%%%%%%%%%%%%%%%%
%%%%%%%%%%%%%%%%%%%%%%%%%%%%%%%%%%%%%%%%%%%%%%%%%%%%%%%%%%%%%%%%%%%%%%%%
\chapter{The fast diffusion equation and the entropy methods}\label{Chapter-2}

In this chapter, we explain how \idx{Gagliardo-Nirenberg-Sobolev inequalities}~\eqref{GNS} can be used in the study of the \idx{fast diffusion equation}, through various \emph{entropy methods}. By \emph{\idx{carr\'e du champ}} computations, the flow can also be used to provide a proof of~\index{Gagliardo-Nirenberg-Sobolev inequalities}{\eqref{GNS}}. In preparation for the next chapters, we will also establish a few additional estimates. Many integrations by parts are needed, that we shall not justify in details, but we will indicate how they can be proved and where the corresponding proofs can be found in the literature.

%%%%%%%%%%%%%%%%%%%%%%%%%%%%%%%%%%%%%%%%%%%%%%%%%%%%%%%%%%%%%%%%%%%%%%%%
%%%%%%%%%%%%%%%%%%%%%%%%%%%%%%%%%%%%%%%%%%%%%%%%%%%%%%%%%%%%%%%%%%%%%%%%
\section{Gagliardo-Nirenberg-Sobolev inequalities and fast diffusion}\label{Sec:GNS-FDE}

Let us consider the diffusion equation
\be{FD}
\frac{\partial u}{\partial t}=\Delta u^m\,,\quad u(t=0,\cdot)=u_0
\ee
acting on nonnegative functions $u$ defined on $\R^+\times\R^d$. The case $m=1$ corresponds to the heat equation and $m>1$ is known as the porous medium case. Here we shall focus on the \index{fast diffusion equation}{\emph{fast diffusion}} case
\[
m<1\,.
\]

%%%%%%%%%%%%%%%%%%%%%%%%%%%%%%%%%%%%%%%%%%%%%%%%%%%%%%%%%%%%%%%%%%%%%%%%
\subsection{Mass, moment, entropy and Fisher information}\label{Sec:APriori-Renyi}

The following properties are well known:
\begin{enumerate}
\item[(i)] \emph{Mass conservation}. If $u_0\in\mathrm L^1_+(\R^d)$ and $m\ge m_c:=(d-2)/d$, then
\[\label{Ch2:Mass}
\frac{\rd}{\dt}\ird{u(t,x)}=0\,.
\]
In other words, the mass $M=\ird{u(t,x)}=\ird{u_0}$ does not depend on $t\ge0$.
\item[(ii)] \emph{Growth of the second moment}. If $u_0\in\mathrm L^1_+\!\(\R^d,(1+|x|^2)\,\dx\)$ and $m>\widetilde m_1$ where $\widetilde m_1:=d/(d+2)$, then
\[\label{Ch2:Moment}
\frac{\rd}{\dt}\ird{|x|^2\,u(t,x)}=2\,d\ird{u^m(t,x)}\,.
\]
Also notice that, by H\"older's inequality,
\begin{multline*}
\ird{u^m}=\ird{\(1+|x|^2\)^m\,u^m\cdot\(1+|x|^2\)^{-m}}\\
\le\(\ird{\(1+|x|^2\)u}\)^m\(\ird{\(1+|x|^2\)^{-\frac m{1-m}}}\)^{1-m}
\end{multline*}
so that $\ird{u^m(t,x)}$ is well defined as soon as $u_0\!\in\!\mathrm L^1_+\!\(\R^d,(1+|x|^2)\dx\)$.
\item[(iii)] \emph{Entropy estimate}. If $u_0\in\mathrm L^1_+\!\(\R^d,(1+|x|^2)\,\dx\)$ is such that $u_0^m\in\mathrm L^1(\R^d)$, and if $m\ge m_1:=(d-1)/d$, then
\be{Ch2:Entropy}
\frac{\rd}{\dt}\ird{u^m(t,x)}=\frac{m^2}{1-m}\ird{u\,|\nabla u^{m-1}|^2}\,.
\ee
\end{enumerate}
Let us define an \emph{entropy} functional $\mathsf E$ and an associated \emph{Fisher information} functional $\mathsf I$ respectively by
\[
\mathsf E[u]:=\ird{u^m}\quad\mbox{and}\quad\mathsf I[u]:=\frac{m^2}{(1-m)^2}\ird{u\,|\nabla u^{m-1}|^2}\,.
\]
If $u$ solves~\eqref{FD} and with the notation $\mathsf E(t)=\mathsf E[u(t,\cdot)]$, $\mathsf E'=\frac{\rd}{\dt}\mathsf E[u(t,\cdot)]$, then~\eqref{Ch2:Entropy} can be rewritten as
\[
\mathsf E'=(1-m)\,\mathsf I\,.
\]

%%%%%%%%%%%%%%%%%%%%%%%%%%%%%%%%%%%%%%%%%%%%%%%%%%%%%%%%%%%%%%%%%%%%%%%%
\subsection{Entropy growth rate}\label{Sec:EntropyGrowth}

In order to relate~\eqref{FD} with \index{Gagliardo-Nirenberg-Sobolev inequalities}{\eqref{GNS}}, it is convenient to introduce
\be{pm}
p=\frac 1{2\,m-1}\quad\Longleftrightarrow\quad m=\frac{p+1}{2\,p}\,,
\ee
consider $f$ such that $u=f^{2\,p}$, and notice that $u^m=f^{p+1}$ and $u\,|\nabla u^{m-1}|^2=(p-1)^2\,|\nabla f|^2$, so that
\[
M=\nrm f{2\,p}^{2\,p}\,,\quad\mathsf E[u]=\nrm f{p+1}^{p+1}\quad\mbox{and}\quad\mathsf I[u]=(p+1)^2\,\nrm{\nabla f}2^2\,.
\]
As a consequence, if $u$ solves~\eqref{FD}, we obtain from~\eqref{Ch2:Entropy} that
\begin{multline*}
\mathsf E'=\frac{p-1}{2\,p}\,\mathsf I=\frac{p-1}{2\,p}\,(p+1)^2\ird{|\nabla f|^2}\\
\ge\frac{p-1}{2\,p}\,(p+1)^2\,\(\mathcal C_{\mathrm{GNS}}(p)\)^\frac2\theta\,\nrm f{2\,p}^\frac2\theta\,\nrm f{p+1}^{-\frac{2\,(1-\theta)}\theta}\,.
\end{multline*}
using~\index{Gagliardo-Nirenberg-Sobolev inequalities}{\eqref{GNS}} with $\theta$ given by~\eqref{Ch1:theta}. Hence
\be{Ch2:Gronwall}
\mathsf E'\ge C_0\,\mathsf E^{1-\,\frac{m-m_c}{1-m}}
\ee
with $C_0:=\frac{p-1}{2\,p}\,(p+1)^2\,\(\mathcal C_{\mathrm{GNS}}(p)\)^\frac2\theta\,M^\frac{(d+2)\,m-d}{d\,(1-m)}$ and $M=\nrm{u_0}1$.
%-----------------------------------------------------------------------
\begin{lemma}\label{Lem:Gronwall} Assume that $d\ge1$, $m\in[m_1,1)\cap(1/2,1)$ and consider a solution of~\eqref{FD} with initial datum $u_0\in\mathrm L^1_+\!\(\R^d,(1+|x|^2)\,\dx\)$ such that $u_0^m\in\mathrm L^1(\R^d)$. Then
\be{Ch2:GrowthEntropy}
\ird{u^m(t,x)}\ge\(\(\ird{u_0^m}\)^\frac{m-m_c}{1-m}+\tfrac{(1-m)\,C_0}{m-m_c}\,t\)^\frac{1-m}{m-m_c}\quad\forall\,t\ge0\,.
\ee
\end{lemma}
%-----------------------------------------------------------------------
The proof follows from an integration of~\eqref{Ch2:Gronwall} on $[0,t]$. It turns out that the estimate~\eqref{Ch2:GrowthEntropy} is sharp. Let us consider the \emph{\index{Barenblatt self-similar solutions}{Barenblatt self-similar solution}} of~\eqref{FD} of mass $M$ defined by
\be{BarenblattM}
B\big(t\,,\,x\,;\,M\big):=\(\tfrac M\Mstar\)^\frac2\alpha\,\frac{\muscal^d}{R(t)^d}\,\mB\(\(\tfrac M\Mstar\)^\frac{1-m}\alpha\frac{\muscal\,x}{R(t)}\)\,,
\ee
where $\Mstar :=\ird{\mB}$,
\be{mB}
\mB(x):=\(1+|x|^2\)^\frac1{m-1}\quad\forall\,x\in\R^d
\ee
and where $\muscal$ and $R$ are given respectively
\be{mu}
\muscal:=\(\tfrac{1-m}{2\,m}\)^{1/\alpha}
\ee
and
\be{R}
R(t):=(1+\alpha\,t)^{1/\alpha}\,,\quad\,\alpha:=d\,(m-m_c)\,.
\ee
Notice that
\be{Mstar}
\Mstar=\pi^\frac d2\,\frac{\Gamma\(\frac1{1-m}-\frac d2\)}{\Gamma\(\frac1{1-m}\)}
\ee
and $\Mstar=\nrm{\mathsf g}{2p}^{2p}$ because $\mathsf g$ defined as in~\eqref{Aubin.Talenti} is such that $\mB=\mathsf g^{2p}$.
%-----------------------------------------------------------------------
\begin{lemma}\label{Lem:Barenblatt-GNS} Assume that $d\ge1$, $m\in[m_1,1)$ and take $M>0$. Then $u=B$ as defined in~\eqref{BarenblattM} solves~\eqref{FD}, with equality in~\eqref{Ch2:GrowthEntropy}.\end{lemma}
%-----------------------------------------------------------------------
\begin{proof} Equation~\eqref{FD} applied to $B$ means that
$\mB$ solves
\[
\nabla\cdot\Big(\mB\(\nabla\mB^{m-1}+2\,x\)\Big)=0\quad\forall\,x\in\R^d\,,
\]
which is easy to check. The function $\mathsf g$ such that $\mathsf g^{2\,p}=\mB$ is optimal for the \idx{Gagliardo-Nirenberg-Sobolev inequalities}~\eqref{GNS}. Taking into account the expression of~$R$ and~\eqref{BarenblattM}, we obtain that $B(\cdot,\cdot;M)$ is also optimal in~\eqref{Ch2:Gronwall} for any $t\ge0$, which means that we have equality in~\eqref{Ch2:Gronwall} and, as a consequence, in~\eqref{Ch2:GrowthEntropy}.\end{proof}

%%%%%%%%%%%%%%%%%%%%%%%%%%%%%%%%%%%%%%%%%%%%%%%%%%%%%%%%%%%%%%%%%%%%%%%%
%%%%%%%%%%%%%%%%%%%%%%%%%%%%%%%%%%%%%%%%%%%%%%%%%%%%%%%%%%%%%%%%%%%%%%%%
\section{R\'enyi entropy powers}\label{Sec:Entropy-Renyi}

\index{R\'enyi entropy powers}{}The results of Lemma~\ref{Lem:Gronwall} and Lemma~\ref{Lem:Barenblatt-GNS} rely on the \idx{Gagliardo-Nirenberg-Sobolev inequalities}~\eqref{GNS} and the corresponding optimality case. It turns out that the evolution equation~\eqref{FD} can be used to establish the inequalities and identify the equality cases. Let us sketch the main steps of the proof of such results, at formal level. We assume here that the solution of~\eqref{FD} has all regularity and decay properties needed to justify the integrations by parts.

%%%%%%%%%%%%%%%%%%%%%%%%%%%%%%%%%%%%%%%%%%%%%%%%%%%%%%%%%%%%%%%%%%%%%%%%
\subsection{Pressure variable and decay of the Fisher information}\label{Sec:Pressure-Fisher}~

Let us introduce the \emph{\idx{pressure variable}}
\[
\P:=\frac m{1-m}\,u^{m-1}\,.
\]
The \emph{\idx{Fisher information}} can be rewritten as
\[
\mathsf I[u]=\ird{u\,|\nabla\P|^2}\,.
\]
If $u$ is a solution of~\eqref{FD}, that is, of
\[
\frac{\partial u}{\partial t}+\nabla\cdot\(u\,\nabla\P\)=0\,,
\]
then $\P$ solves
\be{Eqn:p}
\frac{\partial\P}{\partial t}=(1-m)\,\P\,\Delta\P-|\nabla\P|^2\,.
\ee
Using~\eqref{FD} and~\eqref{Eqn:p}, we compute the time derivative of $\mathsf I(t)=\mathsf I[u(t,\cdot)]$ as
\be{Iprime}
\mathsf I'=\ird{\Delta(u^m)\,|\nabla\P|^2}+\,2\ird{u\,\nabla\P\cdot\nabla\Big((m-1)\,\P\,\Delta\P+|\nabla\P|^2\Big)}\,.
\ee
As a preliminary step, we start by establishing a useful identity. This result and its proof are taken from~\cite{DEL-JEPE} and reproduced here with no changes, for sake of completeness. Similar computations can be found in earlier papers like~\cite{Dolbeault_2016} or~\cite[Appendix~B]{MR3200617}.
%-----------------------------------------------------------------------
\begin{lemma}[\cite{DEL-JEPE}]\label{Lem:DerivFisher} Let $m\in(0,1)$. Assume that $u$ is a smooth, positive and \emph{sufficiently decreasing} function on $\R^d$, with smooth and \emph{sufficiently decreasing} derivatives. If we let $\P=\frac m{1-m}\,u^{m-1}$, then we have
\begin{multline}\label{BLW}
\ird{\Delta(u^m)\,|\nabla\P|^2}+\,2\ird{u\,\nabla\P\cdot\nabla\Big((m-1)\,\P\,\Delta\P+|\nabla\P|^2\Big)}\\
=-\,2\ird{u^m\,\Big(\|\mathrm D^2\P\|^2-\,(1-m)\,(\Delta\P)^2\Big)}\,.
\end{multline}
\end{lemma}
%-----------------------------------------------------------------------
If we take into account~\eqref{pm} and consider $u=f^{2\,p}$, the definition of the \idx{pressure variable} $\P$ is the same as in~\eqref{Eqn:h-BE} and the result of Lemma~\ref{Lem:DerivFisher} coincides with Lemma~\ref{Lem:Carre}. Let us give a more detailed proof.
\begin{proof} Let us start by computing
\begin{eqnarray*}
&&\hspace*{-24pt}\ird{\Delta(u^m)\,|\nabla\P|^2}+\,2\ird{u\,\nabla\P\cdot\nabla\Big((1-m)\,\P\,\Delta\P-|\nabla\P|^2\Big)}\\
&=&\ird{u^m\,\Delta\,|\nabla\P|^2}+\,2\,(1-m)\ird{u\,\P\,\nabla\P\cdot\nabla\Delta\P}\\
&&+\,2\,(1-m)\ird{u\,|\nabla\P|^2\,\Delta\P}-\,2\ird{u\,\nabla\P\cdot\nabla\,|\nabla\P|^2}\\
&=&-\ird{u^m\,\Delta\,|\nabla\P|^2}+\,2\,(1-m)\ird{u\,\P\,\nabla\P\cdot\nabla\Delta\P}\\
&&+\,2\,(1-m)\ird{u\,|\nabla\P|^2\,\Delta\P}
\end{eqnarray*}
where the last line is given by the observation that $u\,\nabla\P=-\,\nabla(u^m)$ and an integration by parts:
\[
-\ird{u\,\nabla\P\cdot\nabla\,|\nabla\P|^2}=\ird{\nabla(u^m)\cdot\nabla\,|\nabla\P|^2}=-\ird{u^m\,\Delta\,|\nabla\P|^2}\,.
\]
1) Using the elementary identity
\[
\frac12\,\Delta\,|\nabla\P|^2=\|\mathrm D^2\P\|^2+\nabla\P\cdot\nabla\Delta\P\,,
\]
we get that
\[
\ird{u^m\,\Delta\,|\nabla\P|^2}=2\ird{u^m\,\|\mathrm D^2\P\|^2}+2\ird{u^m\,\nabla\P\cdot\nabla\Delta\P}\,.
\]
2) Since $u\,\nabla\P=-\,\nabla(u^m)$, an integration by parts gives
\begin{multline*}
\ird{u\,|\nabla\P|^2\,\Delta\P}=-\ird{\nabla(u^m)\cdot\nabla\P\,\Delta\P}\\
=\ird{u^m\,(\Delta\P)^2}+\ird{u^m\,\nabla\P\cdot\nabla\Delta\P}
\end{multline*}
and with $u\,\P=\frac m{1-m}\,u^m$ we find that
\begin{multline*}
2\,(1-m)\ird{u\,\P\,\nabla\P\cdot\nabla\Delta\P}+\,2\,(1-m)\ird{u\,|\nabla\P|^2\,\Delta\P}\\
=2\,(1-m)\ird{u^m\,(\Delta\P)^2}+2\ird{u^m\,\nabla\P\cdot\nabla\Delta\P}\,.
\end{multline*}
Collecting terms establishes~\eqref{BLW}.
\end{proof}
In Lemma~\ref{Lem:DerivFisher}, one can check that the \idx{Barenblatt self-similar solutions} defined by~\eqref{BarenblattM} are \emph{sufficiently decreasing} functions on $\R^d$ in the sense that $B$ and its derivatives decay as $|x|\to+\infty$ sufficiently fast so that no asymptotic boundary term appears in the integrations by parts of the proof of Lemma~\ref{Lem:DerivFisher}.
%-----------------------------------------------------------------------
\begin{corollary}\label{Cor:DerivFisher} Let $d\ge1$ and assume that $m\in[m_1,1)$. If $u$ solves~\eqref{FD} with smooth initial datum $u_0\in\mathrm L^1_+(\R^d)$ such that $\ird{|x|^2\,u_0}<+\infty$ and if, for any $t\ge0$, $u(t,\cdot)$ is a smooth and \emph{sufficiently decreasing} function on $\R^d$, then
\[
\mathsf I'=-\,2\ird{{u^m\,\left\|\,\mathrm D^2\P-\tfrac1d\,\Delta\P\,\mathrm{Id}\,\right\|^2}}-\,2\,(m-m_1)\ird{{u^m\,(\Delta\P)^2}}\quad\forall\,t\ge0\,.
\]
\end{corollary}
%-----------------------------------------------------------------------
\begin{proof} The result follows from~\eqref{Iprime},~\eqref{BLW} and as a consequence of the elementary identity
\[
\|\mathrm D^2\P\|^2=\tfrac1d\,(\Delta\P)^2+\left\|\,\mathrm D^2\P-\tfrac1d\,\Delta\P\,\mathrm{Id}\,\right\|^2\,.
\]
\end{proof}

%%%%%%%%%%%%%%%%%%%%%%%%%%%%%%%%%%%%%%%%%%%%%%%%%%%%%%%%%%%%%%%%%%%%%%%%
\subsection{R\'enyi entropy powers and interpolation inequalities}\label{Sec:GNS-Renyi}

With the notations of Section~\ref{Sec:GNS-FDE}, the Ga\-gliardo-Nirenberg-Sobolev inequalities are equivalent to
\[
\mathsf I[u]^\theta\,\mathsf E[u]^{2\,\frac{1-\theta}{p+1}}\ge(p+1)^{2\,\theta}\,\big(\mathcal C_{\mathrm{GNS}}(p)\big)^{2\,\theta}\,M^\frac{2\,\theta}p\,.
\]
While the right-hand side is independent of $t$ if $u$ solves~\eqref{FD}, it turns out that the left-hand side is monotone with respect to $t$. In order to prove this, it is convenient to rephrase this property using the \emph{\idx{pressure variable}} $\P$.
%-----------------------------------------------------------------------
\begin{lemma}\label{Lem:Renyi} Let $d\ge1$ and assume that $m\in[m_1,1)$. If $u$ solves~\eqref{FD} with smooth initial datum $u_0\in\mathrm L^1_+(\R^d)$ such that $\ird{|x|^2\,u_0}<+\infty$ and if $u(t,\cdot)$ is a smooth and \emph{sufficiently decreasing} function on $\R^d$, then for any $t\ge0$ we have
\begin{multline*}
-\frac{\mathsf I}{2\,\theta}\frac{\rd}{\dt}\log\(\mathsf I^\theta\,\mathsf E^{2\,\frac{1-\theta}{p+1}}\)\\
=\ird{u^m\,\left\|\,\mathrm D^2\P-\frac 1d\,\Delta\P\,\mathrm{Id}\,\right\|^2}
+(m-m_1)\ird{u^m\,\left|\Delta\P+\frac{\mathsf I}{\mathsf E}\right|^2}\,.
\end{multline*}
\end{lemma}
%-----------------------------------------------------------------------
\begin{proof} The result follows from
\[
-\frac1\theta\,\frac{\rd}{\dt}\log\(\mathsf I^\theta\,\mathsf E^{2\,\frac{1-\theta}{p+1}}\)=-\frac{\mathsf I'}{\mathsf I}-\frac{1-\theta}{p+1}\,\frac2\theta\,\frac{\mathsf E'}{\mathsf E}=-\frac1{\mathsf I}\(\mathsf I'+\frac{1-\theta}{p+1}\,\frac2\theta\,(1-m)\,\frac{\mathsf I^2}{\mathsf E}\)
\]
with the observation that $\ird{u^m\,\Delta\P}=\ird{u\,|\nabla\P|^2}$ and that $\frac{1-\theta}{p+1}\,\frac2\theta\,(1-m)=2\,(m-m_1)$, and from Corollary~\ref{Cor:DerivFisher}.\end{proof}
%-----------------------------------------------------------------------
\begin{proposition}\label{Prop:Limit} Let $d\ge1$ and assume that $m\in(m_1,1)$. If $u$ solves~\eqref{FD} with smooth initial datum $u_0\in\mathrm L^1_+(\R^d)$ such that $\ird{|x|^2\,u_0}<+\infty$ and if, for any $t\ge0$, $u(t,\cdot)$ is a smooth and \emph{sufficiently decreasing} function on $\R^d$, then
\begin{multline*}
\lim_{t\to+\infty}\frac{\mathsf I[u(t,\cdot)]^\theta\,\mathsf E[u(t,\cdot)]^{2\,\frac{1-\theta}{p+1}}}{M^\frac{2\,\theta}p}=\lim_{t\to+\infty}\frac{\mathsf I[B(t,\cdot;M)]^\theta\,\mathsf E[B(t,\cdot;M)]^{2\,\frac{1-\theta}{p+1}}}{M^\frac{2\,\theta}p}\\
=\frac{\mathsf I[\mB]^\theta\,\mathsf E[\mB]^{2\,\frac{1-\theta}{p+1}}}{\nrm{\mB}1^\frac{2\,\theta}p}=(p+1)^{2\,\theta}\,\(\mathcal C_{\mathrm{GNS}}(p)\)^{2\,\theta}\,.
\end{multline*}
\end{proposition}
%-----------------------------------------------------------------------
For solutions of~\eqref{FD} with initial data in $\mathrm L^1_+(\R^d)$, it is well known that \idx{Barenblatt self-similar solutions} defined by~\eqref{BarenblattM} play the role of an attractor in various norms if $m\in(m_1,1)$. However, a direct proof that
\[
\lim_{t\to+\infty}\frac{\mathsf I[u(t,\cdot)]}{\mathsf I[B(t,\cdot;M)]}=1\quad\mbox{and}\quad\lim_{t\to+\infty}\frac{\mathsf E[u(t,\cdot)]}{\mathsf E[B(t,\cdot;M)]}=1
\]
is delicate. The scheme of a proof of a similar result can be found in Section~\ref{Sec:RFDE}, in the easier framework of \idx{self-similar variables} and relative entropies. The equivalence of the two approaches is discussed in details in~\cite{DEL-JEPE}. With the results of Lemma~\ref{Lem:Renyi} and Proposition~\ref{Prop:Limit} in hand, we deduce the following estimate for a solution of~\eqref{FD}.
%-----------------------------------------------------------------------
\begin{corollary}\label{Cor:GNSbyRenyi} Let $d\ge1$ and assume that $m\in[m_1,1)$. If $u$ solves~\eqref{FD} with smooth initial datum $u_0\in\mathrm L^1_+(\R^d)$ such that $\ird{|x|^2\,u_0}<+\infty$ and if, for any $t\ge0$, $u(t,\cdot)$ is a smooth and \emph{sufficiently decreasing} function on $\R^d$, then
\[
\frac{\mathsf I[u(t,\cdot)]^\theta\,\mathsf E[u(t,\cdot)]^{2\,(1-\theta)}}{M^\frac{2\,\theta}p}\ge\frac{\mathsf I[\mB]^\theta\,\mathsf E[\mB]^{2\,(1-\theta)}}{\nrm\mB1^\frac{2\,\theta}p}=(p+1)^{2\,\theta}\,\big(\mathcal C_{\mathrm{GNS}}(p)\big)^{2\,\theta}\quad\forall\,t\ge0\,.
\]
\end{corollary}
%-----------------------------------------------------------------------
Altogether, Corollary~\ref{Cor:GNSbyRenyi} provides us with the scheme of a second proof of~\index{Gagliardo-Nirenberg-Sobolev inequalities}{\eqref{GNS}}, under minor restrictions like $\ird{|x|^2\,u_0}<+\infty$. Notice that the estimate holds up to $t=0$, so that it holds true for an arbitrary initial datum: written for $u_0$, the inequality is the generic form of the inequality, which can be extended to an arbitrary function in $\mathcal H_p(\R^d)$ by density.

According to~\cite{MR3200617}, generalized \emph{\idx{R\'enyi entropy powers}} are defined as
\[
\mathsf F[u]:=\mathsf E[u]^{\frac2\theta\,(1-\theta)+1}=\mathsf E[u]^{\frac 2d\,\frac1{1-m}-1}\,.
\]
If $u$ solves~\eqref{FD}, it is straightforward that the concavity of $t\mapsto\mathsf F[u(t,\cdot)]$ is a consequence of Lemma~\ref{Lem:Renyi}. The exponent in the definition of $\mathsf F[u]$ is chosen in order that $t\mapsto\mathsf F[B(t,\cdot;M)]$ is an affine function. We do not really use further properties of the \idx{R\'enyi entropy powers}, but it is convenient to refer to the computations of this section as the \emph{R\'enyi entropy} method in order to distinguish them from other entropy methods.

%%%%%%%%%%%%%%%%%%%%%%%%%%%%%%%%%%%%%%%%%%%%%%%%%%%%%%%%%%%%%%%%%%%%%%%%
\subsection{About the integrations by parts}\label{Sec:IPP}

All computations of Sections~\ref{Sec:EntropyGrowth} and~\ref{Sec:Entropy-Renyi} are formal as we did not justify the integrations by parts. In Section~\ref{Sec:FDE}, we will consider a \emph{\idx{relative entropy}} method which has been proved to be equivalent to the method based on \idx{R\'enyi entropy powers} in~\cite{DEL-JEPE}, and for which integrations by parts are much easier to justify. This is why we have kept the proofs of Lemma~\ref{Lem:Renyi}, Proposition~\ref{Prop:Limit} and Corollary~\ref{Cor:GNSbyRenyi} formal.

%%%%%%%%%%%%%%%%%%%%%%%%%%%%%%%%%%%%%%%%%%%%%%%%%%%%%%%%%%%%%%%%%%%%%%%%
%%%%%%%%%%%%%%%%%%%%%%%%%%%%%%%%%%%%%%%%%%%%%%%%%%%%%%%%%%%%%%%%%%%%%%%%
\section{The fast diffusion equation in \idx{self-similar variables}}\label{Sec:RFDE}

%%%%%%%%%%%%%%%%%%%%%%%%%%%%%%%%%%%%%%%%%%%%%%%%%%%%%%%%%%%%%%%%%%%%%%%%
\subsection{Rescaling and \idx{self-similar variables}}\label{Sec:SSV}

We rewrite the solutions of \eqref{FD} in the scale given by the \emph{\index{Barenblatt self-similar solutions}{Barenblatt self-similar solution}} $B(\cdot,\cdot;M)$, with same mass $M$, which is defined by~\eqref{BarenblattM}: if $v$ is such that
\be{SelfSimilarChangeOfVariables}
u(t,x)=\frac{\muscal^d}{R(t)^d}\,v\(\frac12\,\log R(t),\frac{\muscal\,x}{R(t)}\)
\ee
where $\muscal$ and $R$ are given respectively by~\eqref{mu} and~\eqref{R}, then $v$ solves the \emph{\idx{fast diffusion equation}} in \idx{self-similar variables}
\be{FDr}
\frac{\partial v}{\partial t}+\nabla\cdot\(v\,\nabla v^{m-1}\)=2\,\nabla\cdot(x\,v)\,,\quad v(t=0,\cdot)=v_0
\ee
with nonnegative initial datum $v_0=\muscal^{-d}\,u_0(\cdot/\muscal)\in\mathrm L^1(\R^d)$. Notice that $B(\cdot,\cdot;M)$ is transformed into a stationary solution of~\eqref{FDr}. Assuming that $B$ has a finite mass introduces the limitation $m>m_c$, but we shall further assume that $m\in[m_1,1)$ in view of the application to the \idx{Gagliardo-Nirenberg-Sobolev inequalities}~\eqref{GNS}. With~$m$ and~$p$ related by~\eqref{pm} and $d\ge3$, the range $m\ge m_1$ indeed means $p\le p_\star$ with the notations of Chapter~\ref{Chapter-1}. As a special case, Equation~\eqref{FDr} admits the Barenblatt profile $\mB$ defined by~\eqref{mB} as a stationary solution of mass~$\Mstar$.

%%%%%%%%%%%%%%%%%%%%%%%%%%%%%%%%%%%%%%%%%%%%%%%%%%%%%%%%%%%%%%%%%%%%%%%%
\subsection{The entropy - entropy production inequality}\label{Sec:FDE-Def}

We shall assume that the mass of $v_0$ is taken equal to $\Mstar$. It is well known that the mass is conserved, \emph{i.e.}, if $v$ solves~\eqref{FDr} with initial datum $v_0$, then
\[\label{mass.condition}
\ird{v(t,x)}=\Mstar\quad\forall\,t\ge0\,.
\]
If the initial datum is centered, in the sense that $\ird{x\,v_0(x)}=0$, then the position of the \idx{center of mass} is also conserved and we have
\[\label{center.mass.condition}
\ird{x\,v(t,x)}=0\quad\forall\,t\ge0\,.
\]
The \emph{\idx{free energy}} (or \emph{\idx{relative entropy}}) and the \emph{\idx{Fisher information}} (or \emph{relative entropy production}) are defined respectively by
\[
\mathcal F[v]:=\frac1{m-1}\ird{\(v^m-\mB^m-m\,\mB^{m-1}\,(v-\mB)\)}
\]
and
\[
\mathcal I[v]:=\frac m{1-m}\ird{v\,\left|\nabla v^{m-1}-\nabla\mB^{m-1}\right|^2}\,.
\]
With $v=|f|^{2\,p}$ and $p$ given in terms of $m$ by~\eqref{pm}, Inequality~\eqref{GNS-Intro} is equivalent to the \index{entropy - entropy production inequality}{\emph{entropy - entropy production} inequality}
\be{entropy.eep}
\mathcal I[v]\ge4\,\mathcal F[v]
\ee
by Lemma~\ref{Lem:BasicEntropyProp},~(ii).

If $v$ solves~\eqref{FDr}, it is a straightforward computation to check that
\be{EEP-F-I}
\frac{\rd}{\dt}\mathcal F[v(t,\cdot)]=-\,\mathcal I[v(t,\cdot)]
\ee
after one integration by parts (which has to be justified), and as a consequence, we obtain that
\[
\mathcal F[v(t,\cdot)]\le\mathcal F[v_0]\,e^{-4\,t}\quad\forall\,t\ge0\,.
\]
This estimate is equivalent to~\eqref{entropy.eep}, as one can check by writing at $t=0$ that
\[
\mathcal I[v(t,\cdot)]=-\,\frac{\rd}{\dt}\mathcal F[v(t,\cdot)]\ge-\,\frac{\rd}{\dt}\(\mathcal F[v_0]\,e^{-4\,t}\)=4\,\mathcal F[v_0]\,e^{-4\,t}\quad\forall\,t\ge0\,.
\]
On the other hand, exactly as in the case of the \idx{R\'enyi entropy powers}, the flow associated with~\eqref{FDr} can be used to establish the non-scale invariant form of the \idx{Gagliardo-Nirenberg-Sobolev inequalities}~\eqref{GNS-Intro}. The key estimate goes as follows. Let us define a \emph{relative \idx{pressure variable}} $\relativePressure$ by
\be{RelativePressure}
\relativePressure(t,x):=v^{m-1}(t,x)-\,|x|^2\quad\forall\,(t,x)\in\R^+\times\R^d\,.
\ee
The \emph{relative \idx{Fisher information}} can be rewritten as
\[
\mathcal I[v]=\ird{v\,|\nabla\relativePressure|^2}\,.
\]
%-----------------------------------------------------------------------
\begin{proposition}\label{Prop:EEP} Let $d\ge1$ and assume that $m\in(m_1,1)$. If $v$ solves~\eqref{FDr} with initial datum $v_0\in\mathrm L^1_+(\R^d)$ such that $\ird{|x|^2\,v_0}<+\infty$, then for any $t\ge0$ we have
\begin{multline*}
-\,\frac12\,\frac{\rd}{\dt}\Big(\mathcal I[v(t,\cdot)]-4\,\mathcal F[v(t,\cdot)]\Big)\\
\ge\ird{v^m\,\left\|\,\mathrm D^2\relativePressure-\frac 1d\,\Delta\relativePressure\,\mathrm{Id}\,\right\|^2}
+(m-m_1)\ird{v^m\,|\Delta\relativePressure|^2}\,.
\end{multline*}
\end{proposition}
%-----------------------------------------------------------------------
This result holds with no assumption on the regularity nor on the decay of~$v$. The strategy of the proof goes as follows. Along the flow, $\mathcal I-4\,\mathcal F$ is monotone nonincreasing and strictly monotone unless $\Delta\relativePressure=0$ on the support of $v$, in which case $v$ is a stationary \idx{Barenblatt function}. The case $m=m_1$ is, as usual, more delicate and will not be covered here, but one can formally notice that $\mathrm D^2\relativePressure-\frac 1d\,\Delta\relativePressure=0$ also proves that $v$ is a stationary \idx{Barenblatt function}, but takes into account the conformal invariance associated to Sobolev's inequality. Details on the proof of Proposition~\ref{Prop:EEP} are given below in Section~\ref{Sec:carre}.

An important consequence of Proposition~\ref{Prop:EEP} is the fact that
\be{Eqn:Fisher}
\frac{\rd}{\dt}\mathcal I[v(t,\cdot)]\le-\,4\,\mathcal I[v(t,\cdot)]\,,
\ee
which guarantees that
\[
\mathcal I[v(t,\cdot)]\le\mathcal I[v_0]\,e^{-4\,t}\quad\forall\,t\ge0\,.
\]
As a consequence, this means that $\lim_{t\to+\infty}\mathcal I[v(t,\cdot)]=0$. In addition, we have
\be{Limits:EI}
\lim_{t\to+\infty}\mathcal F[v(t,\cdot)]=0\,.
\ee
Indeed, since $t\mapsto \mathcal F[v(t,\cdot)]$ is monotone, it has a limit as $t\to+\infty$. We then learn from Proposition~\ref{Prop:EEP} that the corresponding relative pressure $\mathcal P$ is strongly converging to $\mathcal P_\infty$ satisfying $\Delta \mathcal P_\infty=0$ a.e.~on the support of the limit $v_\infty$ of~$v$, which implies that $v_\infty$ is the Barenblatt profile and finally~\eqref{Limits:EI} follows from the strong convergence of $v(t,\cdot)$ that can be found for instance in~\cite{Blanchet2009}. Now, since $t\mapsto\mathcal I[v(t,\cdot)]-4\,\mathcal F[v(t,\cdot)]$ is monotone non-increasing, we obtain
\[
\mathcal I[v(t,\cdot)]-4\,\mathcal F[v(t,\cdot)]\ge0\quad\forall\,t\ge0
\]
and, as a special case at $t=0$. This proves the inequality $\mathcal I[v_0]-4\,\mathcal F[v_0]\geq 0$ for an arbitrary function $v_0$ satisfying the assumptions of Proposition~\ref{Prop:EEP}. In other words, this is a proof of~\eqref{GNS-Intro}, and then equivalently of~\index{Gagliardo-Nirenberg-Sobolev inequalities}{\eqref{GNS}}, based on the flow of~\eqref{FDr}. This strategy is a nonlinear version of the \emph{\idx{carr\'e du champ}} method, also known as the \emph{Bakry-Emery} method in the literature. More details are given in Section~\ref{Sec:Bib2}.

%%%%%%%%%%%%%%%%%%%%%%%%%%%%%%%%%%%%%%%%%%%%%%%%%%%%%%%%%%%%%%%%%%%%%%%%
\subsection{The carr\'e du champ method}\label{Sec:carre}

The following proof is inspired, up to the normalization of the \idx{pressure variable}, by~\cite[Section~2]{Dolbeault2013917}. Also see~\cite[Proof~of~Theorem~2.4, pp.~33-36]{MR3497125} and~\cite[Section~2.2]{DEL-JEPE} for similar computations. For more clarity we first adapt the proof done for \idx{R\'enyi entropy powers} to the case of the \idx{relative entropy} $\mathcal F$ using the flow~\eqref{FDr} without taking care of the asymptotic boundary conditions, \emph{i.e.}, by assuming that we deal only with smooth and sufficiently decreasing solutions. Then we briefly sketch an approximation scheme which justifies the computations in the general case.

%.......................................................................
\smallskip\subsubsection{Smooth and sufficiently decreasing solutions}
If $v$ solves~\eqref{FDr}, that is, of
\[
\frac{\partial v}{\partial t}+\nabla\cdot\(v\,\nabla\relativePressure\)=0\,,
\]
then the relative \idx{pressure variable} defined by~\eqref{RelativePressure} solves
\be{Eqn:Q}
\frac{\partial\relativePressure}{\partial t}=(1-m)\,v^{m-2}\,\nabla\cdot\(v\,\nabla\relativePressure\)\,.
\ee
As in Section~\ref{Sec:GNS-Renyi}, let us assume that, for any $t\ge0$, $v(t,\cdot)$ is a smooth and \emph{sufficiently decreasing} function on $\R^d$, so that we can again do the integrations by parts without taking care of asymptotic boundary terms.

Let us compute
\begin{align*}
\frac{\rd}{\dt}\ird{v\,|\nabla\relativePressure|^2}=&\ird{\frac{\partial v}{\partial t}\,|\nabla\relativePressure|^2}+2\ird{v\,\nabla\relativePressure\cdot\nabla\(\frac{\partial\relativePressure}{\partial t}\)}\\
=&\ird{v\,\nabla\relativePressure\,\cdot\nabla\(|\nabla\relativePressure|^2\)}\\
&-\,2\ird{v\,\nabla\relativePressure\cdot\nabla\(\nabla\relativePressure\cdot\nabla v^{m-1}+(m-1)\,v^{m-1}\,\Delta\relativePressure\)}
\end{align*}
using~\eqref{FDr} and~\eqref{Eqn:Q}. By definition of $\nabla\relativePressure$, we have
\begin{align*}
\frac{\rd}{\dt}\ird{v\,|\nabla\relativePressure|^2}=&\ird{v\,\nabla\relativePressure\,\cdot\nabla\(|\nabla\relativePressure|^2\)}\\
&-\,2\ird{v\,\nabla\relativePressure\cdot\nabla\(|\nabla\relativePressure|^2+2\,\nabla\relativePressure\cdot x+(m-1)\,v^{m-1}\,\Delta\relativePressure\)}\\
=&-\ird{v\,\nabla\relativePressure\,\cdot\nabla\(|\nabla\relativePressure|^2\)}\\
&\hspace*{12pt}-\,2\ird{v\,\nabla\relativePressure\cdot\nabla\(2\,\nabla\relativePressure\cdot x+(m-1)\,v^{m-1}\,\Delta\relativePressure\)}\\
=&\ird{\tfrac{m-1}m\,v^m\,\Delta\(|\nabla\relativePressure|^2\)}+2\ird{v\,x\,\cdot\nabla\(|\nabla\relativePressure|^2\)}\\
&\hspace*{12pt}-\,4\ird{v\,\nabla\relativePressure\cdot\nabla(\nabla\relativePressure\cdot x)}\\
&\hspace*{12pt}+\,2\,(1-m)\,\ird{\(\tfrac{1-m}m\,v^m\(\Delta\relativePressure\)^2+\tfrac1m\,v^m\,\nabla\relativePressure\cdot\nabla(\Delta\relativePressure)\)}
\end{align*}
using $v\,\nabla\relativePressure=\frac{m-1}m\,\nabla v^m-\,2\,x\,v$ and $0=\ird{\nabla\cdot\(v^m\,\nabla\relativePressure\,\Delta\relativePressure\)}$. Using the elementary identity
\[
\frac12\,\Delta\(|\nabla\relativePressure|^2\)=\left\|\mathrm D^2\relativePressure\right\|^2+\nabla\relativePressure\cdot\nabla(\Delta\relativePressure)\,,
\]
we obtain that
\[
\ird{v^m\,\Delta\(|\nabla\relativePressure|^2\)}=2\ird{v^m\,\left\|\mathrm D^2\relativePressure\right\|^2}+2\ird{v^m\,\nabla\relativePressure\cdot\nabla(\Delta\relativePressure)}\,.
\]
Using $\partial^2\relativePressure/\partial x_i\partial x_j=\partial^2\relativePressure/\partial x_j\partial x_i$ we can also write
\[
2\ird{v\,x\,\cdot\nabla\(|\nabla\relativePressure|^2\)}-\,4\ird{v\,\nabla\relativePressure\cdot\nabla(\nabla\relativePressure\cdot x)}=-\,4\ird{v\,|\nabla\relativePressure|^2}\,.
\]
Therefore, we have shown that
\begin{multline*}
\frac{\rd}{\dt}\ird{v\,|\nabla\relativePressure|^2}+\,4\ird{v\,|\nabla\relativePressure|^2}\\
=-\,2\,\frac{1-m}m\ird{v^m\(\left\|\mathrm D^2\relativePressure\right\|^2-\,(1-m)\,(\Delta\relativePressure)^2\)}\,.
\end{multline*}
This completes the proof of Proposition~\ref{Prop:EEP} at formal level, \emph{i.e.}, up to asymptotic boundary terms arising from the integrations by parts, as a consequence of the elementary identity
\[
\|\mathrm D^2\relativePressure\|^2=\tfrac1d\,(\Delta\relativePressure)^2+\left\|\,\mathrm D^2\relativePressure-\tfrac1d\,\Delta\relativePressure\,\mathrm{Id}\,\right\|^2
\]
and of the definition of $\mathcal I[v]=\frac m{1-m}\ird{v\,|\nabla\relativePressure|^2}$.

%.......................................................................
\smallskip\subsubsection{An approximation method}

In order to prove Proposition~\ref{Prop:EEP}, boundary terms have to be taken into account. As an approximating problem, we consider~\eqref{FDr} restricted to a centered ball $B_R$ with radius $R>0$, large:
\be{FDrr}
\frac{\partial v}{\partial t}+\nabla\cdot\(v\,\nabla\relativePressure\)=0\quad\mbox{in}\quad B_R\,,\quad \nabla\relativePressure\cdot\omega=0\quad\mbox{on}\quad\partial B_R\,,
\ee
where $\omega=x/|x|$ denotes the unit outgoing normal vector to $\partial B_R$. After redoing the above computations and keeping track of the boundary terms, we find an additional term
\[
\tfrac{1-m}m\int_{\partial B_R}\,u^m\(\omega\cdot\nabla\,|\nabla\relativePressure|^2\)\,d\sigma_R\le0
\]
where $d\sigma_R$ denotes the surface measure on $\partial B_R$ induced by Lebesgue's measure on $\R^d$. This terms is indeed nonpositive by \idx{Grisvard's lemma} (see \cite[Lemma~5.1]{MR2533926} or \cite{MR775683}).

As a second step, we have to approximate the problem on $\R^d$ by the solution of~\eqref{FDrr} in the limit as $R\to+\infty$. This makes sense as the \idx{Barenblatt function} $\mathcal B$, restricted to $B_R$ is always a stationary solution (it satisfies the boundary conditions). However, one has for each $R>0$ to adjust the mass of the solution so that it coincides with $\int_{B_R}\mB\,\dx$. Altogether, we obtain that
\begin{multline*}
\frac{\rd}{\dt}\ird{v\,|\nabla\relativePressure|^2}+\,4\ird{v\,|\nabla\relativePressure|^2}\\
\le-\,2\,\frac{1-m}m\ird{v^m\(\left\|\mathrm D^2\relativePressure\right\|^2-\,(1-m)\,(\Delta\relativePressure)^2\)}\,.
\end{multline*}
Notice here that we only have an inequality, as boundary terms were dropped, which also explains why there is only an inequality in the statement of Proposition~\ref{Prop:EEP}.

%.......................................................................
\smallskip\subsubsection{Changes of variables}

The analogy of the computations of the \idx{R\'enyi entropy powers} of Section~\ref{Sec:Entropy-Renyi} with the above computations by the \emph{\idx{carr\'e du champ}} method applied to the solution of~\eqref{FDr} are striking and this is of course not a coincidence. Up to the \emph{\index{self-similar variables}{self-similar change of variables}}~\eqref{SelfSimilarChangeOfVariables}, the computations are formally the same. Details can be found in~\cite[Section~2.3]{DEL-JEPE}. The above approximation method and~\eqref{SelfSimilarChangeOfVariables} are probably the easiest method to make all computations of Section~\ref{Sec:Entropy-Renyi} rigorous. With these remarks the proof of Proposition~\ref{Prop:Limit} can be completed as a consequence of~\eqref{Limits:EI}.

%%%%%%%%%%%%%%%%%%%%%%%%%%%%%%%%%%%%%%%%%%%%%%%%%%%%%%%%%%%%%%%%%%%%%%%%
%%%%%%%%%%%%%%%%%%%%%%%%%%%%%%%%%%%%%%%%%%%%%%%%%%%%%%%%%%%%%%%%%%%%%%%%
\section{Refined entropy estimates}\label{Sec:FDE}

%%%%%%%%%%%%%%%%%%%%%%%%%%%%%%%%%%%%%%%%%%%%%%%%%%%%%%%%%%%%%%%%%%%%%%%%
\subsection{A quotient reformulation}\label{Sec:Quotient}

Let us consider the quotient
\be{ch2:Q}
\mathcal Q[v]:=\frac{\mathcal I[v]}{\mathcal F[v]}\,,
\ee
which is well defined if $v\neq\mB$. As a consequence of~\eqref{entropy.eep}, we have the bound
\be{Q4}
\mathcal Q[v]\ge4\,,
\ee
and this inequality is in fact equivalent to the \emph{\idx{entropy - entropy production inequality}}~\eqref{entropy.eep}. We read from Proposition~\ref{Prop:EEP},~\eqref{EEP-F-I} and~\eqref{Eqn:Fisher} that
\be{EqQ}
\frac{\rd\mathcal Q}{\dt}\le\mathcal Q\,(\mathcal Q-4)\,.
\ee
Let us recall that for any solution $v$ of~\eqref{FDr}, $\frac{\rd}{\dt}\mathcal F[v(t,\cdot)]=-\,\mathcal I[v(t,\cdot)]$ combined with~\eqref{Eqn:Fisher} establishes~\eqref{EqQ}. Our goal in this section is to prove that the bound~\eqref{Q4} can be improved under additional conditions. We distinguish an \emph{\idx{initial time layer}} $(0,T)$ and an \emph{\idx{asymptotic time layer}} $(T,+\infty)$. In the first case (see Section~\ref{Sec:FDE-Initial}) we exploit~\eqref{EqQ} while the improvement for large values of $t$ is based on spectral considerations (see Section~\ref{Sec:FDE-Asymptotic}).

%%%%%%%%%%%%%%%%%%%%%%%%%%%%%%%%%%%%%%%%%%%%%%%%%%%%%%%%%%%%%%%%%%%%%%%%
\subsection{The initial time layer improvement}\label{Sec:FDE-Initial}

On the interval $(0,T)$, we prove a uniform positive lower bound on \hbox{$\mathcal Q[v(t,\cdot)]-4$} if we know that \hbox{$\mathcal Q[v(T,\cdot)]-4$} is  positive. The precise result goes as follows.
%---------------------------------------------------------------------
\begin{lemma}\label{prop.backward} Let $m\in[m_1,1)$. Assume that $v$ is a solution to~\eqref{FDr} with nonnegative initial datum $v_0\in\mathrm L^1(\R^d)$ such that $\mathcal F[v_0]<+\infty$ and $\ird{v_0}=\Mstar$. If for some $\eta>0$ and $T>0$, we have $\mathcal Q[v(T,\cdot)]\ge4+\eta$, \index{initial time layer}{then} we also have
\be{backward.EEP}
\mathcal Q[v(t,\cdot)]\ge\,4+\frac{4\,\eta\,e^{-4\,(T-t)}}{4+\eta-\eta\,e^{-4\,(T-t)}}\quad\forall\,t\in[0,T]\,.
\ee\end{lemma}
%---------------------------------------------------------------------
Notice that the right-hand side in~\eqref{backward.EEP} is monotone increasing in $t$, so that we also have the lower bound
\[
\mathcal Q[v(t,\cdot)]\ge\,4+\frac{4\,\eta\,e^{-4\,T}}{4+\eta-\eta\,e^{-4\,T}}\quad\forall\,t\in[0,T]\,.
\]
\begin{proof} The estimate~\eqref{backward.EEP} follows by integrating the Bernoulli differential inequality~\eqref{EqQ} on the interval $[t,T]$.\end{proof}

%%%%%%%%%%%%%%%%%%%%%%%%%%%%%%%%%%%%%%%%%%%%%%%%%%%%%%%%%%%%%%%%%%%%%%%%
\subsection{The asymptotic time layer improvement}\label{Sec:FDE-Asymptotic}

Let us define the \emph{\idx{linearized free energy}} and the \emph{\idx{linearized Fisher information}} by
\[\label{linearized.fisher}
\mathsf F[\h]:=\frac m2\ird{|\h|^2\,\mB^{2-m}}\quad\mbox{and}\quad\mathsf I[\h]:=m\,(1-m)\ird{|\nabla \h|^2\,\mB}\,,
\]
in such a way that
\[
\mathsf F[\h]=\lim_{\varepsilon\to0}\varepsilon^{-2}\,\mathcal F[\mB+\varepsilon\,\mB^{2-m}\,\h]\quad\mbox{and}\quad\mathsf I[\h]=\lim_{\varepsilon\to0}\varepsilon^{-2}\,\mathcal I[\mB+\varepsilon\,\mB^{2-m}\,\h]\,.
\]
By the \emph{Hardy-Poincar\'e inequality}~\eqref{Hardy-Poincare} if $d\ge1$ and $m\in(m_1,1)$, then for any function $\h\in\mathrm L^2(\R^d,\mB^{2-m}\,\dx)$ such that $\nabla\h\in\mathrm L^2(\R^d,\mB\,\dx)$ and \hbox{$\ird{\h\,\mB^{2-m}}=0$}, we have
\[
\mathsf I[\h]\ge4\,\mathsf F[\h]\,.
\]
According to the \emph{improved Hardy-Poincar\'e inequality}~\eqref{Improved-Hardy-Poincare}, if additionally we assume that $\ird{x\,\h\,\mB^{2-m}}=0$, then we have
\be{HP-PNAS}
\mathsf I[\h]\ge4\,\alpha\,\mathsf F[\h]
\ee
where $\alpha=2-d\,(1-m)$. 

Inequality~\eqref{HP-PNAS} corresponds to an improved spectral gap. This deserves some remarks. If we consider $v_0(x)=\mB(x-x_0)$ for some $x_0\neq0$, then the solution~$v$ of~\eqref{FDr} with initial datum $v_0$ is $v(t,x)=\mB(x-x(t))$ with $x(t)=x_0\,e^{-2t}$. Asymptotically as $t\to+\infty$, we find that $v(t,x)-\mB(x)\sim x_0\cdot\nabla\mB(x)\,e^{-2t}$ and we find that
\[
\lim_{t\to+\infty}\mathcal Q[v(t,\cdot)]=4\,.
\]
In other words, we have found a solution of the nonlinear evolution equation~\eqref{FDr} which achieves the optimal rate in the asymptotic regime as $t\to+\infty$ so that the minimum of $\mathcal Q$ is determined by the spectral gap associated with the Rayleigh quotient $h\mapsto\mathsf I[\h]/\mathsf F[\h]$. This spectral gap is determined by the properties of the operator $\mathcal L_{\alphaa,d}$ and are given in Proposition~\ref{Prop:Spectrum}. The same result also holds for any direction, which simply means that the eigenspace corresponding to the first positive eigenvalue, $\lambda_{1,0}$, is generated by all infinitesimal translations of the \idx{Barenblatt function}s. The condition $\ird{x\,\h\,\mB^{2-m}}=0$ can therefore be interpreted as an orthogonality condition in the linearized regime, which can also be enforced for the solutions of the nonlinear evolution equation~\eqref{FDr} by assuming \hbox{$\ird{x\,v_0}=0$}, \emph{i.e.}, that the \idx{center of mass} is at $x=0$. There is another special solution~\eqref{FDr} which asymptotically generates the eigenspace corresponding to the next positive eigenvalue, $\lambda_{0,1}$,  when $m\in(m_1,1)$: it is given by the initial datum $v_0(x)=\lambda^d\,\mB(\lambda\,x)$ for some positive $\lambda\neq1$.

Our purpose is to use~\eqref{HP-PNAS} in order to establish an improved lower bound for $\mathcal Q[v(t,\cdot)]$ in the asymptotic time layer as $t\to+\infty$. We have the following result on a time-interval $(T,+\infty)$.

The next result is the first of a set of technical estimates that are needed to produce constructive constants in the stability results of Chapters~\ref{Chapter-5} and~\ref{Chapter-6}. The detail of the proof can be skipped at first reading.
%---------------------------------------------------------------------
\begin{proposition}\label{Prop:Gap} Let $m\in(m_1,1)$ if $d\ge2$, $m\in(1/3,1)$ if $d=1$, $\eta=2\,d\,(m-m_1)$ and $\chi:=\frac1{580}$. If $v$ is a nonnegative solution to~\eqref{FDr} of mass~$\Mstar$,~with
\be{Uniform}\tag{$H_{\varepsilon,T}$}
(1-\varepsilon)\,\mB\le v(t,\cdot)\le(1+\varepsilon)\,\mB\quad\forall\,t\ge T
\ee
for some $\varepsilon\in(0,\chi\,\eta)$ and $T>0$, and such that $\ird{x\,v(T,x)}=0$, \index{asymptotic time layer}{then} we have
\be{Prop.Rayleigh.Ineq}
\mathcal Q[v(t,\cdot)]\ge4+\eta\quad\forall\,t\ge T\,.
\ee
\end{proposition}
%---------------------------------------------------------------------
Notice that $m=m_1$ has to be excluded while the restriction in dimension $d=1$ arises from the fact that $\ird{|x|^2\,\mB}<+\infty$ is needed, which means $m>\widetilde m_1$, and $0=m_1<\widetilde m_1=1/3$ if $d=1$. Notice that Assumption~\eqref{Uniform} is not exactly the same as in~\cite{Blanchet2009}, which motivates the additional restriction $m>1/3$ if $d=1$.
\begin{proof} We estimate the \idx{free energy} $\mathcal F$ and the \idx{Fisher information} $\mathcal I$ in terms of their linearized counterparts $\mathsf F$ and $\mathsf I$. Let
\[
\h:=v\,\mB^{m-2}-\mB^{m-1}\,.
\]
Under Assumption~\eqref{Uniform}, we learn from~\cite[Lemma~3]{Blanchet2009} that
\be{equivalence-entropy-l2}
(1+\varepsilon)^{-b}\,\mathsf F[\h(t,\cdot)]\le\mathcal F[v(t,\cdot)]\le(1-\varepsilon)^{-b}\,\mathsf F[\h(t,\cdot)]\quad\forall\,t\ge T\,,
\ee
where $b=2-m$, and
\be{linear.nonlinear.fisher.inq}
\mathsf I[\h]\le s_1(\varepsilon)\,\mathcal I[v]+s_2(\varepsilon)\,\mathsf F[\h]
\ee
{}from~\cite[Lemma~7]{Blanchet2009}, where
\be{s1-s2}
s_1(\varepsilon):=\frac{(1+\varepsilon)^{2\,b}}{1-\varepsilon}\quad\mbox{and}\quad s_2(\varepsilon):=4\,d\,(1-m)\(\frac{(1+\varepsilon)^{2\,b}}{(1-\varepsilon)^{2\,b}}-1\)\,.
\ee
The estimate~\eqref{equivalence-entropy-l2} follows from a simple Taylor expansion while~\eqref{linear.nonlinear.fisher.inq} is a consequence of some slightly more complicated but still elementary computations, see~\cite[Lemma 3 and Lemma 7]{Blanchet2009}.

Let us explain how to compute $\chi$. Collecting~\eqref{HP-PNAS},~\eqref{equivalence-entropy-l2} and~\eqref{linear.nonlinear.fisher.inq}, elementary computations show that~\eqref{Prop.Rayleigh.Ineq} holds with $\eta=f(\varepsilon)$, where
\[
f(\varepsilon)=\frac{4\,\alpha\,(1-\varepsilon)^b\,-4\,s_1(\varepsilon)-(1+\varepsilon)^b\,s_2(\varepsilon)}{s_1(\varepsilon)}\,.
\]
We claim that
\[
\max_{\varepsilon\in(0,\chi\,\eta)}f(\varepsilon)\ge2\,d\,(m-m_1)\,.
\]
Let us consider
\[
g_1(\varepsilon):=1-\frac{(1-\varepsilon)^{1+b}}{(1+\varepsilon)^{2\,b}}\quad\mbox{and}\quad g_2(\varepsilon):=\frac{1-\varepsilon}{(1+\varepsilon)^b}\(\frac{(1+\varepsilon)^{2\,b}}{(1-\varepsilon)^{2\,b}}-1\)
\]
and observe that $g_1$ is concave and $g_1(\varepsilon)\le g_1'(0)\,\varepsilon=(1+3\,b)\,\varepsilon\le7\,\varepsilon$ for any $\varepsilon\in[0,1]$ and $b\in[1,2]$, while $g_2$ is convex and such that $g_2(\varepsilon)\le g_2'(1/2)\,\varepsilon$ for any $\varepsilon\in[0,1/2]$ with $g_2'(1/2)\le133$ for any $b\in[1,2]$. By writing
\[
f(\varepsilon)=2\,\eta-4\,\alpha\,g_1(\varepsilon)-4\,d\,(1-m)\,g_2(\varepsilon)\,,
\]
and after observing that $4\,\alpha\le8$ and $4\,d\,(1-m)\le4$ if  $m\in(m_1,1)$ and $d\ge1$, we conclude that
\[
f(\varepsilon)\ge2\,\eta-\tfrac\varepsilon\chi\ge\eta\quad\forall\,\varepsilon\in(0,\chi\,\eta)\,.
\]
\end{proof}

%%%%%%%%%%%%%%%%%%%%%%%%%%%%%%%%%%%%%%%%%%%%%%%%%%%%%%%%%%%%%%%%%%%%%%%%
\subsection{Entropy, flow and inequalities}\label{deficit.equivalence}

A very standard consequence of \eqref{entropy.eep} is that a solution of~\eqref{FDr} with nonnegative initial datum $v_0\in\mathrm L^1(\R^d)$ such that $\mathcal F[v_0]<+\infty$ satisfies the estimate
\be{entropy.decay.rate}
\mathcal F[v(t,\cdot)]\le\mathcal F[v_0]\,e^{-4\,t}\quad\forall\,t\ge0\,.
\ee
Under the assumptions of Lemma~\ref{prop.backward} and Proposition~\ref{Prop:Gap}, we have the improved decay estimate in both the initial and asymptotic time layers given by the following Proposition.
%-------------------------------------------------------------------------
\begin{proposition}\label{Prop:Ch2: BetterRates} Under the assumptions of Lemma~\ref{prop.backward} and Proposition~\ref{Prop:Gap}, we have
\[
\mathcal F[v(t,\cdot)]\le\mathcal F[v_0]\,e^{-\,(4+\zeta)\,t}\quad\forall\,t\in[0,T]\,,\quad\mbox{with}\quad\zeta=\frac{4\,\eta\,e^{-4\,T}}{4+\eta-\eta\,e^{-4\,T}}
\]
and
\[
\mathcal F[v(t,\cdot)]\le\mathcal F[v(T,\cdot)]\,e^{-\,(4+\eta)\,(t-T)}\quad\forall\,t\in[T,+\infty)\,.
\]
\end{proposition}
%-------------------------------------------------------------------------
Let us summarize what we have achieved so far. The \index{entropy - entropy production inequality}{\emph{entropy - entropy production} inequality}~\eqref{entropy.eep} is equivalent to~\eqref{GNS}. We look for an improvement of~~\eqref{entropy.eep} using the \index{initial time layer}{initial} and the \idx{asymptotic time layer}s as in Sections~\ref{Sec:FDE-Initial} and~\ref{Sec:FDE-Asymptotic}. Our task is of course to prove that~\eqref{Uniform} is satisfied for some \idx{threshold time} $T>0$, which is the purpose of the next two chapters.

%%%%%%%%%%%%%%%%%%%%%%%%%%%%%%%%%%%%%%%%%%%%%%%%%%%%%%%%%%%%%%%%%%%%%%%%
\subsection{Mass and self-similarity}

Solutions of~\eqref{FD} with initial datum $u_0$ are transformed into solutions of~\eqref{FDr} with initial datum $v_0=\muscal^{-d}\,u_0(\cdot/\muscal)$ by the \emph{\index{self-similar variables}{self-similar change of variables}}~\eqref{SelfSimilarChangeOfVariables}. Reciprocally, the function
\be{Barenblatt-1}
B(t,x):=\frac{\muscal^d}{R(t)^d}\,\mB\(\frac{\muscal\,x}{R(t)}\)=B(t,x;\Mstar)
\ee
is a self-similar solution of~\eqref{FD} which describes the so-called \emph{\idx{intermediate asymptotics}} of the solution $u$ of~\eqref{FD}, that is, the large time behaviour of $u$ under the condition that $\ird{u_0}=\Mstar$. If we relax this condition, any nonnegative solution~$u$ of~\eqref{FD} such that $\ird{u_0}=M$ is attracted by the \index{Barenblatt self-similar solutions}{Barenblatt self-similar solution} of~\eqref{FD} of mass $M$ defined by~\eqref{BarenblattM}.

%%%%%%%%%%%%%%%%%%%%%%%%%%%%%%%%%%%%%%%%%%%%%%%%%%%%%%%%%%%%%%%%%%%%%%%%
\subsection{A Csisz\'ar-Kullback type inequality}\label{Appendix:CK}

The \idx{relative entropy} $\mathcal F[v]$ with respect to the \idx{Barenblatt function} of same mass as $v$ controls the $\mathrm L^1$ distance to the \idx{Barenblatt function}. If $ v\in\mathrm L^1_+(\R^d)$ is such that $\ird{|x|^2\,v}=\ird{|x|^2\,\mB}$ and $\nrm v1=\Mstar$ with $\mB$ defined by~\eqref{mB}, we learn from~\cite{Dolbeault2013917} that
\[
\(\ird{|v-\mB|\(1+|x|^2\)}\)^2\le\frac{16\,\Mstar}{(d+2)\,m-d}\,\mathcal F[v]\,.
\]
In the case $m\in[m_1,1)$, \index{Csisz\'ar-Kullback inequality}{the inequality} that we need, without the second moment condition, appears in~\cite{Carrillo2003}, however with no proof, and also in~\cite{MR1842429}. For completeness, let us give a precise statement with the expression of the constant and an elementary proof, which complements the result and the proof of Lemma~\ref{Lem:CK}.
%-------------------------------------------------------------------------
\begin{lemma}\label{convergence.L1} Let $m\in\big(\widetilde m_1, 1\big)$, for any $v\in\mathrm L^1_+(\R^d)$ such that $\mathcal F[v]$ is finite and $\nrm v1=\Mstar$, we have
\be{CKm}
\nrm{v-\mB}1^2\le\frac{4\,\alpha}m\,\Mstar\,\mathcal F[v]\,.
\ee
\end{lemma}
%-------------------------------------------------------------------------
\begin{proof} Since $\ird{(v-\mB)}=0$, we have that
\[
\nrm{v-\mB}1=\ird{|v-\mB|}=\int_{v\le\mB}(\mB-v)\dx+\int_{v\ge\mB}(v-\mB)\dx=2\int_{v\le\mB}(\mB-v)\dx\,.
\]
Let $\varphi(s)=s^m/(m-1)$. If $0\le t\le s$, a Taylor expansion shows that for some $\xi\in(t,s)$ we have
\[
\varphi(t)-\varphi(s)-\varphi'(s)\,(t-s)=\frac12\,\varphi''(\xi)\,(t-s)^2\ge\tfrac m2\,s^{m-2}\,(s-t)^2\,,
\]
hence
\[
\sqrt{\tfrac m2}\,(s-t)\le s^\frac{2-m}2\,\big(\varphi(t)-\varphi(s)-\varphi'(s)\,(t-s)\big)^{1/2}\,.
\]
Using this inequality with $s=\mB$ and $t=v$ and the Cauchy-Schwarz inequality, we deduce that
\[\label{n.2}\begin{split}
\frac m2\(\int_{v\le\mB}(\mB-v)\dx\)^2&\le\(\int_{v\le\mB}\mB^\frac{2-m}2\,\big(\varphi(v)-\varphi(\mB)-\varphi'(\mB)\,(v-\mB)\big)^{1/2}\dx\)^2\\
&\le\int_{v\le\mB}\mB^{2-m}\dx\int_{v\le\mB}\(\varphi(v)-\varphi(\mB)-\varphi'(\mB)\,(v-\mB)\)\dx\\
&\le\ird{\mB^{2-m}}\;\mathcal F[v]
\end{split}\]
and the conclusion follows from the identity $\ird{\mB^{2-m}}=\frac\alpha2\,\Mstar$ that we shall establish as well as a bunch of useful formulae in the next subsection. \end{proof}

%%%%%%%%%%%%%%%%%%%%%%%%%%%%%%%%%%%%%%%%%%%%%%%%%%%%%%%%%%%%%%%%%%%%%%%%
\subsection{Constants and useful identities}\label{Appendix:Identities}

For the convenience of the reader, we collect some elementary identities and definitions. See~\cite{DelPino2002} or~\cite[Appendix~A]{Dolbeault2013917} for more details. Some computations are reminiscent of the proof of Lemma~\ref{Lem:BestMatchAubinTalenti}. With $\mB$ defined by~\eqref{mB}, using $\nabla\mB^m=-\frac{2\,m}{1-m}\,x\,\mB$ and an integration by parts, we obtain
\[
\ird{\mB^m}=-\frac1d\ird{x\cdot\nabla\mB^m}=\frac{2\,m}{d\,(1-m)}\ird{|x|^2\,\mB}\,.
\]
On the other hand, we deduce from $\mB^m=\mB^{m-1}\,\mB=\(1+|x|^2\)\,\mB$ that
\[
\ird{\mB^m}=\ird{\(1+|x|^2\)\mB}=\Mstar+\ird{|x|^2\,\mB}
\]
where
\[
\Mstar=\ird\mB=\omega_d\int_0^{+\infty}\frac{r^{d-1}}{\(1+r^2\)^\frac1{1-m}}\,dr=\pi^\frac d2\,\frac{\Gamma\big(\frac1{1-m}-\frac d2\big)}{\Gamma\big(\frac1{1-m}\big)}\,.
\]
This gives the expressions
\[
\ird{|x|^2\,\mB}=\frac{d\,(1-m)}{(d+2)\,m-d}\,\Mstar\quad\mbox{and}\quad\ird{\mB^m}=\frac{2\,m}{(d+2)\,m-d}\,\Mstar\,.
\]
With the same method, we find that
\[
\Mstar=\ird\mB=-\frac1d\ird{x\cdot\nabla\mB}=\frac2{d\,(1-m)}\ird{|x|^2\,\mB^{2-m}}
\]
and $\mB=\mB^{m-1}\,\mB^{2-m}=\(1+|x|^2\)\mB^{2-m}$ so that
\[
\Mstar=\ird\mB=\ird{\mB^{2-m}}+\ird{|x|^2\,\mB^{2-m}}\,.
\]
This amounts to
\[
\ird{\mB^{2-m}}=\frac\alpha2\,\Mstar\quad\mbox{and}\quad\ird{|x|^2\,\mB^{2-m}}=\frac d2\,(1-m)\,\Mstar\,.
\]

%%%%%%%%%%%%%%%%%%%%%%%%%%%%%%%%%%%%%%%%%%%%%%%%%%%%%%%%%%%%%%%%%%%%%%%%
%%%%%%%%%%%%%%%%%%%%%%%%%%%%%%%%%%%%%%%%%%%%%%%%%%%%%%%%%%%%%%%%%%%%%%%%
\section{Bibliographical comments}\label{Sec:Bib2}

For an overview of the properties of the nonlinear diffusion~\eqref{FD} in the porous medium regime $m>1$ and in the fast diffusion regime $m<1$, we primarily refer to~\cite{Vazquez2006,Vazquez2007}. Mass conservation for $m\in[m_c,1)$ goes back to~\cite{Herrero1985} by M.~Herrero and M.~Pierre. The \index{intermediate asymptotics}{role of self-similar solutions in large time asymptotics} is, for instance, studied in~\cite{Friedman1980} using comparison methods, and from the point of view of functional inequalities in~\cite{DelPino2002}. However the precise link of the flow with the \idx{Gagliardo-Nirenberg-Sobolev inequalities}~\eqref{GNS} in scale invariant form~\eqref{GNS} only appears in~\cite{Dolbeault_2016}.

\emph{R\'enyi entropies} are inspired by papers in information theory. In~\cite{MR823597}, M.~Costa was studying the R\'enyi entropy power associated to Shannon's entropy and C.~Villani in~\cite{MR1768665} gave a simple proof based on the \emph{\idx{carr\'e du champ}} method. So far, this was an essentially linear setting associated with the heat flow. In~\cite{MR3200617}, G.~Savar\'e and G.~Toscani noted that the formalism of R\'enyi entropy powers could be extended to the nonlinear setting and the precise link, at least from a formal point of view, was made in~\cite{Dolbeault_2016}. The computation of Section~\ref{Sec:EntropyGrowth} can be found in the language of entropy powers in~\cite{MR3291145}. The reader interested in further details is invited to refer to~\cite{MR3255069} and references therein for a more detailed account. For completeness, let us mention that, in~\cite{MR4014730}, Gagliardo-Nirenberg interpolation properties of functions very rapidly decaying in space are used to investigate large time dynamics of some solutions to ultra-fast diffusion equations.

If the notion of entropy in \idx{self-similar variables} can be found as early as in~\cite{Newman1984,Ralston1984}, it has been connected with~\index{Gagliardo-Nirenberg-Sobolev inequalities}{\eqref{GNS}} only in~\cite{DelPino2002} and we refer to~\cite{Blanchet2009} for a precise discussion on the range of the exponent $m$ depending on the various quantities, like second moment, entropy or \idx{Fisher information}, which are at stake for implementing entropy methods. This is usually done in \idx{self-similar variables}, but can be recast in the framework of~\eqref{FD} in \idx{self-similar variables}: see~\cite{DEL-JEPE} for details. The equivalence of the \index{entropy - entropy production inequality}{\emph{entropy - entropy production} inequality} with~\eqref{GNS-Intro} is known from~\cite{DelPino2002}. Spectral considerations and in particular the fact that optimality is achieved in the asymptotic regime as $t\to+\infty$, were later introduced in~\cite{Blanchet2009}. In statistical physics, the entropy associated to~\eqref{FD} is known as \emph{\idx{Tsallis entropy}}, with motivations coming from polytropic gases models and statistical models for gravitating systems: see~\cite{MR2724662} for an overview.

A considerable amount of papers has been published on the \emph{\idx{carr\'e du champ}} method and we shall refer to~\cite{MR3155209} for a rather exhaustive survey. Among earlier important papers by D.~Bakry and his collaborators, one can quote for instance~\cite{Bakry1985,MR2213477}. The link with the entropy point of view in kinetic equations was made by G.~Toscani in~\cite{MR1447044}, which immediately sparked an enormous activity as can be seen for instance from~\cite{MR1842428,Unterreiter_2000} and~\cite{MR2065020}. These results were at that point mostly devoted to the capture of asymptotic rates of convergence and limited to linear evolution equations. In parallel with the variational approach of~\cite{DelPino2002}, J.A.~Carrillo with G.~Toscani and J.L.~V\'azquez started to adapt, at a rather formal level, the \emph{\idx{carr\'e du champ}} method to the \index{fast diffusion equation}{fast diffusion} and the porous media equations in~\cite{Carrillo2000,Carrillo2003}. A complete proof with an approximation scheme and a justification of the boundary terms can be found in~\cite{Carrillo2001} in the context of \index{convex Sobolev inequalities}{\emph{convex Sobolev} inequalities}. Also see~\cite{MR3497125,DEL-JEPE} for results more focused on~\eqref{FD} and~\eqref{FDr}, with more direct proofs. The key of the approximation scheme is \idx{Grisvard's lemma}: see for instance to~\cite[Lemma 5.2]{MR2533926},~\cite[Proposition~4.2]{MR3150642} or~\cite{MR775683} (also see~\cite{MR2435196} or~\cite[Lemma~A.3]{MR3497125}). In an important contribution,~\cite{Demange_2005,Demange_2008}, J.~Demange extended the linear computations attached to Markov processes to nonlinear diffusions on manifolds. These computations are in a sense very close to the ones of~\cite{MR1134481}, and also~\cite{MR1412446}: the point was further clarified in~\cite{Dolbeault20141338}. As a side remark, let us recall that the \idx{fast diffusion equation} can be formally interpreted as the \idx{gradient flow} with respect to \idx{Wasserstein's distance} of the entropy, as arises from the early results of~\cite{McCann1997,MR1842429} and it is not a surprise that mass transportation also plays a role in~\eqref{GNS} as observed in~\cite{MR2032031,MR2053603,MR3263963}. Such an interpretation is possible only if the second moment is controlled, which corresponds to the functional framework introduced in Chapter~\ref{Chapter-1} (see Section~\ref{Sec:UncertaintyPrinciple}). Concerning the connection of nonlinear diffusions with the \emph{curvature-dimension criterion}, we refer to~\cite{Demange_2008,Demange_2005}  and, in an abstract framework, to~\cite{MR2480619,MR2118836,MR2237206,MR2237207,MR3385639}. See~\cite{MR2401600,MR2459454,ABS,FGbook} for considerations on \idx{gradient flow}s. Further efforts were done in, \emph{e.g.},~\cite{Carrillo_2010,MR4044464} to better justify the formal \idx{gradient flow} point of view of~\cite{MR1842429} in cases of practical interest or simply in order to simplify the conceptual framework, while interesting related applications can be found for instance in~\cite{MR3426082,MR4015181,Iacobelli_2018}. 

\medskip Beyond the results of Chapter~1, the \emph{Hardy-Poincar\'e inequality}~\eqref{Hardy-Poincare} and the \emph{improved Hardy-Poincar\'e inequality}~\eqref{Improved-Hardy-Poincare} are spectral gap inequalities that can be obtained by an expansion around the \idx{Barenblatt function} as a consequence of~\eqref{entropy.eep}, under the corresponding constraints: see~\cite{Blanchet2009} and~\cite[Lemma~1]{Bonforte2010c}. We learn from~\cite{Carrillo2002,Blanchet2009} that the linearized problem, \emph{i.e.}, the sharp constant in the Hardy-Poincar\'e inequality, determines the global optimal rate of convergence for~\eqref{FDr}.

The analysis of the spectral gap inequality, or of the decay rate of the entropy in the \emph{\idx{asymptotic time layer}}, is more classical and essentially known from~\cite{Blanchet2007,Blanchet2009} with various improvements in~\cite{Dolbeault2011a,Bonforte2010c,Dolbeault2013917}. It relies on spectral computations that go back to~\cite{Scheffer01,MR1982656,Denzler2005} in the framework of \idx{gradient flow}s with respect to the \index{Wasserstein's distance}{Wasserstein distance}, and~\cite{Denzler2015,Kim2006}.  Further details are available in~\cite[Proposition~1]{Dolbeault2011a}. Quite remarkably, the linearization of the evolution operator was performed first in the context of \idx{Wasserstein's distance}s in~\cite{MR1982656,Denzler2005} and later recast in the context of more standard Lebesgue spaces with weights ~\cite{Blanchet2009,Dolbeault2011a,Bonforte2010c,Dolbeault2013917}. See in particular~\cite{Bonforte2010c} for the equivalence of the spectra.

The improvements in~\cite{Bonforte2010c} and~\cite{Dolbeault2013917} are of course linked with the invariances of the \idx{fast diffusion equation}~\eqref{FD} and the special solutions mentioned in Section~\ref{Sec:FDE-Asymptotic}. Eigenfunctions associated with the lowest eigenvalues can be seen as infinitesimal generators of these invariances, but similar ideas can also be exploited for higher energy levels of the linearized operator, as was shown in~\cite{Denzler2015} (also see~\cite{MR3457597} for a summary).

Improvements of functional inequalities based on \emph{\idx{entropy methods}} and the use of remainder terms in the \emph{\idx{carr\'e du champ}} method have an already long history. As far as we know, the first contribution appeared in~\cite{MR2152502} and was later followed by~\cite{MR2375056,MR2435196} in the case of linear diffusions. For nonlinear diffusions on compact manifolds, we refer to~\cite{Demange_2008,MR3177759,1504,MR4026963,Dolbeault_2020} in the case of the sphere,~\cite{Blanchet2007,Blanchet2009} and~\cite{Bonforte2010c,Dolbeault2011a,1751-8121-48-6-065206,Dolbeault2013917,Dolbeault_2016,Dolbeault_2016,MR3493423} when dealing with the Euclidean space. In these papers, improvements are either of order~$1$ (in terms of powers of a relative entropy)  and obtained under orthogonality conditions which induce limitations on the range of exponents, or of order $2$ (typically given by the square of the relative entropy) and obviously non-optimal. Moreover results so far cover only subcritical regimes of the underlying functional inequalities or weaker notions of distance than the natural strong distance: see for instance~\cite{MR2915466,MR3227280}, in which constants are explicit but the distance is measured in~$\mathrm H^{-1}(\R^d)$ while $\mathrm H^1_0(\R^d)$ would be expected.

The \idx{Csisz\'ar-Kullback inequality} of Section~\ref{Appendix:CK}, also known in the literature as the \emph{Pinsker-Csisz\'ar-Kullback inequality} is an extension of the historical papers~\cite{MR0213190,Csiszar1967,Kullback1967}, which correspond to the limit case $p\to1_+$. There are many variants: see for instance~\cite{Unterreiter_2000,MR1951784,Carrillo2000,Carrillo2001,MR1842429,DelPino2002,BDIK} for some classical extensions. Also see~~\cite{Dolbeault2013917} for the case with a $\lrangle x^2$ weight and Lemma~\ref{Lem:CK}.

\medskip The equivalence of the \index{entropy - entropy production inequality}{entropy - entropy production inequalities} as in~\cite{DelPino2002} and the concavity of the generalized \idx{R\'enyi entropy powers} along the \index{fast diffusion equation}{fast diffusion flow} is established in~\cite{Dolbeault_2016,MR3200617} and further studied in~\cite{DEL-JEPE}. Although elementary, this is one of the key points for controlling the decay rate of the entropy during the \emph{\idx{initial time layer}} (see Section~\ref{Sec:FDE-Initial}). The analysis of the decay rate of the entropy in the \emph{\idx{asymptotic time layer}} is more classical and essentially known from~\cite{Blanchet2009}. The key issue of this paper is to control the \emph{\idx{threshold time}} between these two regimes and relies on a quantitative regularity theory.

%%%%%%%%%%%%%%%%%%%%%%%%%%%%%%%%%%%%%%%%%%%%%%%%%%%%%%%%%%%%%%%%%%%%%%%%
%%%%%%%%%%%%%%%%%%%%%%%%%%%%%%%%%%%%%%%%%%%%%%%%%%%%%%%%%%%%%%%%%%%%%%%%
\chapter{Linear parabolic equations: interpolation and regularity}\label{Chapter-3}

This chapter is devoted to the computation of various explicit constants in functional inequalities and regularity estimates for solutions of parabolic equations, which are not available from the literature. We provide new expressions and simplified proofs of the \idx{Harnack inequality} (Theorem~\ref{Claim:3}) and the corresponding \idx{H\"older continuity} (Theorem~\ref{Claim:4}) of the solution of a linear parabolic equation with measurable coefficients, following J. Moser's original ideas~\cite{Moser1964,Moser1971}.

%%%%%%%%%%%%%%%%%%%%%%%%%%%%%%%%%%%%%%%%%%%%%%%%%%%%%%%%%%%%%%%%%%%%%%%%
%%%%%%%%%%%%%%%%%%%%%%%%%%%%%%%%%%%%%%%%%%%%%%%%%%%%%%%%%%%%%%%%%%%%%%%%
\section{Interpolation inequalities and optimal constants}\label{sec:interpolation.inequalities}

Let us denote by $B_R$ the ball of radius $R>0$ centered at the origin, and define
\begin{align*}
\pcc:=\frac{2\,d}{d-2}\quad&\mbox{if}\quad d\ge3\,,\\
\pcc:=4\quad&\mbox{if}\quad d=2\,,\\
\pcc\in(4,+\infty)\quad&\mbox{if}\quad d=1\,.
\end{align*}
%---------------------------------------------------------------------
\begin{theorem}\label{Thm:K} Let $d\ge1$, $R>0$. For $d=1$, $2$, we further assume that $R\le 1$. With the above notation, the inequality
\be{sob.step2}
\|f\|_{\mathrm L^{\pcc}(B_R)}^2\le\mathcal K\(\|\nabla f\|_{\mathrm L^2(B_R)}^2+\tfrac1{R^2}\,\|f\|_{\mathrm L^2(B_R)}^2\)\quad\forall\,f\in\mathrm H^1(B_R)
\ee
holds for some constant
\be{estim.S-p}
\mathcal K \le
\begin{cases}\begin{array}{ll}\,\frac{4\,\Gamma\big(\tfrac{d+1}2\big)^{2/d}}{2^\frac2d\,\pi^{1+\frac1d}}&\quad\mbox{if}\quad d\ge 3\,,\\
\frac4{\sqrt{\pi}}&\quad\mbox{if}\quad d=2\,,\\
2^{1+\frac2{\pcc}}\max\left\{\tfrac{\pcc-2}{\pi^2},\tfrac14\right\}&\quad\mbox{if}\quad d=1\,.
\end{array}
\end{cases}
\ee
\end{theorem}
%---------------------------------------------------------------------
Inequality~\eqref{sob.step2} is standard and the novelty is the estimate~\eqref{estim.S-p}. The proof of this result is split in several partial results (Lemmas~\ref{Lem:SobolevH1},~\ref{Lem:GNd=2} and~\ref{Lem:GNd=1}) depending on the dimension. 

%%%%%%%%%%%%%%%%%%%%%%%%%%%%%%%%%%%%%%%%%%%%%%%%%%%%%%%%%%%%%%%%%%%%%%
\subsection{A critical interpolation inequality in \texorpdfstring{$d\ge3$}{d ge 3}}\label{Appendix:SobolevH1}

\noindent Since $\pcc=2\,d/(d-2)$, the constant $\mathcal K$ is independent of $R$  and it is sufficient to provide its expression on the unit ball $B=B_1$. We consider the critical \index{interpolation inequalities on bounded domains}{interpolation} inequality, or Sobolev's inequality, associated with the critical embedding $\mathrm H^1(B)\hookrightarrow\mathrm L^{\pcc}(B)$,
\be{SobolevH1}
\|\nabla f\|_{\mathrm L^2(B)}^2+d\,\frac{d-2}{d-1}\,\|f\|_{\mathrm L^2(B)}^2\ge\mathsf S_B^2\,\|f\|_{\mathrm L^{\pcc}(B)}^2\quad\forall\,f\in\mathrm H^1(B)\,.
\ee
%---------------------------------------------------------------------
\begin{lemma}\label{Lem:SobolevH1} For any $d\ge3$, Inequality~\eqref{SobolevH1} holds with $\mathsf S_B^2=\mathsf S_d^2/(d-1)$, where~$\mathsf S_d$ is the optimal constant for Sobolev's inequality~\eqref{SobolevRd}.\end{lemma}
%---------------------------------------------------------------------
As a consequence,~\eqref{sob.step2} holds with $\mathcal K^{-1}=d\,(d-2)\,\mathsf S_d^2$.
\begin{proof} On $\R^d$, Sobolev's inequality~\eqref{SobolevRd} is invariant under \idx{Kelvin's transform}
\[
u\mapsto\tilde u(x):=|x|^{2-d}u\(\frac x{|x|^2}\)
\]
because $\nrm{\nabla\tilde u}2=\nrm{\nabla u}2$ and $\nrm{\tilde u}\pcc=\nrm u\pcc$. Let us consider a function $f\in\mathrm H^1(B)$ and extend it to $\R^d$ by $u(x)=f(x)$ if $x\in B$ and $u(x)=\tilde f(x)$ if $x\in B^c$. We know from~\eqref{SobolevRd} that
\begin{multline*}
0\le\nrm{\nabla u}2^2-\mathsf S_d^2\,\nrm u\pcc^2\\
=\int_B|\nabla f|^2\dx-\mathsf S_d^2\(\int_B|f|^\pcc\dx\)^{2/\pcc}+\int_{B^c}|\nabla\tilde f|^2\dx-\mathsf S_d^2\(\int_{B^c}|\tilde f|^\pcc\dx\)^{2/\pcc}\,.
\end{multline*}
Let us assume that $f$ is smooth and introduce spherical coordinates $r=|x|$ and $\omega=x/r$ for any $x\neq0$. With $f'(r,\omega)=\partial f/\partial r$, $\nabla_\omega f=\nabla f-f'(r)\,\omega$ and $s=1/r$, for a given $\omega\in\mathbb S^{d-1}$, we compute
\begin{align*}
&\int_1^{+\infty}\(|\tilde f'|^2+\frac1{r^2}\,|\nabla_\omega\tilde f|^2\)r^{d-1}\,\mathrm dr\\
&=\int_1^{+\infty}\(|f'(s)+(d-2)\,r\,f(s)|^2+r^2\,|\nabla_\omega f(s)|^2\)r^{-d-1}\,\mathrm dr\\
&=\int_0^1\(|f'|^2+\frac1{s^2}\,|\nabla_\omega f|^2\)s^{d-1}\,\mathrm ds
+\int_0^1\(\frac{(d-2)^2}{s^2}\,|f|^2+\frac{d-2}s\,(f^2)'\)s^{d-1}\,\mathrm ds\\
&=\int_0^1\(|f'|^2+\frac1{s^2}\,|\nabla_\omega f|^2\)s^{d-1}\,\mathrm ds+(d-2)\,|f(1,\omega)|^2\,,
\end{align*}
where the last line arises from an integration by parts. An integration with respect to $\omega\in\mathbb S^{d-1}$ shows that
\be{Sob1}
\int_B|\nabla f|^2\dx+\frac12\,(d-2)\int_{\partial B}|f|^2\,\mathrm d\sigma\ge\mathsf S_d^2\(\int_B|f|^\pcc\dx\)^{2/\pcc}\,,
\ee
where $\mathrm d\sigma$ denotes the measure induced by Lebesgue's measure on $\mathbb S^{d-1}=\partial B$. Similarly, by expanding
\[
0\le\int_0^1\left|f'-\frac r{1+r^2}\,f\right|^2\,r^{d-1}\,\mathrm dr\le\int_0^1|f'|^2\,r^{d-1}\,\mathrm dr+d\int_0^1|f|^2\,r^{d-1}\,\mathrm dr-\frac{|f(1,\omega)|^2}2\,,
\]
we obtain
\be{Sob2}
\int_{\partial B}|f|^2\,\mathrm d\sigma\le2\,\int_B|\nabla f|^2\dx+2\,d\int_B|f|^2\dx\,.
\ee
Collecting the estimates of~\eqref{Sob1} and~\eqref{Sob2} concludes the proof for a smooth function~$f$. The result in $\mathrm H^1(B)$ follows by density.\end{proof}

%%%%%%%%%%%%%%%%%%%%%%%%%%%%%%%%%%%%%%%%%%%%%%%%%%%%%%%%%%%%%%%%%%%%%%
\subsection{A two-dimensional interpolation inequality}\label{Appendix:GNS2}

In dimension $d=2$ and $d=1$, we cannot rely on the Sobolev's inequality of Section~\ref{Appendix:SobolevH1}. This is why direct proofs for subcritical cases have to be established. Here we prove the \index{interpolation inequalities on bounded domains}{result} in dimension $d=2$.
%---------------------------------------------------------------------
\begin{lemma}\label{Lem:GNd=2} Let $d=2$. For any $R>0$, we have
\begin{equation}\label{Disk}
\nrm f{\mathrm L^4(B_R)}^2\le\frac{4\,R}{\sqrt\pi}\left(\nrm{\nabla f}{\mathrm L^2(B_R)}^2+\frac1{R^2}\,\nrm f{\mathrm L^2(B_R)}^2\right)\quad\forall\,f\in\mathrm H^1(B_R)\,.
\end{equation}
\end{lemma}
%---------------------------------------------------------------------
\noindent The constant $4/\sqrt\pi$ is not optimal. Numerically, we find that the optimal constant for the inequality restricted to radial functions is approximatively $0.0564922...<4/\sqrt\pi\approx2.25675$. In the end of this section we explain how we obtained this value.
\begin{proof} Let $\Omega=B_R$ (the proof applies to more general domains, but we do not need such a result). As a first step, we prove the inequality corresponding to the embedding $\mathrm W^{1,1}(\Omega)\hookrightarrow\mathrm L^2(\Omega)$. Using Lebesgue's version of the fundamental theorem of calculus, we get
\[
f(x,y)=f(x_0,y)+\int_{x_0}^x f_x(\xi,y)\,{\rm d}\xi\quad\mbox{and}\quad f(x,y)=f(x,y_0)+\int_{y_0}^y f_y(x,\eta)\,{\rm d}\eta\,,
\]
which implies (letting $\Omega_x$ and $\Omega_y$ be $x$ and $y$ sections of $\Omega$ respectively)
\[
|f(x,y)|\le |f(x_0,y)|+ \int_{\Omega_y}F(\xi,y)\,{\rm d}\xi\quad\mbox{and}\quad|f(x,y)|\le|f(x,y_0)|+\int_{\Omega_x}G(x,\eta)\,{\rm d}\eta
\]
where $F(\xi,y)=|f_x(\xi,y)|$ and $G(x,\eta):=|f_y(x,\eta)|$.
Multiplying the two above expressions, we get
\[
|f(x,y)|^2\le A(x_0,y)\,B(x,y_0)
\]
where
\[
A(x_0,y):=|f(x_0,y)|+\int_{\Omega_y}F(\xi,y)\,{\rm d}\xi\quad\mbox{and}\quad B(x,y_0):=|f(x,y_0)|+\int_{\Omega_x}G(x,\eta)\,{\rm d}\eta\,.
\]
Integrating over $\Omega$ in $\dx\,\dy$ and then again in $\Omega$ in $\dx_0\,\dy_0$ we obtain
\begin{multline*}
|\Omega|\,\nrm f{\mathrm L^2(\Omega)}^2=\iint_\Omega\iint_\Omega|f(x,y)|^2\,\dx\,\dy\;\dx_0\,\dy_0\\
\le\iint_{\Omega}A(x_0,y)\,\dx_0\,\dy\iint_{\Omega}B(x,y_0)\,\dx\,\dy_0\,.
\end{multline*}
Finally, notice that
\begin{multline*}
\iint_{\Omega}A(x_0,y)\,\dx_0\,\dy = \iint_{\Omega}\left(|f(x_0,y)| + \int_{\Omega_y}F(\xi,y)\,{\rm d}\xi\right)\,\dx_0\,\dy\\
\le \nrm f{\mathrm L^1(\Omega)} + {\rm diam}(\Omega)\,\nrm F{\mathrm L^1(\Omega)}
\end{multline*}
and
\begin{multline*}
\iint_{\Omega}B(x,y_0)\,\dx\,\dy_0 = \iint_{\Omega}\left(|f(x,y_0)| +\int_{\Omega_x}G(x,\eta)\,{\rm d}\eta\right)\,\dx\,\dy_0\\
\le \nrm f{\mathrm L^1(\Omega)} + {\rm diam}(\Omega)\,\nrm G{\mathrm L^1(\Omega)}\,.
\end{multline*}
Summing up, we obtain
\[
\nrm f{\mathrm L^2(\Omega)}^2\le\frac1{|\Omega|}\(\nrm f{\mathrm L^1(\Omega)} + {\rm diam}(\Omega)\,\nrm F{\mathrm L^1(\Omega)}\)\(\nrm f{\mathrm L^1(\Omega)} + {\rm diam}(\Omega)\,\nrm G{\mathrm L^1(\Omega)}\)\,.
\]
We recall that $F=|f_x|$, $G=|f_y|$ and $|\nabla f|=\sqrt{F^2+G^2}\ge \max\{F,G\}$, hence
\[
\nrm f{\mathrm L^2(\Omega)}^2\le\frac1{|\Omega|}\(\nrm f{\mathrm L^1(\Omega)}+{\rm diam}(\Omega)\,\nrm{\nabla f}{\mathrm L^1(\Omega)}\)^2\,.
\]
We apply this estimate to $f^2$ to get
\[\begin{split}
\nrm f{\mathrm L^4(\Omega)}^2&\le\frac1{\sqrt{|\Omega|}}\left(\nrm f{\mathrm L^2(\Omega)}^2+\,{\rm diam}(\Omega)\,\|\nabla f^2\|_{\mathrm L^1(\Omega)}\right)\\
&\le\frac1{\sqrt{|\Omega|}}\left(\nrm f{\mathrm L^2(\Omega)}^2+2\,{\rm diam}(\Omega)\,\nrm{\nabla f}{\mathrm L^2(\Omega)}\,\nrm f{\mathrm L^2(\Omega)}\right)\,.
\end{split}
\]
We use the elementary estimate
\[
2\,{\rm diam}(\Omega)\,\nrm{\nabla f}{\mathrm L^2(\Omega)}\,\nrm f{\mathrm L^2(\Omega)}
\le\nrm f{\mathrm L^2(\Omega)}^2+{\rm diam}(\Omega)^2\,\nrm{\nabla f}{\mathrm L^2(\Omega)}^2
\]
and finally obtain
\[
\nrm f{\mathrm L^4(\Omega)}^2\le\frac{{\rm diam}(\Omega)^2}{\sqrt{|\Omega|}}\(\nrm{\nabla f}{\mathrm L^2(B_R)}^2+\frac4{{\rm diam}(\Omega)^2}\,\nrm f{\mathrm L^2(B_R)}^2\)\,,
\]
which completes the proof with ${\rm diam}(\Omega)=2\,R$ and $|\Omega|=\pi\,R^2$.
\end{proof}

We know from the proof that $\mathcal C\le4/\sqrt\pi\approx2.25675$. To compute the constant in~\eqref{Disk} numerically when the inequality is restricted to radial functions (equality case is achieved by compactness), it is enough to solve the Euler-Lagrange equation
\be{ODE}
-f''-\frac{f'}r+f=f^3\,,\quad f(0)=a>0\,,\quad f'(0)=0\,.
\ee
To emphasize the dependence of the solution in the shooting parameter $a$, we denote by $f_a$ the solution of~\eqref{ODE} with $f(0)=a$. We look for the value of $a$ for which $f_a$ changes sign only once (as it is orthogonal to the constants) and such that $f'(1)=0$, which is our shooting criterion. Let $s(a)=f_a'(1)$ for the solution of~\eqref{ODE}. With $a=1$, we find that $f_a\equiv1$. Numerically, a shooting method with $a>1$ provides us with $a_\star\approx7.52449$ such that $s(a_\star)=0$ corresponding to a solution $f_{a_\star}$ with only one sign change. Using
\[
2\,\pi\int_0^1\(|f'_{a_\star}|^2+|f_{a_\star}|^2\)r\,\mathrm dr=2\,\pi\int_0^1|f_{a_\star}|^4\,r\,dr=\frac1{\mathcal C}\(2\,\pi\int_0^1|f_{a_\star}|^4\,r\,\mathrm dr\)^{1/2}\,,
\]
we obtain that the constant is $\big(2\,\pi\int_0^1|f_{a_\star}|^4\,r\,\mathrm dr\big)^{-1/2}\approx0.0564922$.

%%%%%%%%%%%%%%%%%%%%%%%%%%%%%%%%%%%%%%%%%%%%%%%%%%%%%%%%%%%%%%%%%%%%%%
\subsection{One-dimensional interpolation inequalities}\label{Appendix:GNS1}

We prove the following elementary result on an interval. We recall that, in $d=1$, we have that $B_R=(-R,R)$.
%---------------------------------------------------------------------
\begin{lemma}\label{Lem:GNd=1} Let $p\in(2,\infty)$. Then for all $f\in \mathrm H^1(B_R)$ we have
\[\label{GNI.d=1}
\|f\|_{\mathrm L^p(B_R)}^2\le(2\,R)^{1+\frac2p}\(\tfrac{p-2}{\pi^2}\,\|f'\|_{\mathrm L^2(B_R)}^2+\tfrac1{4\,R^2}\,\|f\|_{\mathrm L^2(B_R)}^2\)
\]
and this \index{interpolation inequalities on bounded domains}{inequality} is sharp.\end{lemma}
%---------------------------------------------------------------------
By sharp, we mean that the infimum of the quotient
\[
\mathcal Q_R[f]:=\frac{4\,R^2\,\|f'\|_{\mathrm L^2(I_R)}^2}{(2\,R)^{1-\frac2p}\,\|f\|_{\mathrm L^p(I_R)}^2-\|f\|_{\mathrm L^2(I_R)}^2}
\]
is achieved by $\lim_{n\to+\infty}\mathcal Q_R[f_n]=\tfrac{\pi^2}{p-2}$ with $f_n(x)=1+\frac1n\,\sin\big(\frac{\pi\,x}{2\,R}\big)$.

\begin{proof} Let us denote by $\mathcal C_R$ the infimum of $\mathcal Q_R$ on the set $\mathcal H_R$ of the non-constant functions in $\mathrm H^1(B_R)$. To a function $f\in\mathcal H_R$, we associate a function~$g$ on $B_{2R}$ by considering $g(x-R)=f(x)$ in $B_R$ and $g(x-R)=f(2\,R-x)$ in $(R,3\,R)$. Since $g(2\,R)=g(-\,2\,R)$, $g$ the function can be repeated periodically and considered as a $4\,R$-periodic function on $\R$, or simply a function on $B_{2\,R}$ with periodic boundary conditions. We can easily check that
\[
\mathcal Q_R[f]=\frac14\,\mathcal Q_{2R}[g]\,,
\]
and deduce that $\mathcal C_R=\inf\mathcal Q_{2R}[g]$ where the infimum is taken on the set of the even functions in $\mathcal H_{2R}$. Hence
\be{relax}
\mathcal C_R\ge\frac14\,\inf_{\begin{array}{c}g\in\mathcal H_{2R},\\[-1pt] g\mbox{ is periodic}\end{array}}\mathcal Q_{2R}[g]\,,
\ee
where the inequality arises because we relax the symmetry condition $g(x)=g(-x)$. With the scaling $g(x)=h\big(\frac{\pi\,x}{2\,R}\big)$, we reduce the problem on the periodic functions in $\mathcal H_{2R}$ to the interpolation on the circle $\mathbb S^1$ with the uniform probability measure. The optimal inequality on $\mathbb S^1$ is
\[
\nrm h{\mathrm L^p(\mathbb S^1)}^2-\nrm h{\mathrm L^2(\mathbb S^1)}^2\le(p-2)\nrm{h'}{\mathrm L^2(\mathbb S^1)}^2
\]
for any $p>2$, where $\mathbb S^1\approx B_\pi$ (with periodic boundary conditions), the measure is \hbox{$\mathrm d\mu=\frac{\dx}{2\,\pi}$} and
\[
\nrm h{\mathrm L^p(\mathbb S^1)}^2=\(\int_{-\pi}^{+\pi}|h|^p\,\mathrm d\mu\)^{2/p}\,.
\]
Moreover, the inequality in~\eqref{relax} is actually an equality, because the infimum is obtained on $\mathbb S^1$ among functions which satisfy the symmetry condition $g(x)=g(-x)$: a minimizing sequence is for instance given by $h_n(x)=1+\frac1n\,\cos x$.

With $g(x)=h\big(\frac{\pi\,x}{2\,R}\big)$, we find that
\[
\(\int_{-2R}^{+2R}|g|^p\,\dx\)^{2/p}\le(4\,R)^{\frac 2p-1}\((p-2)\,\frac{4\,R^2}{\pi^2}\int_{-2R}^{+2R}|g'|^2\,\dx+\int_{-2R}^{+2R}|g|^2\,\dx\)\,.
\]
With no restriction, as far as optimal constants are concerned, we can assume that $g(x)=g(-x)$, so that each of the integral in $g$ is twice as big as the integral computed with the restriction $f$ of $g$ to $B_R$:
\[
\(2\int_{-R}^{+R}|f|^p\,\dx\)^{2/p}\le2\,(4\,R)^{\frac 2p-1}\((p-2)\,\frac{4\,R^2}{\pi^2}\int_{-R}^{+R}|f'|^2\,\dx+\int_{-R}^{+R}|f|^2\,\dx\)
\,.
\]
This proves that $\mathcal C_R=\tfrac{p-2}{\pi^2}$.\end{proof}
As an easy consequence of Lemma~\ref{Lem:GNd=1} and to fit better the purpose of Section~\ref{sec:interpolation.inequalities}, we can observe that the following (non optimal) inequality holds
\[\label{ineq:GNd=1}
\|f\|_{\mathrm L^p(B_R)}^2\le(2\,R)^{1+\frac2p}\max\{\tfrac{p-2}{\pi^2},\tfrac{1}{4}\}\(\,\|f'\|_{\mathrm L^2(B_R)}^2+\tfrac1{R^2}\,\|f\|_{\mathrm L^2(B_R)}^2\).
\]

%%%%%%%%%%%%%%%%%%%%%%%%%%%%%%%%%%%%%%%%%%%%%%%%%%%%%%%%%%%%%%%%%%%%%%%%
\subsection{An interpolation between \texorpdfstring{$\mathrm L^p$}{Lp} and \texorpdfstring{$C^\nu$}{Cnu} norms}\label{Appendix:GagliardoCNu}

We give an explicit constant as well as an elementary proof. We claim no originality except for the computation of the constant. Let $\Omega\subset\R^d$ be a bounded, open domain and define the $C^\nu(\Omega)$ semi-norm as
\be{C-alpha-norms}
\lfloor u\rfloor_{C^\nu\left(\Omega\right)}:=\sup_{\substack{x,y\in\Omega\\x\neq y}}\frac{|u(x)-u(y)|}{|x-y|^\nu}\,.
\ee
%---------------------------------------------------------------------
\begin{lemma}\label{interpolation.lemma} Let $d\ge1$, $p\ge1$ and $\nu\in(0,1)$. Then there exists a positive constant $C_{d, \nu, p}$ such that, for any $u\in\mathrm L^p(B_{2R}(x))\cap C^\nu(B_{2R}(x))$, $R>0$ and $x\in\R^d$
\be{interpolation.inequality.cpt6}
\nrm u{\mathrm L^\infty(B_{R}(x))}\,\le\,C_{d, \nu, p}\left(\lfloor u\rfloor_{C^\nu(B_{2R}(x))}^{\frac d{d+p\,\nu}}\,\|u\|_{\mathrm L^p(B_{2R}(x))}^{\frac{p\,\nu}{d+p\,\nu}} + R^{-\frac dp}\,\|u\|_{\mathrm L^p(B_{2R}(x))}\right)\,.
\ee
Analogously, we have
\be{interpolation.inequality.Rn}
\nrm u{\mathrm L^\infty(\R^d)}\,\le\,C_{d, \nu, p}\,\lfloor u\rfloor_{C^\nu(\R^d)}^{\frac d{d+p\,\nu}}\,\|u\|_{\mathrm L^p(\R^d)}^{\frac{p\,\nu}{d+p\,\nu}}\quad\forall\,u\in\mathrm L^p(\R^d)\cap C^\nu(\R^d)\,.
\ee
In both cases, the inequalities hold with the constant
\[\label{constant-interpolation-alphanorm}
C_{d, \nu, p} = 2^\frac{(p-1)(d+p\,\nu)+dp}{p(d+p\,\nu)}\left(1+\tfrac d{\omega_d}\right)^\frac1{p}\,\(1+\big(\tfrac d{p\,\nu}\big)^\frac1{p}\)^\frac d{d+p\,\nu}\,\(\big(\tfrac d{p\,\nu}\big)^{\frac{p\,\nu}{d+p\,\nu}}+\big(\tfrac{p\,\nu}d\big)^{\frac d{d+p\,\nu}}\)^\frac1p\,.
\]
\end{lemma}
%---------------------------------------------------------------------
\begin{proof}
For any $z$, $y \in B_R(x)$, by the triangle inequality and by definition of $\lfloor \cdot\rfloor_{C^\nu\left(B_{2R}\right)}$ given in~\eqref{C-alpha-norms}, we have that
\begin{equation*}\begin{split}
|u(z)|^p &\leq \big(|u(z)-u(y)|+|u(y)|\big)^p\\
&\leq 2^{p-1} \big(|u(z)-u(y)|^p +|u(y)|^p\big)\\
&\leq 2^{p-1}\left[\big(C+\lfloor u\rfloor_{C^\nu\left(B_{2R}(x)\right)}\big)^p\,|z-y|^{p\,\nu}+|u(y)|^p\right]
\end{split}\end{equation*}
for some $C>0$ to be chosen later. Let $0\le\rho < R$. By averaging on a ball $B_{\rho}(z)$, we have
\be{evaluate.gamma.negative}\begin{split}
|u(z)|^p &\leq \frac{2^{p-1}d}{\omega_d\,\rho^d}\,\left[\left(C+\lfloor u\rfloor_{C^\nu\left(B_{2R}(x)\right)}\right)^p\int_{B_{\rho}(z)}|z-y|^{p\,\nu}\,\dy+\int_{B_\rho(z)}|u(y)|^p\,\dy\right]\\
&\leq 2^{p-1}\left(1+\frac d{\omega_d}\right)\Big[\rho^{p\,\nu} \left(C+\lfloor u\rfloor_{C^\nu\left(B_{2R}(x)\right)}\right)^p+\rho^{-d}\,\|u\|^p_{\mathrm L^p(B_{2R}(x))} \Big]\,.
\end{split}\ee
The right-hand side of the above inequality achieves its minimum w.r.t.~$\rho>0$ at
\[
\rho_\star := \left(\frac{d\,\|u\|^p_{\mathrm L^p(B_{2R}(x))}}{p\,\nu\left(C+\lfloor u\rfloor_{C^\nu\left(B_{2R}(x)\right)}\right)^p}\right)^\frac{1}{d+p\,\nu}\,.
\]
With $C>0$, the denominator in the right-hand side is never zero. With the choice
\[\label{definition.c}
C:=\left(\tfrac d{p\,\nu}\right)^\frac1{p} \frac{\|u\|_{\mathrm L^p(B_{2R}(x))}}{R^{\frac{d+p\,\nu}{p}}}\,,
\]
we are sure that $\rho_\star<R$. Hence, by evaluating~\eqref{evaluate.gamma.negative} at $\rho_\star$ we obtain
\begin{multline*}
\nrm u{\mathrm L^\infty(B_{R}(x))}\le2^{1-\frac1{p}}\left(1+\frac d{\omega_d}\right)^\frac{1}{p}\,\(\left(\tfrac d{p\,\nu}\right)^{\frac{p\,\nu}{d+p\,\nu}}+\left(\tfrac{p\,\nu}d\right)^{\frac d{d+p\,\nu}}\)^{1/p}\,\\
\|u\|^\frac{p\,\nu}{d+p\,\nu}_{\mathrm L^p(B_{2R}(x))}\left(C+\lfloor u\rfloor_{C^\nu\left(B_{2R}(x)\right)}\right)^\frac d{d+p\,\nu}\,.
\end{multline*}
Inequality~\eqref{interpolation.inequality.cpt6} is deduced from the above one. Inequality~\eqref{interpolation.inequality.Rn} can be deduced from~\eqref{interpolation.inequality.cpt6} by taking $R\to \infty$. The proof is completed.
\end{proof}

%%%%%%%%%%%%%%%%%%%%%%%%%%%%%%%%%%%%%%%%%%%%%%%%%%%%%%%%%%%%%%%%%%%%%%
\subsection{Weighted Poincar\'e inequalities on bounded domains}

%---------------------------------------------------------------------
\begin{lemma}\label{Lem:weightedPoincare} Let $b\ge 0$ be a continuous compactly supported function such that the domains $\{x\in\R^d\,:\,b(x)\ge \mbox{const}\}$ are convex. Then for any function $f\in\mathrm L^2(\R^d, b(x)\dx)$ with $|\nabla f|\in\mathrm L^2(\R^d, b(x)\dx)$, we have \index{weighted Poincar\'e inequality}{that}
\begin{equation}\label{Lem.Log.Est.4}
\int_{\R^d} \left|f-\overline f_b\right|^2\, b\,\dx \le \lambda_b\,\int_{\R^d} |\nabla f|^2\, b\,\dx\,.
\end{equation}
where $\overline f_b=\int_{\R^d} f\,b\,\dx/\ird b$ and
\be{Lem.Log.Est.4b}
\lambda_b=\frac{|\supp\,b|\,\nrm b\infty}{2\ird b}\,\mathrm{diam}(\supp\,b)^2\,.
\ee
\end{lemma}
%---------------------------------------------------------------------
The proof follows from~\cite[Lemma~3]{Moser1964} and we recall it here as we need an explicit constant $\lambda_b$.
\begin{proof} We first notice that by symmetry in the $x,y$ variables,
\begin{equation}\label{Poincare.b.proof.1}\begin{split}
\int_{\R^d} \left|f(x)-\overline f_b\right|^2\,b(x)\,\dx &= \frac{\iint_{\R^d\times\R^d} \left|f(x)-f(y)\right|^2\,b(x)\,b(y)\,\dx\dy}{2\int_{\R^d} b(y)\,\dy}\\
&= \frac{\iint_{\{b(x)\le b(y)\}} \left|f(x)-f(y)\right|^2\,b(x)\,b(y)\,\dx\dy}{\int_{\R^d} b(y)\,\dy}\\
\end{split}
\end{equation}
Let $I_{xy}$ be the segment connecting $x$ and $y$ and let $\rd s$ be the length element on $I_{xy}$. In the domain $\{b(x)\le b(y)\}$, $x$ and $y$ are such that $b(x)\le b(y)$, then the Cauchy-Schwarz inequality gives
\[\begin{split}
\left|f(x)-f(y)\right|^2\,b(x)\,b(y) &\le\(\int_{I_{xy}} \sum_{i=1}^{d}\sqrt{b}\,(\partial_{x_i}f)\, \frac{\,\dx_i}{\sqrt{b}} \)^2\,b(x)\,b(y)\\
& \le \left(\int_{I_{xy}} b\,|\nabla f|^2\,\rd s\right) \left(\int_{I_{xy}} \frac{\rd s}{b}\right) b(x)\,b(y)\,.
\end{split}
\]
Since $B_x:=\{z\in\R^d:b(z)\ge b(x)\}$ is convex and $I_{xy}\subset B_x$, we deduce that $b$ achieves its minimum at the end point of $I_{xy}$, \emph{i.e.}, at $x$. As a consequence, we have
\[
b(x)\int_{I_{xy}} \frac{\rd s}{b }\le \int_{I_{xy}} \rd s\le\mathrm{diam}(\supp\,b)\,.
\]
Hence we obtain 
\[
\left|f(x)-f(y)\right|^2\,b(x)\,b(y) \le\mathrm{diam}(\supp\,b) \left(\int_{I_{xy}} b\,|\nabla f|^2 \rd s\right) b(y)\,.
\]
The proof can be completed by integrating the above expression first in $x$ and then in $y$\, and using $\int_{\{b>0\}}\,\dy=|\supp\,b|$ and  \eqref{Poincare.b.proof.1}.
\end{proof}

%%%%%%%%%%%%%%%%%%%%%%%%%%%%%%%%%%%%%%%%%%%%%%%%%%%%%%%%%%%%%%%%%%%%%%
%%%%%%%%%%%%%%%%%%%%%%%%%%%%%%%%%%%%%%%%%%%%%%%%%%%%%%%%%%%%%%%%%%%%%%
\section{The constant in Moser's Harnack inequality}

Let $\Omega$ be an open domain and let us consider a positive \emph{weak solution} to
\be{HE.coeff}
\frac{\partial v}{\partial t}=\nabla\cdot\big(A(t,x)\,\nabla v\big)
\ee
on $\Omega_T:=\left(0, T\right)\times \Omega$, where $A(t,x)$ is a real symmetric matrix with bounded measurable coefficients satisfying the \emph{uniform ellipticity condition}
\be{HE.coeff.lambdas}
0\le\lambda_0\,|\xi|^2\le\xi\cdot(A\, \xi)\le\lambda_1\,|\xi|^2\quad\forall\,(t,x,\xi)\in\R^+\times\Omega_T\times\R^d\,,
\ee
where $\xi\cdot(A\, \xi) = \sum_{i,j=1}^d A_{i,j}\xi_i\xi_j $ and $\lambda_0, \lambda_1$ are positive constants. 

%%%%%%%%%%%%%%%%%%%%%%%%%%%%%%%%%%%%%%%%%%%%%%%%%%%%%%%%%%%%%%%%%%%%%%
\subsection{Harnack inequality for linear parabolic equations}

Let us consider the neighborhoods
\be{cylinder.harnack}
\begin{split}
& D_R^+(t_0,x_0):=(t_0+\tfrac34\,R^2,t_0+R^2)\times B_{R/2}(x_0)\,,\\
& D_R^-(t_0,x_0):=\left(t_0-\tfrac34\,R^2,t_0-\tfrac14\,R^2\right)\times B_{R/2}(x_0)\,,
\end{split}\ee
and the constant
\be{h}
\mathsf h:=\exp\left[2^{d+4}\,3^d\,d+c_0^3\,2^{2\,(d+2)+3}\left(1+\frac{2^{d+2}}{(\sqrt2-1)^{2\,(d+2)}}\right)\sigma\right]
\ee
where
\be{c_0}
c_0=3^\frac2{d}\,2^\frac{(d+2)\,(3\,d^2+18\,d+24)+13}{2\,d}\(\tfrac{(2+d)^{1+\frac4{d^2}}}{d^{1+\frac2{d^2}}}\)^{(d+1)(d+2)}\,\mathcal K^\frac{2\,d+4}{d}\,,
\ee
\be{sigma}
\sigma=\sum_{j=0}^{\infty}\left(\tfrac34\right)^j\,\big((2+j)\,(1+j)\big)^{2\,d+4}\,.
\ee
The constant $\mathcal K$ in~\eqref{c_0} is defined in~\eqref{estim.S-p} and it is the constant in the inequality~\eqref{sob.step2}.
 Let us define
\be{h-bar}
\overline{\mathsf h}:=\mathsf h^{\lambda_1+1/\lambda_0}\,.
\ee
We claim that the following \emph{\idx{Harnack inequality}} holds: 
%---------------------------------------------------------------------
\begin{theorem}\label{Claim:3} Let $T>0$, $R\in(0,\sqrt T)$, and take $(t_0,x_0)\in(0,T)\times\Omega$ such that $\left(t_0-R^2, t_0+R^2\right)\times B_{2\,R}(x_0)\subset\Omega_T$. Under Assumption~\eqref{HE.coeff.lambdas}, if $v$ satisfies
\be{weak.solution}
\iint_{(0,T)\times\Omega}\big(-\varphi_t\,v+\nabla\varphi\cdot(A\,\nabla v)\big)\dx\dt=0
\ee
for any $\varphi\in C^{\infty}_c((0,T)\times\Omega)$, then
\be{harnack}
\sup_{D^{-}_R(t_0,x_0)} v\le\overline{\mathsf h}\,\inf_{D^{+}_R(t_0,x_0)} v\,.
\ee
\end{theorem}
%---------------------------------------------------------------------
This result is known from~\cite{Moser1964,Moser1971}. However, to the best of our knowledge, a complete constructive proof and an expression of $\overline{\mathsf h}$ like~\eqref{h-bar} was still missing.

%%%%%%%%%%%%%%%%%%%%%%%%%%%%%%%%%%%%%%%%%%%%%%%%%%%%%%%%%%%%%%%%%%%%%%%%
\subsection{Truncation functions}\label{Appendix:Truncation}

In what follows we introduce a family of particular truncation functions that we shall use as test functions in~\eqref{weak.solution}.
%---------------------------------------------------------------------
\begin{lemma}[\cite{Bonforte2012a}]\label{lem.test.funct} Fix two balls $B_{R_1}\subset B_{R_0}\subset\subset\Omega$. Then there exists a test function $\varphi_{R_1, R_0}\in C_0^1(\Omega)$, with $\nabla \varphi_{R_1, R_0}\equiv0$ on $\partial\Omega$, which is radially symmetric and piecewise $C^2$ as a function of $r$, satisfies $\supp(\varphi_{R_1, R_0})=B_{R_0}$ and $\varphi_{R_1, R_0}=1$ on $B_{R_1}$, and moreover satisfies the bounds
\be{test.estimates}
\|\nabla\varphi_{R_1, R_0}\|_\infty\le\frac2{R_0-R_1}\quad\mbox{and}\quad \|\Delta\varphi_{R_1, R_0}\|_\infty\le\frac{4\,d}{(R_0-R_1)^2}.
\ee
\end{lemma}
%---------------------------------------------------------------------
\begin{proof} With a standard abuse of notation, we write indifferently that a radial function is a function of $x$ or of $|x|$. Let us consider the radial test function defined on $B_{R_0}$
\be{test.funct}
\varphi_{R_1, R_0}(|x|)=\left\{
\begin{array}{lll}
1\,&\quad\mbox{if}\quad0\le |x|\le R_1\\[3mm]
1-\frac{2(|x|-R_1)^2}{(R_0-R_1)^2}\,&\quad\mbox{if}\quad R_1<|x|\le\frac{R_0+R_1}2\\[3mm]
\frac{2(R_0-|x|)^2}{(R_0-R_1)^2}\,&\quad\mbox{if}\quad\frac{R_0+R_1}2<|x|\le R_0\\[3mm]
0\,&\quad\mbox{if}\quad|x|>R_0\\[3mm]
\end{array}
\right.
\ee
for any $0<R_1<R_0$. We have
\begin{equation*}
\nabla \varphi_{R_1, R_0}(|x|)=\left\{
\begin{array}{lll}
0\,&\mbox{if}\quad0\le |x|\le R_1 \mbox{ or if }|x|>R_0\\[3mm]
-\frac{4(|x|-R_1)}{(R_0-R_1)^2}\frac{x}{|x|}\,&\quad\mbox{if}\quad R_1<|x|\le\frac{R_0+R_1}2\\[3mm]
-\frac{4(R_0-|x|)}{(R_0-R_1)^2}\frac{x}{|x|}\,&\quad\mbox{if}\quad\frac{R_0+R_1}2<|x|\le R_0\\[3mm]
\end{array}
\right.
\end{equation*}
and, recalling that $\Delta\varphi(|x|)=\varphi''(|x|)+(d-1)\varphi'(|x|)/|x|$, we have
\begin{equation*}
\Delta \varphi_{R_1, R_0}(|x|)=\left\{
\begin{array}{lll}
0\,&\quad\mbox{if}\quad0\le |x|\le R_1 \mbox{ or if }|x|>R_0\\[3mm]
-\frac4{(R_0-R_1)^2}-\frac{d-1}{|x|}\frac{4(|x|-R_1)}{(R_0-R_1)^2}\,&\quad\mbox{if}\quad R_1<|x|\le\frac{R_0+R_1}2\\[3mm]
-\frac4{(R_0-R_1)^2}-\frac{d-1}{|x|}\frac{4(R_0-|x|)}{(R_0-R_1)^2}\,&\quad\mbox{if}\quad\frac{R_0+R_1}2<|x|\le R_0\\[3mm]
\end{array}
\right.
\end{equation*}
and easily obtain the bounds~\eqref{test.estimates}.\end{proof}

%%%%%%%%%%%%%%%%%%%%%%%%%%%%%%%%%%%%%%%%%%%%%%%%%%%%%%%%%%%%%%%%%%%%%%%%
\subsection{Upper and lower Moser iteration}\label{Sec:MoserIteration}

Let us start by recalling the definition of the parabolic cylinders
\be{Parab.Cylinders}\begin{split}
&Q_\varrho=Q_\varrho(0,0)=\left\{ |t|<\varrho^2\,,\;|x|<\varrho\right\}=(-\varrho^2,\varrho^2)\times B_\varrho(0)\,,\\
&Q^+_\varrho=Q_\varrho(0,0)=\left\{ 0<t<\varrho^2\,,\;|x|<\varrho\right\}=(0,\varrho^2)\times B_\varrho(0)\,,\\
&Q^-_\varrho=Q_\varrho(0,0)=\left\{ 0<-t<\varrho^2\,,\;|x|<\varrho\right\}=(-\varrho^2,0)\times B_\varrho(0)\,.
\end{split}
\ee
The following Lemma is the result of a (nowadays standard) procedure called the \idx{Moser iteration}, which relies on the inequality~\eqref{sob.step2}. Here we provide a quantitative and constructive proof, with explicit constants. From here on, we assume that $u$ is a positive solution.
%---------------------------------------------------------------------
\begin{lemma}\label{Lem.Moser}
Assume that $r$ and $\rho$ are such that $1/2\le \varrho\le r\le1$ and $\mu:=\lambda_1+1/\lambda_0$. Let $v$ be a nonnegative solution to~\eqref{HE.coeff} which satisfies~\eqref{weak.solution}. Then there exists a positive constant $c_1$ depending only on $d$ such that
\be{Lem.Moser.Upper}
\sup_{Q_\varrho} v^p \le\frac{c_1}{(r-\varrho)^{d+2}}\iint_{Q_r}v^p \dx\dt\quad\forall\,p\in\(0,\tfrac1\mu\)
\ee
and
\be{Lem.Moser.Lower}
\sup_{Q^-_\varrho} v^p \le\frac{c_1}{(r-\varrho)^{d+2}}\iint_{Q_r^-}v^p \dx\dt\quad\forall\,p\in\(-\tfrac1\mu,0\).
\ee
\end{lemma}
%---------------------------------------------------------------------
The second estimate is a lower bound on $v$ because $p$ is negative. Our contribution is to establish that the constant $c_1=c_1(d)$ is given by
\be{Lem.Moser.constant}
c_1=3^{\gamma-1}\(2^{2\gamma^2+7(\gamma-1)}\,\gamma^{(\gamma+1)(2\gamma-1)}\,d^{(\gamma+1)(\gamma-1)}\,\mathcal K^{\gamma-1}\)^\frac\gamma{(\gamma-1)^2}\,,
\ee
where $\gamma=(d+2)/d$ if $d\ge3$, $\gamma=5/3$ if $d=1$ or $2$, and $\mathcal K$ is as in~\eqref{estim.S-p}.

\begin{proof} We first notice that it is sufficient to prove the lemma for $\varrho=1/2$ and $r=1$. We can change variables according to
\be{admissible.transformations}
t\mapsto \alpha^2\,t+t_0\quad\mbox{and}\quad x\mapsto \alpha\,x+x_0
\ee
without changing the class of equations: $\lambda_0$ and $\lambda_1$ are invariant under~\eqref{admissible.transformations}. Therefore it is sufficient to prove
\[
\sup_{Q_{\theta/2}}v^p\le\frac{c_1}{\theta^{d+2}}\iint_{Q_{\theta}}v^p\dx\dt\quad\forall\,\theta>0\,.
\]
We recover~\eqref{Lem.Moser.Upper} by setting $\theta=r-\varrho$ and applying the above inequality to all cylinders in $Q_r$ obtained by translation from $Q_\theta$ with admissible transformations~\eqref{admissible.transformations}. The centers of the corresponding cylinders certainly cover $Q_\varrho$ and~\eqref{Lem.Moser.Upper} follows. Analogously, one reduces~\eqref{Lem.Moser.Lower} to the case $\varrho=1/2$ and $r=1$.
\begin{steps}
\stepitem\textit{Energy estimates.} By definition of weak solutions, we have
\be{Lem.Moser.Proof.1}
\iint_{Q_1}\big(-\varphi_t\,v+\nabla\varphi\cdot(A\,\nabla v)\big)\dx\dt=0
\ee
for any test function $\varphi$ which is compactly supported in $B_1=\{x\in\R^d\,:\,|x|<1\}$, for any fixed $t$. For any $p\in\R\setminus\{0,1\}$, we define
\[
w=v^{p/2}\quad\mbox{and}\quad \varphi=p\,v^{p-1}\,\psi^2\,,
\]
where $\psi$ is a $C^\infty$ function. Both $\varphi$ and $\psi$ have compact support in $B_1$ for fixed $t$. We rewrite~\eqref{Lem.Moser.Proof.1} in terms of $w$ and $\psi$ as
\begin{multline}\label{Lem.Moser.Proof.2}
\tfrac14\int_{t_1}^{t_2}\int_{B_1}\psi^2\,\partial_t w^2\dx\dt+\tfrac{p-1}{p}\int_{t_1}^{t_2}\int_{B_1}\psi^2\,\nabla w\cdot(A\,\nabla w)\dx\dt\\
=-\int_{t_1}^{t_2}\int_{B_1}\psi\,w\,\nabla \psi\cdot(A\,\nabla w)\dx\dt
\end{multline}
where we integrate over a slice $t_1<t<t_2$ of $Q_1$, i.e. $|t_1|, |t_2|<1$. Setting $p\ne 1$,
\[
\varepsilon=\tfrac12\left|1-\tfrac1{p}\right|
\]
and recalling that
\[\label{Lem.Moser.Proof.3}
\psi\,w\,\nabla \psi\cdot(A\,\nabla w)\le\frac1{4\,\varepsilon}\,w^2 \nabla \psi\cdot(A\,\nabla \psi)+\varepsilon\,\psi^2\,\nabla w\cdot(A\,\nabla w)\,,
\]
we deduce from~\eqref{Lem.Moser.Proof.2} that
\begin{multline*}
\pm\,\frac14\int_{t_1}^{t_2}\int_{B_1} \partial_t \left(\psi^2\,w^2\right)\dx\dt+\varepsilon\int_{t_1}^{t_2}\int_{B_1}\psi^2\,\nabla w\cdot(A\,\nabla w)\dx\dt\\
\le\frac14\int_{t_1}^{t_2}\int_{B_1}\(\frac1\varepsilon\,\nabla \psi\cdot(A\,\nabla \psi)+2\,|\psi\,\psi_t|\)w^2\dx\dt\,,
\end{multline*}
where the plus sign in front of the first integral corresponds to the case $1/p<1$, while the minus sign corresponds to $1/p>1$. Recall that $p$ can take negative values. Using the ellipticity condition~\eqref{HE.coeff.lambdas} and~\eqref{Lem.Moser.Proof.2}, we deduce
\begin{multline}\label{Lem.Moser.Proof.5}
\pm\,\frac14\int_{t_1}^{t_2}\int_{B_1} \partial_t \left(\psi^2\,w^2\right)\dx\dt+\lambda_0\,\varepsilon\int_{t_1}^{t_2}\int_{B_1}\psi^2 \left|\nabla w\right|^2\dx\dt\\
\le\frac14\int_{t_1}^{t_2}\int_{B_1}\(\frac{\lambda_1}\varepsilon\,|\nabla \psi|^2+2\,|\psi\,\psi_t|\)w^2\dx\dt\,,
\end{multline}
recall that $t_1$ and $t_2$ are arbitrarily chosen for the moment. By choosing a suitable test function $\psi$, compactly supported in $Q_r\subset Q_1$, and such that
\begin{equation}\label{Lem.Moser.Proof.6b}
\| \nabla \psi \|_{\mathrm L^\infty(Q_1)}\le\frac2{r-\varrho}\quad\mbox{and}\quad
\|\psi_t \|_{\mathrm L^\infty(Q_1)}\le\frac{4}{r-\varrho}\,,
\end{equation}
we have
\begin{equation}\label{Lem.Moser.Proof.7}\begin{split}
&\frac14\iint_{Q_1}\(\frac{\lambda_1}\varepsilon\,|\nabla \psi|^2+2\,|\psi\,\psi_t|\)w^2\dx\dt\\
&\le\(\frac{\lambda_1}\varepsilon\,\frac1{(r-\varrho)^2}+\frac1{r-\varrho}\)\iint_{Q_r}\!\!\!\!w^2\dx\dt\\
&\le\frac1{(r-\varrho)^2}\(\frac{\lambda_1}\varepsilon+1\)\iint_{Q_r}\!\!\!\!w^2\dx\dt\,.
\end{split}\end{equation}
for any $r$ and $\varrho$ such that $0<\varrho<r\le 1$. The choice~\eqref{Lem.Moser.Proof.6b} is always possible, see Lemma~\ref{lem.test.funct}. If $1/p>1$, let us take $\tilde{t}\in (-\varrho^2, \varrho^2)$ to be such that
\[
\int_{B_\varrho}w^2(\tilde{t},x)\dx \ge\frac14\sup_{0<|t|<\varrho^2}\int_{B_\varrho}w^2(t,x)\dx
\]
and choose $\psi$ such that $\psi(0,x)=1$ on $Q_\varrho$ and $\psi(0,x)=0$ outside $Q_r$, so that
\be{Lem.Moser.Proof.8}\begin{split}
\sup_{0<|t|<\varrho^2}\int_{B_\varrho}w^2(t,x)\dx &\le 4 \int_{B_\varrho}w^2(\tilde{t},x)\dx \\ &\le 4 \int_{B_r}w^2(\tilde{t},x)\,\psi^2(\tilde{t},x)\dx \\
& \le 4 \iint_{Q_r} \partial_t \left(\psi^2\,w^2\right)\dx\dt\,,
\end{split}\ee
where in the last line we have used the Fundamental Theorem of Calculus. The same holds true if we replace $Q_r$ by $Q^+_r$ and $0<|t|<\varrho^2$ by $0<t<\varrho^2$.

If $1/p<1$ (which includes the case $p<0$), similar arguments yield
\be{Lem.Moser.Proof.9}
\sup_{-\varrho^2<t<0}\int_{B_\varrho}w^2(t,x)\dx
\le 4 \iint_{Q^-_r} \partial_t \left(\psi^2\,w^2\right)\dx\dt\,.
\ee

\stepitem\textit{Space-time Sobolev's inequality.} For any $f\in\mathrm H^1(Q_R)$, we have
\begin{multline}\label{Lem.Moser.Proof.10}
\iint_{Q_R}f^{2\,\gamma}\dx \dt\le\,2\,\pi^2\,\mathcal K
\left[\frac1{R^2}\iint_{Q_R} f^2\dx\dt+\iint_{Q_R}\big|\nabla f\big|^2\dx\dt\right]\\\times\sup_{|s|\in (0,\varrho^2)}\left[\int_{B_R}f^2(s,x)\dx\right]^\frac2d
\end{multline}
with $\gamma=1+2/d$ if $d\ge3$. If $d=1$ or $2$, we rely on~\eqref{sob.step2}, take $\gamma=5/3$, use H\"older's inequality with $2\,\gamma=10/3<4$ and $\pcc\ge4$ if $d=2$, $\pcc>4$ if $d=1$. In order to fix ideas, we take $\pcc=4$ if $d=2$ and $\pcc=8$ if $d=1$. Hence
\[
\iint_{Q_R}f^{2\,\gamma}\dx \dt \le |Q_1|^{1-\frac{2\,\gamma}{\pcc}}
\(\iint_{Q_R}f^{\pcc}\dx \dt\)^{1-\frac{2\,\gamma}{\pcc}}\,.
\]
From the numerical inequality $\omega_d/d\le \pi^2$, we deduce that $|Q_1|=|(-1,1)|\,|B_1|\le 2\,\omega_d/d\le \, 2\,\pi^2$ in any dimension. 
\stepitem\textit{The case $p>0$ and $p\ne 1$.} Assume that $1/2\le\varrho<r\le1$. We work in the cylinder $Q_r=\supp(\psi)$. Here, we choose $\psi(t,x)=\varphi_{\rho, r}(|x|)\,\varphi_{\rho^2, r^2}(|t|)$ where $\varphi_{\rho, r}$ and $\varphi_{\rho^2, r^2}$ are defined in~\eqref{test.funct}, so that $\psi=1$ on $Q_\varrho$ and $\psi=0$ outside~$Q_r$.

Collecting inequalities~\eqref{Lem.Moser.Proof.5},~\eqref{Lem.Moser.Proof.7} and~\eqref{Lem.Moser.Proof.8}, we obtain
\[\label{Lem.Moser.Proof.11}
 \sup_{0<|t|<\varrho^2}\int_{B_\varrho}w^2(t,x)\dx+\lambda_0\,\varepsilon\iint_{Q_\varrho} \left|\nabla w\right|^2\dx\dt
\le\frac{\varepsilon^{-1}\,\lambda_1+1}{(r-\varrho)^2}\iint_{Q_r}\,w^2\dx\dt\,.
\]
Now apply~\eqref{Lem.Moser.Proof.10} to $f=w$ and use the above estimates to get
\begin{align*}\label{Lem.Moser.Proof.12}
&\iint_{Q_\varrho}w^{2\,\gamma}\dx \dt\\
& \le\,2\,\pi^2\,\mathcal K
 \left[\frac1{\varrho^2}\iint_{Q_\varrho}\!\!\!\! w^2\dx\dt+\iint_{Q_\varrho}\big|\nabla w\big|^2\dx\dt\right]
\,\sup_{|s|\in (0,\varrho^2)}\(\int_{B_\varrho}\!\!\!\!w^2(s,x)\dx\)^\frac2d\\
&\le\,2\,\pi^2\,\mathcal K
 \left[\frac1{\varrho^2}\iint_{Q_\varrho} w^2\dx\dt+\frac{\varepsilon^{-1}\,\lambda_1+1}{(r-\varrho)^2\,\lambda_0\,\varepsilon}\iint_{Q_r}\,w^2\dx\dt\right]\\
 &\hspace*{4cm}\times\(\frac{\varepsilon^{-1}\,\lambda_1+1}{(r-\varrho)^2}\iint_{Q_r}\,w^2\dx\dt\)^\frac2d\\
&\le\,2\,\pi^2\,\mathcal K
 \left[\frac1{\varrho^2}+\frac{\varepsilon^{-1}\,\lambda_1+1}{(r-\varrho)^2\,\lambda_0\,\varepsilon}\right]
\left[\frac{\varepsilon^{-1}\,\lambda_1+1}{(r-\varrho)^2}\right]^\frac2d\(\iint_{Q_r}\,w^2\dx\dt\)^{\frac2{d}+1}\\
&\hspace*{4cm}:=A(d,\varrho,r,\lambda_0,\lambda_1,\varepsilon, 2\,\pi^2\,\mathcal K) \left( \iint_{Q_r}\,w^2\dx\dt\right)^{\gamma}\,.
\end{align*}
Using the fact that $\mu=\lambda_1+1/\lambda_0>1$ and $1/2\le\varrho<r\le1$, we can estimate the constant $A$ as follows:
\[\begin{split}
A &\le 2\,\pi^2\,\mathcal K\left[\frac1{\varrho^2}+\frac{\varepsilon^{-1}\,\lambda_1+1}{(r-\varrho)^2\,\lambda_0\,\varepsilon}\right]
\(\frac{\varepsilon^{-1}\,\lambda_1+1}{(r-\varrho)^2}\)^\frac2d\\
 &\le\frac{ 2\,\pi^2\,\mathcal K}{(r-\varrho)^{2\,\gamma}}\(\tfrac12+\tfrac{\lambda_1}{\varepsilon^2\,\lambda_0}\)
\(\tfrac{\lambda_1}\varepsilon\)^\frac2d\\
& \le\frac{2\,\pi^2\,\mathcal K}{(r-\varrho)^{2\,\gamma}}\(1+\tfrac{\mu^2}{\varepsilon^2}\)
\(\tfrac\mu\varepsilon\)^\frac2d
 \le\frac{2^5\,\mathcal K}{(r-\varrho)^{2\,\gamma}}\(1+\tfrac\mu\varepsilon\)^{\gamma+1}
\end{split}\]
where we have used that $\lambda_1/\lambda_0\le\frac12(\lambda_1^2+1/\lambda_0^2)\le\frac12(\lambda_1+1/\lambda_0)^2=\mu^2$ and $\pi\le4$.

\noindent\textit{First iteration step.} Recall that $w=v^{p/2}$, $\varepsilon=\frac12\left|1-\frac1p\right|$, and $\gamma=1+\frac2{d}$ if $d\ge3$, $\gamma=5/3$ if $d=1$ or $2$, $\mu=\lambda_1+1/\lambda_0>1$ and $1/2\le\varrho<r\le1$. We can summarize these results by
\[\label{Lem.Moser.Proof.13}
\left(\iint_{Q_\varrho}v^{\gamma\,p}\dx \dt\right)^{\frac1{\gamma\,p}}
\le\(\frac{(2^5\,\mathcal K)^{\frac1{\gamma}}}{(r-\varrho)^2}\)^\frac1p\left(1+\tfrac\mu\varepsilon\right)^{\frac{\gamma+1}{\gamma\,p}}\left(\iint_{Q_r}\,v^p\dx\dt\right)^\frac1p
\]
for any $p>0$ such that $p\ne 1$. For any $n\in\NN$, let
\[\label{Lem.Moser.Proof.14}
\varrho_n=\frac12\left(1-2^{-n}\right)\,,\quad p_n=\frac{\gamma+1}2\,\gamma^{n-n_0}=p_0\,\gamma^n\,,\quad \varepsilon_n=\frac12\,\left|1-\frac1{p_n}\right|
\]
for some fixed $n_0\in\NN$. Note that $\varrho_0=1$, $p_0=\frac{1+\gamma}{2\,\gamma^{n_0}}$, $\varrho_n$ monotonically decrease to $1/2$, and $p_n$ monotonically increase to $\infty$. We observe that for all $n$, $n_0\in\NN$, we have $p_n\ne 1$ and, as a consequence, $\varepsilon_n>0$. Indeed, if $d\ge3$, $p_n=1$ would mean that
\[
n_0-n=\frac{\log\left(\frac{1+\gamma}2\right)}{\log\gamma}=\frac{\log\(1+\frac1d\)}{\log\(1+\frac2d\)}
\]
and, as a consequence, $0<n_0-n\le\log(4/3)/\log(5/3)<1$, a contradiction with the fact that $n$ and $n_0$ are integers. The same argument holds if $d=1$ or $d=2$ with $n_0-n=\log(4/3)/\log(5/3)$, as $\gamma=5/3$ corresponds to the value of $\gamma$ for $d=1$, $2$ or $3$. It is easy to check that for any $n\ge 0$,
\[\label{Lem.Moser.Proof.15}
|p_n-1|\ge\min\{p_{n_0}- 1,1-p_{n_0-1}\}=\min\left\{\tfrac1d,\tfrac1{d+2}\right\}=\tfrac1{d+2}\,.
\]
For an arbitrary $p\in(0,1/\mu)$, we choose
\[
n_0={\rm i.p.}\left(\frac{\log\left(\frac{1+\gamma}{2\,p}\right)}{\log\gamma}\right)+1
\]
where ${\rm i.p.}$ denotes the integer part, so that $0<p_0\le p<\gamma\,p_0$. By monotonicity of the $\mathrm L^q$ norms, that is,
\[
\left(\iint_{Q_r} v^{p_0}\,\frac{\dx\dt}{|Q_r|}\right)^{\frac1{p_0}}\le \left(\iint_{Q_r} v^p\, \frac{\dx\dt}{|Q_r|}\right)^\frac1p
\le \left(\iint_{Q_r} v^{\gamma\,p_0} \frac{\dx\dt}{|Q_r|}\right)^{\frac1{\gamma\,p_0}}\,,
\]
it is sufficient to prove inequality~\eqref{Lem.Moser.Upper} for $p=p_0$.

Let us define $p_\mu\in(p_0\,\mu,1]$ such that
\be{Lem.Moser.Proof.16}
1+\frac\mu{\varepsilon_n}=1+\frac{2\,\mu\,p_n}{|p_n-1|}=1+\frac{2\,\mu\,p_0\,\gamma^n}{|p_n-1|}\le 1+2\,(d+2)\,\gamma^n\le4\,(d+2)\,\gamma^n=4\,d\,\gamma^{n+1}
\ee
because $d+2=d\,\gamma$ if $d\ge3$ and $\gamma=5/3$ if $d\le3$. Finally, let us define
\[
Y_n:=\left(\iint_{Q_{\varrho_n}}v^{p_n}\dx \dt\right)^\frac1{p_n}\,,\;I_0=(2^5\,\mathcal K)^{\frac1{\gamma}} (4\,d\,\gamma^2)^\frac{\gamma+1}\gamma
\]
and $C=4\,\gamma^\frac{\gamma+1}\gamma$, $\theta=\frac1\gamma\in (0,1)$, and $\xi=\frac1{p_0}$.

\noindent\textit{Iteration.} Summing up, we have the following iterative inequality
\[
Y_n\le\(\frac{(2^5\,\mathcal K)^{\frac1\gamma}}{(\varrho_{n-1}-\varrho_n)^2}
\left(1+\tfrac\mu{\varepsilon_n}\right)^\frac{\gamma+1}\gamma\)^\frac1{p_{n-1}}\,Y_{n-1}\,.
\]
Using $\varrho_{n-1}-\varrho_n=2^{-n}$ and inequality~\eqref{Lem.Moser.Proof.16}, we obtain
\[\label{hyp.num}
Y_n\le I_{n-1}^{\,\xi\,\theta^{n-1}}\,Y_{n-1}\quad\mbox{with}\quad I_{n-1}\le I_0\,C^{\,n-1}\,.
\]
%---------------------------------------------------------------------
\begin{lemma}[\cite{Bonforte2012a}]\label{lem.num.iter} The sequence $(Y_n)_{n\in\N}$ is a bounded sequence and satisfies
\[\label{iteration.num}
Y_\infty:=\limsup_{n\to+\infty}Y_n\le I_0^{\frac{\xi}{1-\theta}}\,C^{\frac{\xi\,\theta}{(1-\theta)^2}}\,Y_0\,.
\]
\end{lemma}
%---------------------------------------------------------------------
The proof follows from the observation that
\begin{multline*}
Y_n \le I_{n-1}^{\,\xi\,\theta^{n-1}}Y_{n-1}\le \left(I_0\,C^{\,n-1}\right)^{\,\xi\,\theta^{n-1}}\,Y_{n-1}
=I_0^{\,\xi\,\theta^{n-1}}\,C^{\,\xi\,(n-1)\,\theta^{n-1}}\,Y_{n-1}\\
\le \prod_{j=0}^{n-1}\,I_0^{\,\xi\,\theta^j} C^{\,\xi\,j\,\theta^j}\,Y_0
=I_0^{\,\xi\sum_{j=0}^{n-1}\theta^j} C^{\,\xi\sum_{j=0}^{n-1}j\,\theta^j}\,Y_0\,.
\end{multline*}

With the estimates
\[
\left(\iint_{Q_1}\,v^{p_0}\dx\dt\right)^{\frac1{p_0}}\le|Q_1|^{\frac1{p_0}-\frac1p}\,\left(\iint_{Q_1}\,v^{p}\dx\dt\right)^\frac1p\,,
\]
$\frac1{p_0}-\frac1p\le\frac{\gamma-1}p$ and $|Q_1|=2\,|B_1|\le2\,\pi^2$, we obtain
\[
\sup_{Q_{1/2}}v\le\(2^5\,\mathcal K\,(4\,d\,\gamma^2)^{\gamma+1 }\)^{\frac1p\,\frac\gamma{\gamma-1}}
\(4^\gamma\,\gamma^{\gamma+1}\)^{\frac1p\,\frac\gamma{(\gamma-1)^2}}\,(2\,\pi^2)^\frac{\gamma-1}p
\left(\iint_{Q_1}\,v^{p}\dx\dt\right)^\frac1p
\]
which, using $2\,\pi^2\le24$ and after raising to the power $p$, is~\eqref{Lem.Moser.Upper} with $c_1$ given by~\eqref{Lem.Moser.constant}.

\stepitem\textit{The case $p<0$.} Assume that $1/2\le\varrho<r\le1$. We work in the cylinder $Q_r^-=\supp(\psi)$. Here, we choose $\phi(t,x)=\varphi_{\rho, r}(|x|)\,\varphi_{\rho^2, r^2}(-t)$, where $\varphi_{\rho,r}$ and $\varphi(\rho^2, r^2)$ are as in~\eqref{test.funct}, so that $\psi=1$ on $Q_\varrho^-$ and $\psi=0$ outside $Q_r^-$.

After collecting~\eqref{Lem.Moser.Proof.5},~\eqref{Lem.Moser.Proof.7} and~\eqref{Lem.Moser.Proof.9}, we obtain
\[\label{Lem.Moser.Proof.19}
\sup_{-\varrho^2<t<0}\int_{B_\varrho}w^2(t,x)+\lambda_0\,\varepsilon\iint_{Q_\varrho^-} \left|\nabla w\right|^2\dx\dt
\le\frac{\varepsilon^{-1}\,\lambda_1+1}{(r-\varrho)^2}\iint_{Q_r^-}\,w^2\dx\dt\,.
\]
Then the proof follows exactly the same scheme as for $p>0$, with the simplification that we do not have to take extra precautions in the choice of $p$. The constant $c_1$ is the same.
\end{steps}
\end{proof}

%%%%%%%%%%%%%%%%%%%%%%%%%%%%%%%%%%%%%%%%%%%%%%%%%%%%%%%%%%%%%%%%%%%%%%%%
\subsection{Logarithmic Estimates}\label{Sec:LogarithmicEstimates}

We prove level set \index{logarithmic estimates}{estimates} on the solutions by means of Cacciop\-poli-type energy estimates in terms of
\[
w=-\log v\,,
\]
where $v$ and $w$ respectively solve~\eqref{HE.coeff} and
\be{HE.coeff.log}
w_t=\nabla\cdot\(A\,\nabla w\) -\nabla w\cdot(A\,\nabla w)\,.
\ee
All computations below are quite standard and can be rigorously justified by testing with $-\log(\delta+v)$ for an arbitrarily small $\delta>0$ and then passing to the limit as $\delta\to 0^+$. Here we are interested in keeping track of the constants. We recall that $\mu=\lambda_1+1/\lambda_0$. Let us choose a test function $\psi$ as follows:
\be{Lem.Log.Est.3}
\psi(x):=\prod_{\nu=1}^d \chi_\nu(x_\nu),\quad\mbox{where}\quad
\chi_\nu(z):=\left\{\begin{array}{lll}
1 &\quad\mbox{if}\quad|z|\le 1\,,\\
2-|z|&\quad\mbox{if}\quad1\le|z|\le 2\,,\\
0 &\quad\mbox{if}\quad|z|\ge 2\,.\\
\end{array}\right.
\ee
Note that this test function has convex super-level sets, or equivalently said, on any straight line segment, $\psi(\cdot)$ assumes its minimum at an end point.

Even if~\eqref{HE.coeff.log} is a nonlinear equation, the nonlinear term actually helps. The reason for that lies in the following result. 
%---------------------------------------------------------------------
\begin{lemma}\label{Lem.Log.Est} Assume that $\psi$ is a smooth compactly supported test function as in~\eqref{Lem.Log.Est.3}. If $w$ is a (sub)solution to~\eqref{HE.coeff.log} in
\[
\Big\{(t,x)\in\R\times\R^d\,:\,|t|<1\,,\;|x|<2\Big\}=(-1,1)\times B_2(0)\,,
\]
then there exist positive constants $a$ and $c_2(d)$ \index{logarithmic estimates}{such that}, for all $s>0$,
\begin{multline}\label{Lem.Log.Est.Ineq.w}
\left|\big\{(t,x)\in Q^+_1\,:\,w(t,x)>s-a\big\}\right|\\
+\left|\big\{(t,x)\in Q^-_1\,:\,w(t,x)<-s-a \big\}\right|\le c_2\,|B_1|\,\frac\mu s\,,
\end{multline}
where
\be{Lem.Log.Est.Ineq.a}
c_2=2^{d+2}\,3^d\,d\quad\mbox{and}\quad a=-\,\frac{\ird{w(0,x)\,\psi^2(x)}}{\ird{\psi^2(x)}}\,.
\ee
\end{lemma}
%---------------------------------------------------------------------
Equivalently, the above inequality stated in terms of $v$ reads
\begin{multline}\label{Lem.Log.Est.Ineq.u}
\left|\big\{(t,x)\in Q^+_1\,:\,\log v(t,x)<-s+a\big\}\right|\\
+\left|\big\{(t,x)\in Q^-_1\,:\,\log v(t,x)>s+a \big\}\right|\le c_2\,|B_1|\,\frac\mu s\,,
\end{multline}
with the choice $a=\ird{\log v(0,x)\,\psi^2(x)}/\ird{\psi^2(x)}$.
\begin{proof}
\begin{steps}
 For better readability, we split the proof into several steps.

\stepitem\textit{Energy estimates.} Testing equation (or inequality)~\eqref{HE.coeff.log} with $\psi^2(x)$, we obtain
\begin{multline*}\label{Lem.Log.Est.1}
\int_{B_2(0)}\psi^2\,w(t_2)\dx-\int_{B_2(0)}\psi^2\,w(t_1)\dx+\frac12\int_{-1}^1\int_{B_2(0)}\psi^2\,\nabla w\cdot(A\,\nabla w)\dx\dt\\\le 2 \int_{-1}^1\int_{B_2(0)} \nabla \psi\cdot(A\,\nabla \psi)\dx\dt\,.
\end{multline*}
Using the conditions~\eqref{HE.coeff.lambdas}, we have that
\begin{align*}
&\lambda_0 \int_{-1}^1\int_{B_2(0)}\psi^2\,|\nabla w|^2 \dx\dt \le\int_{-1}^1\int_{B_2(0)}\psi^2\,\nabla w\cdot(A\,\nabla w)\dx\dt\,,\\
&\int_{-1}^1\int_{B_2(0)} \nabla \psi\cdot(A\,\nabla \psi)\dx\dt \le \lambda_1 \int_{-1}^1\int_{B_2(0)} |\nabla \psi|^2 \dx\dt\,.
\end{align*}
Combining the above two inequalities, we obtain
\be{Lem.Log.Est.2b}\begin{split}
\int_{B_2(0)}\psi^2\,w(t_2)\dx&-\int_{B_2(0)}\psi^2\,w(t_1)+\frac{\lambda_0}2 \int_{-1}^1\int_{B_2(0)}\psi^2\,|\nabla w|^2 \dx\dt\\
&\le 2\,\lambda_1 \int_{-1}^1\int_{B_2(0)} |\nabla \psi|^2 \dx\dt\le 2^d\,\lambda_1\,(t_2-t_1)\,|B_1|\,\|\nabla \psi\|_\infty^2\,.
\end{split}\ee

\stepitem\textit{\index{weighted Poincar\'e inequality}{Poincar\'e inequality} with weight $\psi^2$.} We use the weighted inequality~\eqref{Lem.Log.Est.4} with $b=\psi^2$ and $\psi$ as in~\eqref{Lem.Log.Est.3}. We have that $\mathrm{diam}(\supp\,b)=2\,d$ and since $0\le\psi\le1$, then $\nrm b\infty=\nrm{\psi^2}\infty=1$ and also $|B_1|\le\int_{\R^d}\psi^2\dx\le3^d\,|B_1|$ so that the constant $\lambda_b$ is given by
\[
\lambda_b\le\frac{2\,d\,|B_2|}{2\int_{B_1}\psi^2\,\dx}=\frac{d\,|B_2|}{|B_1|}=2^{d}\,d\,.
\]
Hence we obtain
\be{Lem.Log.Est.5}
\int_{-1}^1\int_{B_2(0)} \left|w(t,x)-\overline{w(t)}_\psi\right|^2\,\psi^2(x)\dx\dt \le 2^d\,d \int_{-1}^1\int_{B_2(0)} |\nabla w(t,x)|^2\,\psi^2(x)\dx\dt\,,
\ee
with
\[
\overline{w(t)}_\psi:=\frac{\int_{\R^d} w(t,x)\,\psi^2(x)\dx}{\int_{\R^d} \psi^2(x)\dx}\,.
\]

\stepitem\textit{Differential inequality.} Let us recall that $\|\nabla \psi\|_\infty^2\le 1$. We combine inequalities~\eqref{Lem.Log.Est.2b} and~\eqref{Lem.Log.Est.5} into
\[\label{Lem.Log.Est.6}\begin{split}
\int_{B_1} \psi^2\,w(t_2)\dx-\int_{B_1} \psi^2\,w(t_1) \dx+\frac{\lambda_0}{2^{d+1}\,d} \int_{t_1}^{t_2}\int_{B_1} \left|w(t,x)-\overline{w(t)}_\psi\right|^2\,\psi^2(x)\dx\dt
\\ \le 2^d\,\lambda_1\,(t_2-t_1)\,|B_1|\,.
\end{split}\]
Recalling that $\psi=1$ on $B_1$ and the expression of $\overline{w(t)}_\psi$ given in~\eqref{Lem.Log.Est.5}, we obtain
\[\label{Lem.Log.Est.7}\begin{split}
\frac{\overline{w(t_2)}_\psi-\overline{w(t_1)}_\psi}{t_2-t_1}+\frac{\lambda_0}{2^{d+1}\,3^d\,d}\frac1{(t_2-t_1)\,|B_1|}\int_{t_1}^{t_2}\int_{B_1} \left|w(t,x)-\overline{w(t)}_\psi\right|^2 \dx\dt\\
\le\frac{2^d\,\lambda_1\,|B_1|}{\int_{\R^d}\psi^2\dx}\le\,2^d\,\lambda_1\,.
\end{split}\]
Here we have used that $|B_1|\le \int_{\R^d}\psi^2\dx\le3^d\,|B_1|$. Recalling that $\mu=\lambda_1+1/\lambda_0$, so that $\lambda_0\,\mu>1$, we obtain
\[\label{Lem.Log.Est.7b}\begin{split}
\frac{\overline{w(t_2)}_\psi-\overline{w(t_1)}_\psi}{t_2-t_1}+\frac1{2^{d+1}\,3^d\,d\,\mu}\frac1{(t_2-t_1)\,|B_1|}\int_{t_1}^{t_2}\int_{B_1} \left|w(t,x)-\overline{w(t)}_\psi\right|^2 \dx\dt \\
\le\frac{2^d\,\lambda_1\,|B_1|}{\int_{\R^d}\psi^2\dx}\le\,2^d\,\mu\,.
\end{split}\]

Letting $t_2\to t_1$ we obtain the following differential inequality for $\overline{w(t)}_\psi $
\be{Lem.Log.Est.8}
\frac{\rd}{\dt}\overline{w(t)}_\psi+\frac1{2^{d+1}\,3^d\,d\,\mu}\frac1{|B_1|}\int_{B_1} \left|w(t,x)-\overline{w(t)}_\psi\right|^2\dx
\le\,2^d\,\mu\,.
\ee
The above inequality can be applied to
\[
\underline{w}(t,x)=w(t,x)-\overline{w(0)}_\psi-2^d\,\mu\,t\,.
\]
Notice that $\underline{w}$ is a subsolution to~\eqref{HE.coeff.log} since $w$ is. With $a=-\overline{w(0)}_\psi$, we can write~\eqref{Lem.Log.Est.8} in terms~of
\[
W(t)=\overline{w(t)}_\psi+a-2^d\,\mu\,t\quad\mbox{such that}\quad W(0)=0
\]
as
\[
\frac{\rd}{\dt}W(t)+\frac1{2^{d+1}\,3^d\,d\,\mu}\frac1{|B_1|}\int_{B_1} \left|\underline{w}(t,x)-W(t)\right|^2\dx
\le 0\,.
\]
An immediate consequence of the above inequality is that $W(t)\le W(0)=0$ for all $t\in (0,1)$.

Let $Q_s(t)=\{x\in B_1\,:\,w(t,x)>s\}$, for a given $t\in (0,1)$. For any $s>0$, we have
\[
\underline{w}(t,x)-W(t)\ge s-W(t)\ge 0\quad\forall\,x\in Q_s(t)\,,
\]
because $W(t)\le 0$ for $t\in (0,1)$. Using $\frac{\rd}{\dt}W=-\frac{\rd}{\dt}(s-W)$, the integration restricted to $Q_s$ in~\eqref{Lem.Log.Est.8} gives
\[\label{Lem.Log.Est.9}
\frac{\rd}{\dt}\big(s-W(t)\big) \ge\frac1{2^{d+1}\,3^d\,d\,\mu}\frac{|B_s(t)|}{|B_1|}\,\big(s-W(t)\big)^2\,.
\]
By integrating over $(0,1)$, it follows that
\begin{multline*}
\left|\left\{(t,x)\in Q_1^+\,:\,\underline{w}(t,x)>s\right\}\right|
=\iint_{\{\underline{w}>s\}\cap Q_1^+}\dx\dt=\int_0^1 |Q_s(t)|\dt\\
\le 2^{d+1}\,3^d\,d\,\mu\,|B_1|\left(\frac1{s-W(0)}-\frac1{s-W(1)}\right)\le 2^{d+1}\,3^d\,d\,|B_1|\,\frac\mu s\,,
\end{multline*}
which proves the first part of inequality~\eqref{Lem.Log.Est.Ineq.w}.

\stepitem\textit{Estimating the second term of inequality~\eqref{Lem.Log.Est.Ineq.w}.} We just replace $t$ by $-t$ and repeat the same proof. Upon setting $a=-\overline{w(0)}_\psi$, we obtain
\[\label{Lem.Log.Est.11}
 \left|\left\{(t,x)\in Q_1^-\,:\,w<-s-a\right\}\right| \le2^{d+1}\,3^d\,d\,|B_1|\,\frac\mu s\,.
 \]
\end{steps}
\end{proof}

%%%%%%%%%%%%%%%%%%%%%%%%%%%%%%%%%%%%%%%%%%%%%%%%%%%%%%%%%%%%%%%%%%%%%%%%%%
\subsection{A lemma by Bombieri and Giusti}\label{Sec:Bombieri-Giusti}

The following lemma, attributed to E.~Bombieri and E.~Giusti~\cite{Bombieri1972} and adapted by J.~Moser~\cite{Moser1971}, is the cornerstone of our method. We remark that the result applies to any measurable function~$f$, not necessarily solutions to a PDE, and to any family of nested domains $(\OO_r)_{0<r\le R}$, \textit{i.e.}, such that $\OO_{r_0}\subset\OO_{r_1}$ whenever $r_0<r_1$.
%---------------------------------------------------------------------
\begin{lemma}\label{BGM.Lemma}
Let $\beta$, $c_1$, $\mu>0$, $c_2\ge 1/\mathrm{e}$, $\theta\in [1/2,1)$ and $p\in(0,1/\mu)$ be positive constants, and let $f>0$ be a positive measurable function defined on a neighborhood of $\OO_1$ for which
\be{hyp.BGM.Lemma.1}
\sup_{\OO_\varrho}f^p<\frac{c_1}{(r-\varrho)^\beta\,|\OO_1|}\iint_{\OO_r}f^p\dx\dt
\ee
for any $r$ and $\varrho$ such that $\theta\le\varrho<r\le 1$, and
\be{hyp.BGM.Lemma.2}
\big|\big\{ (t,x)\in\OO_1\,:\,\log f>s\big\}\big|<c_2\,|\OO_1|\,\frac\mu s\quad\forall\,s>0\,.
\ee
Let $\sigma$ be as in~\eqref{sigma}. Then we have
\be{BGM.Lemma.ineq}
\sup_{\OO_\theta} f<\kappa_0^\mu\,,\quad\mbox{where}\quad \kappa_0:=\exp\left[2\,c_2\vee\frac{8\,c_1^3}{(1-\theta)^{2\,\beta}}\right]\,.
\ee
\end{lemma}
%---------------------------------------------------------------------
The difference between the upper bounds~\eqref{hyp.BGM.Lemma.1} and~\eqref{BGM.Lemma.ineq} is subtle. The first inequality depends on the solution on the whole space-time set $\OO_r$ and is somehow implicit. By assumption~\eqref{hyp.BGM.Lemma.2}, if the set where $f$ is super-exponential has small measure, then on a slightly smaller set the solution is quantitatively bounded by an explicit and uniform constant, given by~\eqref{BGM.Lemma.ineq}.

\begin{proof} We sketch the relevant steps of the proof of~\cite[Lemma~3]{Moser1971}. Our goal is to provide some minor technical improvements and quantify all constants. Without loss of generality, after replacing $s$ by $s\,\mu$, we reduce the problem to the case $\mu=1$. Analogously, we also assume that $|\OO_1|=1$. We define the nondecreasing function
\[
\varphi(\varrho)=\sup_{\OO_\varrho}(\log f)\quad\forall\,\varrho\in[\theta,1)\,.
\]
We will prove that assumptions~\eqref{hyp.BGM.Lemma.1} and~\eqref{hyp.BGM.Lemma.2} imply the following dichotomy:\\
-- either $\varphi(r)\le 2\,c_2$ and there is nothing to prove: $\kappa_0=\mathrm{e}^{2\,c_2}$,\\
-- or $\varphi(r)>2\,c_2$ and we have
\be{BGM.Lemma.01}
\varphi(\varrho)\le\frac34\,\varphi(r)+\frac{8\,c_1^3}{(r-\varrho)^{2\,\beta}}
\ee
for any $r$ and $\varrho$ such that $\theta\le\varrho<r\le 1$. We postpone the proof of~\eqref{BGM.Lemma.01} and observe that~\eqref{BGM.Lemma.01} can be iterated along a monotone increasing sequence $(\varrho_k)_{k\ge0}$ such that
\[
\theta\le \varrho_0<\varrho_1<\dots<\varrho_k\le 1
\]
for any $k\in\N$ to get
\[
\varphi(\varrho_0)<\frac34\,\varphi(\varrho_k)+8\,c_1^3\sum_{j=0}^{k-1}\left(\tfrac34\right)^j\frac1{\left(\varrho_{j+1}-\varrho_j\right)^{2\,\beta}}\,.
\]
By monotonicity, we have that $\varphi(\varrho_k)\le \varphi(1)<\infty$, so that in the limit as $k\to+\infty$, we obtain
\[
\varphi(\theta)\le \varphi(\varrho_0)\le 8\,c_1^3\sum_{j=0}^{\infty}\left(\tfrac34\right)^j\frac1{\left(\varrho_{j+1}-\varrho_j\right)^{2\,\beta}}
\]
provided the right-hand side converges. This convergence holds true for the choice
\[
\varrho_j=1-\frac{1-\theta}{1+j}\,,
\]
and in that case, the estimate
\[\label{BGM.Lemma.03}
\varphi(\theta)\le\frac{8\,c_1^3\,\sigma}{(1-\theta)^{2\,\beta}}:=\tilde\kappa_0
\]
implies inequality~\eqref{BGM.Lemma.ineq} with $\mu=1$ because
\[
\sup_{\OO_\theta} f \le \exp\(\sup_{\OO_\theta}(\log f)\)=\mathrm{e}^{\varphi(\theta)}\le \mathrm{e}^{\tilde\kappa_0}:=\kappa_0\,.
\]
In order to complete the proof, we have to prove inequality~\eqref{BGM.Lemma.01}.

\medskip\noindent\textit{Proof of Inequality~\eqref{BGM.Lemma.01}.} We are now under the assumption $\varphi(r)>2\,c_2$. We first estimate the integral
\be{BGM.Lemma.04}\begin{split}
\iint_{\OO_r} f^p\dx\dt
&=\iint_{\{\log f>\frac12\,\varphi(r)\}} f^p\dx\dt
+\iint_{\{\log f\le\frac12\,\varphi(r)\}} f^p\dx\dt\\
&\le \mathrm{e}^{p\,\varphi(r)}\left|\left\{ (t,x)\in\OO_1\,:\,\log f>\tfrac12\,\varphi(r)\right\}\right|+|\OO_1|\,\mathrm{e}^{\frac p2\,\varphi(r)}\\
&\le\frac{2\,c_2}{\varphi(r)}\,\mathrm{e}^{p\,\varphi(r)}+\mathrm{e}^{\frac p2\,\varphi(r)}\,,
\end{split}
\ee
where we have estimated the first integral using that
\[
\sup_{\OO_r} f^p \le \sup_{\OO_r} \mathrm{e}^{p \log f} \le \mathrm{e}^{p \sup_{\OO_r} \log f}=\mathrm{e}^{p\,\varphi(r)}\,.
\]
In the present case, assumption~\eqref{hyp.BGM.Lemma.2} reads:
\[
\left|\left\{ (t,x)\in\OO_1\,:\,\log f>\tfrac12\,\varphi(r)\right\}\right|<\frac{2\,c_2}{\varphi(r)}.
\]
We choose
\[
p=\frac2{\varphi(r)}\log\left(\frac{\varphi(r)}{2\,c_2}\right)
\]
such that the last two terms of~\eqref{BGM.Lemma.04} are equal, which gives
\be{BGM.Lemma.05}
 \iint_{\OO_r} f^p\dx\dt \le 2\,\mathrm{e}^{\frac p2\,\varphi(r)}\,.
\ee
The exponent $p$ is admissible, that is, $0<p<1/\mu=1$, if $\varphi(r)>2/\mathrm e$, which follows from the assumption $c_2>1/\mathrm e$. Now, using assumption~\eqref{hyp.BGM.Lemma.1} and inequality~\eqref{BGM.Lemma.05}, we obtain
\begin{multline*}
\varphi(\varrho)=\frac1p \sup_{\OO_\varrho}\log(f^p)=\frac1p \log\(\sup_{\OO_\varrho} f^p\)\\
\le\frac1p\log\(\frac{c_1}{(r-\varrho)^\beta}\iint_{\OO_r}f^p \dx \dt\)\hspace*{4cm}\\
\le\frac1p\log\(\frac{2\,c_1\,\mathrm{e}^{\frac p2\,\varphi(r)}}{(r-\varrho)^\beta}\)
=\frac1p\log\(\frac{2\,c_1}{(r-\varrho)^\beta}\)+\frac12\,\varphi(r)\\
\hspace*{4cm}=\frac12\,\varphi(r)\(1+\frac{\log(2\,c_1)-\log(r-\varrho)^{\beta}}{\log(\varphi(r))-\log(2\,c_1)}\)\\
 \le\frac12\,\varphi(r)\(1+\frac12\)=\frac34\,\varphi(r)\,.
\end{multline*}
In the last line, we take
\[
\varphi(r)\ge\frac{8\,c_1^3}{(r-\varrho)^{2\,\beta}}
\]
so that
\be{BGM.Lemma.07}
\frac{\log(2\,c_1)-\log(r-\varrho)^{\beta}}{\log(\varphi(r))-\log(2\,c_1)}\le\frac12\,.
\ee
We again have that either $\varphi(r)<\frac{8\,c_1^3}{(r-\varrho)^{2\,\beta}}$ and~\eqref{BGM.Lemma.01} holds, or $\varphi(r)\ge\frac{8\,c_1^3}{(r-\varrho)^{2\,\beta}}$ and~\eqref{BGM.Lemma.07} holds, hence $\varphi(\varrho)\le\frac34\,\varphi(r)$. We conclude that~\eqref{BGM.Lemma.01} holds in all cases and this completes the proof.
\end{proof}

%%%%%%%%%%%%%%%%%%%%%%%%%%%%%%%%%%%%%%%%%%%%%%%%%%%%%%%%%%%%%%%%%%%%%%%%%%
\subsection{Proof of Moser's Harnack inequality}\label{Sec:Moser-Harnack}

\begin{proof}[Proof of Theorem~\ref{Claim:3}]
\begin{steps}
We prove the \idx{Harnack inequality}
\[\label{Harnack.Ineq}
\sup_{D^-}v\le \mathsf h^\mu\,\inf_{D^+} v
\]
where $\mathsf h$ is as in~\eqref{h} and $D^\pm$ are the parabolic cylinders given by
\[\label{Parab.Cylinders.unitary}\begin{split}
&D^+=\left\{\tfrac34<t<1\,,\;|x|<\tfrac12\right\}=\left(\tfrac34,1\right)\times B_{1/2}(0)\,,\\
&D^-=\left\{ -\tfrac34<t<-\tfrac14\,,\;|x|<\tfrac12\right\}=\left(-\tfrac34,-\tfrac14\right)\times B_{1/2}(0)\,.
\end{split}
\]
The general inequality~\eqref{harnack} follows by applying the change of variables~\eqref{admissible.transformations}, which do not alter the values of $\lambda_0$, $\lambda_1$ and $\mu=\lambda_1+1/\lambda_0$.

Let $v$ be a positive solution to~\eqref{HE.coeff} and $a=\frac{\int_{\R^d} \log v(0,x)\,\psi^2(x)\dx}{\int_{\R^d}\psi^2(x)\dx}$.
In order to use Lemma~\ref{Lem.Moser} and Lemma~\ref{Lem.Log.Est}, we apply Lemma~\ref{BGM.Lemma} to
\[
v_+(t,x)=\mathrm{e}^{-a}\,v(t,x)\quad\mbox{and}\quad v_-(t,x)=\frac{\mathrm{e}^{+a}}{v(t,x)}\,.
\]
\stepitem\textit{Upper estimates.} Let us prove that
\be{Proof.Harnack.claim.1}
\sup_{D^-}v_+\le \ka_0^{\,\mu}
\ee
where $\ka_0$ has an explicit expression, given below in~\eqref{proof.Harnack.5}. For all $\varrho\in\left[1/2,1\right)$, let
\[\begin{split}
\OO_\varrho&:=\left\{(t,x)\in\OO_1\,:\,\left|t+\tfrac12\right|<\tfrac12\,\varrho^2\,,\;|x|<\varrho/\sqrt2\right\}\\
&=\left(-\tfrac12\,(\varrho^2+1),\tfrac12\,(\varrho^2-1)\right)\times B_{\varrho/\sqrt2}(0)=Q_{\varrho/\sqrt2}\left(-\tfrac12,0\right)\,.
\end{split}\]
Note that if $\varrho=1/\sqrt2$, then $\OO_\varrho=\left(-3/4,-1/4\right)\times B_{1/2}(0)=D^-$, and also that $\OO_\varrho\subset\OO_1=(-1,0)\times B_1(0)=Q_1^-$ for any $\varrho\in\left[1/2,1\right)$.

The first assumption of Lemma~\ref{BGM.Lemma}, namely inequality~\eqref{hyp.BGM.Lemma.1} with $\beta=d+2$ is nothing but inequality~\eqref{Lem.Moser.Upper} of Lemma~\ref{Lem.Moser} applied to $\OO_\varrho=Q_{\varrho/\sqrt2}\left(-1/2,0\right)$, that is,
\[\label{proof.Harnack.1}
\sup_{\OO_\varrho} v_+^p \le\frac{c_1\,2^\frac{d+2}2}{(r-\varrho)^{d+2}}\iint_{\OO_r}v_+^p \dx\dt\quad\forall\,p\in(0,1/\mu)\,.
\]
Note that the results of Lemma~\ref{Lem.Moser} hold true for these cylinders as well, with the same constants, since $Q_{\varrho/\sqrt 2}(-1/2,0)$ can be obtained from $Q_{\varrho}(0,0)$ by means of change of variables~\eqref{admissible.transformations} which leave the class of equations unchanged, \emph{i.e.}, such that $\lambda_1$, $\lambda_0$ and $\mu$ are the same.

The second assumption of Lemma~\ref{BGM.Lemma}, namely inequality~\eqref{hyp.BGM.Lemma.2} of Lemma~\ref{BGM.Lemma}, if stated in terms of super-level sets of $\log v_+$, reads
\[\label{proof.Harnack.2}
\left|\left\{x\in\OO_1\,:\,\log v_+>s\right\}\right|=\left|\{(t,x)\in Q^-_1\,:\,\log v>s+a \}\right|\le c_2\,|B_1|\,\frac\mu s
\]
according to Lemma~\ref{Lem.Log.Est}. Hence we are in the position to apply Lemma~\ref{BGM.Lemma} with $\theta=1/\sqrt 2$ to conclude that~\eqref{Proof.Harnack.claim.1} is true with
\be{proof.Harnack.5}
\ka_0:=\exp\left[2\,c_2\vee\frac{8\,c_1^3(\sqrt 2)^{3\,(d+2)}\,\sigma}{(1-1/\sqrt 2)^{2\,(d+2)}}\right]\,.
\ee
This concludes the first step.

\stepitem\textit{Lower estimates.} Let us prove that
\be{Proof.Harnack.claim.2}
\sup_{D^+}v_- \le \kb_0^{\,\mu}
\ee
where $\kb_0$ has an explicit expression, given below in~\eqref{proof.Harnack.10}. For all $\varrho\in\left[1/2,1\right)$, let
\[
\OO_\varrho=\left\{(t,x)\in\OO_1\,:\,0<1-t<\varrho^2\,,\;|x|<\varrho\right\}
=\left(1-\varrho^2, 1\right)\times B_{\varrho}(0)=Q^-_\varrho(1,0)\,.
\]
Note that if $\varrho=1/2$ then $\OO_\varrho=\left(3/4,1\right)\times B_{1/2}(0)=D^+$, and $\OO_\varrho\subset\OO_1=(0,1)\times B_1(0)=Q_1^+$ for any $\varrho\in\left[1/2,1\right)$.

The first assumption of Lemma~\ref{BGM.Lemma}, namely inequality~\eqref{hyp.BGM.Lemma.1} with $\beta=d+2$ is nothing but inequality~\eqref{Lem.Moser.Lower} of Lemma~\ref{Lem.Moser} applied to $\OO_\varrho=Q^-_\varrho(1,0)$
\[\label{proof.Harnack.6}
\sup_{\OO_\varrho} v_-^p \le\frac{c_1}{(r-\varrho)^{d+2}}\iint_{\OO_r}v_-^p \dx\dt\quad\forall\,p\in(-\tfrac1\mu,0)\,.
\]
Note that the results of Lemma~\ref{Lem.Moser} hold true for these cylinders as well, with the same constants, since $Q^-_\varrho(1,0)$ can be obtained from $Q_{\varrho}(0,0)$ by means of change of variables~\eqref{admissible.transformations}.

The second assumption of Lemma~\ref{BGM.Lemma}, namely inequality~\eqref{hyp.BGM.Lemma.2} of Lemma~\ref{BGM.Lemma}, if stated in terms of super-level sets of $\log v_-$, reads
\[\label{proof.Harnack.7}
\left|\left\{x\in\OO_1\,:\,\log v_->s\right\}\right|=\left|\{(t,x)\in Q^+_1\,:\,\log v<-s+a\}\right|\le c_2\,|B_1|\,\frac\mu s\,.
\]
and follows from inequality~\eqref{Lem.Log.Est.Ineq.u} of Lemma~\ref{Lem.Log.Est}. With the same~$a$ and $c_2$, we are in the position to apply Lemma~\ref{BGM.Lemma} with $\theta=1/2$ to conclude that~\eqref{Proof.Harnack.claim.2} is true with
\be{proof.Harnack.10}
\kb_0:=\exp\left[2\,c_2\vee c_1^3 2^{2\,(d+2)+3}\,\sigma\right]\,.
\ee
This concludes the second step.

\stepitem\textit{\idx{Harnack inequality} and its constant.} We deduce from~\eqref{Proof.Harnack.claim.1} and~\eqref{Proof.Harnack.claim.2} that
\[\label{proof.Harnack.11}
\ka_0^{-\mu}\sup_{D^-}v\,\le \mathrm{e}^{a} \le \kb_0^{\,\mu}\,\inf_{D^+}v
\]
or, equivalently,
\[
\sup_{D^-}v\,\le (\ka_0\,\kb_0)^\mu\,\inf_{D^+}v=\widetilde{\mathsf h}^\mu\,\inf_{D^+}v\,.
\]
Using~\eqref{proof.Harnack.5} and~\eqref{proof.Harnack.10}, we compute
\[\label{HE.Harnack.Constant.proof}\begin{split}
\widetilde{\mathsf h}=\ka_0\,\kb_0&=\exp\left[2\,c_2\vee c_1^3\,2^{2\,(d+2)+3}\,\sigma\right]\,\exp\left[2\,c_2\vee\tfrac{8\,c_1^3\,(\sqrt 2)^{3\,(d+2)}}{(1-1/\sqrt 2)^{2\,(d+2)}}\,\sigma\right]\\
&\le \exp\left[4\,c_2+c_1^3\left( 2^{2\,(d+2)+3}+\tfrac{8\,(\sqrt 2)^{3\,(d+2)}}{(1-1/\sqrt2)^{2\,(d+2)}}\right)\,\sigma\right]\\
&=\exp\left[4\,c_2+c_1^3\,2^{2\,(d+2)+3}\left( 1+\tfrac{2^{ d+2 }}{(\sqrt 2-1)^{2\,(d+2)}}\right)\,\sigma\right]:=\mathsf h\,.
\end{split}
\]
The expressions of $c_1$ and $c_2$ are given in~\eqref{Lem.Moser.constant} and~\eqref{Lem.Log.Est.Ineq.a} respectively. The above expression of $\mathsf h$ agrees with the simplified expression of~\eqref{h}, which completes the proof.
\end{steps}\end{proof}

%%%%%%%%%%%%%%%%%%%%%%%%%%%%%%%%%%%%%%%%%%%%%%%%%%%%%%%%%%%%%%%%%%%%%%
\subsection{Harnack inequality implies H\"older continuity}\label{sec:holder}

In this section, we shall show a standard application of the \idx{Harnack inequality}~\eqref{harnack}. It is well known that~\eqref{harnack} implies \idx{H\"older continuity} of solutions to~\eqref{HE.coeff}. The novelty is that, here, we keep track of all constants and obtain a quantitative expression of the \idx{H\"older continuity} exponent, which only depends on the Harnack constant, \emph{i.e.}, only depends on the dimension~$d$ and on the ellipticity constants $\lambda_0$ and $\lambda_1$ in~\eqref{HE.coeff.lambdas}.

Let $\Omega_1\subset\Omega_2\subset \R^d$ two bounded domains and let us consider $ Q_1:=(T_2, T_3)\times \Omega_1\subset (T_1, T_4)\times \Omega_2=:Q_2$, where $0\le T_1<T_2<T_3<T<4$. We define the \emph{parabolic distance} between $Q_1$ and $Q_2$ as
\be{parabolic-distance}
dist(Q_1, Q_2):=\inf_{\substack{(t, x)\in Q_1 \\ (s, y)\in [T_1, T_4]\times\partial\Omega_2 \cup \{T_1, T_4\}\times \Omega_2}} |x-y|+|t-s|^\frac12\,.
\ee
In what follows, for simplicity, we shall consider $\Omega_1$ and $\Omega_2$ as convex sets. However, this is not necessary and the main result of this section holds without such restriction.
%---------------------------------------------------------------------
\begin{theorem}\label{Claim:4} Let $v$ be a nonnegative solution of~\eqref{HE.coeff} on $Q_2$ which satisfies~\eqref{weak.solution} and assume that $A(t,x)$ satisfies~\eqref{HE.coeff.lambdas}.
Then we have
\be{holder-continuity-inequality}
\sup_{(t,x),(s,y)\in Q_1}\frac{|v(t,x)-v(s,y)|}{\big(|x-y|+|t-s|^{1/2}\big)^\nu}\le\,2\(\frac{128}{dist(Q_1, Q_2)}\)^\nu\,\|v\|_{\mathrm L^\infty(Q_2)}\,.
\ee
where
\be{nu}
\nu:=\log_4\Big(\frac{\overline{\mathsf h}}{\overline{\mathsf h}-1}\Big)\,,
\ee
and $\overline{\mathsf h}$ is as in~\eqref{h-bar}.
\end{theorem}
%---------------------------------------------------------------------
From the expression of $\mathsf h$ in~\eqref{h} it is clear that $\overline{\mathsf h}\ge\frac43$, from which we deduce that $\nu\in(0,1)$.
%---------------------------------------------------------------------
\begin{remark}
We stated the above Theorem under the simplifying assumption that both $\Omega_1$ and $\Omega_2$ are convex. As it shall appear in the proof, this hypothesis is just technical and not needed. A more general statement, without any further assumption on $\Omega_1$ and $\Omega_2$, can be proved by using a covering argument where $Q_1$ is covered by convex sets, open balls for instance. In this last case, the constant in the right-hand-side of inequality~\eqref{holder-continuity-inequality} will change. However, the geometrical information needed to compute the constant is the number of balls needed to cover~$Q_1$ and the parabolic distance between $Q_1$ and $Q_2$.
\end{remark}
%---------------------------------------------------------------------

\begin{proof}
 We proceed in steps: in step 1 we shall show that inequality~\eqref{harnack} implies a \emph{reduction of oscillation} on cylinders of the form~\eqref{cylinder.harnack}. In step 2 we will iterate such reduction of oscillation and directly show estimate~\eqref{holder-continuity-inequality}.

\begin{steps}
\stepitem\textit{Reduction of oscillation.} Let us define $D_R(t_0,x_0)=(t_0-R^2,t_0+R^2)\times B_{2\,R}(x_0)$ and let $D_R^+(t_0,x_0)$, $D_R^-(t_0,x_0)$ be as in~\eqref{cylinder.harnack}. Let us define
\[\label{max.min}
M:=\max_{D_R(t_0,x_0)}v\,,\quad M^{\pm}=\max_{D_R^\pm(t_0,x_0)}v\,,\quad m=\min_{D_R(t_0,x_0)}v\,,\quad m^{\pm}=\max_{D_R^\pm(t_0,x_0)}v\,,
\]
and let us define the oscillations $\omega$ and $\omega^{+}$ namely
\[\label{oscillation}
\omega=M-m\quad\mbox{and}\quad\omega^{+}=M^+-m^+\,.
\]
We observe that the function $M-v$ and $v-m$ are nonnegative solution to~\eqref{HE.coeff} which also satisfy~\eqref{HE.coeff.lambdas} with $\lambda_0$ and $\lambda_1$ as in~\eqref{HE.coeff.lambdas}. We are therefore in the position to apply inequality~\eqref{harnack} to those functions and get
\[\begin{split}
M-m^{-}=\sup_{D_R^-(t_0,x_0)} M-v\,\le \overline{\mathsf h} \inf_{D_R^+(t_0,x_0)}v-M\,=\overline{\mathsf h}\left(M-M^{+}\right)\,,\\
M^{-}-m=\sup_{D_R^-(t_0,x_0)} v-m\,\le \overline{\mathsf h} \inf_{D_R^+(t_0,x_0)}v-m\,=\overline{\mathsf h}\left(m^{+}-m\right)\,.
\end{split}\]
Summing up the two above inequalities we get
\[
\omega\,\le\,\omega+(M^{-}-m^{-}) \le\,\overline{\mathsf h}\,\omega- \overline{\mathsf h}\,\omega^{+}
\]
which can be rewritten as
\begin{equation}\label{reduction.oscillation}
\omega^{+} \le\frac{\overline{\mathsf h}-1}{\overline{\mathsf h}}\,\omega\,=: \zeta\,\omega\,,
\end{equation}
which means that the oscillation on $D_R^+(t_0,x_0)$ is smaller then the oscillation on $D_R(t_0,x_0)$, recall that $\zeta<1$. In the next step we will iterate such inequality in a sequence of nested cylinders to get a \emph{geometric} reduction of oscillations.
\stepitem\textit{Iteration.} Let us define $\delta=dist(Q_1,Q_2)/64$. The number $\delta$ has the following property:
\begin{equation}\label{property.distance}\tag{P}\begin{split}
&\mbox{Let $(t, x)\in Q_1$ and $(s, y)\in(0, \infty)\times\R^d$.}\\
&\mbox{If}\,\,|x-y|+|t-s|^\frac12\le \delta\,\,\mbox{then,}\,\,(s,y)\in Q_2\,.
\end{split}\end{equation}
Let us consider $(t,x), (s,y)\in Q_1$ such that $(t,x)\neq(s,y)$, then either
\begin{equation*}\label{A}\tag{A}
|x-y|+|t-s|^\frac12<\delta\,,
\end{equation*}
or
\begin{equation*}\label{B}\tag{B}
|x-y|+|t-s|^\frac12 \ge \delta\,.
\end{equation*}
If~\eqref{A} happens, then there exists an integer $k\ge0$ such that
\[
\frac{\delta}{4^{k+1}}\le|x-y|+|t-s|^\frac12\le\frac{\delta}{4^k}\,.
\]
Let us define $z=\frac{x+y}2$ and $\tau_0=\frac{t+s}2$. Since $Q_1$ is a convex set we have that $z, \tau_0\in Q_1$. Let us define,
\[\label{choices.radii.times}
R_{i+1}:=4\,R_i\quad\tau_{i+1}:=\tau_{i}-14\,R_i^2\quad\forall\,i \in\{0, \cdots, k-1\}\,\,\mbox{where}\,\,R_0=\frac{\delta}{4^{k-1}}\,.
\]
With such choices we have that
\be{incapsulation}
D_{R_i}(z, \tau_{i})\subset D_{R_{i+1}}^{+}(z, \tau_{i+1})\quad\forall\,i\in\{0, \cdots, k-1\}\,,
\ee
and
\[
(t, x)\,,(s, y) \in D_{R_0}(z, \tau_0)\subset D_{R_1}^{+}(z, \tau_1)\,.
\]
We also observe that, as a consequence of property~\eqref{property.distance} we have that $D_{R_k}(z, \tau_k)\subset Q_2$. Let us define, for any $i\in\{0, \cdots, k-1\}$
\[
\omega_i:=\max_{D_{R_i}(z, \tau_i)}v - \min_{D_{R_i}(z, \tau_i)}v\quad\mbox{and}\quad\omega_i^{+}:=\max_{D^+_{R_i}(z, \tau_i)}v-\min_{D^+_{R_i}(z, \tau_i)}v\,.
\]
As a consequence of~\eqref{incapsulation}
\be{incapsulation-2}
\omega_i\le\omega^+_{i+1}\,.
\ee
By iterating inequalities~\eqref{incapsulation-2} -~\eqref{reduction.oscillation}, we obtain that
\[\begin{split}
|v(t,x)-v(s,y)|\le \omega_0 \le \omega_1^{+}&\le \xi\,\omega_1 \\
& \le\,\xi^k\,\omega_k=\(\frac14\)^{k\,\nu} \omega_k \\
& \le\(\frac4{\delta}\)^\nu\(\frac{\delta}{4^{k+1}}\)^\nu\,\omega_k \\
& \le 2\(\frac4{\delta}\)^\nu\(|x-y|+|t-s|^\frac12\)^\nu\,\|v\|_{\mathrm L^\infty(Q_2)}\,.
\end{split}\]
This concludes the proof of~\eqref{holder-continuity-inequality} under Assumption~\eqref{A}.

Let us now assume that~\eqref{B} happens. In this case we have that
\[\begin{split}
|v(t,x)-v(s,y)|\le 2\,\|v\|_{\mathrm L^\infty(Q_2)}\,\frac{\delta^\nu}{\delta^\nu} &\le 2\,\|v\|_{\mathrm L^\infty(Q_2)}\(\frac{|x-y|+|t-s|^\frac12}{\delta}\)^\nu\\
&\le\,2\(\frac4{\delta}\)^\nu\(|x-y|+|t-s|^\frac12\)^\nu\,\|v\|_{\mathrm L^\infty(Q_2)}\,.
\end{split}\]
\end{steps}
The proof is then completed.
\end{proof}

%%%%%%%%%%%%%%%%%%%%%%%%%%%%%%%%%%%%%%%%%%%%%%%%%%%%%%%%%%%%%%%%%%%%%%
%%%%%%%%%%%%%%%%%%%%%%%%%%%%%%%%%%%%%%%%%%%%%%%%%%%%%%%%%%%%%%%%%%%%%%
\section{Bibliographical comments}\label{Sec:Bib3}

Sobolev's inequalities on domains are classical and can be established in many ways, see for instance~ \cite{zbMATH06817095}. The critical \index{interpolation inequalities on bounded domains}{interpolation} inequality~\eqref{SobolevH1} follows from~\cite[Section~5.6]{MR2597943}, The constant in~\eqref{SobolevH1} is new, to our knowledge. The proof of inequality~\eqref{sob.step2} in dimension $2$ is due to E.~Gagliardo and L.~Nirenberg in~\cite{Gagliardo1958,Nirenberg1959}, while, the proof in dimension $1$ seems to be new. For an introduction to Gagliardo-Nirenberg inequalities in dimension $d=1$, we refer to~\cite{MR3014099,MR3263963}. In arbitrary dimensions $d\ge1$, the interpolation inequalities of Lemma~\ref{interpolation.lemma} are due to E.~Gagliardo and L.~Nirenberg~\cite{Gagliardo1958,Nirenberg1959,Nirenberg1966} but, to the best of our knowledge, the constants are not present in the existing literature. The proof of the weighted Poincar\'e inequality is due to J.~Moser~\cite[Lemma~3]{Moser1964}, see also~\cite[p~488]{zbMATH00051509} for a previous contribution.

The constant in Lemma~\ref{Lem:SobolevH1} is not optimal and, in the critical exponent case, one can expect various regimes in Inequality~\eqref{sob.step2} depending on the radius $R$ of the ball. This can be seen by direct methods when the space is restricted to radial functions, but the problem is that optimal functions are expected to satisfy Neumann homogeneous boundary conditions and that we cannot use the \idx{P\'olya–Szeg\H o principle} in this setting. Alternatively, in the large $R$ regime, Sobolev's inequality should play a role, at least asymptotically, while in the small $R$ regime, if there are optimal functions corresponding to the inequality written with the optimal constant, then these solutions enter in the framework of the Lin-Ni conjecture, and it is expected that the only solutions are the constants. This is actually some rather delicate matter, in which dimension plays a role: see~\cite{Druet2012} and~\cite{MR3746647} for results in the entropy framework. Hence in Lemma~\ref{Lem:SobolevH1}, and also in Lemma~\ref{Lem:GNd=2} for similar reasons, we claim no optimality.

\smallskip The \emph{\idx{Harnack inequality}} for the classical heat equation was first established by B.~Pini~\cite{Pini1954} and J.~Hadamard~\cite{Hadamard1954} independently. The first \emph{regularity result} for linear parabolic equation with uniformly elliptic measurable coefficients was established by J.~Nash~\cite{Nash58}. The \idx{Harnack inequality} of Theorem~\ref{Claim:3} (without explicit constants) goes back to J.~Moser~\cite{Moser1964,Moser1967,Moser1971}. After the seminal papers of J.~Nash and J.~Moser, several other results appeared for more general linear parabolic equations including the case with unbounded lower order terms~\cite{Aronson1967a, Ish99, Trudinger1968, Ladyzenskaja1995}. A proof of {M}oser's parabolic {H}arnack inequality using the founding ideas of {N}ash, can be found in~\cite{FS86}. The relation between heat kernel bounds and Harnack inequalities was investigated by D.G.~Aronson, E.B.~Fabes and D.W.~Stroock,~\cite{Ar67, FS84}, and the complete equivalence was established by M.T.~Barlow, A.~Grigor'yan and T.~Kumagai in~\cite{BGK12}.

Several regularity results were obtained also for linear operators with degenerate or singular weights. Harnack inequalities where established in~\cite{Chiarenza1984a,Chiarenza1984b,Chiarenza1985} using similar techniques to~\cite{Moser1971}. Pointwise estimates on solutions were later obtained in~\cite{Chiarenza1987}. Bounds on the heat kernel for such kind of operators were obtained in~\cite{Gutierrez1988}. Those papers deal with rather particular class of weights. We remark that in this setting functional inequalities with weights of Sobolev and Poincar\'e play a key role, see for instance~\cite{Franchi1994}.

The correspondences between elliptic and parabolic Harnack inequalities were investigated in~\cite{HSC01}. L.~Saloff-Coste discovered an intriguing equivalence between Harnack inequalities and families of Poincar\'e inequalities with doubling measures~\cite{SC92,SC95}, somehow already foreseen in~\cite{Gr91}, together with the intriguing relation with Sobolev type inequalities~\cite{MR1872526,SC09}. 

In the Riemannian manifold setting, the differential \idx{Harnack inequality} of \cite{Yau94,LiYau86} opened the door to a prolific field of research, see for instance the monograph~\cite{Gr09} and references therein. We also refer to the survey~\cite{Kas07} for more results on Harnack inequalities in various settings ranging from local and nonlocal elliptic and parabolic equations, possibly in non-divergence form, to Harnack inequalities in discrete settings, or for Brownian and Levy diffusion processes.
 
\medskip The dependence of the constant $\overline{\mathsf h}$ on the ellipticity constants $\lambda_0$ and $\lambda_1$ was not clear before the paper of J.~Moser~\cite{Moser1971}, where he showed that such a dependence is optimal by providing an explicit example,~\cite[p.~729]{Moser1971}. The fact that $\mathsf h$ only depends on the dimension $d$ is also pointed out by C.E.~Gutierrez and R.L.~Wheeden in~\cite{Gutierrez1990} after the statement of their Harnack inequalities,~\cite[Theorem~A]{Gutierrez1990}. 
The dependence of the constant $\mathsf h$ on the geometric parameters of the cylinders $D^+_R$ and $D^-_R$ were provided in~\cite{Gutierrez1990, Gutierrez1991}. 

The method of this chapter relies on several estimates of~\cite{Moser1964,Moser1971}. These two papers contain two different proofs of Theorem~\ref{Claim:3}. In~\cite{Moser1964}, Moser established a proof based on the same \index{Moser iteration}{iteration} which we use in Lemma~\ref{Lem.Moser}. To conclude the proof, however, he used a parabolic version of the BMO estimates contained in the John-Nirenberg Lemma~\cite{John1961}. It turns out that the original proof contained a mistake which was corrected only in~\cite{Moser1967}. A slightly different BMO approach can be found in~\cite{FG85}. The proof of the parabolic version of the John-Nirenberg Lemma is technical and provides no explicit constants. For the purpose of giving constructive estimates, it was desirable to use a different method. Finally, in~\cite{Moser1971}, a simplified proof was obtained: the \idx{Moser iteration} method provides $\mathrm L^p-\mathrm L^\infty$ and $\mathrm L^{\kern-2pt-p}-\mathrm L^{\kern-2pt-\infty}$ bounds for arbitrarily small $p>0$, for some constant which is independent of $p$. To combine the two bounds in a constructive way and obtain Harnack inequalities, J.~Moser used \idx{logarithmic estimates} as in Lemma \ref{BGM.Lemma}, attributed to E.~Bombieri, and first appeared in the Lecture Notes of a course \cite{bombieri1970theory}, then published in the celebrated paper with E.~Giusti~\cite{Bombieri1972}. This is the approach that we have chosen for obtaining constructive \idx{H\"older continuity} continuity estimates: we use exactly this strategy with minor improvements, detail each step, and quantify explicitly all constants.

%%%%%%%%%%%%%%%%%%%%%%%%%%%%%%%%%%%%%%%%%%%%%%%%%%%%%%%%%%%%%%%%%%%%%%%%
%%%%%%%%%%%%%%%%%%%%%%%%%%%%%%%%%%%%%%%%%%%%%%%%%%%%%%%%%%%%%%%%%%%%%%%%
\chapter{Uniform convergence in relative error and threshold time}\label{Chapter-4}

In this chapter, we aim at proving a uniform convergence result for the solutions of the \idx{fast diffusion equation}~\eqref{FD}
\[
\frac{\partial u}{\partial t}=\Delta u^m\,,\quad u(t=0,\cdot)=u_0\,.
\]
After some explicit \idx{threshold time} $t_\star$, we establish \idx{uniform convergence in relative error} on the basis of the results of Chapter~\ref{Chapter-3}.

%%%%%%%%%%%%%%%%%%%%%%%%%%%%%%%%%%%%%%%%%%%%%%%%%%%%%%%%%%%%%%%%%%%%%%%%
%%%%%%%%%%%%%%%%%%%%%%%%%%%%%%%%%%%%%%%%%%%%%%%%%%%%%%%%%%%%%%%%%%%%%%%%
\section{Statement}

Fix $m\in\left(m_c, 1\right)$, let us define, for any $u\in \mathrm L^1(\R^d)$ the quantity
\be{X.norm}
\|u\|_{\mathcal{X}_m}:=\sup_{r>0}\, r^\frac{\alpha}{(1-m)}\int_{|x|>r}|u|\,\dx\,,
\ee
where $\alpha=d\,(m-m_c)$ is as in~\eqref{R} so that $\alpha/(1-m)>0$ when $\in\left(m_c, 1\right)$; let
\[\label{X.space}
\X:=\{u\in \mathrm L^1(\R^d):\|u\|_{\mathcal{X}_m}<\infty \}\,.
\]
After an explicit \emph{\idx{threshold time}} $t_\star$, the solution $u$ of~\eqref{FD} converges \emph{uniformly in \index{uniform convergence in relative error}{relative error}} to the self-similar function $B$ defined by~\eqref{Barenblatt-1}.
%---------------------------------------------------------------------
\begin{theorem}\label{Thm:RelativeUniform} Let $m\in\big[m_1,1\big)$ if $d\ge2$ and $m\in\big(\widetilde m_1, 1\big)$ if $d=1$, $A$, $G>0$, and let $u$ be a solution of~\eqref{FD} corresponding to the nonnegative initial datum $u_0\in\mathrm L^1_{+}(\R^d)\cap\mathcal{X}_m$ such that
\be{hyp:Harnack}
\ird{u_0}=\ird{\mB}\,,\quad \|u_0\|_{\X}\le A\,,\quad \mathcal F[\muscal^{-d}\, u_0(\cdot/\muscal)]\le G\,.
\ee
There exists an explicit $\varepsilon_{m,d}\in (0,1/2)$, such that for any $\varepsilon\in(0,\varepsilon_{m,d}]$ 
\be{inq.RelativeUniform}
\sup_{x\in\R^d}\left|\frac{u(t,x)}{B(t,x)}-1\right|\le\varepsilon\quad\forall\,t\ge t_\star:=\taustarA\frac{1+A^{1-m}+G^\frac\alpha2}{\varepsilon^{\mathsf a}} \,,
\ee
where $\mathsf a$ and $\taustarA$ are explicit constants depending only on $m$, $d$.
\end{theorem}
%---------------------------------------------------------------------
The scaling parameter $\muscal$ is defined by~\eqref{mu}. It is taken into account in~\eqref{hyp:Harnack} only in order to simplify the conditions, in rescaled variables, for stability, which are needed in the next two chapters. See~Section~\ref{Sec:InitialData} for details. In dimension $d=1$ the restriction on $m$ is a consequence of $\ird{|x|^2\,\mB}<+\infty$, which means $m>\widetilde m_1$, and $0=m_1<\widetilde m_1=1/3$ if $d=1$.

%%%%%%%%%%%%%%%%%%%%%%%%%%%%%%%%%%%%%%%%%%%%%%%%%%%%%%%%%%%%%%%%%%%%%%%%
%%%%%%%%%%%%%%%%%%%%%%%%%%%%%%%%%%%%%%%%%%%%%%%%%%%%%%%%%%%%%%%%%%%%%%%%
\section{Local estimates and an interpolation estimate}\label{sec:local.estimates}

In this section we prove local estimates on solutions to~\eqref{FD} with $\mathrm L^1$ initial datum. 

%%%%%%%%%%%%%%%%%%%%%%%%%%%%%%%%%%%%%%%%%%%%%%%%%%%%%%%%%%%%%%%%%%%%%%%%
\subsection{\texorpdfstring{$\mathrm L^1$}{L1} mass displacement estimates}\label{sec.herrero.pierre}

The following lemma gives information on how mass moves both locally and at infinity. 
%---------------------------------------------------------------------
\begin{lemma} \label{HP-Lemma}
Let $m \in (0, 1)$ and $u(t,x)$ be a nonnegative solution to the Cauchy problem~\eqref{FD}. Then, for any $ t, \tau \ge 0$ and $r$, $R>0$ such that $\varrho_0\,r\ge 2\,R$ for some $\varrho_0>0$, we have
\be{Herrero.Pierre.123}
\int_{B_{2\,R}(x_0)} u(t, x)\dx \le 2^{\frac{m}{1-m}}\,\int_{B_{2\,R+r}(x_0)}{u(\tau,x)\dx}+\cc\,\frac{|t-\tau|^{\frac1{1-m}}}{r^{\frac{2-d\,(1-m)}{1-m}}}\,,
\ee
where
\be{C3.constant}
\cc:=2^{\frac{m}{1-m}}\,\omega_d\(\frac{16\,(d+1)\,(3+m)}{1-m}\)^{\frac1{1-m}} (\varrho_0+1)\,.
\ee
Under the same assumptions, we have that
\be{Herrero.Pierre.opposite}
\int_{\R^d\setminus B_{2\,R+r}(x_0)} u(t, x)\dx \le 2^{\frac{m}{1-m}}\,\int_{\R^d\setminus B_{2\,R}(x_0)}{u(\tau,x)\dx}+\cc\,\frac{|t-\tau|^{\frac1{1-m}}}{r^{\frac{2-d\,(1-m)}{1-m}}}\,.
\ee
\end{lemma}
%---------------------------------------------------------------------
\begin{proof} We begin by proving~\eqref{Herrero.Pierre.123}. Let $ \phi=\varphi^\beta $, for some $\beta>0$ (sufficiently large, to be chosen later) be a radial cut-off function supported in $ B_{2\,R+r}(x_0)$ and let $\varphi=1 $ in $ B_{2\,R}(x_0) $. We can take, for instance, $\varphi=\varphi_{2\,R, 2\,R+r}$, where $\varphi_{2\,R, 2\,R+r}$ is defined in~\eqref{test.funct}.
By Lemma~\ref{lem.test.funct} we have that
\be{test.estimates.1}
\|\nabla\varphi\|_\infty\le\frac2{r}\quad\mbox{and}\quad \|\Delta\varphi\|_\infty\le\frac{4\,d}{r^2}.
\ee
Without loss of generality, we can assume that $x_0=0$ by translation invariance, and write $B_{R}$ instead of $B_{R}(x_0)$. Let us compute
\be{derivation.L1.norm}\begin{split}
 \left|\frac{\rd}{\dt} \int_{B_{2\,R+r}}{u( t,x ) \phi\left( x\right)\dx}\right|
&=\left| \int_{B_{2 R+r}}\Delta\left( u^m\right) \phi \dx\right|
=\left| \int_{B_{2 R+r}} u^m \Delta\phi \dx\right| \\
& \le \int_{B_{2 R+r}} u^m \big|\Delta\phi\big| \dx\\
&\le \( \int_{B_{2 R+r}} u\,\phi \dx\)^{m} \( \int_{B_{2\,R+r}}\frac{\left|\Delta \phi\right|^{\frac1{1-m}}}{\phi ^{\frac{m}{1-m}}} \dx\)^{1-m} \\
 &\quad :=C\left( \phi\right) \( \int_{B_{2\,R+r}}{u\,\phi \left( x\right)\dx}\)^{m}\,,
\end{split}\ee
where we have used H\"older's inequality with conjugate exponents $\frac1{m}$ and $\frac1{1-m}$.
We have obtained the following closed differential inequality
\begin{equation*}
 \left|\frac{\rd}{\dt} \int_{B_{2 R+r}}{u\left( t,x\right) \phi\left( x\right)\dx}\right| \le C(\phi) \( \int_{B_{2 R+r}}{u\left(t,x\right) \phi\left( x\right)\dx}\)^m.
\end{equation*}
An integration in time shows that, for all $t$, $\tau\ge 0$, we have
\begin{equation*}
\left( \int_{B_{2 R}}{u\left( t,x\right) \phi\left( x\right)\dx}\right)^{1-m} \le \left( \int_{B_{2 R}}{u\left( \tau,x\right) \phi\left( x\right)\dx}\right)^{1-m}+\left( 1-m\right)C\left( \phi\right) \left| t-\tau\right|.
\end{equation*}
Since $ \phi $ is supported in $ _{2\,R+r} $ and equal to $ 1$ in $ B_{2\,R}$, this implies~\eqref{Herrero.Pierre.123}, indeed, using
\[
(a+b)^{\frac1{1-m}}\le 2^{\frac1{1-m}-1} \left(a^{\frac1{1-m}}+b^{\frac1{1-m}}\right),
\]
we get
\begin{equation*}\begin{split}
 \int_{B_{2 R}} u(t,x)\dx &\le 2^{\frac{m}{1-m}}\( \int_{B_{2 R+r}} u(\tau,x)\dx+\big( (1-m)\,C(\phi)\big)^{\frac1{1-m}} \left| t-\tau\right|^{\frac1{1-m}}\)\\
&\le 2^{\frac{m}{1-m}}\,\int_{B_{2 R+r}} u(\tau,x)\dx+\cc\,\frac{|t-\tau|^{\frac1{1-m}}}{r^{\frac{2-d\,(1-m)}{1-m}}}\,,
\end{split}
\end{equation*}
where
\[
\cc(r):=2^{\frac{m}{1-m}}\,\big((1-m)\,C(\phi)\big)^{\frac1{1-m}}r^{\frac{2-d\,(1-m)}{1-m}}\,.
\]
The above proof is formal when considering weak or very weak solutions, in which case, it is quite lengthy (although standard) to make it rigorous, \emph{cf.}~\cite[Proof of Lemma 3.1]{Herrero1985}; indeed, it is enough to consider the time-integrated version of estimates~\eqref{derivation.L1.norm}, and conclude by a Gr\"onwall-type argument.

The proof is completed once we show that the quantity $\cc(r)$ is bounded and provide the expression~\eqref{C3.constant}. Recall that $\phi=\varphi^\beta$, so that
\be{ineq.Delta}\begin{split}
\left|\Delta\left(\phi(x)\right)\right|^{\frac1{1-m}}\phi(x)^{-\frac{m}{1-m}}&=
\varphi(x)^{-\frac{\beta\,m}{1-m}}\left|\beta\,(\beta-1)\,\varphi^{\beta-2}\,\left|\nabla\,\varphi\,\right|^2+\beta\,\varphi^{\beta-1}\,\Delta\varphi\right|^{\frac1{1-m}}\\
&\le\big(\beta\,(\beta-1)\big)^{\frac1{1-m}}\,\varphi^{\frac{\beta-2-\beta\,m}{1-m}}
\left|\,\left|\nabla\,\varphi\,\right|^2+\left|\Delta\varphi\right|\right|^{\frac1{1-m}}\\
&\le\(\tfrac{4\,(3+m)}{(1-m)^2}\)^{\frac1{1-m}}
\left(\tfrac{4\,(d+1)}{r^2}\right)^{\frac1{1-m}}\,.
\end{split}
\ee
The first inequality follow from the fact that we are considering a radial function $0\le \varphi(x)\le 1$, and we take $\beta=\frac4{1-m}>\frac2{1-m}$. The last one follows by~\eqref{test.estimates.1}. Finally:
\[\begin{split}
&\big( (1-m)\,C(\phi)\big)^{\frac1{1-m}}\,r^{\frac{2-d\,(1-m)}{1-m}}\\
&=(1-m)^{\frac1{1-m}} \( \int_{B_{2\,R+r}\setminus B_{2\,R}}\frac{\left|\Delta \phi\right|^{\frac1{1-m}}}{\phi ^{\frac{m}{1-m}}} \dx\)r^{\frac{2-d\,(1-m)}{1-m}}\\
&\le (1-m)^{\frac1{1-m}} \(\tfrac{4\,(3+m)}{(1-m)^2}\)^{\frac1{1-m}}\left(\tfrac{4\,(d+1)}{r^2}\right)^{\frac1{1-m}}\big|B_{2\,R+r}\setminus B_{2\,R}\big|\,r^{\frac{2-d\,(1-m)}{1-m}}\\
&\quad=\omega_d\(\tfrac{16\,(d+1)\,(3+m)}{1-m}\)^{\frac1{1-m}}\frac{(2\,R+r)^d-(2\,R)^d}{d\,r^d}\\
&\quad\le\omega_d\(\tfrac{16\,(d+1)\,(3+m)}{1-m}\)^{\frac1{1-m}} (\varrho_0+1)
\end{split}
\]
where we have used that the support of $\Delta\phi$ is contained in the annulus $B_{2\,R+r}\setminus B_{2\,R}$, inequality~\eqref{ineq.Delta} and in the last step we have used that $ \varrho_0\,r\ge 2\,R$ and
\[
(2\,R+r)^d-(2\,R)^d \le d\,(2\,R+r)^{d-1}\,r \le d\,(\varrho_0+1)\,r^d\,.
\]
This concludes the proof of~\eqref{Herrero.Pierre.123}. \par

The proof of~\eqref{Herrero.Pierre.opposite} is very similar to the previous one. Indeed, it is sufficient to take $\psi=\left(1-\phi\right)^\beta$ and to derive in time the quantity
\[
\int_{\R^d\setminus B_{2\,R+r}} u(t,x)\phi(x)\dx\,,
\]
and perform the same computations as in the previous case. The proof of the Lemma is now completed.\end{proof}

%%%%%%%%%%%%%%%%%%%%%%%%%%%%%%%%%%%%%%%%%%%%%%%%%%%%%%%%%%%%%%%%%%%%%%%%
\subsection{Local upper bounds}\label{Sec:LocalUpperBounds}
Our main result in this subsection is a local $\mathrm L^1-\mathrm L^\infty$ smoothing effect for solutions to~\eqref{FD}.  We recall that $\alpha=d\,(m-m_c)$ is as in~\eqref{R}.
%---------------------------------------------------------------------
\begin{lemma}\label{Lem:LocalSmoothingEffect} Let $d\ge1$, $m\in\big[ m_1,1\big)$. Then there exists a positive constant~$\overline\kappa$ such that for any solution $u$ of~\eqref{FD} with nonnegative initial datum $u_0\in\mathrm L^1(\R^d)$ satisfies for all $(t,R)\in(0,+\infty)^2$ the estimate
\be{BV-1}
\sup_{y\in B_{R/2}(x)}u(t,y)\le\overline{\kappa}\left(\frac1{t^{d/\alpha}}\(\int_{B_R(x)}u_0(y)\,\dy\)^{2/\alpha}+\(\frac t{R^2}\)^\frac1{1-m}\right)\,.
\ee
\end{lemma}
%---------------------------------------------------------------------
Even if the above estimate is well known, \emph{cf.}~\cite{DiBenedetto1993, DiBenedetto2012, Daskalopoulos2007,Bonforte2010b, Bonforte2019a}, the expression of the constant $\overline{\kappa}$ was unknown, to the best of our knowledge. We provide it here:
\be{kappa}
\overline\kappa=\mathsf k\,\mathcal K^\frac{2\,q}\beta
\ee
where $\mathsf k=\mathsf k(m,d,\beta,q)$ is such that
\[
\mathsf k^\beta=\big(\tfrac{4\,\beta}{\beta+2}\big)^\beta\,\big(\tfrac4{\beta+2}\big)^2\,\pi^{\,8\,(q+1)}
\,e^{8\sum_{j=0}^{\infty}\log(j+1)\,\left(\frac q{q+1}\right)^j}\,2^\frac{2\,m}{1-m}\,(1+\mathtt{a}\,\omega_d)^2\,\mathtt{b}
\]
with $\mathtt{a}=\tfrac{3\,(16\,(d+1)\,(3+m))^\frac1{1-m}}{(2-m)\,(1-m)^\frac m{1-m}}+\tfrac{2^\frac{d-m\,(d+1)}{1-m}}{3^d\,d}\quad\mbox{and}\quad \mathtt{b}=\tfrac{38^{2\,(q+1)}}{\big(1-(2/3)^{\frac{\beta}{4\,(q+1)}}\big)^{4\,(q+1)}}$.\\
The constant $\mathcal K$ is the same constant as in~\eqref{sob.step2} and corresponds to the inequality
\be{GNS111}
\|f\|^2_{\mathrm L^{\pc}(B)}\le\mathcal K\(\|\nabla f\|^2_{\mathrm L^2(B)}+\|f\|^2_{\mathrm L^2(B)}\).
\ee
In other words,~\eqref{GNS111} is~\eqref{sob.step2} written for $R=1$. The other parameters are given in Table~\ref{table.k.bar}.
%---------------------------------------------------------------------
\begin{table}[ht]
\begin{center}
\begin{tabular}{|c||c|c|c|c|}\hline
~ & $\pc$ & $\mathcal K$ & $q$ & $\beta$\\[6pt]
\hline\hline
$d\ge 3$ & $\frac{2\,d}{d-2}$ & $\frac{4\,\Gamma\big(\tfrac{d+1}2\big)^{2/d}}{2^\frac2d\,\pi^{1+\frac1d}}$ & $\frac d2$& $\alpha$ \\[6pt]
\hline
$d=2$ & $4$ & $\frac4{\sqrt\pi}$ & $2$ & $2\,(\alpha-1)$ \\[6pt]
\hline
$d=1$ & $\frac4m$ & $2^{1+\frac m2}\,\max\{\frac{2\,(2-m)}{m\,\pi^2},\frac14\}$ & $\frac2{2-m}$ & $\frac{2\,m}{2-m}$\\[6pt]
\hline
\end{tabular}
\vspace{2mm}
\caption{\label{table.k.bar} Table of the parameters and the constant $\mathcal K$ in dimensions $d=1$, $d=2$ and $d\ge3$. The latter case corresponds to the critical Sobolev exponent while the inequality for $d\le2$ is subcritical. In dimension $d=1$, $\pc=4/m$, which makes the link with~\eqref{estim.S-p}.}
\end{center}
\end{table}
%---------------------------------------------------------------------

\begin{proof}[Proof of Lemma~\ref{Lem:LocalSmoothingEffect}] The proof presented here follows closely the scheme of~\cite{Bonforte2010b,Bonforte2019a} so we shall only sketch the main steps, keeping track of the explicit expression of all constants. Without loss of generality we can assume that $x=0$, by translation invariance. We also recall that $u$ always possesses the regularity needed to perform all computations throughout the following steps.

Let us introduce the rescaled function
\be{realation-u-hatu}
\hat u(t,x)=\left(\frac{R^2}\tau\right)^{\frac1{1-m}}\,u(\tau\,t,R\,x)
\ee
which solves~\eqref{FD} on the cylinder $\left(0, 1\right]\times B_1$. In Steps 1-3 we establish a $\mathrm L^2-\mathrm L^\infty$ smoothing inequality for $\hat v=\max\{\hat u,1\}$, which we improve to a $\mathrm L^1-\mathrm L^\infty$ smoothing in Step 4, by means of a {\em de Giorgi}-type iteration. In Step 5, we rescale back the estimate and obtain the desired result for $u$.

\begin{steps}
\stepitem We observe that $\hat v=\max\{\hat u,1\}$ solves $\frac{\partial \hat v}{\partial t}\le\Delta \hat v^m$. According to~\cite[Lemma~2.5]{Bonforte2010b}, we know that
\begin{multline*}
\sup_{s\in[T_1,T]}\int_{B_{R_1}}\hat v^{p_0}(s,x)\,\dx+\iint_{Q_1}\left|\nabla \hat v^\frac{p_0+m-1}2\right|^2\,\dx\,\dt\\
\le\frac8{c_{m,p_0}}\iint_{Q_0}\(\hat v^{m+p_0-1}+\hat v^{p_0}\)\dx\,\dt
\end{multline*}
where $Q_k=(T_k,T]\times B_{R_k}$ with $0<T_0<T_1<T\le 1$, $0<R_1<R_0\le 1$ and $c_{m,p_0}=\min\left\{1-\tfrac1{p_0},\tfrac{2\,(p_0-1)}{p_0+m-1}\right\}\ge\tfrac12$. We have $\hat v^{m+p_0-1}\le \hat v^{p_0}$ because $\hat v\ge 1$, so that
\be{Lemma2.5:BV2010}
\sup_{s\in[T_1,T]}\int_{B_{R_1}}\hat v^{p_0}(s,x)\,\dx+\iint_{Q_1}\left|\nabla \hat v^\frac{p_0+m-1}2\right|^2\,\dx\,\dt\le\mathcal C_0\iint_{Q_0}\hat v^{p_0}\,\dx\,\dt
\ee
where
\[
\mathcal C_0=32\(\frac1{(R_0-R_1)^2}+\frac1{T_1-T_0}\).
\]

\stepitem Let $\pc$ be as defined in Table~\ref{table.k.bar} and $\mathcal K$ be the constant in the inequality~\eqref{sob.step2}. Let $q=\pc/(\pc-2)$ and $Q_i=(T_i,T]\times B_{R_i}$ as in Step 1. We claim that
\be{induction.step}
\iint_{Q_1}\hat v^{p_1}\,\dx\,\dt\le\mathcal K_0\(\iint_{Q_0}\hat v^{p_0}\,\dx\,\dt\)^{1+\frac1q}\quad\mbox{with}\quad\mathcal K_0=\mathcal K\(R_1^{-2}+\mathcal C_0\)^{1+\frac1q}\,.
\ee
Let us prove~\eqref{induction.step}. Using H\"older's inequality, for any $a\in(2,\pc)$ we may notice that
\[
\int_{B_{R_1}}|f(s,x)|^a\,\dx=\int_{B_{R_1}}|f(s,x)|^2\,|f(s,x)|^{a-2}\,\dx\le\|f\|_{\mathrm L^{\pc}(B_{R_1})}^2\,\|f\|_{\mathrm L^b(B_{R_1})}^{a-2}
\]
with $b=q\,(a-2)$. Using~\eqref{sob.step2}, this leads to
\begin{multline*}
\iint_{Q_1}|f(t,x)|^a\,\dx\,\dt\\
\le\mathcal K\(\|\nabla f\|_{\mathrm L^2(Q_1)}^2+\tfrac1{R_1^2}\,\|f\|_{\mathrm L^2(Q_1)}^2\)\sup_{s\in(T_1,T)}\(\int_{B_{R_1}}|f(s,x)|^b\,\dx\)^\frac1q.
\end{multline*}
Choosing $f^2=\hat v^{p_0+m-1}$ with $a=2\,p_1/(p_0+m-1)$ and $b=2\,p_0/(p_0+m-1)$ we get
\[
\iint_{Q_1}\hat v^{p_1}\,\dx\,\dt\le\mathcal K\iint_{Q_1}\(\left|\nabla\hat v^\frac{p_0+m-1}2\right|^2+\frac{\hat v^{p_0}}{R_1^2}\)\dx\,\dt\sup_{s\in(T_1,T)}\(\int_{B_{R_1}}\hat v^{p_0}\,\dx\)^\frac1q
\]
where
\[
p_1=\(1+\frac1q\)p_0-1+m>p_0\,.
\]
Letting $X=\nrm{\nabla\hat v^{(p_0+m-1)/2}}2^2$, $Y_i=\iint_{Q_1}\kern-3pt\hat v^{p_i}\,\dx\,\dt$ and $Z=\sup_{s\in(T_1,T)}\int_{B_{R_1}}\kern-4pt\hat v^{p_0}\,\dx$, we get $Y_1\le\mathcal K\,(X+R_1^{-2}\,Y_0)\,Z^{1/q}$, while~\eqref{Lemma2.5:BV2010} reads $X+Z\le\mathcal C_0\,Y_0$. Hence $Y_1\le\mathcal K\,\big((R_1^{-2}+\mathcal C_0)\,Y_0-Z\big)\,Z^{1/q}\le\,\mathcal K\,\big((R_1^{-2}+\mathcal C_0)\,Y_0\big)^{(q+1)/q}$, that is inequality~\eqref{induction.step}.

\stepitem We perform a \index{Moser iteration}{Moser-type iteration}. In order to iterate~\eqref{induction.step}, fix $R_\infty<R_0<1$, $T_0<T_\infty<1$ and also assume that $2\,R_\infty\ge R_0$. We shall consider the sequences $(p_k)_{k\in\N}$, $(R_k)_{k\in\N}$, $(T_k)_{k\in\N}$ and $(\mathcal K_k)_{k\in\N}$ defined as follows:
\begin{align*}
&p_k=\(1+\frac1q\)^k\(2-q\,(1-m)\)+q\,(1-m)\,,\\
&R_k-R_{k+1}=\frac6{\pi^2}\,\frac{R_0-R_\infty}{(k+1)^2}\,,\quad T_{k+1}-T_k=\frac{90}{\pi^4}\,\frac {T_\infty-T_0}{(k+1)^4}\,,\\
&\mathcal K_k=\mathcal K\(R_{k+1}^{-2}+\mathcal C_k\)^{1+\frac1q}\,,\quad\mathcal C_k=32\(\frac1{(R_k-R_{k+1})^2}+\frac1{T_{k+1}-T_k}\),
\end{align*}
using the Riemann sums $\sum_{k\in\N}(k+1)^{-2}=\frac{\pi^2}6$ and $\sum_{k\in\N}(k+1)^{-4}=\frac{\pi^4}{90}$. It is clear that $\lim\limits_{k\to+\infty}R_k=R_\infty$, $\lim\limits_{k\to+\infty}T_k=T_\infty$ and $\mathcal C_k$ diverge as $k\to+\infty$. In addition, the assumption $2\,R_\infty\ge R_0$ leads to $R_{k+1}^{-2}\le (R_0-R_\infty)^{-2}$ hence $\mathcal K_k$ is explicitly bounded by
\[
\mathcal K_k \le \mathcal K \left(\pi^4\,(k+1)^4 L_{\infty}\right)^{1+\frac1q},\quad\mbox{where}\quad L_{\infty}:=\frac1{(R_0-R_\infty)^2}+\frac1{\(T_\infty-T_0\)}\,.
\]
Set $Q_\infty=(T_\infty, T)\times B_{R_\infty}$ and notice that $Q_\infty\subset Q_k$ for any $k\ge 0$. By iterating~\eqref{induction.step}, we find that
\begin{multline*}
\nrm{\hat v}{\mathrm L^{p_{k+1}}(Q_{\infty})}\le \nrm{\hat v}{\mathrm L^{p_{k+1}}(Q_{k+1})}\\
\le\mathcal K_k^\frac1{p_{k+1}}\,\nrm{\hat v}{\mathrm L^{p_k}(Q_k)}^\frac{(q+1)\,p_k}{q\,p_{k+1}} \le \prod_{j=0}^k \mathcal K_j^{\frac1{p_{k+1}}\left(\frac{q+1}{q}\right)^{k-j}} \nrm{\hat v}{\mathrm L^2(Q_0)}^\frac{2\,(q+1)^{k+1}}{q^{k+1}\,p_{k+1}}
\end{multline*}
and
\[
\prod_{j=0}^k \mathcal K_j^{\frac1{p_{k+1}}\left(\frac{q+1}{q}\right)^{k-j}} \le \left[\mathcal K \left(\pi^4\,L_{\infty}\right)^{1+\frac1q}\right]^{\frac1{p_{k+1}}\sum_{j=0}^k\(\frac{q+1}{q}\)^j}\,\prod_{j=1}^{k+1}\,j^\frac{4\(\frac{q+1}{q}\)^{k+2-j}}{p_{k+1}}.
\]
By lower semicontinuity of the $\mathrm L^\infty$ norm, letting $k\to+\infty$, we obtain
\be{step1to3}
\|\hat v\|_{\mathrm L^{\infty}((T_\infty, T]\times B_{R_\infty})} \le \mathcal C\,\|\hat v\|_{\mathrm L^2((T_0, T]\times B_{R_0})}^{\frac2{2-q\,(1-m)}}
\ee
where $0<T_0<T_\infty<T\le 1$, $1/2<R_\infty<R_0 \le 1$, $R_0 \le 2\,R_\infty$, and
\[
\mathcal C=\mathcal K^\frac{q}{2-q\,(1-m)}\left(\pi^4\,L_{\infty}\right)^\frac{(q+1)}{2-q\,(1-m)}\,e^{\frac{4\,(q+1)}{q\(2-q\,(1-m)\)}\sum_{j=1}^{\infty}\left(\frac{q}{q+1}\right)^j\log j}\,.
\]

\stepitem We show how to improve the $\mathrm L^2-\mathrm L^\infty$ smoothing estimate~\eqref{step1to3} to a $\mathrm L^1-\mathrm L^\infty$ estimate, using a de Giorgi-type iteration. Let us set
\be{beta}
\beta=2-2\,q\,(1-m)\,=\begin{cases}\begin{array}{ll}
\alpha\quad&\mbox{if}\quad d\ge 3\,,\\
2\,(\alpha-1)\quad&\mbox{if}\quad d=2\,,\\
\frac{2\,m}{2-m}\quad&\mbox{if}\quad d=1\,,
\end{array}\end{cases}
\ee
we recall that $\beta>0$ for any $m\in(m_1, 1)$ and $d\ge1$. Then, from~\eqref{step1to3}, we obtain, using H\"older's and Young's inequalities,
\begin{multline}\label{iteration.step4}
\|\hat v\|_{\mathrm L^{\infty}((1/9, 1]\times B_{1/2})} \le \mathcal C\,\|\hat v\|_{\mathrm L^{\infty}((\tau_1, 1]\times B_{r_1})}^{\frac1{2-q\,(1-m)}}\,\|\hat v\|_{\mathrm L^1((\tau_1, 1]\times B_{r_1})}^{\frac1{2-q\,(1-m)}} \\
\le\frac12\,\|\hat v\|_{\mathrm L^\infty((\tau_1, 1]\times B_{r_1})}+\mathfrak C_1\,\|\hat v\|_{\mathrm L^1((\tau_1, 1]\times B_{r_1}))}^{\frac2\beta}
\end{multline}
where $1/9<\tau_1<1$, $1/2<r_1<1$ and
\[
\mathfrak C_1=X\left(\frac1{\(r_1-\frac12\)^2}+\frac1{\frac19-\tau_1}\right)^{\frac{2\,(q+1)}{\beta}}
\]
with
\[
X=\tfrac\beta{\beta+2}\,\big(\tfrac4{\beta+2}\big)^\frac2{\beta}\,\mathcal K^\frac{2\,q}{\beta}\(\pi^q\,e^{\sum_{j=1}^{\infty}\left(\frac q{q+1}\right)^j\log j}\)^\frac{8\,(q+1)}{q\,\beta}
\,.
\]
%----------------------------------------------
To iterate~\eqref{iteration.step4} we shall consider sequences $(r_i)_{i \in\N}, (\tau_i)_{i\in\N}$ such that
\[
r_{i+1}-r_i=\tfrac16\,(1-\xi)\,\xi^i\,,\quad\tau_i-\tau_{i+1}=\tfrac19\,(1-\xi^2)\,\xi^{2i}\,.
\]
with $\xi=(2/3)^{\frac{\beta}{4\,(q+1)}}$. Since $2/3\le\xi\le 1$, we have
\[
\frac1{1-\xi^2}\le\frac1{5\,(1-\xi)^2}\,,
\]
and this iteration gives us
\[
\|\hat v\|_{\mathrm L^{\infty}((1/9, 1]\times B_{1/2})} \le\frac1{2^k}\,\|\hat v\|_{\mathrm L^\infty((\tau_k, 1]\times B_{r_k})}+\|\hat v\|_{\mathrm L^1((\tau_k, 1]\times B_{r_k})}^{\frac2\beta}\,\sum_{i=0}^{k-1}\frac{\mathfrak C_{i+1}}{2^i}
\]
where for all $i\ge 0$
\[
\frac{\mathfrak C_{i+1}}{2^i}\le\(\tfrac{38}{(1-\xi)^2}\)^\frac{2\,(q+1)}\beta X\left(\tfrac34\right)^i\,.
\]
In the limit $k\rightarrow\infty$ we find
\be{step4}
\|\hat v\|_{\mathrm L^{\infty}((1/9, 1]\times B_{1/2})} \le \mathfrak C\,\|\hat v\|_{\mathrm L^1((0, 1]\times B_{2/3})}^{\frac2\beta}
\ee
where
\be{goth.C}
\mathfrak C=4\(\tfrac{38}{(1-\xi)^2}\)^\frac{2\,(q+1)}\beta X\,.
\ee
%------------------------------------------------
\stepitem In this step we complete the proof of~\eqref{BV-1}. We recall that $\hat v=\max\{\hat u, 1\}$ and then, using inequality~\eqref{step4} and the fact that $\hat u\le\hat v\le \hat u\,d\,(1-m)+1$, we find
\be{inequality.notrescaled}
\sup_{y\in B_{1/2}}\hat u(1,y)\le\|\hat u\|_{\mathrm L^{\infty}\((1/9, 1]\times B_{1/2})\)} \le \mathfrak C\,\|\hat u+1\|_{\mathrm L^1\((0, 1]\times B_{2/3}\)}^{\frac2\beta}\,.
\ee
The function $\hat u$ satisfies the following inequality for any $s \in\left[0,1\right]$
\begin{equation}\label{herrero.pierre.2}
\int_{B_{2/3}}\hat u(s,x)\,\dx \le 2^{\frac{m}{1-m}} \int_{B_1}\hat u_0\,\dx+\mathscr C\,s^{\frac1{1-m}}\,,
\end{equation}
where
\be{calligrafic.C}
\mathscr C=2^{\frac{m}{1-m}}\left(3\omega_d\left[\frac{16\,(d+1)\,(3+m)}{1-m}\right]^{\frac1{1-m}}\right)\,.
\ee
We recall that $\omega_d=|\mathbb S^{d-1}|=\frac{2\,\pi^{d/2}}{\Gamma(d/2)}$. Inequality~\eqref{herrero.pierre.2} is obtained by applying Lemma~\ref{HP-Lemma} with $R=1/3$, $r=1/3$ and $\rho=2$. Integrating inequality~\eqref{herrero.pierre.2} over $\left[0,1\right]$ we find
\be{HP.hat}
\|\hat u\|_{\mathrm L^1((0, 1]\times B_{2/3})} \le 2^\frac{m}{1-m}\int_{B_1}\hat u_0\,\dx+\tfrac{1-m}{2-m}\,\mathscr C\,.
\ee
We deduce from inequalities~\eqref{inequality.notrescaled}-\eqref{HP.hat} that
\be{inequality.notrescaled-12}
\sup_{y\in B_{1/2}(x)}\hat u(1,y) \le \mathfrak C\,\left[2^\frac{m}{1-m} \left( \int_{B_1}\hat u_0\,\dx\right)
+\tfrac{1-m}{2-m}\,\mathscr C+\left(\tfrac23\right)^d\,\frac{\omega_d}d\right]^{\frac2\beta}\,.
\ee
where $\beta$ is as in~\eqref{beta}. Let us define
\[\label{kd12}
\overline{\kappa}:=\mathfrak C\,\left[2^\frac{m}{1-m}+\tfrac{1-m}{2-m}\,\mathscr C+\left(\tfrac23\right)^d\,\frac{\omega_d}d\right]^{\frac2\beta}\,,
\]
with $\mathfrak C$ given in~\eqref{goth.C} and $\mathscr C$ in~\eqref{calligrafic.C}. We first prove inequality~\eqref{BV-1} assuming
\[
\tau \ge \tau_\star:=R^\alpha \|u_0\|_{\mathrm L^1(B_R)}^{1-m}\,,
\]
which, by~\eqref{realation-u-hatu}, is equivalent to the assumption $\|\hat u_0\|_{\mathrm L^1(B_1)}\le 1$.
Indeed, together with~\eqref{inequality.notrescaled-12}, we get
\be{inequality.rescaled-12}
\sup_{y\in B_{R/2}}u(\tau,y)\le \overline{\kappa}\(\frac\tau{R^2}\)^\frac1{1-m}\le\overline{\kappa}\left(\frac1{\tau^{\frac d{\alpha}}}\| u_0\|_{\mathrm L^1(B_R)}^\frac 2\alpha+\(\frac\tau{R^2}\)^\frac1{1-m}\right)\,,
\ee
which is exactly~\eqref{BV-1}. Now, for any $0<t\le\tau_\star$, we use the time monotonicity estimate
\[u(\tau)\le u(\tau_\star) \left(\frac{\tau_\star}\tau\right)^{\frac d{\alpha}}
\]
obtained by integrating in time the estimate $u_t\ge -\(d/\alpha\)(u/t)$ of Aronson and Benilan (see~\cite{MR524760}). Combined with the estimate~\eqref{inequality.rescaled-12} at time $\tau_\star$, this leads to
\[\begin{split}
\sup_{y\in B_{R/2}}u(\tau,y) & \le \sup_{y\in B_{R/2}}u(\tau_\star,y) \left(\frac{\tau_\star}\tau\right)^\frac d\alpha\le \overline{\kappa}\(\frac{\tau_\star}{R^2}\)^\frac1{1-m}\left(\frac{\tau_\star}\tau\right)^\frac d\alpha\\
&=\overline{\kappa}\frac{\|u_0\|_{\mathrm L^1(B_R)}^\frac2\alpha}{\tau^\frac d{\alpha}}\le\overline{\kappa}\left(\frac1{\tau^{\frac d{\alpha}}}\| u_0\|_{\mathrm L^1(B_R)}^\frac 2\alpha+\(\frac\tau{R^2}\)^\frac1{1-m}\right)
\end{split}
\]
and concludes the proof.
\end{steps}
\end{proof}

%%%%%%%%%%%%%%%%%%%%%%%%%%%%%%%%%%%%%%%%%%%%%%%%%%%%%%%%%%%%%%%%%%%%%%%%
\subsection{A comparison result based on the Aleksandrov reflection principle}\label{sec:Aleksandrov-Reflection-Principle}

In this section, we state and prove a version of the \emph{\idx{Aleksandrov reflection principle}}, or \emph{moving plane method}, a key tool for proving the lower bounds of Lemma~\ref{Posit.Thm.FDE}. Analogous results have been proven in~\cite{Bonforte2006,Bonforte2010b,Galaktionov2004} for similar purposes, and we borrow some ideas from those proofs.
%---------------------------------------------------------------------
\begin{proposition}\label{Local.Aleks}
Let $d\ge1$, $m\in\big(m_c,1\big)$ and let $B_{\lambda R}(x_0)\subset\RR^d$ be an open ball with center in $x_0\in\RR^d$ of radius $\lambda\,R$ with $R>0$ and $\lambda>2$. Let $u$ be a solution of~\eqref{FD} with nonnegative initial datum $u_0\in\mathrm L^1(\R^d)$ such that
$\supp(u_0)\subset B_{R}(x_0)$. Then
\be{inequality-0}
u(t,x_0)\ge u(t,x)
\ee
for any $t>0$ and for any $x\in D_{\lambda, R}(x_0)=B_{\lambda
R}(x_0)\setminus B_{2\,R}(x_0)$. Hence,
\be{Aleks.Mean}
u(t,x_0)\ge \left|D_{\lambda, R}(x_0)\right|^{-1}\int_{D_{\lambda,
R}(x_0)}u(t,x)\dx\,.
\ee
\end{proposition}
%---------------------------------------------------------------------
We use the mean value inequality (\ref{Aleks.Mean}) in following form:
\begin{equation} \label{Aleks.Mean.r}
\int_{B_{2\,R+r}(x_0)\setminus B_{2^b R}(x_0)}u(t,x)\dx\le A_d\,r^d\,u(t,x_0)\,,
\ee
with $b=2-(1/d)$, $r>2\,R(2^{1-\frac1{d}}-1)=:r_0$ and a suitable positive constant $A_{d}$. This inequality can easily be obtained from~\eqref{Aleks.Mean}. Let us first assume $d\ge2$, note that in this case $b-1\ge1/2 $ and therefore $r\ge 2\,R\(\sqrt2-1\)$. By Taylor expansion we obtain that for some $\xi\in\(r_0, r\)$ that
\[\begin{split}
\left|B_{2\,R+r}(x_0)\setminus
B_{2^bR}(x_0)\right|&=\frac{\omega_d}{d}\left[(2\,R+r)^d-2^{bd}\,R^d\right]=\omega_d\,(2\,R+\xi)^{d-1}\(r-r_0\)\\
& \le\,\omega_d\,(2\,R+\xi)^{d-1}\,r \le\,\omega_d\,r^d\(\frac{\sqrt2}{\sqrt2-1}\)^{d-1}\,,
\end{split}\]
a simple computation shows that $\sqrt2/\(\sqrt2-1\)\approx3.4142135\le4 $. In the case $d=1$, we have that $b=1$ and therefore
\[
\left|B_{2\,R+r}(x_0)\setminus
B_{2\,R}(x_0)\right|=\omega_1\,r\,.
\]
In conclusion we obtain that, for $r\ge2\,R\,(2^{1-\frac1{d}}-1)$, we have
\be{AD}
\left|B_{2\,R+r}(x_0)\setminus
B_{2^bR}(x_0)\right| \le A_d\,r^d\quad\mbox{where}\quad A_d:=\omega_d\,4^{d-1}\,.
\ee

\begin{proof}Without loss of generality we may assume that $x_0=0$ and write $B_R$ instead of $B_R(0)$. Let us recall that the support of $u_0$ is contained in $B_R$. Let us consider an hyperplane $\Pi$ of equation $\Pi=\{x\in\RR^d~|~ x_1=a\}$ with $a\ge R>0$, in this way $\Pi$ is tangent to the the sphere of radius $a$ centered in the origin. Let us as well define $\Pi_+=\{x\in\RR^d~|~ x_1>a\}$ and $\Pi_-=\{x\in\RR^d~|~x_1<a\}$, and the reflection $\sigma(z)=\sigma(z_1,z_2,\ldots,z_n)=(2a-z_1,z_2,\ldots,z_n)$. By these definitions we have that $\sigma(\Pi_+)=\Pi_-$ and $\sigma(\Pi_-)=\Pi_+$.
Let us denote $Q=(0, \infty)\times \Pi_-$ and the parabolic boundary $\partial_pQ:=\partial Q$. We now consider the \emph{Boundary Value Problem} (BVP) defined as
\be{FDE.Problem.BVP}\tag{BVP}
\begin{split}
\left\{\begin{array}{lll}
u_t=\Delta (u^m) &~{\rm in}~ Q,\\
u(t,x)=g(t,x) &~{\rm in}~ \partial_p Q,\\
\end{array}\right.
\end{split}
\ee
for some (eventually continuous) function $g(t,x)$. Let us define $u_1(t,x)$ to be the restriction of $u(t,x)$ to $Q$ and $u_2(t,x)=u(t, \sigma(x))$. We recall that $u_2(t,x)$ is still a solution to~\eqref{FD}. Also, both $u_1(t,x)$ and $u_2(t,x)$ are solutions to~\eqref{FDE.Problem.BVP} with boundary values $g_1(t,x)$ and $g_2(t,x)$. Furthermore, for any $t>0$ and for any $x\in\Pi$, we have that $g_1(t,x)=g_2(t,x)$, as well $g_1(t,x)=u_0\ge g_2(t,x)=0$ for any $x\in\Pi_-$. By comparison principle we obtain for any $(t,x)\in Q$
\[
u_1(t,x)\ge u_2(t,x)\,,
\]
which implies that for any $t>0$
\[
u(t,0)\ge u(t,(2a, \dots, 0)).
\]
By moving $a$ in the range $(R,\lambda R/2)$ we find that $u(t,0)\ge u(t,x)$ for any $x\in D_{\lambda,R}$ such that $x=(x_1, 0, \dots, 0)$. It is clear that by rotating the hyperplane $\Pi$ we can generalize the above argument and obtain inequality~\eqref{inequality-0}. Lastly, we observe that inequality~\eqref{Aleks.Mean} can be easily deduced by averaging inequality~\eqref{inequality-0}. The proof is complete. \end{proof}

%%%%%%%%%%%%%%%%%%%%%%%%%%%%%%%%%%%%%%%%%%%%%%%%%%%%%%%%%%%%%%%%%%%%%%%%
\subsection{Local lower bounds}\label{Sec:Locallowerbounds}

The main estimate of this subsection is a lower bound in which, again, the novelty is the explicit form of the numerical constants.
%---------------------------------------------------------------------
\begin{lemma}\label{Posit.Thm.FDE} 
Let $d\ge1$ and $m\in\big[m_1,1\big)$. Let $x_0\in\R^d$, $u(t,x)$ be a solution to~\eqref{FD} with nonnegative initial datum $u_0\in\mathrm L^1(\R^d)$ and let $R>0$ such that $M_R(x_0):=\|u_0\|_{\mathrm L^1(B_R(x_0))}>0$. Then the inequality
\be{BV-3}
\inf_{|x-x_0|\le R}u(t,x)\ge\kappa\left(R^{-2}\,t\right)^\frac1{1-m}\quad\forall\,t\in[0,2\,\underline t]
\ee
holds with
\[\label{BV-3t1a}
\underline t=\tfrac12\,\kappa_\star\,M_R^{1-m}(x_0)\,R^{\alpha}\,.
\]
\end{lemma}
%---------------------------------------------------------------------
Following the scheme of the proof of~\cite{Bonforte2006,Vazquez2006}, and exploiting the precise results of Sections~\ref{sec.herrero.pierre} and~\ref{sec:Aleksandrov-Reflection-Principle}, we are able to provide an explicit expression of:
\be{kappaExpr-kappastarExpr}
\kappa_\star=2^{\,3\,\alpha+2}\,d^{\,\alpha}\quad\mbox{and}\quad\kappa=\alpha\,\omega_d\(\frac{(1-m)^4}{2^{38}\,d^{\,4}\,\pi^{16\,(1-m)\,\alpha}\,\overline\kappa^{\,\alpha^2\,(1-m)}}\)^\frac2{(1-m)^2\,\alpha\,d}\,.
\ee

\begin{proof}\begin{steps} Without loss of generality we assume that
$x_0=0$. The proof is a combination of several steps. Different positive constants that depend on $m$ and $d$ are denoted by $C_i$.

\stepitem {\sl Reduction.} By comparison we may assume $\supp(u_0)\subset B_{R}(0)$. Indeed, a general $u_0\ge 0$ is greater than $u_0\chi_{B_R}$, $\chi_{B_R}$ being the characteristic function of $B_{R}$. If $v$ is the solution of the \idx{fast diffusion equation} with initial data $u_0\chi_{B_R}$ (existence and uniqueness are well known in this case), then we obtain by comparison:
\[
\inf_{x\in B_{R}}u(t,x)\ge \inf_{x\in B_{R}}v(t,x)\,.
\]

\stepitem{\sl A priori estimates.} The so-called smoothing effect (see e.g.~\cite[Theorem~2.2]{Herrero1985} , or~\cite{Vazquez2006}) asserts that, for any $t>0$ and $x \in\RR^d$, we have:
\be{est.hp}
u(t,x)\le \overline{\kappa}\,\frac{\|u_0\|_1^\frac2\alpha}{t^\frac d\alpha}\,.
\ee
where $\alpha=2-d\,(1-m)$. We remark that~\eqref{est.hp} can be deduced from inequality~\eqref{BV-1} of Lemma~\ref{Lem:LocalSmoothingEffect} by simply taking the limit $R\rightarrow \infty$. The explicit expression of the constant $\overline{\kappa}$ is given in~\eqref{kappa}. We remark that $\|u_0\|_1=M_{R}$ since $u_0$ is nonnegative and supported in $B_{R}$, so that we get $ u(t,x)\le \overline{\kappa} M_{R}^\frac2\alpha\,t^{-\frac d\alpha}$. Let $b=2-1/d$, an integration over $B_{2^bR}$ gives then:
\be{HP.1.Apriori}
\int_{B_{2^bR}}u(t,x)\dx\le \overline{\kappa}\,\frac{\omega_d}{d}\,\frac{M_R^\frac2\alpha}{t^\frac d\alpha} \left(2^b\,R\right)^d\le C_2\,\frac{M_R^\frac2\alpha}{t^\frac d\alpha}\,R^d\,,
\ee
where $C_2$ can be chosen as
\be{C2}
C_2:=2^d\,\max\Big\{1,\overline{\kappa}\,\frac{\omega_d}{d} \Big\}\,.
\ee

\stepitem In this step we use the so-called {\emph\idx{Aleksandrov reflection principle}}, see Proposition~\ref{Local.Aleks} in Section~\ref{sec:Aleksandrov-Reflection-Principle} for its proof. This principle reads:
\be{Posit.Alex}
\int_{B_{2\,R+r}\setminus B_{2^bR}}u(t,x)\dx\le A_{d}\,r^d u(t,0)
\ee
where $A_d$ is as in~\eqref{AD} and $b=2-1/d$. One has to remember of the condition
\be{r-condition}
r\,\ge\,(2^{(d-1)/d}-1)\,2\,R.
\ee
We refer to Proposition~\ref{Local.Aleks} and formula {\rm (\ref{Aleks.Mean.r})} in section~\ref{sec:Aleksandrov-Reflection-Principle} for more details.

\stepitem{\sl Integral estimate.} Thanks to Lemma~\ref{HP-Lemma}, for any $r$, $R>0$ and $s,t\ge 0$ one has
\[
\int_{B_{2\,R}}u(s,x)\dx\le C_3
\left[\int_{B_{2\,R+r}}u(t,x)\dx+\frac{|s-t|^{1/(1-m)}}{r^{(2-d\,(1-m))/(1-m)}}\right],
\]
where the constant $C_3$ has to satisfy $C_3\ge\max\{1,\cc\}$ and $\cc$ is defined in~\eqref{C3.constant}. In what follows we prefer to take a larger constant (for reasons that will be clarified later) and put
\[\label{C_3}
C_3=\left(\frac{16}{1-m}\right)^{\frac1{1-m}}\max\left\{1,2\,\omega_d\,\left[\frac{16\,(d+1)\,(3+m)}{1-m}\right]^{\frac1{1-m}}\right\}.
\]
We let $s=0$ and rewrite it in a form more useful for our purposes:
\be{Quasi.Mass.Posit}
\int_{B_{2\,R+r}}u(t,x)\dx\ge\frac{M_{R}}{C_3}
-\frac{t^{\frac1{1-m}}}{r^{\frac{\alpha}{1-m}}}\,.
\ee
We recall that $M_{2\,R}=M_R$ since $u_0$ is nonnegative and supported in $B_{R}$.

\stepitem We now put together all previous calculations:
\[\begin{split}
\int_{B_{2\,R+r}}u(t,x)\dx&=\int_{B_{2\,R}}u(t,x)\dx+\int_{B_{2\,R+r}\setminus
B_{2^bR}}u(t,x)\dx\\
&\le C_2\,\frac{M_{R}^\frac2\alpha\,R^d}{t^\frac d\alpha}+A_d\,r^d\,u(t,0)\,.
\end{split}\]
This follows from (\ref{HP.1.Apriori}) and (\ref{Posit.Alex}). Next, we use (\ref{Quasi.Mass.Posit}) to obtain:
\[
\frac{M_{R}}{C_3}
-\frac{t^{\frac1{1-m}}}{r^{\frac{\alpha}{1-m}}}\le\int_{B_{2\,R+r}}u(t,x)\dx\le
C_2\,\frac{M_{R}^\frac2\alpha\,R^d}{t^\frac d\alpha}+A_d\,r^d u(t,0)\,.
\]
Finally we obtain
\[\label{Posit.Non.Opt}
u(t,0) \ge\frac1{A_d}\left[\left(\frac{M_{R}}{C_3}-C_2\,\frac{M_{R}^\frac2\alpha\,R^d}{t^\frac d\alpha}
\right)\frac1{r^d}-\frac{t^{\frac1{1-m}}}{r^{\frac2{1-m}}}\right]=\frac1{A_d}\left[\frac{B(t)}{r^d}-\frac{t^{\frac1{1-m}}}{r^{\frac2{1-m}}}\right]\,.
\]

\stepitem The function $B(t)$ is positive when
\[
B(t)=\frac{M_{R}}{C_3}-C_2\,\frac{M_{R}^\frac2\alpha\,R^d}{t^\frac d\alpha}>0
\Longleftrightarrow t>\left(C_3\,C_2\right)^{\frac{\alpha}{d}}\,.
M_{R}^{1-m} R^\alpha
\]
Let us define
\be{tilde-kappa-star}
\tilde{\kappa}_\star:=4\,\left(C_3\,C_2\right)^\frac{\alpha}{d}\quad\mbox{and}\quad \tilde{\underline{t}}=\tfrac12\,\tilde{\kappa}_\star\,M_R^{1-m}\,R^{\alpha}\,.
\ee
We assume that $t\ge 2\,\tilde{\underline{t}}$ and optimize the function
\[
f(r)=\frac1{A_d}\left[\frac{B(t)}{r^d}-\frac{t^{\frac1{1-m}}}{r^{\frac2{1-m}}}\right]
\]
with respect to $r(t)=r>0$. The function $f$ reaches its maximum at $r=r_{max}(t)$ given by
\[
r_{max}(t)=\left(\frac2{d\,(1-m)}\right)^\frac{1-m}{\alpha}\,\frac{t^\frac1{\alpha}}{B(t)^\frac{1-m}{\alpha}}\,.
\]
We recall that we have to verify that $r_{max}$ satisfies condition~\eqref{r-condition}, namely that $r_{max}(t)>\left(2^{(d-1)/d}-1\right)2\,R$. To check this we optimize in $t$ the function $r_{max}(t)$ with respect to $t\in(2\tilde{\underline{t}},+\infty)$. The minimum of $r_{max}(t)$ is attained at a time $t=t_{min}$ given by
\[
t_{min}=\left(\frac2\alpha\,C_2\,C_3\right)^\frac{\alpha}{d}\,M_R^{1-m}\,R^\alpha\,.
\]
We compute $r_{max}(t_{min})$ and find that
\[
r_{max}(t_{min})=\left(\frac2{d\,(1-m)}\right)^\frac{2\,(1-m)}{\alpha}\,\left(\frac2\alpha\,C_2\right)^\frac1{d}\,C_3^\frac2{d\alpha}\,R\,.
\]
Therefore the condition $r_{max}(t_{min})>\left(2^{(d-1)/d}-1\right)2\,R$ is nothing more than a lower bound on the constants $C_2$ and $C_3$, namely that
\[
\left(\frac2{d\,(1-m)}\right)^\frac{2\,(1-m)}{\alpha}\,\left(\frac2\alpha\,C_2\right)^\frac1{d}\,C_3^\frac2{d\alpha}\,\ge 2^{(d-1)/d}-1\,.
\]
Such a lower bound is easily verified, by using the fact $m\in(m_1,1)$, we have $(1-m)^{-1}>d$ and therefore we have the following inequalities
\be{first-estimates}
\frac2{d\,(1-m)}\ge2\,,\quad\frac2\alpha=\frac2{2-d\,(1-m)}\ge1\,,\quad C_2\ge2^d\quad\mbox{and}\quad C_3\ge 16^d\,d^d\,,
\ee
therefore, from the above inequalities we find that
\[
\left(\frac2{d\,(1-m)}\right)^\frac{2\,(1-m)}{\alpha}\,\left(\frac2\alpha\,C_2\right)^\frac1{d}\,C_3^\frac2{d\alpha} \ge 32\,d\ge2^{(d-1)/d}-1\,,
\]
and so the such a lower bound is verified. Let us now continue with the proof.
\stepitem After a few straightforward computations, we show that the maximum value is attained for all $t>2\,\tilde{\underline{t}}$ as follows:
\[
f(r_{max})=\alpha\,A_d\frac{\left[
d\,(1-m)\right]^{\frac{d\,(1-m)}{\alpha}}}{2^\frac 2\alpha}
\left[\frac1{C_3}-C_2\,\frac{M_{R}^{\frac{d\,(1-m)}{\alpha}}R^d}{t^\frac d\alpha}\right]^\frac 2\alpha
\frac{M_{R}^\frac 2\alpha}{t^\frac d\alpha}>0\,.
\]
We get in this way the estimate:
\[
\begin{split}
u(t,0) &\ge K_1\,H_1(t)\,\frac{M_R^\frac 2\alpha}{t^\frac d\alpha}\,,
\end{split}
\]
where
\[
H_1(t)=\left[\frac1{C_3}-C_2\,\frac{M_{R}^{\frac{d\,(1-m)}{\alpha}}R^d}{t^\frac d\alpha}\right]^\frac 2\alpha\quad\mbox{and}\quad K_1=\alpha\,A_d\frac{\left[
d\,(1-m)\right]^{\frac{d\,(1-m)}{\alpha}}}{2^\frac 2\alpha}.
\]
A straightforward calculation shows that the function is non-decreasing in time, thus if $t\ge 2\,\tilde{\underline{t}}$:
\[
H_1(t)\ge H_1(2\,\tilde{\underline{t}})=C_3^{-\frac2\alpha}\,\left(1-4^{-\frac d\alpha}\right)^\frac2\alpha\,,
\]
and finally we obtain for $t\ge 2\,\tilde{\underline{t}}$ that
\be{Posit.Center}
u(t,0) \ge\,K_1\,C_3^{-\frac2\alpha}\,\left(1-4^{-\frac d\alpha}\right)^\frac2\alpha\,\frac{M_{R}^\frac 2\alpha}{t^\frac d\alpha}=\tilde{\underline{\kappa}}\,\frac{M_{R}^\frac 2\alpha}{t^\frac d\alpha}\,.
\ee

\stepitem {\sl From the center to the infimum.} Now we want to obtain a positivity estimate for the infimum of the solution $u$ in the ball $B_R=B_R(0)$. Suppose that the infimum is attained in some point $x_m\in\overline{B_R}$, so that $\inf_{x\in B_R}u(t,x)=u(t,x_m)$, then one can apply {\rm(\ref{Posit.Center})} to this point and obtain:
\[\label{Posit.Min}
u(t,x_m) \ge
\tilde{\underline{\kappa}}\,\frac{M_{2\,R}(x_m)^\frac 2\alpha}{t^\frac d\alpha}
\]
for $t>\tilde{\kappa}_\star M_{R}^{1-m}(x_m) R^\alpha$. Since the point $x_m\in\overline{B_R(0)}$ then it is clear that $B_R(0)\subset B_{2\,R}(x_m)\subset B_{4R}(x_0)$, and this leads to the inequality:
\[
M_{2\,R}(x_m)\ge M_R(0)\quad{\rm and}\quad M_{2\,R}(x_m)\le M_{4R}(0)
\]
since $M_{\varrho}(y)=\int_{B_{\varrho}(y)}u_0(x)\dx$ and $u_0\ge 0$. Thus, we have found that:
\[\label{Posit.Final}
\inf_{x\in B_R(0)}u(t,x) =u(t,x_m) \ge
\tilde{\underline{\kappa}}\,\frac{M_{2\,R}^\frac 2\alpha(x_m)}{t^\frac d\alpha}\ge
\tilde{\underline{\kappa}}\,\frac{M_{2\,R}^\frac 2\alpha(0)}{t^\frac d\alpha}=
\tilde{\underline{\kappa}}\,\frac{M_{R}^\frac 2\alpha(0)}{t^\frac d\alpha}\,.
\]
for $t>2\,\tilde{\underline{t}}(0)=\tilde{\kappa}_\star M_{4R}^{1-m}(0) R^\alpha=\tilde{\kappa}_\star\,M_{R}^{1-m}(0) R^\alpha$, after noticing that $M_{4R}(0)=M_{2\,R}(0)=M_R(0)$, since $\supp(u_0)\subset B_R(0)$. Finally we obtain the claimed estimate
\[\label{FDE.R}
\inf_{x\in B_R(0)}u(t,x)\ge
\tilde{\underline{\kappa}}\,\frac{M_{R}^\frac 2\alpha}{t^\frac d\alpha}\quad\forall\,t\ge 2\,\tilde{\underline{t}}\,.
\]

\stepitem The last step consists in obtaining a lower estimate when $0\le t\le 2\,\tilde{\underline{t}}$. To this end we consider the fundamental estimate of B\'enilan-Crandall~\cite{Benilan1981}:
\[
u_t(t,x)\le\frac{u(t,x)}{(1-m)t}\,.
\]
This easily implies that the function:
\[
u(t,x)t^{-1/(1-m)}
\]
is non-increasing in time. Thus, for any $t\in (0,2\,\tilde{\underline{t}})$, we have that
\[
u(t,x)\ge u(2\,\underline{t},x)\,\frac{t^{1/(1-m)}}{(2\,\tilde{\underline{t}})^{1/(1-m)}}\ge\,\tilde{\underline{\kappa}}\,\tilde{\kappa}_\star^{-\frac2{1-m}}\,\left(t\,R^{-2}\right)^\frac1{1-m}\,.
\]
which is exactly inequality~\eqref{BV-3}. It is straightforward to verify that the constant~$\tilde{\kappa}$ has the value
\begin{equation}\label{kappa-tilde}
\tilde{\kappa}
=\tilde{\underline{\kappa}}\,\tilde{\kappa}_\star^{-\frac2{1-m}}\,=\alpha\,A_d\frac{\left[
d\,(1-m)\right]^{\frac{d\,(1-m)}{\alpha}}}{2^\frac 2\alpha}\,C_3^{-\frac2\alpha}\,\left(1-4^{-\frac d\alpha}\right)^\frac2\alpha\,\tilde{\kappa}_\star^{-\frac2{1-m}}\,.
\ee

\stepitem\textit{Simplification of the constants}. In this step we are going to simplify the expression of some constants in order to obtain the expression in~\eqref{kappaExpr-kappastarExpr}. This translates into estimates from below of the actual values of constants $\tilde{\overline{\kappa}}$ and $\tilde{\kappa}_\star$, and in order to do so, we need to estimate $C_2$ and $ C_3$. Let us begin with $C_2$, since we only need an estimate from below. 
For any $d\ge1$ the numerical inequality $\omega_d/d\le \pi^2$ holds. It is then clear from~\eqref{C2} that
\[\label{C2-2}
2^d\le C_2\le 2^d\,\overline{\kappa}\,\pi^2\,.
\]

In the case of $C_3$ we already have a lower bound given in~\eqref{first-estimates}, in what follows we compute the upper bound. A computation shows that the numerical inequality
 $\omega_d\le 16\pi^3\,/15$ holds for any $d\ge1$. Since $m<1$, we have that
\[
16\(d+1\)\(3+m\)\le 64\(d+1\)\le 128\,d\,.
\]

Combining the above inequality, with the estimates on $\omega_d$ and the defintion of $C_3$ we get
\[\label{second-estimate}
\(4d\)^d\,\le C_3\,\le\(\frac{128\,d}{1-m}\)^\frac2{1-m}\,4\,\pi^3\,.
\]
Therefore, we can estimate $\tilde{\kappa}_\star$ and obtain the expression of $\kappa_\star$
\[\label{kappa-star}
\tilde{\kappa}_\star=4\(C_2\,C_3\)^\frac{\alpha}{d} \ge 2^2\(2^{5\,d}\,d^d\)^\frac{\alpha}{d}=2^{3\,\alpha+2}\,d^{\alpha}=:\kappa_\star\,.
\]
Let us simplify $\tilde{\kappa}$. By combining~\eqref{kappa-tilde},~\eqref{tilde-kappa-star} and~\eqref{AD}, we get that
\[\label{second-definition-kappa}
\tilde{\kappa}\ge\alpha\,\omega_d\,2^{2d-2-\frac{2\,(1-m)+4\,\alpha}{\alpha(1-m)}}\,\left[
d\,(1-m)\right]^{\frac{d\,(1-m)}{\alpha}}\,C_3^{-\frac{4}{\alpha\,d\,(1-m)}}\,C_2^{-\frac{2\,\alpha}{d\,(1-m)}}\,\left(1-4^{-\frac d\alpha}\right)^\frac2\alpha\,.
\]
Let us begin simplifying the expression $\left(1-4^{-\frac d\alpha}\right)$. We first notice that, since $\alpha \in\(1,2\)$, we have that $1-4^{-\frac d\alpha} \ge 1-4^{-\frac{d}2} $, which is an expression monotone increasing in $d$. We have therefore that
\[
\left(1-4^{-\frac d\alpha}\right)^\frac2\alpha\ge \left(1-4^{-\frac12}\right)^\frac2\alpha=2^{-\frac2\alpha}\,.
\]
Combining all together we find
\[
\tilde{\kappa}\ge\alpha\,\omega_d\,2^{-\mathfrak{a}}\,\pi^{-\mathfrak{b}}\,\overline{\kappa}^{-\frac{2\,\alpha}{d\,(1-m)}}\,d^{\frac{d\,(1-m)}{\alpha}-\frac{8}{\alpha(1-m)^2d}}\(1-m\)^\frac{d^2\,(1-m)^3+8}{\alpha\,(1-m)^2\,d}\,,
\]
where
\[
\mathfrak{a}=\frac{56+8\,(1-m)+2\,\alpha^2\,d\,(1-m)+2\,\alpha\,(1-m)^2\,d}{\alpha\,(1-m)^2\,d}-2\,d\quad\mbox{and}\quad
\mathfrak{b}=\frac{12+4\,\alpha^2}{d\,(1-m)}\,.
\]
Since $m_1<m<1$, and $d\,(1-m)<1$, we can simplify the expression of $\mathfrak{a}$ and $\mathfrak{b}$ into
\[
\mathfrak{a}\le\frac{76}{\alpha\,(1-m)^2\,d}\quad\mbox{and}\quad\mathfrak{b}\le\frac{32}{d\,(1-m)}\,.
\]
By summing up all estimates above and estimating the exponents of $(1-m)$ and~$d$, we get
\[\label{kappa-definition}
\tilde{\kappa} \ge\frac{\alpha\,\omega_d\,\left(\frac{1-m}{d}\right)^\frac8{\alpha\,(1-m)^2\,d}}{2^\frac{76}{\alpha\,(1-m)^2\,d}\,\pi^\frac{32}{d\,(1-m)}\,\overline{\kappa}^\frac{2\,\alpha}{d\,(1-m)}}=\kappa\,.
\]
\end{steps}
\end{proof}

%%%%%%%%%%%%%%%%%%%%%%%%%%%%%%%%%%%%%%%%%%%%%%%%%%%%%%%%%%%%%%%%%%%%%%%%
%%%%%%%%%%%%%%%%%%%%%%%%%%%%%%%%%%%%%%%%%%%%%%%%%%%%%%%%%%%%%%%%%%%%%%%%
\section{\idx{Global Harnack Principle}}\label{ssec:refinedGHP}

In this section we show that the solution of~\eqref{FD} can be bounded from above (Proposition~\ref{GHP:UpperBound}) and from below (Proposition~\ref{GHP-3}) by two Barenblatt functions defined in~\eqref{BarenblattM}, with (slightly) different masses and with (small) shifts in time. Compared to the existing literature~\cite{Bonforte2006,Bonforte2019a,Vazquez2006}, we provide a simpler proof and explicit constants.

%%%%%%%%%%%%%%%%%%%%%%%%%%%%%%%%%%%%%%%%%%%%%%%%%%%%%%%%%%%%%%%%%%%%%%%%
\subsection{Control in terms of Barenblatt profile from above}

The profile $B(t,x;M)$ as defined by~\eqref{BarenblattM} and translated in time by a parameter $\tau\ge -1/\alpha$ is still a solution to~\eqref{FD}. In particular, we have that
\[
B\big(t-\tfrac1\alpha\,,\,x\,;\,M\big)\rightharpoonup M\,\delta_{x=0}\quad\mbox{as}\quad t\to{0}_+\,,
\]
in the sense of distributions. Such a profile can be written as
\be{BarenblattM-delta}
B\(t-\tfrac1\alpha\,,\,x\,;\,M\)=\(\tfrac M\Mstar\)^\frac2\alpha\frac{\lambdaBarenblatt^d}{t^{\frac d{\alpha}}}\,\mB\(\(\tfrac M\Mstar\)^\frac{1-m}\alpha\frac{\lambdaBarenblatt}{t^{\frac1\alpha}}\,x\)\quad\mbox{with}\quad\lambdaBarenblatt=\(\tfrac{1-m}{2\,m\,\alpha}\)^\frac1\alpha\,.
\ee

%---------------------------------------------------------------------
\begin{proposition}\label{GHP:UpperBound}
Under the assumptions of Theorem~\ref{Thm:RelativeUniform}, there exist positive constants~$\overline t$ and~$\overline M$ such that any solution $u$ satisfies
\be{GHP-1}
u(t,x)\le\,B\big(t+\overline t-\tfrac1\alpha\,,\,x\,;\,\overline M\big)\quad\forall\,(t,x)\in[\,\overline t,+\infty)\times\R^d\,.
\ee
\end{proposition}
%---------------------------------------------------------------------
\noindent The expressions of $\overline t$ and $\overline M$ are given in~\eqref{toverline} and in~\eqref{moverline} respectively. Here~$\overline M$ is a numerical constant. The reader may notice that the factor $1/\alpha$ causes no harm in the definition of the Barenblatt profile, see~\eqref{BarenblattM-delta} of Chapter~\ref{Chapter-2}.
\begin{proof}[Proof of Proposition~\ref{GHP:UpperBound}] The proof is divided in several steps, and follows the standard strategy, but here we keep track of all constants.
\begin{steps}
\stepitem \emph{A priori estimates on the solution.} By taking $R\to\infty$ in~\eqref{BV-1}, we deduce that
\be{estimate1}
u(t,x)\le\overline\kappa\,\Mstar^{\frac2\alpha}\,t^{-\frac d\alpha}\quad\forall\,(t,x)\in(0,+\infty)\times\R^d\,,
\ee
where $\overline\kappa$ is as in Lemma~\ref{Lem:LocalSmoothingEffect}.
Let us choose
\be{GHP-1-time}
t_0:=\,A^{1-m}\,,
\ee
$x_0 \neq 0$ and $R=|x_0|/4$, so that $B_R(x_0)\subset B_R^c(0)$. Using~\eqref{hyp:Harnack} and~\eqref{GHP-1-time}, we deduce from~\eqref{BV-1} that
\be{estimate2}
u(t_0,x_0)\le\overline\kappa\left(\frac{4^{\frac2{1-m}}}{t_0^{\frac d\alpha}} \frac{A^{\frac2\alpha}}{\left|x_0\right|^\frac2{1-m}} + 2^\frac4{1-m}\left(\frac{t_0}{\left|x_0\right|^2}\right)^{\frac1{1-m}}\right)\le2^{1+\frac4{1-m}}\,\frac{t_0^\frac1{1-m}}{\left|x_0\right|^\frac2{1-m}}\,\overline\kappa\,.
\ee

\stepitem \emph{Proof of~\eqref{GHP-1} at time $t_0$.} Let us define
\be{toverline}
c:=\max\big\{1, 2^{5-m}\,\overline\kappa^{1-m}\,\lambdaBarenblatt^\alpha\big\}\,,\quad\overline t:=c\,t_0\,,
\ee
and
\be{moverline}
\overline M:=2^\frac\alpha{2\,(1-m)}\,\overline\kappa^\frac\alpha2\,\(1+c\)^\frac d2\,\lambdaBarenblatt^{-\frac{d\,\alpha}2}\,\Mstar^2\,,
\ee
where $\lambdaBarenblatt$ is as in~\eqref{BarenblattM-delta}. Let us also define the auxiliary function
\[
\lambda(t):=\(\tfrac{\overline M}\Mstar\)^\frac{1-m}\alpha\lambdaBarenblatt\,t^{-\frac 1\alpha}\,\(1+c\)^{-\frac 1\alpha}\quad\mbox{so that}\quad\lambda(t_0)=\(\tfrac{\overline M}\Mstar\)^\frac{1-m}\alpha\lambdaBarenblatt\,\(t_0+\overline t\)^{-\frac 1\alpha}\,.
\]
If $\lambda(t_0)\,|x|\le1$, we deduce from~\eqref{estimate1},~\eqref{moverline} and~\eqref{BarenblattM-delta} that
\[
u(t_0,x)\le\overline\kappa\,\Mstar^{\frac2\alpha}\,t_0^{-\frac d\alpha}=\(\tfrac{\overline M}\Mstar\)^\frac2\alpha\frac{\lambdaBarenblatt^d}{2^\frac1{1-m}}\,t_0^{-\frac d\alpha}\,\(1+c\)^{-\frac d\alpha}\,\le B\big(t_0+\overline t-\tfrac1\alpha\,,\,x\,;\,\overline M\big)\,.
\]
If $\lambda(t_0)\,|x|\ge1$, we deduce from~\eqref{estimate2},~\eqref{toverline} and~\eqref{BarenblattM-delta} that
\[
u(t_0,x)\le2^{1+\frac4{1-m}}\,\frac{t_0^\frac1{1-m}}{\left|x_0\right|^\frac2{1-m}}\,\overline\kappa\le \(\frac{1+c}{2\,\lambdaBarenblatt^\alpha}\)^\frac1{1-m}\,\frac{t_0^\frac1{1-m}}{\left|x_0\right|^\frac2{1-m}}\,\le B\big(t_0+\overline t-\tfrac1\alpha\,,\,x\,;\,\overline M\big)\,.
\]

\stepitem \emph{Comparison.} Once we have obtained~\eqref{GHP-1} at time $t=t_0$, by comparison it also holds for any $t\ge t_0$. In particular~\eqref{GHP-1} holds for any $t\ge\overline t\ge t_0$, which completes the proof.
\end{steps}
\end{proof}

%%%%%%%%%%%%%%%%%%%%%%%%%%%%%%%%%%%%%%%%%%%%%%%%%%%%%%%%%%%%%%%%%%%%%%%%
\subsection{Control in terms of Barenblatt profile from below}

%---------------------------------------------------------------------
\begin{proposition}\label{GHP-3} Under the assumptions of Theorem~\ref{Thm:RelativeUniform}, there exist positive constants~$\underline t$ and~$\underline M$ such that any solution $u$ satisfies
\be{GHP-2}
u(t,x)\ge B\big(t-\underline t-\tfrac1\alpha\,,\,x\,;\,\underline M\big)\quad\forall\,(t,x)\in[2\,\underline t,+\infty)\times\R^d\,.
\ee
\end{proposition}
%---------------------------------------------------------------------
Compared to~\cite{Bonforte2006,Vazquez2006, Bonforte2019a}, the novelty here is that, again, we provide constructive estimates of the constants. The value of the constant $\underline M$ is given below by~\eqref{underlineM}: it is a numerical constant, which is independent of $u$. An upper bound on $\underline t$ is given by~\eqref{BV-3t12}. This bound depends only on $A$ and various numerical constants. As an intermediate quantity, we define $R_\star>0$ such that
\be{Rstarlower}
\int_{|x|\le R_\star}u_0\,\dx=\frac12\,\Mstar\,.
\ee
\begin{proof} Our task is essentially to keep track of the constants and we claim no originality in the strategy of the proof.
\begin{steps}

\stepitem Let $R_\star>0$ be as in~\eqref{Rstarlower}. We read from~\eqref{hyp:Harnack} that
\be{Rstar}
R_\star^\alpha\le\(\frac{2\,A}\Mstar\)^{1-m}\,.
\ee
From Lemma~\ref{Posit.Thm.FDE} we have that
\be{BV-3-2}
\inf_{|x|\le R_\star}u(t,x)\ge\kappa\(R_\star^{-2}\,t\)^\frac1{1-m}\quad\forall\,t\in[0,2\,\underline t]
\ee
for $\kappa$ and $\kappa_\star$ given in~\eqref{kappaExpr-kappastarExpr} and $\underline t$ given by
\be{BV-3t1b}
\underline t=\frac12\,\kappa_\star\,\(\frac\Mstar2\)^{1-m}\,R_\star^\alpha\,.
\ee
After taking into account~\eqref{Rstar}, we obtain
\be{BV-3t12}
\underline t\le\frac12\,\kappa_\star\,\(\frac\Mstar2\)^{1-m}\,\(\frac{2\,A}\Mstar\)^{1-m}=\frac{\kappa_\star}{2}\,A^{1-m}.
\ee

\stepitem If $|x|\le R_\star$, we deduce from~\eqref{BarenblattM-delta},~\eqref{Rstar} and~\eqref{BV-3-2} that
\[
u(2\,\underline t,x)\ge\kappa\(\frac{2\,\underline t}{R_\star^2}\)^\frac1{1-m}\ge B\big(t-\tfrac1\alpha\,,0\,;\,M\big)=\(\frac M\Mstar\)^\frac2\alpha\lambdaBarenblatt^d\,t^{-\frac d\alpha}\ge B\big(t-\tfrac1\alpha\,,x\,;\,M\big)
\]
for any $t>0$ and $M>0$ such that
\be{Mt}
M^\frac2\alpha\,t^{-\frac d\alpha}\le\frac\kappa{\lambdaBarenblatt^d}\(\frac{2\,\underline t}{R_\star^2}\)^\frac1{1-m}\,\Mstar^\frac2\alpha\,.
\ee
Let us notice that~\eqref{Mt} at $t=\underline t$ with $\underline t$ given by~\eqref{BV-3t1b} amounts to
\be{underlineMsoln1}
M\le\frac{\kappa_\star^{\frac1{1-m}}}{2^{\,\frac d2}}\(\frac\kappa{\lambdaBarenblatt^d}\)^\frac\alpha2\Mstar^2\,.
\ee
This condition is independent of $R_\star$.

\stepitem If $|x|=R_\star$, we enforce the condition that
\[
u(t+\underline t,x)\ge\kappa\(\frac{t+\underline t}{R_\star^2}\)^\frac1{1-m}\ge B\big(t-\tfrac1\alpha\,,x\,;\,M\big)
\]
for any $t\in[0,\underline t]$ by requesting that
\[
\kappa\(\frac{\underline t}{R_\star^2}\)^\frac1{1-m}\ge\frac M\Mstar\,R_\star^{-d}\,\sup_{\lambda>0}\lambda^d\,\(1+\lambda^2\)^\frac1{m-1}\ge\frac M\Mstar\,\frac{\lambda(t)^d}{R_\star^d}\(1+\lambda(t)^2\)^\frac1{m-1}
\]
where the left-hand side is the estimate of $u(\underline t,x)$ deduced from~\eqref{BV-3-2}, while the right-hand side is the value of $B(t-\tfrac1\alpha,x;M)$ for $|x|=R_\star$ and
\[
\lambda(t):=\big(\tfrac M\Mstar\big)^\frac{1-m}\alpha\lambdaBarenblatt\,t^{-\frac1\alpha}\,R_\star\,.
\]
After taking into account~\eqref{BV-3t1b}, we obtain the condition
\be{underlineMsoln2}
M\le\frac{\kappa\,\kappa_\star^\frac1{1-m}}{\(d\,(1-m)\)^{d/2}\,\alpha^\frac\alpha{2\,(1-m)}}\,\Mstar^2\,,
\ee
which is also independent of $R_\star$.

\stepitem We adapt~\cite[Lemma~3.4]{Herrero1985} as follows. We choose
\be{underlineM}
\underline M:=\min\left\{2^{-\,d/2}\,\Big(\frac\kappa{\lambdaBarenblatt^d}\Big)^{\alpha/2},\frac\kappa{\(d\,(1-m)\)^{d/2}\,\alpha^\frac\alpha{2\,(1-m)}}\right\}\;\kappa_\star^\frac1{1-m}\,\Mstar^2
\ee
so that~\eqref{underlineMsoln1} and~\eqref{underlineMsoln2} are simultaneously true. Notice that $\underline M$ is independent of~$R_\star$.

The function $\underline u(t,x):=B(t-\underline t-\tfrac1\alpha,x;\underline M)$ is such that
\be{underlineu-inner}
\underline u(2\,\underline t,x)\le u(2\,\underline t,x)\quad\mbox{if}\quad|x|\le R_\star
\ee
by Step 2,
\[\label{74b}
\underline u(t,x)\le u(t,x)\quad\mbox{if}\quad (t,x)\in(\underline t,\,2\,\underline t)\times\R^d\,,\quad|x|=R_\star
\]
by Step 3, and, in the sense of distributions,
\[
\underline u(t,\cdot)\rightharpoonup\underline M\,\delta_{x=0}\quad\mbox{as}\quad t\to{\underline t}_+\,.
\]
As a consequence, we also have that
\[
\lim_{t\to{\underline t}_+}\underline u(t,\cdot)\le u\(\underline t\,,x\)\quad\mbox{for any}\quad x\in\R^d\quad\mbox{such that}\quad|x|\ge R_\star\,.
\]
The functions $\underline u$ and $u$ solve~\eqref{FD}. By arguing as in~\cite[Lemma~3.4]{Herrero1985}, we find that
\[
\underline u(t,x)\le u(t,x)\quad\forall\,(t,x)\in[\underline t,\,2\,\underline t]\times\R^d\quad\mbox{such that}\quad|x|\ge R_\star\,.
\]
This inequality holds in particular for $t=2\,\underline t$, which can be combined with~\eqref{underlineu-inner} to prove that
\[
\underline u(2\,\underline t,x)\le u(2\,\underline t,x)\quad\forall\,x\in\R^d\,.
\]
Notice that~\cite[Lemma~3.4]{Herrero1985} holds only for smooth functions, so that an approximation scheme is needed, which is standard and will be omitted here.

\stepitem By standard comparison methods, if~\eqref{GHP-2} is true at $t=2\,\underline t$, it is also true at any $t\ge 2\,\underline t$. This completes the proof of~\eqref{GHP-2}.
\end{steps}\end{proof}

%%%%%%%%%%%%%%%%%%%%%%%%%%%%%%%%%%%%%%%%%%%%%%%%%%%%%%%%%%%%%%%%%%%%%%%%
%%%%%%%%%%%%%%%%%%%%%%%%%%%%%%%%%%%%%%%%%%%%%%%%%%%%%%%%%%%%%%%%%%%%%%%%
\section{Convergence in relative error and the threshold time \texorpdfstring{$t_\star$}{tstar}}\label{Sec:PfUniform}

This Section is devoted to the prove of the \idx{uniform convergence in relative error} and to the computation of the \idx{threshold time} $t_\star$ defined in Theorem~\ref{Thm:RelativeUniform}.

We begin by observing that the prove of Theorem~\ref{Thm:RelativeUniform} boils down to the computation of $t_\star$. Let us give some definitions
\begin{equation}\label{epsilon.md.def}
\overline\varepsilon:=\(\overline M/\Mstar\)^\frac2\alpha-1\,,\quad\underline\varepsilon:=1-\(\underline M\,/\Mstar\)^\frac2\alpha\,,\quad\mbox{and}\quad\varepsilon_{m,d}:=\min\big\{\overline\varepsilon,\,\underline\varepsilon,\,\tfrac12\big\}\,,
\end{equation}
where $\overline M$ and~$\underline M$ are defined respectively by~\eqref{moverline} and~\eqref{underlineM}.
As a byproduct of Proposition~\ref{GHP:UpperBound}, by integrating over $\R^d$, we deduce from~\eqref{GHP-1} that $\overline M\,/\Mstar>1$, which proves that $\overline\varepsilon>0$. In the same way, by integrating over $\R^d$, we deduce from~\eqref{GHP-2} that $\underline M\,/\Mstar<1$, which proves that $\underline\varepsilon>0$. We conclude that $\varepsilon_{m,d}>0$. With this definition, notice that $\underline\varepsilon, \overline\varepsilon$ and $\varepsilon_{m,d}$ are numerical constants that depend only on $d$ and $m$. We recall that $R(t)=(1+\alpha t)^\frac{1}{\alpha}$ and is defined in~\eqref{R}.

%%%%%%%%%%%%%%%%%%%%%%%%%%%%%%%%%%%%%%%%%%%%%%%%%%%%%%%%%%%%%%%%%%%%%%%%
\subsection{The outer estimates}\label{Ssec:Thm:RelativeUniformOuter}

So far, the upper and lower estimates of Propositions~\ref{GHP:UpperBound} and~\ref{GHP-3} correspond to Barenblatt functions which do not have the same mass $\Mstar$ as $u$. The present sub-section is devoted to the comparison of the solution $u$ of~\eqref{FD} with the Barenblatt function~$\mB$ of mass $\Mstar$ but up to a multiplicative factor. This comparison will be done outside of a large ball in $x$, or for large values of $t$.
With the notation of~\eqref{Barenblatt-1} and~\eqref{BarenblattM}, we recall that $B(t,x)=B(t\,,\,x\,;\,\Mstar)$.
%---------------------------------------------------------------------
\begin{corollary}\label{Cor:ControlTailsLower} Under the assumptions of Theorem~\ref{Thm:RelativeUniform} and for any $\varepsilon\in\(0,\underline\varepsilon\)$, there are some $\underline T(\varepsilon)$ and $\underline\rho(\varepsilon)$ for which any solution $u$ of~\eqref{FD} satisfies
\be{GHP-6Lower}
u(t,x)\ge(1-\varepsilon)\,B(t\,,\,x\,)\quad\mbox{if}\quad |x|\ge R(t)\,\underline\rho(\varepsilon)\quad\mbox{and}\quad t\ge\underline T(\varepsilon)\,.
\ee
Furthermore, there exists $\underline{C}>0$ such that, for all $x\in\R^d$,
\be{GHP-7Lower}
u(t,x)\ge\underline{C}\,B\big(t-\tfrac1\alpha\,,\,x\,\big)\quad\mbox{if}\quad t\ge\underline T(\varepsilon)\,.
\ee
\end{corollary}
%---------------------------------------------------------------------
\noindent The constants $\underline T(\varepsilon)$, $\underline\rho(\varepsilon)$ and $\underline{C}$ have an explicit expression which will be given below in~\eqref{t1bis},~\eqref{Rbis} and~\eqref{underline-C} respectively.
\begin{proof} The Barenblatt solution of mass $M$ as defined in~\eqref{BarenblattM} can be rewritten as
\[\label{BarenblattM2}
B(t\,,\,x\,;\,M)=\lambda(t)^d\(\(\Mstar/M\)^{2\,\frac{1-m}\alpha}+\lambda(t)^2\,|x|^2\)^\frac1{m-1}
\]
where $\lambda(t):=\muscal\,R(t)^{-1}$, $\muscal$ is a constant given by~\eqref{mu}, and $R(t)=(1+\alpha\,t)^{1/\alpha}$, so that
\[
\frac{B\big(t-\underline t-\tfrac1\alpha\,,\,x\,;\,\underline M\big)}{B(t\,,\,x\,)}
=\frac{\lambda\big(t-\underline t-\tfrac1\alpha\big)^d}{\lambda(t)^d}\(\frac{1+\lambda(t)^2\,|x|^2}{(1-\underline\varepsilon)^{m-1}+\lambda\big(t-\underline t-\tfrac1\alpha\big)^2\,|x|^2}\)^\frac1{1-m}.
\]
With
\[
\eta(t):=\(\frac{t+\tfrac1\alpha}{t-\underline t}\)^{\frac1\alpha}\quad\mbox{and}\quad s(t,x):=\lambda(t)^2\,|x|^2\,,
\]
Inequality~\eqref{GHP-6Lower} amounts to
\[
\eta^d\(\frac{1+s}{(1-\underline\varepsilon)^{m-1}+\eta^2\,s}\)^\frac1{1-m}\ge1-\varepsilon\,.
\]
It is sufficient to have
\[\label{c1}
{\eta(t)}^\alpha\,(1-\varepsilon)^{1-m}<1\quad\mbox{and}\quad s(t,x)=\(\frac{\muscal\,|x|}{R(t)}\)^2\ge\frac{{\eta(t)}^{-d\,(1-m)}\,\big(\frac{1-\varepsilon}{1-\underline\varepsilon}\big)^{1-m}-1}{1-{\eta(t)}^\alpha\,(1-\varepsilon)^{1-m}}\,.
\]
Using~\eqref{BV-3t12}, the first condition is satisfied if $t\ge\underline T(\varepsilon)$ with
\be{t1bis}
\tau(\varepsilon):=\frac{2\,\underline t + \tfrac1\alpha\big(1+(1-\varepsilon)^{1-m}\big)}{1-(1-\varepsilon)^{1-m}} \le\frac{\kappa_\star\(2\,A\)^{1-m}+\frac2\alpha}{1-(1-\varepsilon)^{1-m}} =: \underline T(\varepsilon)\,.
\ee
Notice that $\underline T(\varepsilon)=O(1/\varepsilon)$ as $\varepsilon\to0$ and also that Condition~\eqref{t1bis} guarantees that $\underline T(\varepsilon)\ge2\,\underline t$. Next, using
\[
1\le\eta(t)\le\eta\big(\underline T(\varepsilon)\big)\le\frac{2^\frac1\alpha}{\big(1+(1-\varepsilon)^{1-m}\big)^\frac1\alpha}=\eta\big(\tau(\varepsilon)\big)
\]
for any $t\ge\underline T(\varepsilon)$, the second condition follows from $s(t,x)\ge\muscal^2\,\underline\rho^2(\varepsilon)$ with
\be{Rbis}
\underline\rho(\varepsilon):=\frac1\muscal\(\big(1+(1+\varepsilon)^{1-m}\big)\,\frac{\(\tfrac{1-\varepsilon}{1-\underline\varepsilon}\)^{1-m}-1}{1-(1-\varepsilon)^{1-m}}\)^{1/2}\,.
\ee
It follows from a Taylor expansion that $\underline\rho(\varepsilon)=O(1/\sqrt\varepsilon)$ as $\varepsilon\to0$.

With the above notation, we remark that
\[
\frac{B\big(t-\underline t-\tfrac1\alpha\,,\,x\,;\,\underline M\big)}{B\big(t-\tfrac1\alpha\,,\,x\,\big)}=\gamma(t)^d\,\(\frac{1+\sigma}{(1-\underline\varepsilon)^{m-1}+\gamma^2\,\sigma}\)^\frac1{1-m}
\]
where
\[
\gamma(t)=\left(\frac{t}{t-\underline t}\right)^\frac1\alpha\quad\mbox{and}\quad\sigma=\lambda\(t-\tfrac1\alpha\)^2\,|x|^2\,.
\]
Since $\gamma>1$, inequality~\eqref{GHP-7Lower} amounts to find
\[
\inf_{\sigma\ge0}\(\frac{1+\sigma}{(1-\underline\varepsilon)^{m-1}+\gamma^2\,\sigma}\)^\frac1{1-m}\,.
\]
A straightforward computation shows that such an infimum is achieved either at~$0$ or at infinity. Since for any $t\ge\underline T(\varepsilon)\ge 2\,\underline t$ we have that $\gamma(t)\le\gamma(2\,\underline t)=2^\frac1\alpha $, we obtain
\be{underline-C}
\inf_{\sigma\ge0}\(\frac{1+\sigma}{(1-\underline\varepsilon)^{m-1}+\gamma^2\,\sigma}\)^\frac1{1-m}\ge \min\left\{(1-\underline\varepsilon)\,, \gamma^{-\frac2{1-m}}\right\} \ge \frac{1-\underline\varepsilon}{2^{\frac2{(1-m)\,\alpha} }}=:\underline{C}\,,
\ee
where we have used that $\underline\varepsilon<1$ and $2^\frac2{(1-m)\,\alpha}>1$. The proof is completed.
\end{proof}

Next we prove lower bounds. We recall that $R(t)=(1+\alpha t)^\frac{1}{\alpha}$ is as in~\eqref{R}.
%---------------------------------------------------------------------
\begin{corollary}\label{Cor:ControlTailsUpper} Under the assumptions of Theorem~\ref{Thm:RelativeUniform} and for any $\varepsilon\in\(0,\overline\varepsilon\)$, there are some $\overline T(\varepsilon)$ and $\overline\rho(\varepsilon)$ for which any solution $u$ of~\eqref{FD} satisfies
\be{GHP-6Upper}
u(t,x)\le(1+\varepsilon)\,B(t\,,\,x\,)\quad\mbox{if}\quad |x|\ge\,R(t)\,\overline\rho(\varepsilon)\quad\mbox{and}\quad t\ge\overline T(\varepsilon)\,.
\ee
Furthermore, there exists $\overline C>0$ such that, for all $x\in\R^d$,
\[\label{GHP-7Upper}
u(t,x)\le\overline C\,B\big(t-\tfrac1\alpha\,,\,x\,\big)\quad\mbox{if}\quad t\ge\overline T(\varepsilon)\,.
\]
\end{corollary}
%---------------------------------------------------------------------
\noindent The constants $\overline T(\varepsilon)$, $\overline\rho(\varepsilon)$ and $\overline C$ have an explicit expression given below in~\eqref{Tter}, \eqref{Rter} and~\eqref{overline-C} respectively.
\begin{proof} The beginning of the proof is the same as for Corollary~\ref{Cor:ControlTailsLower}. With the same notation except for $\eta$ which is now defined by
\[
\eta(t):=\(\frac{\tfrac1\alpha+t}{t+\overline t}\)^{\frac1\alpha}\,,
\]
where $\overline t$ is as in~\eqref{toverline}, Inequality~\eqref{GHP-6Upper} amounts to
\[
\eta^d\(\frac{1+s}{(1+\overline\varepsilon)^{m-1}+\eta^2\,s}\)^\frac1{1-m}\le1+\varepsilon\,.
\]
To prove the above inequality it is sufficient to have
\be{condition-upper}
{\eta(t)}^\alpha\,(1+\varepsilon)^{1-m}>1\quad\mbox{and}\quad s(t,x)=\(\frac{\muscal\,|x|}{R(t)}\)^2\ge\frac{{1-\eta(t)}^{-d\,(1-m)}\,\big(\frac{1+\varepsilon}{1+\overline\varepsilon}\big)^{1-m}}{{\eta(t)}^\alpha\,(1+\varepsilon)^{1-m}-1}\,.
\ee
where  $R(t)=(1+\alpha t)^\frac{1}{\alpha}$ is as in~\eqref{R}. Let us define
\be{Tter}
\overline T(\varepsilon):=\frac{2\,\overline t}{(1+\varepsilon)^{1-m}-1}
\ee
where $\overline t$ is as in~\eqref{toverline} and $\overline T(\varepsilon)=O(1/\varepsilon)$ as $\varepsilon\to0$ follows from a Taylor expansion. If $\overline t < 1/\alpha$ then the first condition in~\eqref{condition-upper} is always satisfied, while in the case $\overline t \ge 1/\alpha$, we need to ask that $t\ge\overline T(\varepsilon)$. In both cases we have that
\[
\eta(t)\ge\left(\frac2{(1+\varepsilon)^{1-m}+1}\right)^\frac1\alpha
\]
for any $t\ge\overline T(\varepsilon)$. As a consequence, a sufficient condition the second inequality in~\eqref{condition-upper} is $s(t,x)\ge\muscal^2\,\overline\rho^2(\varepsilon)$ with
\be{Rter}
\overline\rho(\varepsilon):=\frac1\muscal\left(\,\frac{(1+\varepsilon)^{1-m}+1}{(1+\varepsilon)^{1-m}-1}\right)^\frac12\,.
\ee
It follows from a second order Taylor expansion that $\overline\rho(\varepsilon)=O(1/\sqrt\varepsilon)$ as $\varepsilon\to0$.

As in Corollary~\ref{Cor:ControlTailsLower}, we remark that
\[
\frac{B\big(t+\overline t-\tfrac1\alpha\,,\,x\,;\,\underline M\big)}{B\big(t-\tfrac1\alpha\,,\,x\,\big)}=\gamma(t)^d\,\(\frac{1+\sigma}{(1+\overline\varepsilon)^{m-1}+\gamma^2\,\sigma}\)^\frac1{1-m}
\]
where
\[
\gamma(t)=\left(\frac{t}{t+\overline t}\right)^\frac1\alpha\quad\mbox{and}\quad\sigma=\lambda\(t-\tfrac1\alpha\)^2\,|x|^2\,.
\]
Since $\gamma(t)\le1$, inequality~\eqref{GHP-7Lower} amounts to find
\[
\sup_{\sigma\ge0}\(\frac{1+\sigma}{(1+\overline\varepsilon)^{m-1}+\gamma^2\,\sigma}\)^\frac1{1-m}\,.
\]
A straightforward computation shows that such infimum is achieved either at $0$ or at infinity. Since $\gamma(t)\ge (2/3)^{1/\alpha}$ for any $t\ge \overline T(\varepsilon)\ge 2\,\overline t$, we can argue that
\be{overline-C}
\sup_{\sigma\ge0}\(\frac{1+\sigma}{(1+\overline\varepsilon)^{m-1}+\gamma^2\,\sigma}\)^\frac1{1-m}\kern-4pt\le\max\left\{(1+\overline\varepsilon)\,, \gamma^{-\frac2{1-m}}\right\}\le \(1+\overline\varepsilon\)\(\tfrac32\)^\frac2{(1-m)\,\alpha} =:\overline C\,.
\ee
The proof is completed.
\end{proof}

%%%%%%%%%%%%%%%%%%%%%%%%%%%%%%%%%%%%%%%%%%%%%%%%%%%%%%%%%%%%%%%%%%%%%%%%
\subsection{The inner estimate}\label{Ssec:Thm:RelativeUniformInner}

Here we prove the \idx{uniform convergence in relative error} inside a finite ball. Let us recall that
\begin{equation*}
\varepsilon_{m,d}:=\min\big\{\overline\varepsilon,\,\underline\varepsilon,\,\tfrac12\big\}
\end{equation*}
where, as in Section~\ref{Ssec:Thm:RelativeUniformOuter}, $\overline\varepsilon=\(\overline M/\Mstar\)^\frac2\alpha-1$ and $\underline\varepsilon=1-\(\underline M\,/\Mstar\)^\frac2\alpha$ are given in terms of $\overline M$ and~$\underline M$ defined respectively by~\eqref{moverline} and~\eqref{underlineM}. For any $\varepsilon\in(0,\varepsilon_{m,d})$, let us define
\be{RE}
\rho(\varepsilon):=\max\big\{\overline\rho(\varepsilon),\,\underline\rho(\varepsilon)\big\}\quad\mbox{and}\quad T(\varepsilon):=\max\big\{\overline T(\varepsilon),\,\underline T(\varepsilon)\big\}
\ee
where $\overline\rho(\varepsilon)$, $\underline\rho(\varepsilon)$, $\overline T(\varepsilon)$, and $\underline T(\varepsilon)$ are defined by~\eqref{t1bis},~\eqref{Rbis},~\eqref{Tter}, and~\eqref{Rter}. We know that $\rho(\varepsilon)=O(1/\sqrt\varepsilon)$ and $T(\varepsilon)=O(1/\varepsilon)$ as $\varepsilon\to0$. The main result of this sub-section is the following. Recall that $R(t)=(1+\alpha t)^\frac{1}{\alpha}$ is as in~\eqref{R}.
%---------------------------------------------------------------------
\begin{proposition}\label{Thm:ControlRadius} Under the assumptions of Theorem~\ref{Thm:RelativeUniform}, there exist a numerical constant $\mathsf K>0$ and an exponent $\vartheta\in(0,1)$ such that, for any $\varepsilon\in(0, \varepsilon_{m,d})$ and for any $t\ge 4\,T(\varepsilon)$, any solution $u$ of~\eqref{FD} satisfies
\be{control-radius-inequality2}
\left|\frac{u(t,x)}{B(t,x)}-1\right|\le\frac{\mathsf K}{\varepsilon^\frac1{1-m}}\,\left(\frac1t+\frac{\sqrt G}{R(t)}\right)^\vartheta\quad\mbox{if}\quad |x|\le 2\,\rho(\varepsilon)\,R(t)\,.
\ee
\end{proposition}
%---------------------------------------------------------------------
\noindent The exponent is $\vartheta=\nu/(d+\nu)$ and the numerical constants $\nu=\nu(m,d)$ and $\mathsf K(m,d)$ are explicit and given below in~\eqref{chapter4-nu} and~\eqref{constant-c-control-radius.suppl.1} respectively.
\begin{proof}
\begin{steps}
By the triangle inequality, the left-hand-side of~\eqref{control-radius-inequality2} can be estimated by
\be{inequality-f-1}\begin{split}
\left|\frac{u(t,x)}{B(t,x)}-1\right| & \le\,\left|\frac{B\big(t-\tfrac1\alpha,x\big)}{B(t,x)}\right|\\
&\quad\times\,\left(\left|\frac{u(t,x)}{B(t-\tfrac1\alpha\,,\,x\,)}-1\right|+\left|\frac{B(t,x)}{B(t-\tfrac1\alpha\,,\,x\,)}-1\right|\right)\,.
\end{split}\ee
The supremum of the quotient $B(t-1/\alpha,x)/B(t,x)$ is achieved at $x=0$ for any $t\ge0$. Using~\eqref{Barenblatt-1} and~\eqref{BarenblattM-delta}, we have that
\[\label{c-1-inequality}
\left\|\frac{B\big(t-\tfrac1\alpha,x\big)}{B(t,x)}\right\|_{\mathrm L^\infty(\R^d)} \le \frac{R(t)^d}{\alpha^\frac d{\alpha} t^\frac d{\alpha}}=:\overline{c_1}(t)\,.
\]
The supremum of the quotient $B(t,x)/B(t-1/\alpha,x)$ is achieved at infinity, therefore a simple computation shows that
\begin{equation*}\begin{split}
\left\|\frac{B(t,x)}{B(t-\tfrac1\alpha\,,\,x\,)}-1\right\|_{\mathrm L^\infty(\R^d)} = \left(1+\frac1{\alpha\,t}\right)^\frac1{1-m}-1\le\frac{\overline c_3}{t}+\frac{\overline c_2}{t^2}\,.
\end{split}\end{equation*}
A Taylor expansion shows that the values of $\overline c_2$ and $\overline c_3$ are
\begin{equation*}
\overline c_3=\frac1{1-m}\quad\mbox{and}\quad\overline c_2=\frac{m}{2\,(1-m)^2\,\alpha^2}\,.
\end{equation*}
Our task is to estimate the missing term $|{u(t,x)}/{B(t-1/\alpha\,,\,x\,)}-1|$. This is done by interpolating the above quantity between its $\mathrm L^p$ and $C^\nu$ norms, by means of inequality~\eqref{interpolation.inequality.cpt6}, in Section~\ref{Appendix:GagliardoCNu}. In order to do so, we use parabolic regularity theory to estimate the $C^\nu$ norm of the quotient $u(t,x)/B(t-1/\alpha\,,\,x\,)$.

\stepitem We recall some elements of linear parabolic regularity theory. Let us define the cylinders
\begin{align*}\label{cilinders}
Q_1&:=\left(1/2, 3/2\right)\times B_{1}(0)\,,\quad Q_2:=\left(1/4, 2\right)\times B_{8}(0)\,,\\
Q_3&:=\left(1/2, 3/2\right)\times B_{1}(0)\setminus B_{1/2}(0)\,\quad\mbox{and}\quad Q_4:=\left(1/4, 2\right)\times B_{8}(0)\setminus B_{1/4}(0)\,.
\end{align*}
By Theorem~\ref{Claim:4} of Chapter~\ref{Chapter-3} any nonnegative weak solution to~\eqref{HE.coeff} defined on~$Q_2$ satisfies the following inequality 
\be{chapter4-holder-continuity-inequality}
\sup_{(t,x),(s,y)\in Q_{i}}\frac{|v(t,x)-v(s,y)|}{\big(|x-y|+|t-s|^{1/2}\big)^\nu}\le\,2\(\frac{128}{dist(Q_i, Q_{i+1})}\)^\nu\,\|v\|_{\mathrm L^\infty(Q_{i+1})}
\ee
where
\be{chapter4-nu}
i\in\{1,2\}\,,\quad
\nu:=\log_4\(\frac{\overline{\mathsf h}}{\overline{\mathsf h}-1}\) \in (0,1)\,\quad\mbox{and}\quad\overline{\mathsf h}:=\mathsf h^{\lambda_1+1/\lambda_0}\,.
\ee
 Notice that $\nu\ge1/\overline{\mathsf h}$
\[
\nu\ge \frac1{\mathsf h^{ \lambda_1+\lambda_0^{-1} }}
\]
is a positive number, which only depends only on $d$, through $\mathsf h$ as defined in~\eqref{h}, and on $\lambda_0$, $\lambda_1$. We remark that a solution $u(t,x)$ to~\eqref{FD} is, also, a solution to the linear equation~\eqref{HE.coeff} with coefficients
\[
a(t,x)=m\,u^{m-1}(t,x)\,,\quad A(t,x)=a(t,x)\,\mathrm{Id}\,,
\]
where $\mathrm{Id}$ is the identity matrix on $\R^d\times \R^d$. In what follows, we apply this linear theory to nonlinear equations by choosing $\lambda_0$ and $\lambda_1$ appropriately, depending only on $m$ and $d$: see below~\eqref{lambdas} and equality~\eqref{chapter4-holder-continuity-inequality} we deduce that a nonnegative weak solution to~\eqref{HE.coeff} defined on $Q_2$ satisfies, for any $s\in(1/2,3/2)$
\be{inequality-C-alpha-solution}
\max\{\lfloor v(s, \cdot)\rfloor_{C^\nu\left(B_{1}(0)\right)}, \lfloor v(s, \cdot)\rfloor_{C^\nu\left(B_{1}(0)\setminus B_{1/2}(0)\right)}\} \le\,c_1\,\|v\|_{\mathrm L^\infty(Q_2)}
\ee
where $\lfloor \cdot \rfloor_{C^\nu(\Omega)}$ is defined in~\eqref{C-alpha-norms} and $c_1=2^{10}$. In this first step we show how to compute $c_1$. By estimates~\eqref{chapter4-holder-continuity-inequality} applying it is clear that the only ingredient needed is to estimate from below $dist(Q_1,Q_2)$ and $dist(Q_3,Q_4)$, where $dist(\cdot, \cdot)$ is defined in~\eqref{parabolic-distance}. Let us consider the case of $dist(Q_1, Q_2)$. By symmetry, it is clear that the infimum in~\eqref{parabolic-distance} is achieved by a couple of points $(t,x)\in\overline{Q_1}$, $(s, y)\in\partial Q_2$ such that either $|x|=1, t\in\left(1/2, 3/2\right)$ and $|y|=8, s\in\left(1/2, 3/2\right)$ or $t=1/2, y=1/4$ and $x, y\in B_1$. In both cases we have that $dist(Q_1, Q_2)=|x-y|+|t-s|^\frac12\ge1/4$. By a very similar argument we can also conclude that $dist(Q_3, Q_4)\ge1/4$. Therefore, we conclude that, in both cases, $c_1$ can be taken (accordingly to inequality~\eqref{chapter4-holder-continuity-inequality}).
\be{c-1}
2\(128\)^\nu \max\left\{\frac1{dist(Q_1, Q_2)}\,,\frac1{dist(Q_3, Q_4)}\,\right\}^\nu\le 2\(512\)^\nu \le 2^{10}=:c_1\,,
\ee
where we have used the fact that $\nu\in(0,1)$.
\stepitem We estimate the $C^\nu$ norm of $u(t,x)$. For any $k>0$ and $ \tau>0$, let us define the re-scaled function
\be{scaling}
\hat u_{\tau, k}(t,x):=k^\frac{2}{1-m}\,\tau^\frac d\alpha\,u\(\tau\,t, k\,\tau^\frac1\alpha\,x\)\,.
\ee
The function $\hat u_{\tau, k}$ solves~\eqref{FD}. Similarly, the Barenblatt profile $B$ as defined in~\eqref{BarenblattM-delta} is rescaled according to
\[
\hat{B}_{\tau, k}\(t-\tfrac1\alpha,x\)=B\(t-\tfrac1\alpha,x; k^\frac\alpha{1-m}\,\Mstar\)\,.
\]
In this step we obtain estimates for the $C^\nu$-norm of $\hat u_{\tau, 1}(1,\cdot)$ and $B(1-\tfrac1\alpha,\cdot)$. Let us begin with the latter: for any $\gamma\in(0,1)$, we have
\begin{align}\label{gamma-norm-barenblatt}
\lfloor B(1-\tfrac1\alpha,x)\rfloor_{C^\gamma\left(\R^d\right)}
&\le 2 \max\Big\{\|B(1-\tfrac1\alpha,\cdot)\|_{\mathrm L^\infty(\R^d)},\|\nabla B(1-\tfrac1\alpha,\cdot)\|_{\mathrm L^\infty(\R^d)}\Big\} \nonumber\\
&=2\,\lambdaBarenblatt\,\max\Big\{1\,,2^\frac{3-2\,m}{1-m}\frac{\lambdaBarenblatt^d\,(2-m)^{2-m}}{\sqrt{1-m}\,(3-m)^\frac{5-3\,m}{2\,(1-m)}}\Big\}=:c_2\,,
\end{align}
where $\lambdaBarenblatt$ is as in~\eqref{BarenblattM-delta} and $\|\nabla B(1-\tfrac1\alpha,\cdot)\|_{\mathrm L^\infty(\R^d)}$ can be estimated as
\begin{multline*}
\|\nabla B(1-\tfrac1\alpha,\cdot)\|_{\mathrm L^\infty(\R^d)}=\(\tfrac\muscal{\alpha^{1/\alpha}}\)^{d+1}\,\sup_{z>0}\tfrac{2\,z}{1-m}\(1+z^2\)^{-2\,\frac{2-m}{1-m}}\\
=\frac{\muscal^{d+1}}{\alpha^\frac{d+1}\alpha}\,\tfrac{2^\frac1{m-1}}{\sqrt{(1-m)(3-m)}}\(\tfrac{3-m}{2-m}\)^\frac{2-m}{1-m}\,.
\end{multline*}
 By the results of Corollaries~\ref{Cor:ControlTailsLower} and~\ref{Cor:ControlTailsUpper}, there exist positive constants $\underline{C}$ and $\overline C$ such that, for all $x\in\R^d$, all $t\ge T(\varepsilon)/\tau$ and all $k \ge 1$,
\be{ghp-rescaled}
0<\underline{C} \le \frac{\hat u_{\tau , k}(t,x)}{B(t-\tfrac1\alpha,\,x;\,k^\frac\alpha{1-m}\,\Mstar)} \le\overline{C}<\infty\,,
\ee
where the expressions of $\underline{C}$ and $\overline C$ are given in~\eqref{underline-C} and in~\eqref{overline-C} respectively, and depend only on $m$ and $d$. Let us define
\be{lambdas}\begin{split}
\lambda_0^{\frac1{m-1}}&:= m^\frac1{m-1}\,\overline C\,\max\big\{\sup_{Q_2}B\big(t-\tfrac1\alpha,x\big),\,\sup_{k\ge1}\sup_{Q_4}B(t-\tfrac1\alpha,x;k^\frac\alpha{1-m}\,\Mstar)\big\}\,,\\
\lambda_1^{\frac1{m-1}}&:=\,m^\frac1{m-1}\,\underline{C}\,\min\big\{\inf_{Q_2}B\big(t-\tfrac1\alpha,x\big)\,, \inf_{k\ge1}\inf_{Q_4}B(t-\tfrac1\alpha,x;k^\frac\alpha{1-m}\,\Mstar)\big\}\,.
\end{split}
\ee
We remark that
\[
\sup_{k\ge1}\sup_{Q_4}B\big(t-\tfrac1\alpha,x;k^\frac\alpha{1-m}\,\Mstar\big)\quad\mbox{and}\quad\inf_{k\ge1}\inf_{Q_4}B\big(t-\tfrac1\alpha,x;k^\frac\alpha{1-m}\,\Mstar\big)
\]
are bounded and bounded away from zero. Indeed, let us consider first the case of $B(t-\tfrac1\alpha,x)$: for any $(t,x)\in(0, \infty)\times \R^d$, we have that
\[
B(t-\tfrac1\alpha,x)=\frac{t^\frac1{1-m}}{\lambdaBarenblatt^\frac\alpha{1-m}}\(\frac{t^\frac2\alpha}{\lambdaBarenblatt^2}+|x|^2\)^\frac1{m-1}\mbox{where}\quad\lambdaBarenblatt=\(\tfrac{1-m}{2\,m\,\alpha}\)^\frac1\alpha\,.
\]
We deduce therefore that, for any $(t,x)\in Q_2$, we have that
\[
\frac1{4^\frac1{1-m}\,\lambdaBarenblatt^\frac{\alpha}{1-m}}\(\frac{2^\frac2\alpha}{\lambdaBarenblatt^2}+2^6\)^\frac1{m-1}\le B(t-\tfrac1\alpha,x)\le \lambdaBarenblatt^d\,4^\frac d\alpha\,.
\]
This is enough to prove that $\lambda_0>0$ and $\lambda_1<\infty$. Let us consider $B(t-\tfrac1\alpha,x;k^\frac\alpha{1-m}\,\Mstar)$, we recall that
\[
B(t-\tfrac1\alpha,x;k^\frac\alpha{1-m}\,\Mstar)=\frac{t^\frac1{1-m}}{\lambdaBarenblatt^\frac{\alpha}{1-m}}\(\frac{t^\frac2\alpha}{k^2\,\lambdaBarenblatt^2}+|x|^2\)^\frac1{m-1}\,.
\]
Let us consider $(t,x)\in Q_4$, we have therefore
\begin{multline*}
\frac1{4^\frac1{1-m}\,\lambdaBarenblatt^\frac{\alpha}{1-m}}\(\frac{2^\frac2\alpha}{k^2\,\lambdaBarenblatt^2}+64\)^\frac1{m-1} \le B(t-\tfrac1\alpha,x;k^\frac\alpha{1-m}\,\Mstar)\\
 \le\frac{2^\frac1{1-m}}{\lambdaBarenblatt^\frac{\alpha}{1-m}}\(\frac1{\lambdaBarenblatt^2\,k^2\,4^\frac2\alpha}+\frac1{16}\)^\frac1{m-1}\,.
\end{multline*}
From the above computation we deduce that
\[
\sup_{k\ge1}\sup_{Q_4}B(t-\tfrac1\alpha,x;k^\frac\alpha{1-m}\,\Mstar)\big\}\le\frac{2^\frac1{1-m}}{\lambdaBarenblatt^\frac{\alpha}{1-m}}\(\frac1{\lambdaBarenblatt^2\,4^\frac2\alpha}+\frac1{16}\)^\frac1{m-1}\,,
\]
while
\[
\frac1{2^\frac7{1-m}\,\lambdaBarenblatt^\frac{\alpha}{1-m}}\le \inf_{k\ge1}\inf_{Q_4}B(t-\tfrac1\alpha,x;k^\frac\alpha{1-m}\,\Mstar)\,.
\]
Combining all estimates together we obtain that $0<\lambda_0\le \lambda_1<\infty$.

As a consequence of~\eqref{ghp-rescaled} we obtain that, for any $\tau \ge 4\,T(\varepsilon)$ and for any $k\ge1$, we have
\[
\left(\tfrac{\lambda_1}{m}\right)^\frac1{m-1} \le \hat u_{\tau,k}(t,x) \le \left(\tfrac{\lambda_0}{m}\right)^\frac1{m-1}\quad\forall\,(t,x)\in\,Q_2\,,\,Q_4\,.
\]
The function $\hat u_{\tau,k}$ is a nonnegative weak solution to~\eqref{HE.coeff} and, as a consequence of inequality~\eqref{inequality-C-alpha-solution}, we get that for any $\tau \ge 4\,T(\varepsilon)$,
\be{holder-estimate-rescaled}\begin{split}
&\lfloor \hat u_{\tau,1}(1,\cdot)\rfloor_{C^\nu\left(B_1(0)\right)}\le c_1\,\|\hat u_{\tau, 1}\|_{\mathrm L^\infty(Q_2)}\,,\\
&\lfloor \hat u_{\tau,k}(1,\cdot)\rfloor_{C^\nu\(B_1(0)\setminus B_{1/2}(0)\)} \le c_1\,\|\hat u_{\tau, k}\|_{\mathrm L^\infty(Q_4)}\quad\forall\,k\ge2\,,
\end{split}\ee
where $\nu$ is as in~\eqref{chapter4-nu} and $\lambda_0$, $\lambda_1$ are as in~\eqref{lambdas}. We observe that
\be{inequality-bigger-ball}\begin{split}
&\lfloor \hat u_{\tau,k}(1,\cdot)\rfloor_{C^\nu\(B_1(0)\setminus B_{1/2}(0)\)} = k^{\frac{2}{1-m}+\nu}\,\lfloor \hat u_{\tau,1}(1,\cdot)\rfloor_{C^\nu\(B_k(0)\setminus B_{k/2}(0)\)}\,,\\
&\|\hat u_{\tau,k}\|_{\mathrm L^\infty(Q_4)}= k^\frac{2}{1-m}\,\|\hat u_{\tau\,,1}\|_{\mathrm L^\infty\left([1/4\,,2]\times B_k(0)\setminus B_{k/2}(0)\right)}\le k^\frac{2}{1-m}\,\|\hat u_{\tau,1}\|_{\mathrm L^\infty\left([1/4\,,2]\times\R^d\right)}\,.
\end{split}
\ee
We finally estimate the $C^\nu$-norm of $\hat u_{\tau, 1}(1,\cdot)$, combining~\eqref{holder-estimate-rescaled} with~\eqref{inequality-bigger-ball} we get
\[\label{c-norm-u}\begin{split}
\lfloor \hat u_{\tau\,,1}(1,\cdot)\rfloor_{C^\nu\(\R^d \)}&\le \lfloor \hat u_{\tau\,,1}(1,\cdot)\rfloor_{C^\nu\(B_1(0)\)}+\sum_{j=0}^\infty\lfloor \hat u_{\tau\,,1}(1,\cdot)\rfloor_{C^\nu\(B_{2^{j+1}}(0)\setminus B_{2^j}(0)\)}\\
&\le c_1\,\|\hat u_{\tau\,,1}\|_{\mathrm L^\infty\left([1/4\,,2]\times\R^d\right)}\,\tfrac{2^\nu}{2^\nu-1}\,.
\end{split}\]

Lastly, we notice that, as a consequence of~\eqref{scaling} and inequality~\eqref{BV-1} where we take the limit $R\to \infty$, we have
\be{linfinity-rescaled}
\|\hat u_{\tau\,,1}\|_{\mathrm L^\infty\left([1/4\,,2]\times\R^d\right)} \le \tau^\frac d\alpha\,\|u\|_{\mathrm L^\infty([1/4\,,2]\times\R^d)} \le \tau^\frac d\alpha\,\overline\kappa\,\frac{4^\frac d\alpha\,\Mstar^\frac2\alpha}{\tau^\frac d\alpha}=4^\frac d\alpha\,\overline\kappa\,\Mstar^\frac2\alpha\,.
\ee

\stepitem In this step we shall show that for any $t\ge 4\,T(\varepsilon)$, the following inequality
\be{control-radius-inequality}
\left|\frac{u(t,x)}{B(t-\tfrac1\alpha\,,\,x\,)}-1\right|\le C\,\|u(t,x) - B(t-\tfrac1\alpha\,,\,x\,)\|_{\mathrm L^1(\R^d)}^{\vartheta}\quad\mbox{if}\quad |x|\le 2\,Z\,\rho(\varepsilon)\,t^{\frac1\alpha}
\ee
holds for any $Z\ge1$, with $C$ as in~\eqref{constant-radius-control} and
\be{theta}
\vartheta=\frac{\nu}{d+\nu}\,.
\ee
Let us define
\[\label{constant-1}
\mathsf{C}:=C_{d, \nu, 1} \left(\(c_1\,4^\frac d\alpha\,\overline\kappa\,\Mstar^\frac2\alpha\,\frac{2^\nu}{2^\nu-1}+c_2\)^\frac{d}{d+\nu}+\frac1{(2\,Z\,\rho(\varepsilon))^d}\,(2\,\Mstar)^\frac d{d+\nu}\right)\,.
\]
By inequalities~\eqref{gamma-norm-barenblatt},~\eqref{holder-estimate-rescaled} -~\eqref{linfinity-rescaled} and~\eqref{interpolation.inequality.cpt6} we deduce that for any $\tau \ge 4\,T(\varepsilon)$
\be{interpolation}\begin{split}
\|\hat u_{\tau\,,1}(1,x)-\hat{B}_{\tau\,,1}(1-\tfrac1\alpha,x)\|_{\mathrm L^\infty(B_{2\,Z\,\rho(\varepsilon)})}\le\mathsf{C}\,\|\hat u_{\tau\,,1}(1,x)-\hat{B}_{\tau\,,1}(1-\tfrac1\alpha,x)\|_{\mathrm L^1(\R^d)}^{\vartheta}\,,
\end{split}\ee
where $C_{d, \nu, 1}$ is as in~\eqref{interpolation.inequality.cpt6} and $\vartheta$ as in~\eqref{theta}.
Let us define
\be{constant-radius-control}
C:=\lambdaBarenblatt^d\left(1+4\,\lambdaBarenblatt^2\,Z^2\,\rho(\varepsilon)^2\right)^\frac1{1-m}\,\mathsf{C}=\left\|\frac1{\hat{B}_\tau(1-\tfrac1\alpha,\cdot)}\right\|_{\mathrm L^\infty(B_{2\,Z\,\rho(\varepsilon)})}\,\mathsf{C}\,.
\ee
From inequality~\eqref{interpolation}, we deduce that, for any $x\in\R^d$ such that $|x|\le 2\,Z\,\rho(\varepsilon) $,
\be{radius-control}
\Big|\frac{\hat u_{\tau\,,1}(1,x)-\hat{B}_{\tau\,,1}(1-\tfrac1\alpha,x)}{\hat{B}_{\tau\,,1}(1-\tfrac1\alpha,x)}\Big| \le\,C\,\nrm{\hat u_{\tau\,,1}(1,\cdot)-\hat{B}_{\tau\,,1}(1-\tfrac1\alpha,\cdot)}{\mathrm L^1(\R^d)}^\vartheta\,.
\ee
Let us define $y=\tau^\frac1\alpha\,x$. By using~\eqref{scaling}, we can see that the left-hand-side of~\eqref{radius-control} is as the left-hand-side of~\eqref{control-radius-inequality}, indeed we can write
\[
\left|\frac{u(\tau,y)}{B(\tau-\tfrac1\alpha\,,\,y\,)}-1\right|=\Big|\frac{\hat u_{\tau\,,1}(1,x)-\hat{B}_{\tau\,,1}(1-\tfrac1\alpha,x)}{\hat{B}_{\tau\,,1}(1-\tfrac1\alpha,x)}\Big|\,.
\]
The same holds for the right-hand-sides of those inequalities, indeed from~\eqref{scaling} we deduce that
\[
\nrm{\hat u_{\tau\,,1}(1,\cdot)-\hat{B}_{\tau\,,1}(1-\tfrac1\alpha,\cdot)}{\mathrm L^1(\R^d)}=\nrm{u(\tau,\cdot)-B(\tau-\tfrac1\alpha,\cdot)}{\mathrm L^1(\R^d)}\,.
\]
Combining the above observation we deduce that inequality~\eqref{control-radius-inequality} holds.
\stepitem For $t\ge T (\varepsilon)\ge2/\alpha$ we have that $R(t)\le \(2\,\alpha\)^\frac1\alpha t^{1/\alpha}$. By combining~\eqref{control-radius-inequality} (where we set $Z = \(2\,\alpha\)^{1/\alpha}$) and~\eqref{inequality-f-1} (where we have estimated $\overline{c}_1(t)\le 2^{d/\alpha}$) we find that for any $t\ge T(\varepsilon)\ge2/\alpha$, we have that
\[\label{inequality-to-BDNS-Supp}
\left|\frac{u(t,x)}{B(t,x)}-1\right| \le 2^\frac d\alpha\,\left(C\,\nrm{u(t,\cdot)-B(t-\tfrac1\alpha,\cdot)}1^\vartheta+\(\overline c_3+\tfrac2\alpha\,\overline c_2\)\frac1t\right)
\]
if $|x|\le2\,R(t)\,\rho(\varepsilon)$. By the triangle inequality and the above estimate on the quotients of two delayed Barenblatt solutions we obtain that for any $t\ge T(\varepsilon)$
\[
\nrm{u(t,x)-B\big(t-\tfrac1\alpha,x\big)}1\le\nrm{u(t,x)-B(t,x)}1 + \left(\overline c_3+\tfrac2\alpha\,\overline c_2\right)\frac{\Mstar}{t}\,.
\]

We take advantage of the Csisz\'{a}r-Kullback inequality~\eqref{CKm} of Section~\ref{Appendix:CK} in Chapter~\ref{Chapter-2}, namely
\[\label{CK}
\nrm{u(t,x)-B(t,x)}1\le \sqrt{\tfrac{4\,\alpha\,\Mstar}m}\,\frac{\sqrt G}{R(t)}\,,
\]
where $R(t)$ is as in~\eqref{R}. This inequality can be obtained from~\eqref{CKm} and~\eqref{entropy.decay.rate} by applying the change of variables~\eqref{SelfSimilarChangeOfVariables}. This yields~\eqref{control-radius-inequality2} with a constant in the right-hand-side given by
\[\label{constant-c-control-radius}
\overline C= \,2^{\frac{d}\alpha}\,\left(C\,+\(\overline c_3+\tfrac2\alpha\,\overline c_2\)\right)
\left(\sqrt{\tfrac{4\,\alpha\,\Mstar}m}+\(\overline c_3+\tfrac2\alpha\,\overline c_2\)\Mstar\right)^\vartheta
\]
where the constant $C$ is given by
\begin{align*}
C:=\lambdaBarenblatt^d&\left(1+4\,\lambdaBarenblatt^2\,Z^2\,\rho(\varepsilon)^2\right)^\frac1{1-m}\,C_{d, \nu, 1}\\
&\times\left(\(c_1\,4^\frac d\alpha\,\overline\kappa\,\Mstar^\frac2\alpha\,\frac{2^\nu}{2^\nu-1}+c_2\)^\frac{d}{d+\nu}+\frac1{\(2\,Z\,\rho(\varepsilon)\)^d}\,(2\,\Mstar)^\frac d{d+\nu}\right)\,.
\end{align*}
where $c_1$, $c_2$ are defined in~\eqref{c-1} and in~\eqref{gamma-norm-barenblatt} respectively, and $ Z=\(2\,\alpha\)^\frac1\alpha$.
Recall that $\alpha\in (1,2)$ so that
\[
\overline c_3+\frac{2\,\overline c_2}{\alpha}=\frac1{1-m}+\frac{2m}{2\,(1-m)^2\,\alpha^3}
\le\frac1{1-m}+\frac{m}{(1-m)^2}=\frac1{(1-m)^2}\,,
\]
hence
\begin{align*}\label{constant-c-control-radius.est}
\overline C&\le\,2^{\frac{d}\alpha}\left(C+\tfrac1{(1-m)^2}\right)
\left(\frac{2\,\alpha}{m}+\Mstar+\frac{\Mstar}{(1-m)^2}\right)^\vartheta\\
&\le\,2^{\frac{d}\alpha+\vartheta}\,\frac{\,\left(1+C\right)}{m^\vartheta(1-m)^{2(1+\vartheta)}}
(\alpha+\Mstar)^\vartheta\,.
\end{align*}
Then, for any $\varepsilon\in(0, \varepsilon_{m,d})\subset(0,1/2)$, we have that
\[\label{Rbis.est}
\sqrt{
\tfrac{(1-\underline\varepsilon)^m(\underline\varepsilon-\varepsilon)}{2^{1+m}}}\,\frac1{\muscal\,\sqrt\varepsilon} \le \underline \rho(\varepsilon)=\frac1\muscal\(\big(1+(1+\varepsilon)^{1-m}\big)\,\tfrac{\(\tfrac{1-\varepsilon}{1-\underline\varepsilon}\)^{1-m}-1}{1-(1-\varepsilon)^{1-m}}\)^{1/2}\le\frac{2\,\sqrt{\underline \varepsilon}}{\muscal\,\sqrt\varepsilon}
\]
and
\[\label{Rter.est}
\frac1{\sqrt{1-m}}\,\frac1{\muscal\,\sqrt\varepsilon}
\le\overline \rho(\varepsilon)=\frac1\muscal\left(\,\frac{(1+\varepsilon)^{1-m}+1}{(1+\varepsilon)^{1-m}-1}\right)^\frac12\,\le\frac4{ \sqrt{1-m}}\frac1{\muscal\,\sqrt\varepsilon}\,.
\]
We recall that $\underline \varepsilon<1$, we obtain therefore that
\[\label{rho.eps.est}\begin{split}
\rho(\varepsilon)^2&:=\max\{\overline \rho(\varepsilon), \underline \rho(\varepsilon)\}^2
\le \max\left\{\frac4{\muscal^2\,\varepsilon}
\,,\,\frac{16}{(1-m)}\frac1{\muscal^2\,\varepsilon}\right\}
\le\frac{16}{(1-m)^2\,\muscal^2\,\varepsilon}
\end{split}\]
and also, since $\varepsilon<1/2$, we have that
\[\label{rho.eps.est.low}\begin{split}
\rho(\varepsilon)^2&
\ge\frac1{(1-m)\muscal^2\,\varepsilon} \ge\frac2{(1-m)\,\muscal^2}\,.
\end{split}\]
Combining all above estimates together we find that
\[\begin{split}
\left(1+4\,\lambdaBarenblatt^2\,(2\,\alpha)^\frac2\alpha\,\rho(\varepsilon)^2\right)^\frac1{1-m} &\le \left(\frac{\muscal^2+2^{\frac{6\,\alpha+2}{\alpha}} \lambdaBarenblatt^2}{(1-m)^2, \muscal^2\,\varepsilon}\right)^\frac1{1-m} \\ &\le\frac{2^\frac{2+6\alpha}{\alpha(1-m)}}{(1-m)^\frac2{1-m}\varepsilon^\frac1{1-m}}\left(\frac{\muscal^2+\alpha^\frac2\alpha\lambdaBarenblatt^2}{\muscal^2}\right)^\frac1{1-m} \\
& \le\frac{2^\frac{3+6\alpha}{\alpha(1-m)}}{(1-m)^\frac2{1-m}\varepsilon^\frac1{1-m}}\,,
\end{split}\]
where in the last step we have used the identity $\muscal=\lambdaBarenblatt\,\alpha^\frac1{\alpha}$.

Altogether, we finally obtain
\begin{equation}\label{constant-c-control-radius.suppl.1}\begin{split}
\overline C & \le\frac{2^{\frac{d}\alpha+\frac{3+6\alpha}{\alpha(1-m)}+\vartheta}}{\varepsilon^\frac1{1-m}}\,\frac{(\alpha+\Mstar)^\vartheta}{m^\vartheta(1-m)^{2(1+\vartheta)+\frac2{1-m}}}\\
&\quad\times \left[1+\lambdaBarenblatt^d\,C_{d, \nu, 1} \left(\(2^{10+\frac{2d}{\alpha}}\,\overline\kappa\,\Mstar^\frac2\alpha\,\frac{2^\nu}{2^\nu-1}+c_2\)^\frac{d}{d+\nu}+\(\frac{\muscal^2}{\alpha^\frac1{\alpha}}\)^d\,(2\,\Mstar)^\frac d{d+\nu}\right)\right] \\
&\le\frac{2^{\frac{3d}\alpha+\frac{3+6\alpha}{\alpha(1-m)}+\vartheta+10}}{\varepsilon^\frac1{1-m}}\,\frac{(\alpha+\Mstar)^\vartheta}{m^\vartheta(1-m)^{2(1+\vartheta)+\frac2{1-m}}} \\
&\quad \times\left[1+\lambdaBarenblatt^d\,C_{d, \nu, 1} \left(\(\overline\kappa\,\Mstar^\frac2\alpha\,\frac{2^\nu}{2^\nu-1}+c_2\)^\frac{d}{d+\nu}+\frac{\muscal^{2d}}{\alpha^\frac d\alpha}\,\Mstar^\frac d{d+\nu}\right)\right]=:\frac{\mathsf K}{\varepsilon^\frac1{1-m}}\,.
\end{split}\end{equation}
We recall that $c_2$, $\overline \kappa$, $\muscal$ and $\lambdaBarenblatt$ are all numerical constants, which have been introduced earlier in~\eqref{gamma-norm-barenblatt},~\eqref{kappa},~\eqref{mu} and~\eqref{BarenblattM-delta}. The proof is completed.
\end{steps}
\end{proof}

%%%%%%%%%%%%%%%%%%%%%%%%%%%%%%%%%%%%%%%%%%%%%%%%%%%%%%%%%%%%%%%%%%%%%%%%
\subsection{Proof of Theorem~\ref{Thm:RelativeUniform}}

\begin{proof} By definitions~\eqref{t1bis},~\eqref{Tter} and~\eqref{RE}, we have
\[
T(\varepsilon)=\max\Big\{\frac{2\,c\,A^{1-m}}{(1+\varepsilon)^{1-m}-1}\,,\frac{\kappa_\star\(2\,A\)^{1-m}+\frac2\alpha}{1-(1-\varepsilon)^{1-m}}\Big\}\le\frac14\big(\kappa_1(\varepsilon,m)\,A^{1-m}+\kappa_3(\varepsilon,m)\big)
\]
where $c$ and $\kappa_\star$ are as in~\eqref{toverline} and~\eqref{kappaExpr-kappastarExpr}, and
\[
\kappa_1(\varepsilon,m):=\max\Big\{\frac{8\,c}{(1+\varepsilon)^{1-m}-1}\,,\frac{2^{3-m}\,\kappa_\star}{1-(1-\varepsilon)^{1-m}}\Big\}\,,\;\kappa_3(\varepsilon,m):=\frac{8\,\alpha^{-1}}{1-(1-\varepsilon)^{1-m}}\,.
\]
{}From Corollaries~\ref{Cor:ControlTailsLower} and~\ref{Cor:ControlTailsUpper}, we obtain that the inequality
\be{relative-error}
\left|\frac{u(t,x)}{B(t,x)}-1\right|\le\varepsilon
\ee
holds if $t\ge\kappa_1(\varepsilon,m)\,A^{1-m}+\kappa_3(\varepsilon,m)$ and $|x|\,\ge\,\rho(\varepsilon)\,R(t)$. We also know from inequality~\eqref{control-radius-inequality2} that~\eqref{relative-error} also holds for any $|x|\le 2\,\rho(\varepsilon)\,R(t)$ if $t\ge4\,T(\varepsilon)$ and $t$ is such that
\be{tkappa3}
\frac{\mathsf K}{\varepsilon^\frac1{1-m}}\,\left(\frac1t+\frac{\sqrt G}{R(t)}\right)^\vartheta\le\varepsilon\,.
\ee
Let us define 
\be{a}
\mathsf a:= \frac{\alpha}{\vartheta}\,\frac{2-m}{1-m}\,,
\ee
where $\vartheta$ is as in~\eqref{theta}. Since $R(t)\le(2\,\alpha\,t)^{1/\alpha}$ for any $t\ge2/\alpha$ and $2^{\alpha-1}\,(1+G^{\alpha/2}\,)\ge(1+\sqrt G\,)^\alpha$,~\eqref{tkappa3} holds if
\[
t\ge\max\big\{\kappa_2(\varepsilon,m)\(1+G^\frac\alpha2\)\,,\tfrac2\alpha\big\}\quad\mbox{with}\quad\kappa_2(\varepsilon,m):=\frac{\(4\,\alpha\)^{\alpha-1}\,\mathsf K^\frac\alpha\vartheta}{\varepsilon^{\mathsf a}}\,.
\]
Let us define
\be{taustarabstract}
\taustarA(m,d)=\sup_{\varepsilon\in(0,\varepsilon_{m,d})}\max\big\{\varepsilon\,\kappa_1(\varepsilon,m),\,\varepsilon^{\mathsf a}\kappa_2(\varepsilon,m),\,\varepsilon\,\kappa_3(\varepsilon,m)\big\}\quad
\ee
and let us recall that $t_\star=\taustarA\,\varepsilon^{-\mathsf a}\, \left(1+A^{1-m}+G^\frac\alpha2\right)$ . 
Since $\kappa_2(\varepsilon,m)+\kappa_3(\varepsilon,m)\ge2/\alpha$, then~\eqref{relative-error} holds for any $x\in\R^d$ if
\[\label{second-time}
t\ge t_\star\ge \kappa_1(\varepsilon,m)\,A^{1-m}\,+\kappa_2(\varepsilon,m)\,G^\frac\alpha2+\kappa_2(\varepsilon,m)+\kappa_3(\varepsilon,m)\,.
\]
With $(m,\varepsilon)\in(0,1)\times(0,\varepsilon_{m,d})$, we deduce from the elementary estimates
\[
\tfrac1{(1-m)\,\varepsilon}\le\tfrac1{(1+\varepsilon)^{1-m}-1}\le\tfrac{4}{(1-m)\,\varepsilon}\quad\mbox{and}\quad\tfrac12\,\tfrac1{(1-m)\,\varepsilon}\le\tfrac1{1-(1-\varepsilon)^{1-m}}\le\tfrac1{(1-m)\,\varepsilon}
\]
that $\kappa_2$ dominates $\kappa_1$ and $\kappa_3$ as either $\varepsilon\to0_+$. Up to elementary computations, this proves that $\taustarA(m,d)$ is finite, which completes the proof.\end{proof}

\begin{remark}\label{remarkcstar}
From the expression of $\taustarA(m,d)$ we obtain that
\[
\taustarA(m,d)\ge\varepsilon\,\kappa_3(\varepsilon, m)\ge \frac8{\alpha\,(1-m)}\to \infty\quad\mbox{as}\quad m\to1^{-}\,.
\]
As a consequence, we know that $t_\star\to+\infty$ as $m\to1_-$ for fixed $\varepsilon>0$, which is a limitation of the method that will be further discussed in Chapter~\ref{Chapter-5}.
\end{remark}

%%%%%%%%%%%%%%%%%%%%%%%%%%%%%%%%%%%%%%%%%%%%%%%%%%%%%%%%%%%%%%%%%%%%%%%%
\subsection{The threshold time in self-similar variables}\label{Sec:Stability.improvedEEP}~

The results of Theorem~\ref{Thm:RelativeUniform}, in original variables, for a solution $u$ of~\eqref{FD}, can be rephrased as follows for a solution $v$ of~\eqref{FDr}, using~\eqref{SelfSimilarChangeOfVariables}. Let us define
\be{Tstar}
\taustar:=\taustarA\,\muscal^{-\alpha}\quad\mbox{and}\quad T_\star:= \frac1{2\,\alpha}\,\log\(1+\alpha\,\taustar\,\frac{1+A^{1-m}+G^\frac\alpha2}{\varepsilon^\mathsf{a}}\)\,.
\ee
%-----------------------------------------------------------------------
\begin{proposition}\label{Prop:TT} Let $m\in\big[m_1,1\big)$ if $d\ge2$, $m\in\big(\widetilde m_1,1\big)$ if $d=1$, $A$, $G>0$. Let $\varepsilon\in(0,\min\{\chi\,\eta, \varepsilon_{m,d}\})$, with $\eta$ and $\chi$ as in Proposition~\ref{Prop:Gap}. For any solution $v$ to~\eqref{FDr} with nonnegative initial datum $v_0\in\mathrm L^1(\R^d)$, $\ird{v_0}=\Mstar$, $\ird{x\,v_0}=0$ which satisfies
\be{hyp:Harnack.self}
\|v_0\|_{\mathcal{X}_m}\le A\,,\quad \mathcal F[v_0]\le G\,,
\ee
we have that
\be{uniformFDr}
(1-\varepsilon)\,\mB(x)\le v(t,x)\le(1+\varepsilon)\,\mB(x)\quad\forall\,(t,x)\in[T_\star,+\infty)\times\R^d\,.
\ee
\end{proposition}
%-----------------------------------------------------------------------
\begin{proof} As a consequence of~\eqref{inq.RelativeUniform},~\eqref{SelfSimilarChangeOfVariables} and~\eqref{Barenblatt-1}, the result holds for any
\[
t\ge\frac12\,\log R(t_\star)=\frac1{2\,\alpha}\,\log\(1+\alpha\,\taustarA\,\frac{1+A^{1-m}\,\muscal^{-\alpha}+G^\frac\alpha2}{\varepsilon^{\mathsf a}}\)\ge T_\star
\]
by the definition~\eqref{R} of $R$, where $t_\star$ is computed from $\varepsilon\in(0,\min\{\chi\,\eta, \varepsilon_{m,d}\})$ as in Theorem~\ref{Thm:RelativeUniform}. The last inequality follows from taking into account that~\eqref{hyp:Harnack.self} implies $\|u_0\|_{\mathcal{X}_m}=\muscal^{-\alpha/(1-m)}\,\|v_0\|_{\mathcal{X}_m}$ where $u_0=\muscal^d\,v_0(\muscal\cdot)$, and that \hbox{$\muscal^{-\alpha}\ge1$}. \end{proof}

%%%%%%%%%%%%%%%%%%%%%%%%%%%%%%%%%%%%%%%%%%%%%%%%%%%%%%%%%%%%%%%%%%%%%%%%
%%%%%%%%%%%%%%%%%%%%%%%%%%%%%%%%%%%%%%%%%%%%%%%%%%%%%%%%%%%%%%%%%%%%%%%%
\section{Side results of interest and further observations}

%%%%%%%%%%%%%%%%%%%%%%%%%%%%%%%%%%%%%%%%%%%%%%%%%%%%%%%%%%%%%%%%%%%%%%%%
\subsection{A different assumption for Theorem~\texorpdfstring{\ref{Thm:RelativeUniform}}{4.1}}\label{Sec:InitialData}

Let us define
\[
D:=\frac1{1-m}\ird{\mB^m}-\frac{m}{1-m}\ird{|x|^2\,\mB}=\frac{m\,\Mstar}{(d+2)\,m-d}\left(\frac2{1-m}-d\right)>0\,.
\]
We can avoid the introduction of $\muscal$ in~\eqref{hyp:Harnack} by considering the following lemma. 
%-----------------------------------------------------------------------
\begin{lemma} Let $m \in (\widetilde m_1, 1)$ and let $u_0\in \mathrm L^1_{+}(\R^d)$ be a nonnegative function such that
\[
\ird{u_0}=\ird{\mB}\quad\mbox{and}\quad\ird{|x|^2\,u_0}<\infty\,.
\]
Then we have
\[
\mathcal F\big[\muscal^{-d}\,u_0(\cdot/\muscal)\big] \le\,\max\left\{\muscal^{d\,(1-m)}\,, D\(1-\muscal^{d\,(1-m)}\)\right\}\, \big(1+\mathcal F[u_0]\big)\,.
\]
\end{lemma}
%-----------------------------------------------------------------------
As a consequence, \eqref{inq.RelativeUniform} holds under the more elegant assumption $\mathcal F[u_0]\le G$ and with slightly bigger constant $\taustarA$.
\begin{proof}
Let us define $X:=\frac{m}{1-m}\ird{|x|^2\,u_0}$, $Y:=\frac1{1-m}\ird{u_0^m}$. By simple computations we deduce that $\mathcal F\big[\muscal^{-d}\,u_0(\cdot/\muscal)\big]=\muscal^{2}\,X-\muscal^{d\,(1-m)}\,Y+D$ and $\mathcal F[u_0]=X-Y+D\ge0$. Since $\muscal<1$ and $2-d\,(1-m)>0$, we have that
\[\textstyle
\frac{\muscal^2\,X-\muscal^{d\,(1-m)}\,Y+D}{1+X-Y+D}\le \frac{\muscal^{d\,(1-m)}\,(X-Y)+D}{1+X-Y+D} \le \max\left\{\muscal^{d\,(1-m)}\,, D\(1-\muscal^{d\,(1-m)}\)\right\}\,,
\]
where the last inequality can be obtained by computing the supremum of the function $f(Z)=\big(\muscal^{d\,(1-m)}\,Z+D\big)/(1+Z+D)$ on the interval $Z\ge-\,D$.
\end{proof}

%%%%%%%%%%%%%%%%%%%%%%%%%%%%%%%%%%%%%%%%%%%%%%%%%%%%%%%%%%%%%%%%%%%%%%%%
\subsection{Decay rates in relative error}\label{Sec:Decay.RE}

%-----------------------------------------------------------------------
\begin{corollary}\label{Cor:RelativeUniform} Let $m\in\big[m_1,1\big)$ if $d\ge2$, $m\in\big(\widetilde m_1,1\big)$ if $d=1$, $A$, $G>0$, and $u$ be a solution of~\eqref{FD} corresponding to the nonnegative initial datum $u_0\in\mathrm L^1_{+}(\R^d)\cap\mathcal{X}_m$ such that
\[
\ird{u_0}=\ird{\mB}\,,\quad \|u_0\|_{\X}\le A\,,\quad \mathcal F[\muscal^{-d}\, u_0(\cdot/\muscal)]\le G\,.
\]
Then we have
\be{inq.RelativeUniform-t}
\sup_{x\in\R^d}\left|\frac{u(t,x)}{B(t,x)}-1\right|\le\kappa\,t^{-1/{\mathsf a}}\quad\forall\,t\ge t_{m,d}\,,
\ee
where $t_{m,d}:=\varepsilon_{m,d}^{-\mathsf a}$ and $\kappa(\mathsf a,A,G):=\Big(\taustarA\(1+A^{1-m}+G^{\alpha/2}\)\Big)^{1/\mathsf a}$.
\end{corollary}
%---------------------------------------------------------------------
\begin{proof} In~\eqref{inq.RelativeUniform}, at $t=t_\star$, we compute $\varepsilon$ in terms of $t$ and eliminate it, provided $\varepsilon\le\varepsilon_{m,d}$. This proves~\eqref{inq.RelativeUniform-t} at time $t=t_\star$. Iterating the same procedure for smaller values of $\varepsilon$ proves~\eqref{inq.RelativeUniform-t} for any $t>t_{m,d}$. \end{proof}

%%%%%%%%%%%%%%%%%%%%%%%%%%%%%%%%%%%%%%%%%%%%%%%%%%%%%%%%%%%%%%%%%%%%%%%
\subsection{On the uniform H\"older continuity of solutions to~\texorpdfstring{\eqref{FD}}{FD}}

One of the most technical parts of this chapter is Proposition~\ref{Thm:ControlRadius}. In the proof, we obtain a uniform bound on the modulus of continuity where the exponent $\nu$ does not depend on the solution nor on the initial datum. This allows us to obtain inequality~\eqref{control-radius-inequality2} with an exponent $\vartheta$ which does not depend on $u$. As a consequence, the exponent~$\mathsf a$ only depends on $m$, $d$. 

In the case of the linear parabolic equation~\eqref{HE.coeff} the uniform modulus of continuity with an explicit exponent is obtained as a consequence of the Harnack inequality~\eqref{harnack}, as in Theorem~\ref{Claim:4}. For the fast diffusion equation, this is more subtle. The Harnack inequality for solutions of~\eqref{FD} holds on \emph{intrinsic cylinders}, whose size depend on the solution itself. This fact has been thoroughly studied in~\cite{Daskalopoulos2007,Bonforte2010b,DiBenedetto2012, Bonforte2019a}. In general, the intrinsic geometry can cause a dependence in the \idx{H\"older continuity} exponent on some norms of the solution, see for instance~\cite{Bonforte2019a} for a thorough explanation. We overcome this difficulty by using the \idx{Global Harnack Principle} that allows us to exploit the linear parabolic estimates of Section~\ref{sec:holder} and obtain a uniform exponent of the \idx{H\"older continuity} for the class of solutions considered.

%%%%%%%%%%%%%%%%%%%%%%%%%%%%%%%%%%%%%%%%%%%%%%%%%%%%%%%%%%%%%%%%%%%%%%%
\subsection{Bibliographical comments}

Many results of this chapter are known, however without explicit constants. Let us review our sources.

Up to minor changes in the proof, Inequality~\eqref{Herrero.Pierre.123} about \emph{local mass displacement} is originally containing~\cite[Lemma 3.1]{Herrero1985}. Inequality~\eqref{Herrero.Pierre.opposite} appeared in \cite{bonforte2020fine, Simonov2020}. Local $\mathrm L^p-\mathrm L^\infty$ \emph{smoothing estimates} of Lemma~\ref{Lem:LocalSmoothingEffect} were essentially known, by means of De Giorgi type iterations in~\cite{DiBenedetto1993, DiBenedetto2012}, and via \index{Moser iteration}{Moser type iterations} in~\cite{Daskalopoulos2007,Bonforte2010b,Bonforte2019a}. Proposition~\ref{Local.Aleks} is stated in~\cite{Bonforte2006, Bonforte2010}, see also for~\cite{Galaktionov2004} for a previous contribution. The proof of Proposition~\ref{Local.Aleks} is inspired by a comparison principle which is nowadays a classical tool in the theory of parabolic equations, but for which we are not aware of a precise reference. For a general overview of the subject we refer to the books~\cite{Daskalopoulos2007,Vazquez2006,Vazquez2007, Galaktionov2004} and to~\cite[Lemma~3.4]{Herrero1985},~\cite[Remark~1.5]{MR712265} for closer statements.
The \emph{local lower bounds}~\eqref{BV-3} of Lemma~\ref{Posit.Thm.FDE} are originally contained in~\cite{Bonforte2006}, see also~\cite{Vazquez2006} and~\cite{Bonforte2019a} for a different proof. \emph{Global existence} of nonnegative solutions of~\eqref{FD} is established in~\cite{Herrero1985}, while uniqueness for $\mathrm L^1$ solutions follows by $\mathrm L^1$-contractivity estimates, see~\cite{CP82, BenilanCrandall,Vazquez2007}. Much more is known on~\eqref{FD} and we refer to the monographs~\cite{Vazquez2006,Vazquez2007} for a general overview. The \emph{\idx{Global Harnack Principle}} in the form of Propositions~\ref{GHP-3} and ~\ref{GHP:UpperBound} goes back to~\cite{Bonforte2006}, see also~\cite{Vazquez2003} for a previous contribution. Those results are essentially constructive, even if not all constants were explicitly computed. We refer to~\cite{DiBenedetto1991} for the \idx{Global Harnack Principle} in the case of the Dirichlet problem on a bounded domain.

\emph{\index{uniform convergence in relative error}{Uniform convergence in relative error}} was first established in~\cite{Carrillo2003} in the case of radial data, rates of convergence were also provided. The result was extended to a bigger class of initial data in~\cite{Kim2006}. The \idx{Global Harnack Principle} allows to connect nonlinear and linearized entropy estimates, as unveiled in~\cite{Blanchet2007,Blanchet2009,Bonforte2010,Bonforte2010c}, and provides sharp decay rates in uniform relative error and in weaker norms. A slightly more general version of Theorem~\ref{Thm:RelativeUniform} can be found in~\cite{bonforte2020fine, Simonov2020}, without explicit constants. Analogous results have been obtained in the more general case of fast diffusion equation with Caffarelli-Kohn-Nirenberg weights in~\cite{Bonforte2017,Bonforte2017a,bonforte2020fine}.

%%%%%%%%%%%%%%%%%%%%%%%%%%%%%%%%%%%%%%%%%%%%%%%%%%%%%%%%%%%%%%%%%%%%%%%% 
%%%%%%%%%%%%%%%%%%%%%%%%%%%%%%%%%%%%%%%%%%%%%%%%%%%%%%%%%%%%%%%%%%%%%%%%
\chapter{Stability in Gagliardo-Nirenberg inequalities}\label{Chapter-5}

In this chapter, we focus on the subcritical case $p\in(1,p^\star)$ of the Gagliardo-Nirenberg-Sobolev inequalities~\eqref{GNS} and prove a \idx{stability} result similar to Theorem~\ref{Thm:StabSubCriticalNormalized}, however with an explicit constant $\mathcal C$. Using the \idx{threshold time} $T_\star$ of Proposition~\ref{Prop:TT}, and the fast diffusion equation~\eqref{FDr} with exponent $m$ given by~\eqref{pm}, we invoke the entropy estimates of the initial time layer $(0,T_\star)$ and of the asymptotic time layer $(T_\star,+\infty)$ to establish an improved decay rate of the \idx{relative entropy} the self-similar variables for any $t\in(0,+\infty)$, that is, an improved \idx{entropy - entropy production inequality}, which is interpreted as a \idx{stability} result for the Gagliardo-Nirenberg-Sobolev written in the non-scale invariant form~\eqref{GNS-Intro}.

%%%%%%%%%%%%%%%%%%%%%%%%%%%%%%%%%%%%%%%%%%%%%%%%%%%%%%%%%%%%%%%%%%%%%%%%
%%%%%%%%%%%%%%%%%%%%%%%%%%%%%%%%%%%%%%%%%%%%%%%%%%%%%%%%%%%%%%%%%%%%%%%%
\section{Stability results in relative entropy}

%%%%%%%%%%%%%%%%%%%%%%%%%%%%%%%%%%%%%%%%%%%%%%%%%%%%%%%%%%%%%%%%%%%%%%%%
\subsection{Improved entropy-entropy production inequality}\label{Sec:Stability.improvedEEP2}~

%---------------------------------------------------------------------
\begin{theorem}\label{Thm:ImprovedE-EPinequality} Let $m\in(m_1,1)$ if $d\ge2$, $m\in(1/2,1)$ if $d=1$, $A>0$ and $G>0$. Then there is a positive \index{stability}{number} $\zeta$ such that
\be{improved.entropy.eep}
\mathcal I[v]\ge(4+\zeta)\,\mathcal F[v]
\ee
for any nonnegative function $v\in\mathrm L^1(\R^d)$ such that $\mathcal F[v]=G$, $\ird v=\Mstar$, $\ird{x\,v}=0$ and $v$ satisfies~\eqref{hyp:Harnack.self}.\end{theorem}
%---------------------------------------------------------------------
\noindent An expression of $\zeta$ is given below in~\eqref{ZAG} in terms of $A$ and $G$. Inequality~\eqref{improved.entropy.eep} is an improvement of the \emph{\idx{entropy - entropy production inequality}}~\eqref{entropy.eep}. We prove that the inequality holds at any time $t\ge0$ for any solution of the evolution equation~\eqref{FDr} and, as a special case, for its initial datum.
\begin{proof} Proposition~\ref{Prop:TT} determines an asymptotic time layer improvement: according to Proposition~\ref{Prop:Gap}, Inequality~\eqref{Prop.Rayleigh.Ineq} holds with $\eta=2\,d\,(m-m_1)$ for $\varepsilon\in(0,\chi\,\eta)$, that is,
\[
\mathcal I[v(t,.)]\ge(4+\eta)\,\mathcal F[v(t,.)]\quad\forall\,t\ge T_\star\,.
\]
With the initial time layer improvement of Lemma~\ref{prop.backward}, we obtain that
\be{IFzeta}
\mathcal I[v(t,.)]\ge(4+\zeta)\,\mathcal F[v(t,.)]\quad\forall\,t\in[0,T_\star]\,,\quad\mbox{where}\quad\zeta=\frac{4\,\eta\,e^{-4\,T_\star}}{4+\eta-\eta\,e^{-4\,T_\star}}\,.
\ee
As a consequence,~\eqref{improved.entropy.eep} holds for $v(t,.)$, for any $t\ge0$, because $\zeta\le\eta$, under the condition
\[\label{epsilon}
\varepsilon\in\(0,\,2\,\varepsilon_\star\)\quad\mbox{with}\quad\varepsilon_\star:=\frac12\,\min\big\{\varepsilon_{m,d},\,\chi\,\eta\big\}\,,
\]
where $\varepsilon_{m,d}$ is defined in~\eqref{epsilon.md.def}. As a special case, it is true at $t=0 $ with $\varepsilon=\varepsilon_\star$ and for an arbitrary initial datum satisfying the assumptions of Theorem~\ref{Thm:ImprovedE-EPinequality}. This completes the proof. \end{proof}

The fact that Inequality~\eqref{improved.entropy.eep} holds true at any $t\ge0$ for a solution of~\eqref{FDr} is a \idx{stability} property under the action of the nonlinear fast diffusion flow. The improvement in inequality~\eqref{improved.entropy.eep} has an interesting counterpart in terms of rates, which goes as follows.
%---------------------------------------------------------------------
\begin{corollary}\label{Cor:Improvedrate} Let $m\in(m_1,1)$ if $d\ge2$, $m\in(1/2,1)$ if $d=1$, $A>0$ and $G>0$ and let $\zeta$ be as in Theorem~\ref{Thm:ImprovedE-EPinequality}. If $v$ is a solution of~\eqref{FDr} with nonnegative initial datum $v_0\in\mathrm L^1(\R^d)$ such that $\mathcal F[v_0]=G$, $\ird{v_0}=\Mstar$, $\ird{x\,v_0}=0$ and $v_0$ satisfies~\eqref{hyp:Harnack}, then
\be{improved.rate}
\mathcal F[v(t,.)]\le\mathcal F[v_0]\,e^{-\,(4+\zeta)\,t}\quad\forall\,t\ge0\,.
\ee
\end{corollary}
%---------------------------------------------------------------------
Let us give the sketch of a proof and some comments. We know from Theorem~\ref{Thm:ImprovedE-EPinequality} that
\[
\frac{\rd}{\dt}\mathcal F[v(t,\cdot)]=-\,\mathcal I[v(t,\cdot)]\le-\,(4+\zeta)\,\mathcal F[v(t,\cdot)]
\]
and obtain~\eqref{improved.rate} by a Gr\"onwall estimate. Inequality~\eqref{improved.entropy.eep} can be recovered as a consequence of Corollary~\ref{Cor:Improvedrate}. It is indeed enough to notice that~\eqref{improved.rate} is an equality at $t=0$ and differentiate it at $t=0_+$. Notice that the optimal decay rate in~\eqref{improved.rate} is the optimal constant in~\eqref{improved.entropy.eep}, as in~\cite{DelPino2002}, in the non-improved version of the inequality.

%%%%%%%%%%%%%%%%%%%%%%%%%%%%%%%%%%%%%%%%%%%%%%%%%%%%%%%%%%%%%%%%%%%%%%%%
\subsection{Improved estimate for the deficit functional}\label{Sec:Stability.improvedDeficit}

The \emph{deficit} $\delta[f]$ and the \emph{\idx{free energy}} or \emph{\idx{relative entropy} functional} $\mathcal E[f|\mathsf g]$ are defined by~\eqref{Deficit} and~\eqref{RelativeEntropy}.
%-----------------------------------------------------------------------
\begin{theorem}\label{Thm:MainCh5} Let $d\ge1$, $p\in(1,p^\star)$, $A>0$ and $G>0$. \index{stability}{There is} a positive \index{deficit functional}{constant}~$\mathcal C$ such that
\be{stability.entropy.1.cpt6}
\delta[f]\ge\mathcal C\,\mathcal E[f|\mathsf g]
\ee
for any nonnegative $f\in\mathcal W_p(\R^d)$ such that
\be{Hyp1}
\ird{|f|^{2p}}=\ird{|\mathsf g|^{2p}}\quad\mbox{and}\quad\ird{x\,|f|^{2p}}=0\,,
\ee
\be{Hyp2}
\sup_{r>0}r^\frac{d-p\,(d-4)}{p-1}\int_{|x|>r}|f|^{2p}\,\dx\le A\quad\mbox{and}\quad\mathcal E[f|\mathsf g]\le G\,.
\ee
\end{theorem}
%---------------------------------------------------------------------
\begin{proof} Using~\eqref{deficit-entropy-production},~\eqref{pm} where, $p=1/(2\,m-1)$ and $v=|f|^{2p}$, we learn from Theorem~\ref{Thm:ImprovedE-EPinequality} that
\[
\mathcal I[v]-4\,\mathcal F[v]=\tfrac{p+1}{p-1}\,\delta[f]\ge\zeta\,\mathcal F[v]=\zeta\,\mathcal E[f|\mathsf g]
\]
under Condition~\eqref{Hyp1}. As a consequence, Theorem~\ref{Thm:MainCh5} holds with $\mathcal C=\frac{p-1}{p+1}\,\zeta$.\end{proof}

The expression of $\zeta$ given in~\eqref{IFzeta} can be rewritten using~\eqref{Tstar} as
\[
\zeta=\frac{4\,\eta}{(4+\eta)\,e^{4\,T_\star}-\eta}\,,
\]
where $T_\star$ is given by~\eqref{Tstar}. Since $t_\star\ge2/\alpha$ and $e^{2\alpha\,T_\star}\ge 1+ \alpha\,t_\star$, we obtain
\[
\zeta\ge\frac{4\,\eta}{4+\eta}\(\frac{\varepsilon^{\mathsf a}}{2\,\alpha\,\taustar}\)^\frac2\alpha\(1+A^{1-m}+G^\frac\alpha2\)^{-\frac2\alpha}
\]
with the notation of Theorem~\ref{Thm:RelativeUniform} and where $\taustar$ is defined in~\eqref{Tstar}. Let
\be{calpha}
c_\alpha:=\inf_{x>0,\,y\ge0}\frac{1+x^{2/\alpha}+y}{\(1+x+y^{\alpha/2}\)^{2/\alpha}}\,.
\ee
Then we have
\[
\zeta\ge\mathsf Z\big(A,\mathcal F[u_0]\big)
\]
with
\be{ZAG}
\mathsf Z(A,G):=\frac{\zeta_\star}{1+A^{\,(1-m)\,\frac2\alpha}+G}\,.
\ee
We adopt the convention that
\[
\mathsf Z(A,G)=0\quad\mbox{if}\quad A=+\infty\,.
\]
We also make the choice $\varepsilon=\varepsilon_\star$ as in Section~\ref{Sec:Stability.improvedEEP2}, so that the numerical constant~$\zeta_\star$ is defined as
\be{zetastar}
\zeta_\star:=\frac{4\,\eta}{4+\eta}\(\frac{\varepsilon_\star^{\mathsf a}}{2\,\alpha\,\taustar}\)^\frac2\alpha c_\alpha\,.
\ee
This is the explicit expression of the constant in Theorem~\ref{Thm:MainCh5}.

The constant $\zeta_\star$ deserves some comments. First of all, we know that $\varepsilon_\star\le\frac\chi2\,\eta$, where $\eta=2\,d\,(m-m_1)$ so that $\zeta_\star=\zeta_\star(m)$ is at most of the order of $(m-m_1)^{1+2\,\mathsf a/\alpha}$ where $\frac{\mathsf a}\alpha=\frac{2-m}{\vartheta\,(1-m)}$. As a consequence, we know that $\lim_{m\to m_1}\zeta_\star(m)=0$. This also means that the estimate of the constant $\mathcal C$ in Theorem~\ref{Thm:MainCh5} decays to $0$ if $p\to p^\star$ if $d\ge2$. On the other hand, it appears from~\eqref{zetastar} that $\lim_{m\to1_-}\zeta_\star(m)=0$, see Remark~\ref{remarkcstar}. Our method is therefore limited to the strictly subcritical range $\max\{1/2,m_1\}<m<1$, or $1<p<p^\star$.

Exactly as in Section~\ref{Comparison:Bianchi-Egnell}, we also have a \idx{stability} result in a stronger sense, measured by the relative Fisher information.
%---------------------------------------------------------------------
\begin{corollary}\label{Cor:DefFisher} Under the assumptions of Theorem~\ref{Thm:MainCh5} and for the same positive constant~$\mathcal C$, we have
\be{stability.entropy.1.cpt6.StrongForm}
\delta[f]\ge\frac{(p-1)^2\,(p+1)\,\mathcal C}{(p+1)\,\mathcal C+4\,(p-1)}\ird{\big|\nabla f+\tfrac1{p-1}\,|f|^{p-1}\,f\,\nabla\mathsf g^{1-p}\big|^2}
\ee
for any $f\in\mathcal W_p(\R^d)$ satisfying~\eqref{Hyp1} and~\eqref{Hyp2}
\end{corollary}
%---------------------------------------------------------------------
\begin{proof} Inequality~\eqref{improved.entropy.eep} can be rewritten as
\be{IEP-Fisher}
\mathcal I[v]-\,4\,\mathcal F[v]\ge\frac\zeta{4+\zeta}\,\mathcal I[v]\,.
\ee
After taking into account
\[
\mathcal I[v]=\(p^2-1\)\ird{\big|\nabla f+\tfrac1{p-1}\,f^p\,\nabla\mathsf g^{1-p}\big|^2}\,,
\]
this proves~\eqref{stability.entropy.1.cpt6.StrongForm} in the case of non-negative functions. For a \index{sign-changing functions}{sign-changing solution} $f\in\mathcal W_p(\R^d)$, one can notice that~\eqref{stability.entropy.1.cpt6.StrongForm} holds as soon as it is established for $|f|$.\end{proof}

%%%%%%%%%%%%%%%%%%%%%%%%%%%%%%%%%%%%%%%%%%%%%%%%%%%%%%%%%%%%%%%%%%%%%%%%
%%%%%%%%%%%%%%%%%%%%%%%%%%%%%%%%%%%%%%%%%%%%%%%%%%%%%%%%%%%%%%%%%%%%%%%%
\section{Stability of scale invariant Gagliardo-Nirenberg inequalities}\label{Sec:GNstab2}

In this section, we establish a \index{stability}{consequence} of Theorem~\ref{Thm:MainCh5} for the scale invariant Gagliardo-Nirenberg inequalities~\eqref{GNS} which is a deep but technical result that requires further notation.

%%%%%%%%%%%%%%%%%%%%%%%%%%%%%%%%%%%%%%%%%%%%%%%%%%%%%%%%%%%%%%%%%%%%%%%%
\subsection{Invariances and related parameters}

For any $f\in\mathcal W_p(\R^d)$, let us consider the \emph{\idx{best matching}} \emph{\index{Aubin-Talenti functions}{Aubin-Talenti profile}} $\mathsf g_f$ defined by~\eqref{gf}. With $\mu[f]$, $\lambda[f]$, and $y[f]=x_f$ as in~\eqref{lambdamuy}, let us further define
\begin{align}
\kappa[f]&\,:=\frac{\Mstar^\frac1{2\,p}}{\nrm f{2\,p}}=\mu[f]^{-\frac1{2p}}\,,\nonumber\\
\lambdasigma[f]&\,:=\(\frac{2\,d\,\kappa[f]^{p-1}}{p^2-1}\,\frac{\nrm f{p+1}^{p+1}}{\nrm{\nabla f}2^2}\)^\frac{2\,p}{d-p\,(d-4)},\nonumber\\
\mathsf A_p[f]&\,:=\frac\Mstar{\lambdasigma[f]^\frac{d-p\,(d-4)}{p-1}\,\nrm f{2\,p}^{2p}}\,\sup_{r>0}r^\frac{d-p\,(d-4)}{p-1}\int_{|x|>r}|f(x+x_f)|^{2p}\,\dx\,,\label{Apf}\\
\mathsf E_p[f]&\,:=\tfrac{2\,p}{1-p}\ird{\left(\tfrac{\kappa[f]^{p+1}}{\lambdasigma[f]^{d\,\frac{p-1}{2\,p}}}\,|f|^{p+1}-\mathsf g^{p+1}-\tfrac{1+p}{2\,p}\,\mathsf g^{1-p}\,\Big(\tfrac{\kappa[f]^{2p}}{\lambdasigma[f]^2}\,|f|^{2p}-\mathsf g^{2p}\Big)\right)}\,.\label{Epf}
\end{align}
A computation shows that $\lambda[g_{\lambda,\mu,y}]\,\lambdasigma[g_{\lambda,\mu,y}]^\frac{d-p\,(d-4)}{2\,p}=1$ for any \index{Aubin-Talenti functions}{Aubin-Talenti function} in $g_{\lambda,\mu,y}\in\mathfrak M$, but for a function $f$ which is not an \index{Aubin-Talenti functions}{Aubin-Talenti function}, the scales $\lambdasigma[f]$ and~$\sigmalambda[f]$ generically differ.

On $\mathcal W_p(\R^d)$, let us consider the \emph{normalization map} $\mathsf N$ such that
\be{scaling3}
\mathsf N f(x):=\lambdasigma[f]^\frac d{2\,p}\,\kappa[f]\,f\big(\lambdasigma[f]\,x+x_f\big)\quad\forall\,x\in\R^d\,.
\ee
With these definitions, we have the following useful relations.
%---------------------------------------------------------------------
\begin{lemma}\label{Lem:Nf} For any $f\in\mathcal W_p(\R^d)$, we have $\nrm{\mathsf N f}{2p}=\nrm{\mathsf g}{2p}$, $x_{\mathsf N f}=0$ and
\[
\delta\big[\mathsf Nf\big]=C(p,d)\(\left(\frac{\nrm{\nabla f}2^\theta\,\nrm f{p+1}^{1-\theta}}{\nrm f{2p}}\right)^{2\,p\,\gamma}-\(\mathcal C_{\mathrm{GNS}}\)^{2\,p\,\gamma}\)\quad\mbox{and}\quad\mathsf E_p[f]=\mathcal E[\mathsf N f|\mathsf g]
\]
where $\gamma$ and $C(p,d)$ are defined respectively by~\eqref{Ch1:gamma} and~\eqref{Ch1:Cpd}.
\end{lemma}
%---------------------------------------------------------------------
\begin{proof} Take some $f\in\mathcal W_p(\R^d)$ and let us compute
\[
\nrm{\mathsf N f}{2p}^{2p}=\nrm{\mathsf g}{2p}^{2p}\,,\quad\ird{|x|^2\,|\mathsf N f|^{2p}}=\tfrac{\kappa[f]^{2p}}{\lambdasigma[f]^2}\ird{|x|^2\,|f|^{2p}}\,,
\]
\[
\nrm{\mathsf N f}{p+1}^{p+1}=\tfrac{\kappa[f]^{p+1}}{\lambdasigma[f]^{d\,\frac{p-1}{2\,p}}}\,\nrm f{p+1}^{p+1}\,,\quad\nrm{\nabla\mathsf N f}2^2=\kappa[f]^2\,\lambdasigma[f]^\frac{d-p\,(d-2)}p\,\nrm{\nabla f}2^2\,.
\]
The expression of $\mathsf E_p[f]$ follows. With the notations of Lemma~\ref{Lem:GNscaling} and of its proof, we also find that $\delta[\mathsf Nf]=h\big(\sigma\big[f/\nrm f{2p}\big]\big)$ where $\sigma=\sigma\big[f/\nrm f{2p}\big]$ is precisely the choice of the scaling which optimizes the deficit.\end{proof}

%%%%%%%%%%%%%%%%%%%%%%%%%%%%%%%%%%%%%%%%%%%%%%%%%%%%%%%%%%%%%%%%%%%%%%%%
\subsection{A stability result in relative entropy}

We can get rid of the constraints~\eqref{Hyp1} and~\eqref{Hyp2} of Theorem~\ref{Thm:MainCh5} using the invariances of~\eqref{GNS}.
%----------------------------------------------------------------------
\begin{theorem}\label{Thm:stabilityDraft2} Let $d\ge1$ and $p\in(1,p^\star)$. For any $f\in\mathcal W_p(\R^d)$, we \index{relative entropy}{have}
\be{stability.entropy.1.cpt6.2}
\(\nrm{\nabla f}2^\theta\,\nrm f{p+1}^{1-\theta}\)^{2\,p\,\gamma}-\(\mathcal C_{\mathrm{GNS}}\,\nrm f{2\,p}\)^{2\,p\,\gamma}\ge\mathfrak S[f]\,\nrm f{2\,p}^{2\,p\,\gamma}\,\mathsf E_p[f]
\ee
\index{stability}{where} $\gamma$ is defined by~\eqref{Ch1:gamma} and, with $C(p,d)$ as in~\eqref{Ch1:Cpd} and $\mathsf Z$ as in~\eqref{ZAG},
\[
\mathfrak S[f]=\frac{\Mstar^\frac{p-1}{2\,p}}{p^2-1}\,\frac1{C(p,d)}\,\mathsf Z\(\mathsf A_p[f],\,\mathsf E_p[f]\)
\]
with $\mathsf A_p[f]$ and $\mathsf E_p[f]$ given respectively by~\eqref{Apf} and~\eqref{Epf}.
\end{theorem}
%---------------------------------------------------------------------
\begin{proof} Since $\nrm{\mathsf Nf}{2\,p}^{2p}=\Mstar$ and $\ird{x\,|\mathsf Nf|^{2p}}=0$, Theorem~\ref{Thm:ImprovedE-EPinequality} applies and, as a consequence,
\[
\delta[\mathsf Nf]\ge\mathcal C\,\mathcal E[\mathsf Nf|\mathsf g]\,.
\]
We learn from Section~\ref{Sec:Stability.improvedDeficit} that
\[
\mathcal C=\mathcal C[f]=\tfrac{p-1}{p+1}\,\mathsf Z\big(A,\mathcal E[\mathsf Nf|\mathsf g]\big)\quad\mbox{where}\quad A=\sup_{r>0}r^\frac{d-p\,(d-4)}{p-1}\int_{|x|>r}|\mathsf Nf|^{2p}\,\dx
\,.
\]
Undoing the change of variables~\eqref{scaling3} with the help of Lemma~\ref{Lem:Nf} completes the proof.\end{proof}

%%%%%%%%%%%%%%%%%%%%%%%%%%%%%%%%%%%%%%%%%%%%%%%%%%%%%%%%%%%%%%%%%%%%%%%%
\subsection{A stability result in relative Fisher information}\label{Stab:RelFisher}

We also have a scale invariant form of~\eqref{stability.entropy.1.cpt6.StrongForm}, that is, a \idx{stability} result with respect to a strong norm. Let us define
\[
\mathsf J[f]:=\ird{\left|\,\lambdasigma[f]^\frac{d-p\,(d-2)}{2\,p}\,\nabla f+\kappa[f]^{p-1}\,\lambdasigma[f]^{-1}\,(x-x_f)\,|f|^{p-1}\,f\,\right|^2}\,.
\]
%---------------------------------------------------------------------
\begin{theorem}\label{Thm:stabilityDraft2bis} Let $d\ge1$ and $p\in(1,p^\star)$. For any $f\in\mathcal W_p(\R^d)$, we \index{stability}{have}
\be{stability.entropy.1.cpt6.2bis}
\nrm{\nabla f}2^\theta\,\nrm f{p+1}^{1-\theta}-\mathcal C_{\mathrm{GNS}}\,\nrm f{2\,p}\ge\frac{\mathfrak S_\star[f]\,\nrm f{2\,p}^{4\,(p\,\gamma-1)}}{\(\nrm{\nabla f}2^\theta\,\nrm f{p+1}^{1-\theta}\)^{2\,p\,\gamma-1}}
\,\mathsf J[f]
\ee
where
\[
\mathfrak S_\star[f]:=\frac{\mathcal C_{\mathrm{GNS}}^{2\,p\,\gamma-2}}{2\,p\,\gamma\,\Mstar^{1/p}}\,\frac{\(p^2-1\)}{C(p,d)}\,\frac{\mathsf Z\(\mathsf A_p[f],\,\mathsf E\big[|f|\big]\)}{4+\mathsf Z\(\mathsf A_p[f],\,\mathsf E\big[|f|\big]\)}\,.
\]
\end{theorem}
%---------------------------------------------------------------------
\begin{proof}By writing~\eqref{IEP-Fisher} for $\mathsf Nf$, we find that
\begin{multline*}\label{Stab-final}
\(\nrm{\nabla f}2^\theta\,\nrm f{p+1}^{1-\theta}\)^{2\,p\,\gamma}-\(\mathcal C_{\mathrm{GNS}}\,\nrm f{2\,p}\)^{2\,p\,\gamma}\\
\ge\nrm f{2\,p}^{2\,p\,\gamma}\,{\textstyle\frac{\(p^2-1\)}{C(p,d)}\,\frac{\mathsf Z\(\mathsf A_p[f],\,\mathsf E_p[f]\)}{4+\mathsf Z\(\mathsf A_p[f],\,\mathsf E_p[f]\)}}\hspace*{4cm}\\
\textstyle\times\ird{\left|\,\kappa[f]\,\lambdasigma[f]^\frac{d-p\,(d-2)}{2\,p}\,\nabla f+\tfrac{\kappa[f]^p}{\lambdasigma[f]}\,(x-x_f)\,|f|^{p-1}\,f\,\right|^2}
\end{multline*}
for any function nonnegative $f\in\mathcal W_p(\R^d)$. Up to the \index{sign-changing functions}{replacement} of $f$ by $|f|$, there is no loss of generality in assuming that $f$ is nonnegative. In terms of $\mathsf J[f]$, this means
\[
\(\nrm{\nabla f}2^\theta\,\nrm f{p+1}^{1-\theta}\)^{2\,p\,\gamma}-\(\mathcal C_{\mathrm{GNS}}\,\nrm f{2\,p}\)^{2\,p\,\gamma}\ge\tfrac{\nrm f{2\,p}^{2\,p\,\gamma-2}}{\Mstar^{1/p}}\,{\textstyle\frac{\(p^2-1\)}{C(p,d)}\,\frac{\mathsf Z\(\mathsf A_p[f],\,\mathsf E_p[f]\)}{4+\mathsf Z\(\mathsf A_p[f],\,\mathsf E_p[f]\)}}\,\mathsf J[f]\,.
\]
The conclusion holds using $(1+t)^q-1\le q\,t\,(1+t)^{q-1}$ with $q=2\,p\,\gamma>1$ and $t=\tfrac{\nrm{\nabla f}2^\theta\,\nrm f{p+1}^{1-\theta}}{\mathcal C_{\mathrm{GNS}}\,\nrm f{2\,p}}-1$, 
so that
\begin{multline*}
\nrm{\nabla f}2^\theta\,\nrm f{p+1}^{1-\theta}-\mathcal C_{\mathrm{GNS}}\,\nrm f{2\,p}\\
\ge\frac{\(\mathcal C_{\mathrm{GNS}}\,\nrm f{2\,p}\)^{2\,p\,\gamma-2}}{2\,p\,\gamma\(\nrm{\nabla f}2^\theta\,\nrm f{p+1}^{1-\theta}\)^{2\,p\,\gamma-1}}\left[\(\nrm{\nabla f}2^\theta\,\nrm f{p+1}^{1-\theta}\)^{2\,p\,\gamma}-\(\mathcal C_{\mathrm{GNS}}\,\nrm f{2\,p}\)^{2\,p\,\gamma}\right]\,.
\end{multline*}
\end{proof}

%%%%%%%%%%%%%%%%%%%%%%%%%%%%%%%%%%%%%%%%%%%%%%%%%%%%%%%%%%%%%%%%%%%%%%%%
%%%%%%%%%%%%%%%%%%%%%%%%%%%%%%%%%%%%%%%%%%%%%%%%%%%%%%%%%%%%%%%%%%%%%%%%
\section{Stability results: some remarks}\label{Sec:Conclusion}

Theorems~\ref{Thm:stabilityDraft2} and~\ref{Thm:stabilityDraft2bis} are not easy to read and deserve some comments.
\\[4pt]
(i) The constant $\mathfrak S[f]$ in the right-hand side of~\eqref{stability.entropy.1.cpt6.2} measures the \idx{stability}. Although it has a complicated expression, we have shown that it can be written in terms of well-defined quantities depending on $f$ and purely numerical constants. As a special case, it is straightforward to check that
\[
\mathfrak S[\mathsf g]>0
\]
where $\mathsf g$ is the \index{Aubin-Talenti functions}{Aubin-Talenti function}~\eqref{Aubin.Talenti}.
\\[4pt]
(ii) \index{stability}{Stability} results known so far from either the \emph{\idx{carr\'e du champ}} method or from the scaling properties of the deficit functional, according to~\cite{MR3493423,Dolbeault2013917}, involve in the right-hand side an $\mathcal E[f|\mathsf g]^2$ term, while here we achieve a linear lower estimate in terms of $\mathcal E[f|\mathsf g]$. In fact, among all possible inequalities
\[
\delta[f]\ge\mathcal C_A\,\mathcal E[f|\mathsf g]^A
\]
for some appropriate constant $\mathcal C_A$ depending on $A>0$, we have proved in~\eqref{stability.entropy.1.cpt6} that $A=1$ is achieved while only the case $A=2$ was previously known. This of course raises the question of the best possible $A$. The example of Section~\ref{Sec:FDE-Asymptotic} corresponding to the solution of~\eqref{FDr} with the initial datum $v_0(x)=\lambda^d\,\mB(\lambda\,x)$, $\lambda\neq1$, shows that $A<1$ is impossible, so that $A=1$ is optimal. On the other hand, there is a lot of space for improvements on the estimate of $\mathcal C=\mathcal C_1$.
\\[4pt]
(iii) An easy consequence of~\eqref{stability.entropy.1.cpt6.2} is an estimate of the deficit in the Gagliardo-Nirenberg inequality~\eqref{GNS}, namely
\[
\nrm{\nabla f}2^\theta\,\nrm f{p+1}^{1-\theta}-\mathcal C_{\mathrm{GNS}}\,\nrm f{2\,p}\ge\frac{\(2\,p\,\gamma\)^{-1}\,\mathfrak S[f]\,\nrm f{2\,p}^{2\,p\,\gamma}}{\(\nrm{\nabla f}2^\theta\,\nrm f{p+1}^{1-\theta}\)^{2\,p\,\gamma-1}}\,\mathsf E_p[f]\quad\forall\,f\in\mathcal W_p(\R^d)\,.
\]
The proof is similar to the proof of Theorem~\ref{Thm:stabilityDraft2bis}.\\[4pt]
(iv) While the restriction~\eqref{Hyp1} has been lifted in Theorem~\ref{Thm:stabilityDraft2}, Condition~\eqref{Hyp2} is deeply rooted in our method. It is an open question to decide if this assumption can be removed. Note that it is present in Theorem~\ref{Thm:stabilityDraft2} because $\mathsf A_p[f]=+\infty$ or $\mathsf E_p[f]=+\infty$ means $\mathfrak S[f]=0$. The same remark applies to~\eqref{stability.entropy.1.cpt6.2bis}.
\\[4pt]
(v) More subtle is the fact the natural space is not the space of all functions $f$ in $\mathrm L^{2p}(\R^d)$ with gradient in $\mathrm L^2(\R^d)$, but we also need that $\ird{|x|^2\,|f|^{2p}}$ is finite, for instance to define the \idx{free energy}. Up to the condition that $\mathsf A_p[f]<+\infty$, we are therefore working in the space $\mathcal W_p(\R^d)$. If $p=p^\star$, this is not the space of the \idx{stability} result in the critical case by G.~Bianchi and H.~Egnell. It is however consistent with the use of the Fisher information. See Section~\ref{Sec:UncertaintyPrinciple} for some properties enjoyed by entropy related functionals in  $\mathcal W_p(\R^d)$.
\\[4pt]
(vi) The two quantities $\lambdasigma[f]$ and $\sigmalambda[f]$ define length scales. They are equal if $f$ is an \index{Aubin-Talenti functions}{Aubin-Talenti function} but they are not generically equal.
\\[4pt]
(vii) The notion of \emph{\idx{best matching} \idx{Aubin-Talenti functions}} in the sense of relative entropies is well defined for nonnegative functions $f\in\mathcal W_p(\R^d)$. The \idx{relative entropy} cannot be defined as a convex functional if $f$ changes \index{sign-changing functions}{sign}. However, our method applies with no change when $f$ is replaced by $|f|$. By the Csisz\'ar-Kullback inequality of Lemma~\ref{Lem:CK} and Lemma~\ref{convergence.L1}, we obtain a standard measure of the distance to the manifold $\mathfrak M$ of the \idx{Aubin-Talenti functions}. Concerning \index{sign-changing functions}{sign-changing solutions}, see Corollary~\ref{Cor:StabSubCriticalNormalized}. The result of Theorem~\ref{Thm:stabilityDraft2bis} applies to sign-changing solutions because neither the deficit nor the relative Fisher information are sensitive to the sign. The relative Fisher information measures a notion of strong distance to the manifold $\mathfrak M$, but it is then clear that the \idx{best matching} \index{Aubin-Talenti functions}{Aubin-Talenti function} with respect to $|f|$ is in general not optimal if one optimizes $\mathcal J[f|g]$ as defined in Chapter~\ref{Chapter-1} with respect to~$g$. This again leaves some space for improvements.
\\[4pt]
(viii) In Theorem~\ref{Thm:StabSubCriticalNormalized}, we assume that
\begin{equation}\label{smc}
\ird{|x|^2\,f^{2p}}=\ird{|x|^2\,\mathsf g^{2p}}\,.
\end{equation}
Such a condition was needed to take advantage of a previous stability result where the deficit controls $\mathcal E[f|\mathsf g]^2$ as a distance towards $\mathfrak M$. It is also used in the proof of Theorem~\ref{Thm:StabSubCriticalNormalized} as a technical tool to obtain the limit of a minimization sequence and, as well, to get an improved spectral gap in the linearization analysis, although a mode detailed analysis could eventually be done without this assumption. Condition~\eqref{smc} is not required in the constructive \idx{stability} estimate of Theorem~\ref{Thm:MainCh5}.  Reciprocally, Condition~\eqref{Hyp2} in Theorem~\ref{Thm:MainCh5} is not present in Theorem~\ref{Thm:StabSubCriticalNormalized}. The boundedness of $\ird{|x|^2\,f^{2p}}$ will be further discussed in Section~\ref{ssec:secondmomentdiscussion} of Chapter~\ref{Chapter-7}.
\\[4pt]
(ix) The whole strategy of our proof is to reduce the issue of \idx{stability} results for \idx{Gagliardo-Nirenberg-Sobolev inequalities} to \emph{improved \index{entropy - entropy production inequality}{entropy - entropy production inequalities}} as in Theorem~\ref{Thm:ImprovedE-EPinequality}. An important consequence is given in  Corollary~\ref{Cor:Improvedrate}: under the appropriate conditions, improved decay rates are obtained along the fast diffusion flow, which put into light the role the eigenstates corresponding to the lowest eigenvalues. \index{best matching}{Best matching} is the counterpart of this observation: the choice of an adapted \idx{Barenblatt function} instead of the standard \index{Barenblatt self-similar solutions}{Barenblatt self-similar solution} provides us with a more accurate expansion of the solutions to the \idx{fast diffusion equation} for large time values.

%%%%%%%%%%%%%%%%%%%%%%%%%%%%%%%%%%%%%%%%%%%%%%%%%%%%%%%%%%%%%%%%%%%%%%%%
%%%%%%%%%%%%%%%%%%%%%%%%%%%%%%%%%%%%%%%%%%%%%%%%%%%%%%%%%%%%%%%%%%%%%%%%
\chapter{Constructive stability for Sobolev's inequality}\label{Chapter-6}

This chapter is devoted to a \idx{stability} estimate for \idx{Sobolev's inequality} corresponding to~\eqref{GNS} in the critical case $p=p^\star$. The main novelty compared to the \index{Bianchi-Egnell result}{result} of G.~Bianchi and H.~Egnell is that we provide a constructive estimate with an explicit constant and prove that the \index{deficit functional}{deficit} in \idx{Sobolev's inequality} controls a strong distance to the \index{Aubin-Talenti functions}{Aubin-Talenti manifold} of optimal functions measured by a nonlinear relative \idx{Fisher information}.

Compared to the subcritical case $p<p_\star$, \emph{i.e.}, $m>m_1$, the main difference when $p=p_\star$ and  $m=m_1$ with $d\ge3$ is that the Hardy-Poincar\'e inequality~\eqref{HP-PNAS} admits no improved spectral gap under the additional constraint \hbox{$\ird{x\,h\,\mathcal B^{2-m}}=0$}. As a consequence, there is no improved  entropy-entropy production inequality as in Theorem~\ref{Thm:ImprovedE-EPinequality}. An improved spectral gap holds only if one further assumes that $\ird{|x|^2\,h\,\mathcal B^{2-m}}=0$, which arises only as a consequence of second moment estimates of the solutions of~\eqref{FDr}. This is an important difficulty, as $\ird{|x|^2\,v}$ is not conserved. We recall that $\ird{x\,v}=0$ if $\ird{x\,v_0}=0$ where $v$ is the solution of~\eqref{FDr} with initial datum $v_0$. The key idea is to use the \idx{best matching} Barenblatt function, which involves an additional, finite scaling compared to the self-similar solutions. This can be interpreted as the next order term in an asymptotic expansion, that is, a refinement of the self-similar scaling: among all \idx{Barenblatt self-similar solutions} written for the unscaled problem~\eqref{FD}, at main order, the scale is $R(t)\sim t^{1/\alpha}\to+\infty$ as $t\to+\infty$ given by~\eqref{R}. However, best matching is in general achieved up to a lower order term: an additional rescaling of order $o(R(t))$ gives rise to a time-dependent rescaled equation~\eqref{FDR} which preserves the second moment. This new rescaling can also be seen as a simple time-shift. Our goal is to estimate the corresponding \emph{delay} $\tau(t)$ in the rescaled variables.

This chapter is organised as follows: results are stated in Section~\ref{Sub:StabCritResult} and the definition of the delay is given in Section~\ref{Sub:StabCritFlow}. The core of the argument is to prove that the time \idx{delay} is finite, which is done in Section~\ref{sec:new} using a phase portrait analysis. Up to technicalities, the remainder of the proof is similar to the sub-critical case of Chapter~\ref{Chapter-5}: see Sections~\ref{sec:proofs},~\ref{Sec:6-Harnack} and~\ref{Sec:6-Constant}. Because of the criticality, we have $\theta=1$ in~\eqref{Ch1:theta} and the computations needed to deduce the stability in case of \idx{Sobolev's inequality} from the \idx{entropy - entropy production inequality}, that is, for the conclusion of the proof of Theorem~\ref{Thm:Main}, are even simpler.

%%%%%%%%%%%%%%%%%%%%%%%%%%%%%%%%%%%%%%%%%%%%%%%%%%%%%%%%%%%%%%%%%%%%%%
%%%%%%%%%%%%%%%%%%%%%%%%%%%%%%%%%%%%%%%%%%%%%%%%%%%%%%%%%%%%%%%%%%%%%%
\section{A result in the critical case}\label{sec:Ch6-Result}

%%%%%%%%%%%%%%%%%%%%%%%%%%%%%%%%%%%%%%%%%%%%%%%%%%%%%%%%%%%%%%%%%%%%%%
\subsection{Functional framework and definitions}

On the Euclidean space $\R^d$ with $d\ge3$, the \index{Bianchi-Egnell result}{Bianchi-Egnell result} \idx{stability} estimate~\eqref{Bianchi-Egnell} holds for any function~$f$ in the Beppo Levi space $\big\{f\in\mathrm L^{2^*}(\R^d)\,:\,\nabla f\in\mathrm L^2(\R^d)\big\}$, where $\nabla f$ is defined in the sense of distributions. With $2\,p^\star=2d/(d-2)=2^*$, we work in the slightly smaller space
\[
\mathcal W_{p^\star}(\R^d)=\left\{f\in\mathrm L^{2^*}(\R^d)\,:\,\nabla f\in\mathrm L^2(\R^d)\,,\;\,|x|\,|f|^{\,p^\star}\in\mathrm L^2(\R^d)\right\}\,.
\]

%%%%%%%%%%%%%%%%%%%%%%%%%%%%%%%%%%%%%%%%%%%%%%%%%%%%%%%%%%%%%%%%%%%%%%
\subsection{A constructive stability result}\label{Sub:StabCritResult}

%---------------------------------------------------------------------
\begin{theorem}\label{Thm:Main} Let $d\ge3$ and $A>0$. Then for any nonnegative function $f\in\mathcal W_{p^\star}(\R^d)$ such that
\[
\ird{(1,x, |x|^2)\,f^{2^*}}=\ird{(1,x, |x|^2)\,\mathsf g}\quad\mbox{and}\quad
\sup_{r>0}r^d\int_{|x|>r}\,f^{2^*}\dx\le A\,,
\]
we have
\be{Stability-entropy}
\nrm{\nabla f}2^2-\mathsf S_d^2\,\nrm f{2^*}^2\ge\mathcal{C}_\star(A) \ird{\(\mathsf{g}^{2\,\frac{d-1}{d-2}}-f^{2\,\frac{d-1}{d-2}}\)}\ge0
\ee
and
\be{stability-fisher}
\nrm{\nabla f}2^2-\mathsf S_d^2\,\nrm f{2^*}^2\ge\frac{\mathcal C_\star(A)}{4+\mathcal C_\star(A)} \ird{\left|\nabla f+\tfrac{d-2}2\,f^\frac d{d-2}\,\nabla\mathsf g^{-\frac2{d-2}}\right|^2}\,.
\ee
The \idx{stability} constant is $\mathcal C_\star(A)=\mathfrak C_\star\,\big(1\!+\!A^{1/(2\,d)}\big)^{-1}$ where $\mathfrak C_\star>0$ depends only on~$d$. \end{theorem}
%---------------------------------------------------------------------
As in Chapter~\ref{Chapter-5}, the main point in this result is that $\mathcal C_\star(A)$ is explicit. Let us define 
\[\label{Af}
A[f]:=\sup_{r>0}\,r^d\,\int_{r>0} |f|^{2^*}(x+x_f)\quad\mbox{and}\quad Z[f]:=\Big(1+\mu[f]^{-d}\,\lambda[f]^d\,A[f]\Big)
\]
where $\mu[f]$, $\lambda[f]$ and $x_f=y[f]$ are as in~\eqref{lambdamuy}. With $\mathcal E[f|g]$ and $\mathcal J[f|g]$ and $\mathsf g_f$ as in~\eqref{RelativeEntropy},~\eqref{RelativeFisher} and~\eqref{gf}, we have the following result. 
%---------------------------------------------------------------------
\begin{corollary}\label{cor:main1}
Let $d\ge3$. Then for any nonnegative function $f\in\mathcal W_{p^\star}(\R^d)$ such that $A[f]<\infty$, we have
\be{stability-general-entropy}
\delta[f]\ge\mathfrak C_\star\,\frac{\mu[f]^\frac{d-3}{d}\,\lambda[f]^\frac12}{Z[f]}\,\mathcal E[f | \mathsf g_f]\,.
\ee
and
\be{stability-general-fisher}
\delta[f]\ge\mathfrak C_\star\,\frac{Z[f]}{4+Z[f]}\,\mathcal J[f | \mathsf g_f]\,.
\ee
\end{corollary}
%---------------------------------------------------------------------
Up to the replacement of $\mathsf g_f$ by $\mathsf g_{|f|}$, Inequality~\eqref{stability-general-fisher} applies to \idx{sign-changing functions} $f\in\mathcal W_{p^\star}(\R^d)$ using the \idx{Fisher information} $\mathcal J[f|g]$ defined by~\eqref{Fisher-Sign}, providing a constructive version of Corollary~\ref{Cor:StabSubCriticalNormalized}. As a consequence, we have a stability result of Bianchi-Egnell type, that can be written~as
\[
\delta[f]\ge\mathfrak C_\star\,\frac{Z[f]}{4+Z[f]}\,\inf_{g\in\mathfrak M}\mathcal J[f|g]\,.
\]
At this point is it important to notice that the \idx{best matching} \index{Aubin-Talenti functions}{Aubin-Talenti function} $\mathsf g_f$ in terms of the \idx{relative entropy} $\mathcal E[f|g]$ does not minimize $\mathcal J[f|g]$. \index{best matching}{Best matching} \index{Aubin-Talenti functions}{Aubin-Talenti function} in terms of the \idx{relative Fisher information} can be characterized as follows.
%---------------------------------------------------------------------
\begin{proposition}\label{Prop:bestFisher} Let $d\ge3$. Then for any function $f\in\mathcal W_{p^\star}(\R^d)$, we have
\[
\inf_{g\in\mathfrak M}\mathcal J[f|g]=\mathcal J[f|\mathsf g_{\lambda,\mu,x_f}]=4\,\tfrac{d-1}{(d-2)^2}\,\nrm{\nabla f}2^2-\frac{d^2\,\nrm f{2\,\frac{d-1}{d-2}}^{2\,\frac{d-1}{d-2}}}{(d-1)\,\ird{|f|^{2^*}\,|x-x_f|^2}}
\]
where
\[
\mu=\nrm{f}{2^*}^{2^*}\quad\mbox{and}\quad\lambda=\frac{d}{2\,(d-1)}\,\frac{\nrm f{2^*}^\frac2{d-2}\,\nrm f{2\,\frac{d-1}{d-2}}^{2\,\frac{d-1}{d-2}}}{\ird{|f|^{2^*}\,|x-x_f|^2}}\,.
\]
\end{proposition}
%---------------------------------------------------------------------
Notice that $\mathcal J[f|\mathsf g_{\lambda,\mu,x_f}]$ is nonnegative as it is the integral of the square, but the expanded version is also nonnegative according to the generalized \idx{Heisenberg uncertainty principle}~\eqref{Heisenberg}.

\medskip An interesting alternative to a measure of the distance to a single function of the Aubin-Talenti manifold is suggested by the observation that entropy methods are efficient for nonnegative functions. Let us denote by $f_\pm=(|f|\pm f)/2$ the positive and the negative part of a \index{sign-changing functions}{sign-changing function} $f\in\mathcal W_{p^\star}(\R^d)$. Using Minkowski's inequality applied to $f^2=f_+^2+f_-^2$, we have that
\[
\nrm f{2\,p^\star}^2=\nrm{f^2}{p^\star}\le\nrm{f_+^2}{p^\star}+\nrm{f_-^2}{p^\star}=\nrm{f_+}{2\,p^\star}^2+\nrm{f_-}{2\,p^\star}^2\,,
\]
which proves that the deficit~\eqref{Deficit} has the property that 
\[
\delta[f]\ge\delta[f_+]+\delta[f_-]
\]
in the critical case $p=p^\star$. As an easy consequence, we have the following result.
%---------------------------------------------------------------------
\begin{corollary}\label{cor:main-pm}
Let $d\ge3$. Then for any function $f\in\mathcal W_{p^\star}(\R^d)$ such that $A[f]<\infty$, we have
\[
\delta[f]\ge\mathfrak C_\star\(\frac{Z[f_+]}{4+Z[f_+]}\,\mathcal J[f_+|\mathsf g_{f_+}]+\frac{Z[f_-]}{4+Z[f_-]}\,\mathcal J[f_-|\mathsf g_{f_-}]\)\,.
\]\end{corollary}
%---------------------------------------------------------------------

%%%%%%%%%%%%%%%%%%%%%%%%%%%%%%%%%%%%%%%%%%%%%%%%%%%%%%%%%%%%%%%%%%%%%%
\subsection{Flow: a refined rescaling and consequences}\label{Sub:StabCritFlow}

Let us consider the \emph{\idx{fast diffusion equation}} in \idx{self-similar variables}~\eqref{FDr}, which we recall for convenience:
\[
\frac{\partial v}{\partial t}+\nabla\cdot\(v\,\nabla v^{m-1}\)=2\,\nabla\cdot(x\,v)\,,\quad v(t=0,\cdot)=v_0\,.
\]
For any $x\in\R^d$, let us consider the \index{Barenblatt function}{Barenblatt profile} $\mathcal B_\lambda$ defined by
\be{rescaled-second-moment}
\mathcal B_\lambda(x)=\lambda^{-\frac d2}\,\mathcal B\(\frac x{\sqrt\lambda}\)
\ee
where $\mB(x)=\(1+|x|^2\)^{1/(m-1)}$ as in~\eqref{mB} is the unique stationary solution of~\eqref{FDr} with $\ird\mB=\Mstar$. We are interested in the specific choice of $\lambda$ corresponding to
\be{sigma-6}
\lambda(t):=\frac1{\J\,\Rr(t)^2}\ird{|x|^2\,v(t,x)}\quad\mbox{with}\quad \J:=\ird{|x|^2\,\mathcal B}\,,
\ee
where $v$ solves~\eqref{FDr} and $t\mapsto\Rr(t)$ is obtained by solving
\be{ODE-6}
\frac{\rd\tau}{\dt}=\(\frac 1\J\ird{|x|^2\,v}\)^{-\frac d2\,(m-m_c)}-1\,,\quad\tau(0)=0\quad\mbox{and}\quad\Rr(t)=e^{2\,\tau(t)}\quad\forall\,t\ge0\,.
\ee
With these definitions, let us consider the change of variables
\be{ChVar}
v(t,x)=\frac1{\Rr(t)^d}\,w\(t+\tau(t),\frac x{\Rr(t)}\)\quad\forall\,(t,x)\in\R^+\times\R^d\,.
\ee
The role of the \emph{\idx{delay}} $t\mapsto\tau(t)$ is to introduce an additional, finite rescaling on the set of the \index{Barenblatt function}{Barenblatt profiles} in order to produce a better matching in an asymptotic expansion of the relative entropy. Technically, the advantage of this change of variables is a gain of control of the $|x|^2$ moment.
%---------------------------------------------------------------------
\begin{lemma}\label{Lem:BestMatching} Let $d\ge3$ and $m\in[m_1,1)$. Assume that $v_0$ is a nonnegative function in $\mathrm L^1\big(\R^d,(1+|x|^2)\dx\big)$ such that $v_0^m\in\mathrm L^1(\R^d,\dx)$ and $\ird{(1,x)\,v_0}=(\Mstar,0)$. If $v$ solves~\eqref{FDr} with initial datum $v_0$ and $w$ is obtained by~\eqref{sigma-6},~\eqref{ODE-6} and~\eqref{ChVar}, then $w$ solves
\be{FDR}
\frac{\partial w}{\partial s}+\lambda_\star(s)^{\frac d2\,(m-m_c)}\,\nabla\cdot\(w\,\nabla w^{m-1}\)=2\,\nabla\cdot(x\,w)\,,\quad w(t=0,\cdot)=v_0\,,
\ee
where the function $t\mapsto s(t):=t+\tau(t)$ is monotone increasing on $\R^+$, $\lambda_\star$ is defined~by
\[
\lambda_\star\big(s(t)\big)=\lambda(t)\quad\forall\,t\ge0
\]
and the function $\mB_\star(s,x):=\mB_{\lambda_\star(s)}(x)$ is such that for all $s\ge 0$ 
\be{Conservations}
\ird{\(1,x,|x|^2\)w(s,x)}=\ird{\(1,x,|x|^2\)\mB_\star(s,x)}.
\ee
\end{lemma}
%---------------------------------------------------------------------
\begin{proof} We read from~\eqref{ChVar} that $w$ solves
\[
\(1+\frac{\rd\tau}{\dt}\)\frac{\partial w}{\partial s}+\Rr(t)^{-d\,(m-1)-2}\,\nabla\cdot\(w\,\nabla w^{m-1}\)=2\(1+\frac1{2\,\Rr}\,\frac{d\Rr}{dt}\)\nabla\cdot(x\,w).
\]
Taking into account~\eqref{sigma-6} and~\eqref{ODE-6} concludes the proof after simple changes of variables based on~\eqref{ODE-6}.\end{proof}

Identity~\eqref{Conservations} means that $\mB_\star$ is the \emph{\idx{best matching} \idx{Barenblatt function}} associated with $w(s,\cdot)$ and solves
\[
\lambda_\star^{\frac d2\,(m-m_c)}\,\nabla\cdot\(w\,\nabla w^{m-1}\)=2\,\nabla\cdot(x\,w)\,,
\]
so that $\mB_\star$ is not a solution of~\eqref{FDR} if $\lambda_\star$ depends on $s$, \emph{i.e.}, if $\lambda$ depends on $t$. The fact that $\tau$ is bounded for any $t>0$ in terms of the initial data was proved in~\cite{1751-8121-48-6-065206}, however in a rather implicit form. Our task here is now to provide explicit estimates.

%%%%%%%%%%%%%%%%%%%%%%%%%%%%%%%%%%%%%%%%%%%%%%%%%%%%%%%%%%%%%%%%%%%%%%
%%%%%%%%%%%%%%%%%%%%%%%%%%%%%%%%%%%%%%%%%%%%%%%%%%%%%%%%%%%%%%%%%%%%%%
\section{Time \idx{delay} estimates}\label{sec:new}

In this Section, we prove estimates on $t\mapsto\tau(t)$ based on the study of the solutions of~\eqref{FDr}.

%%%%%%%%%%%%%%%%%%%%%%%%%%%%%%%%%%%%%%%%%%%%%%%%%%%%%%%%%%%%%%%%%%%%%%
\subsection{Definitions and constraints}\label{sec:defandcon}

Let us define the \emph{relative second moment} $\mathcal K$ and the \emph{entropy} $\mathcal S$ respectively by
\[
\mathcal K[v]:=\ird{|x|^2\,v(x)}-\ird{|x|^2\,\mB(x)}=\ird{|x|^2\,v(x)}-\,\J\,,
\]
\[
\mathcal S[v]:=\ird{v^m}-\ird{\mB^m}=\ird{v^m}-\E\,,
\]
where the constants $\J$ and $\E$ are given according to Section~\ref{Appendix:Identities} by
\begin{align*}
&\J:=\ird{|x|^2\,\mB(x)}=\frac{d\,(1-m)}{(d+2)\,m-d}\,\Mstar\,,\\
&\E:=\ird{\mB^m}=\frac{2\,m}{(d+2)\,m-d}\,\Mstar\,.
\end{align*}
%---------------------------------------------------------------------
\begin{lemma}\label{Lem:Basic} Let $d\ge1$, $m\in[m_1,1)$ and assume that $v\in\mathrm L^1_+(\R^d)$ is such that $\nrm v1=\Mstar$ and $(|x|^2\,v+v^m)\in\mathrm L^1(\R^d)$. Then we have
\be{Constraints:KS}
-\,\J\le\mathcal K[v] \;\mbox{ and }\;-\,\E\le\mathcal S[v]\le\E\(1+d\,\frac{1-m}{2\,m\,\J}\,\mathcal K[v]\)^m-\E\le m\,\mathcal K[v]\,.
\ee
\end{lemma}
%---------------------------------------------------------------------
\begin{proof} The lower bounds are consequences of the positivity of $v$. By H\"older's inequality, we have that
\[
\ird{v^m}=\ird{\mB^{m\,(m-1)}\,v^m\cdot\mB^{m\,(1-m)}}\le\(\ird{\(1+|x|^2\)\,v}\)^m\,\E^{1-m}\,,
\]
which can be rephrased into
\be{UpperCurve}
\mathcal S[v]\le\psi\big(\mathcal K[v]\big)
\ee
where
\[
\psi(z):=\E\(1+\frac z{\E}\)^m-\E=\E\(1+d\,\frac{1-m}{2\,m}\frac z{\J}\)^m-\E\,.
\]
The last inequality, $\psi\big(\mathcal K[v]\big)\le m\,\mathcal K[v]$ is a consequence of the inequality $(1+x)^m-1-m\,x\le0$ for any $x>-1$.\end{proof}

%%%%%%%%%%%%%%%%%%%%%%%%%%%%%%%%%%%%%%%%%%%%%%%%%%%%%%%%%%%%%%%%%%%%%%
\subsection{Dynamical estimates} We consider the solution $v$ of~\eqref{FDr}. Using
\begin{align*}
&\frac{\rd}{\dt}\ird{|x|^2\,v}=2\,d\,\frac{1-m}m\ird{v^m}-4\ird{|x|^2\,v}\,,\\
&\frac{\rd}{\dt}\ird{v^m}=(1-m)\,\mathcal I[v]+m\(2\,d\,\frac{1-m}m\ird{v^m}-4\ird{|x|^2\,v}\)\,,
\end{align*}
we can rewrite these two identities on $\mathcal K(t)=\mathcal K[v(t,\cdot)]$ and $\mathcal S(t)=\mathcal S[v(t,\cdot)]$ as
\[\label{KS}
\mathcal K'=\aA\,\mathcal S-4\,\mathcal K\,,\quad\mathcal S'=-\,\bB\,\mathcal S+m\,\delta\,,
\]
with
\[
\aA=2\,d\,\frac{1-m}m\,,\quad \bB=2\,d\,(m-m_c)=2\,\alpha\,,
\]
and $m\,\delta=(1-m)\,\mathcal I-4\(m\,\mathcal K-\mathcal S\)\ge0$. With $\mathcal I[v]$ and $\mathcal F[v]$ defined as in Section~\ref{Sec:FDE-Def}, we have indeed that
\[
\delta[v]=\frac{1-m}m\,\big(\mathcal I[v]-\,4\,\mathcal F[v]\big)
\]
where $\delta$ coincides with the notion of \index{deficit functional}{deficit} in~\eqref{Deficit}, the above relation is stated in~\eqref{deficit-entropy-production} up to a slight abuse of notations because $\delta$ is defined in Chapter~\ref{Chapter-1} as $f\mapsto\delta[f]$ with $f^{2p}=v$, and not as a functional acting directly on $v$: also see Lemma~\ref{Lem:BasicEntropyProp} with $\mathcal I[v]=\mathcal J[f|\mathsf g]$ and $\mathcal F[v]=\mathcal E[f|\mathsf g]$.

%%%%%%%%%%%%%%%%%%%%%%%%%%%%%%%%%%%%%%%%%%%%%%%%%%%%%%%%%%%%%%%%%%%%%%
\subsection{Dynamical system and phase portrait}

Let us consider the system of differential equations
\be{XY}
X'=\aA\,Y-4\,X\,,\quad Y'=-\,\bB\,Y\,.
\ee
%---------------------------------------------------------------------
\begin{lemma}\label{Lem:XY} Let $d\ge1$, $m\in[m_1,1)$. The region
\be{Region}
X\ge-\,\J\quad\mbox{and}\quad Y\ge-\,\E
\ee
is stable under the action of the flow corresponding to~\eqref{XY}. Within this region, $\{Y\ge0\}$, $\{Y\le0\}$ and $\{X\ge 0,\,0\le Y\le4\,X/\aA\} $ are all stable. Moreover, for any solution of~\eqref{XY}, we have
\be{Energy}
\frac12\,\frac{\rd}{\dt}\((\aA\,Y-4\,X)^2+4\,\bB\,X^2\)=-(\bB+4)\,(\aA\,Y-4\,X)^2\,.
\ee
\end{lemma}
%---------------------------------------------------------------------
\begin{proof} If either $X=-\,\J$ or $Y=-\,\E$, the flow~\eqref{XY} is entering the region defined by~\eqref{Region}. The stability of the other regions follows from similar reasons. The energy estimate~\eqref{Energy} follows by a direct computation.\end{proof}
See Fig.~\ref{PhasePortrait}. The straight line $\aA\,Y-4\,X=0$ contains the points $(-\,\J,-\,\E)$ and $(0,0)$. Whether $Y>4\,X/a$ or not decides the sign of $X'$. Also notice that $\E\,\big(1+\E^{-1}\,X\big)^m-\E\ge4\,X/a$ if and only if $X\le0$. The next results deals with a special energy level.
%---------------------------------------------------------------------
\begin{lemma}\label{Lem:Ellipse} The ellipse defined by
\be{Ellipse}
(\aA\,Y-4\,X)^2+4\,\bB\,X^2=\frac{\aA^2\,\bB}{4+\bB}\,\E^2
\ee
has the following properties:
\begin{enumerate}
\item[(i)] it is contained in the region $X>-\,\J$ and $Y\ge-\E$,
\item[(ii)] it is tangent to the line $Y=-\,\E$ at $X=-\frac\aA{4+\bB}\,\E>-\,\J$,
\item[(iii)] it is tangent to the line $X=X_\star:=-\frac\aA{2\sqrt{4+\bB}}\,\E>-\,\J$ at a point which intersects the line $\aA\,Y-4\,X=0$, \emph{i.e.}, $(X,Y)=(X_\star,Y_\star)$ with $Y_\star=\frac4\aA\,X_\star$.
\end{enumerate}
\end{lemma}
%---------------------------------------------------------------------
This lemma follows from simple computations which are not detailed here. The motivation for studying~\eqref{XY} comes from the following comparison result.
%---------------------------------------------------------------------
\begin{lemma}\label{Lem:BasicXY} Let $d\ge1$, $m\in[m_1,1)$ and assume that $v_0\in\mathrm L^1_+(\R^d)$ is such that $\nrm{v_0}1=\Mstar$ and $(|x|^2\,v_0+v_0^m)\in\mathrm L^1(\R^d)$. Then the solution of~\eqref{FDr} with initial datum $v_0$ satisfies
\[\label{XYv}
\mathcal K(t)=\mathcal K[v(t,\cdot)]\ge X(t)\quad\mbox{and}\quad\mathcal S(t)=\mathcal S[v(t,\cdot)]\ge Y(t)\quad\forall\,t\ge0
\]
if $(X,Y)$ solves~\eqref{XY} with initial data
\[\label{XY0}
X(0)=\mathcal K[v_0]\,,\quad Y(0)=\mathcal E[v_0]\,.
\]\end{lemma}
%---------------------------------------------------------------------
\begin{proof} It is a consequence of $\delta\ge0$ and
\begin{align*}
&\frac{\rd}{\dt}\(e^{\bB\,t}\,\mathcal S(t)\)=m\,e^{\bB\,t}\,\delta\ge0=\frac d{dt}\(e^{\bB\,t}\,Y(t)\)\,,\\
&\frac{\rd}{\dt}\(e^{4\,t}\,\mathcal K(t)\)=\aA\,e^{4\,t}\,\mathcal S(t)\ge \aA\,e^{4\,t}\,Y(t)=\frac d{dt}\(e^{4\,t}\,X(t)\)\,.
\end{align*}
\end{proof}
%---------------------------------------------------------------------
\setlength\unitlength{1cm}
\begin{figure}[ht]
\begin{center}\vspace*{-0.4cm}\begin{picture}(12,6)
\put(3,0){\includegraphics[width=6cm]{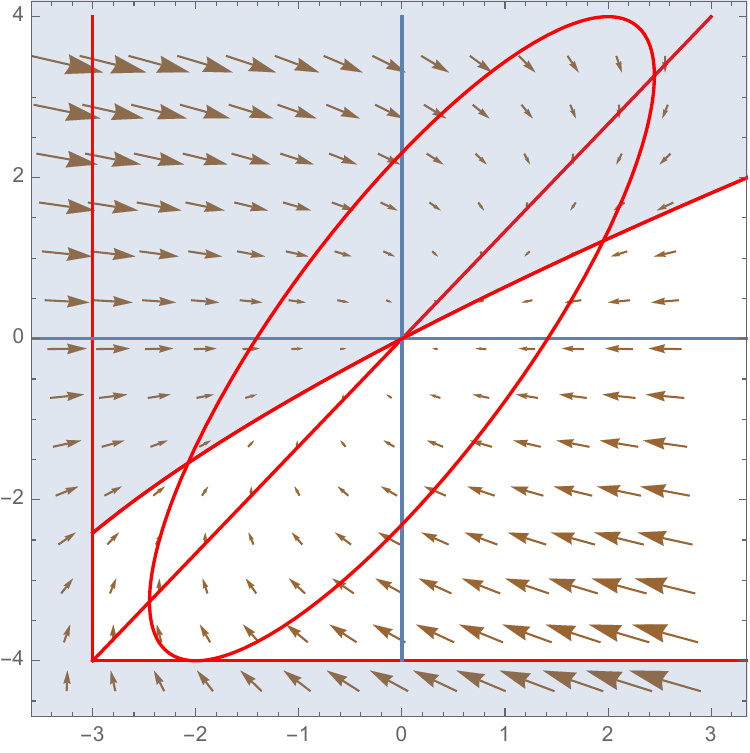}}
\put(9.25,3.5){$X$}
\thicklines
\put(6.215,0.25){\vector(0,1){6}}
\put(6.4,6.25){$Y$}
\put(3.22,3.31){\vector(1,0){6}}
\put(2.2,0.645){$-\E$}
\put(3.4,-0.25){$-\J$}
\end{picture}
\caption{\label{PhasePortrait} In $(X,Y)$ coordinates, arrows represent the vector field which corresponds to~\eqref{XY}, \emph{i.e.}, $(\aA\,Y-4\,X,-\,\bB\,Y)$. The line $\aA\,Y-4\,X=0$ determines the sign of $X'$. The ellipse defined by~\eqref{Ellipse} is also represented. White areas are the regions such that~\eqref{Constraints:KS} holds: the upper limit is $Y=\psi(X)$. Here we have $d=3$, $m=m_1=2/3$, $\aA=2\,d\,(1-m)/m=3$ and \hbox{$\bB=2\,d\,(m-m_c)=2\,\alpha$} and scales are in units of $\Mstar$. The qualitative properties are independent of the dimension $d\ge3$ and of the exponent $m\in[m_1,1)$.}
\end{center}\vspace*{-10pt}
\end{figure}
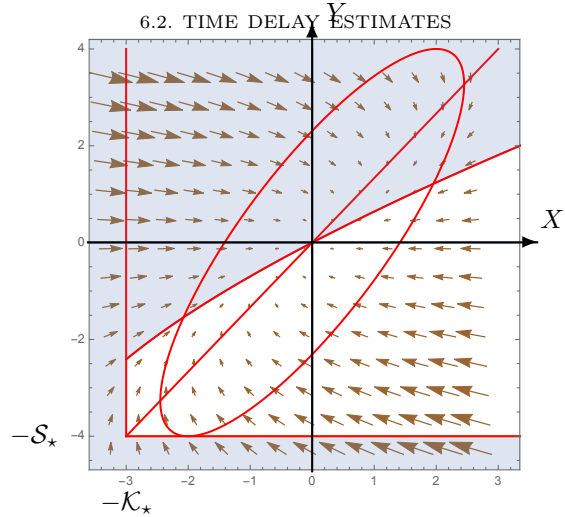
%---------------------------------------------------------------------
Let us define three regions:\\[4pt]
$\bullet$ Region A:
\[
-\,\J<X\le0\quad\mbox{and}\quad\frac4\aA\,X\le Y\le\psi(X)
\]
with $\psi$ as in~\eqref{UpperCurve}.\\[4pt]
$\bullet$ Region B:
\begin{multline*}
-\,\E\le Y<\min
\left\{\frac4\aA\,X,\psi(X)\right\}\\[4pt]
\mbox{and}\quad\Bigg\{\begin{array}{ll}
\mbox{either}\quad&\displaystyle(\aA\,Y-4\,X)^2+4\,\bB\,X^2\le\frac{\aA^2\,\bB}{4+\bB}\,\E^2\,,\\[4pt]
\mbox{or}\quad&\displaystyle X>-\frac\aA{4+\bB}\,.
\end{array}
\end{multline*}
$\bullet$ Region C:
\[
-\,\J<X<-\frac\aA{4+\bB}\,\E\,,\;-\,\E<Y<\frac4\aA\,X\;\mbox{and}\;(\aA\,Y-4\,X)^2+4\,\bB\,X^2>\frac{\aA^2\,\bB}{4+\bB}\,\E^2\,.
\]
See Fig.~\ref{Regions} for an illustration of the three regions.
%---------------------------------------------------------------------
\begin{figure}[ht]
\includegraphics[width=4cm]{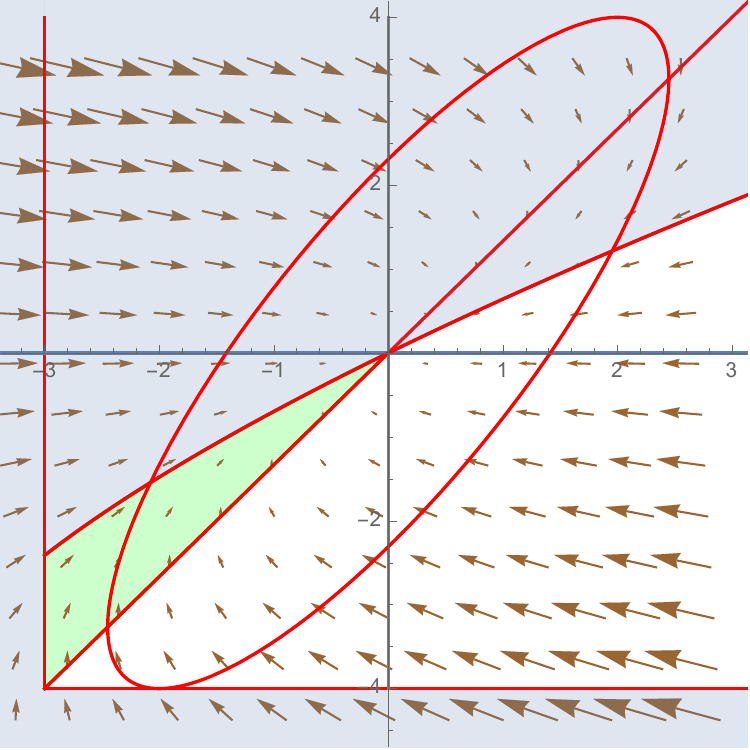}\hspace*{5pt}\includegraphics[width=4cm]{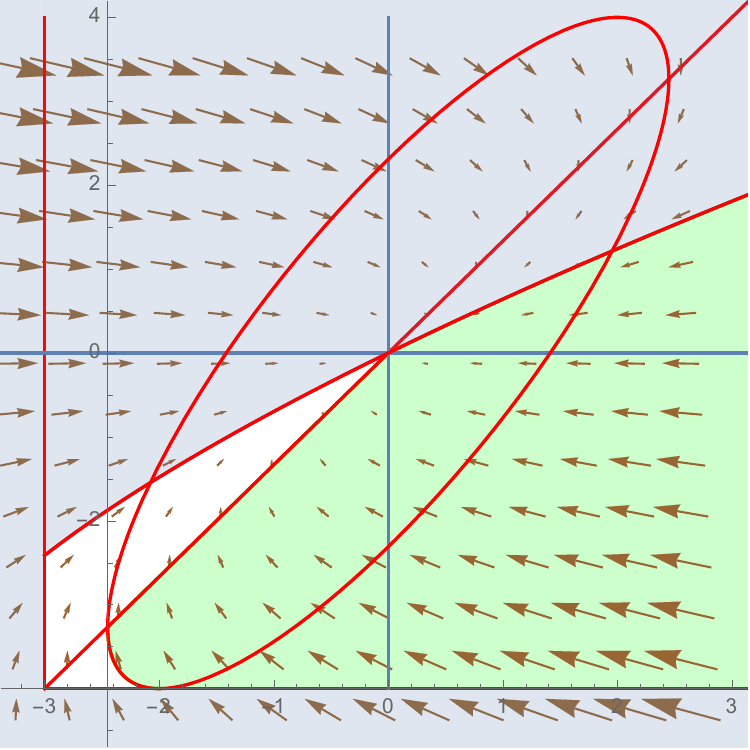}\hspace*{5pt}\includegraphics[width=4cm]{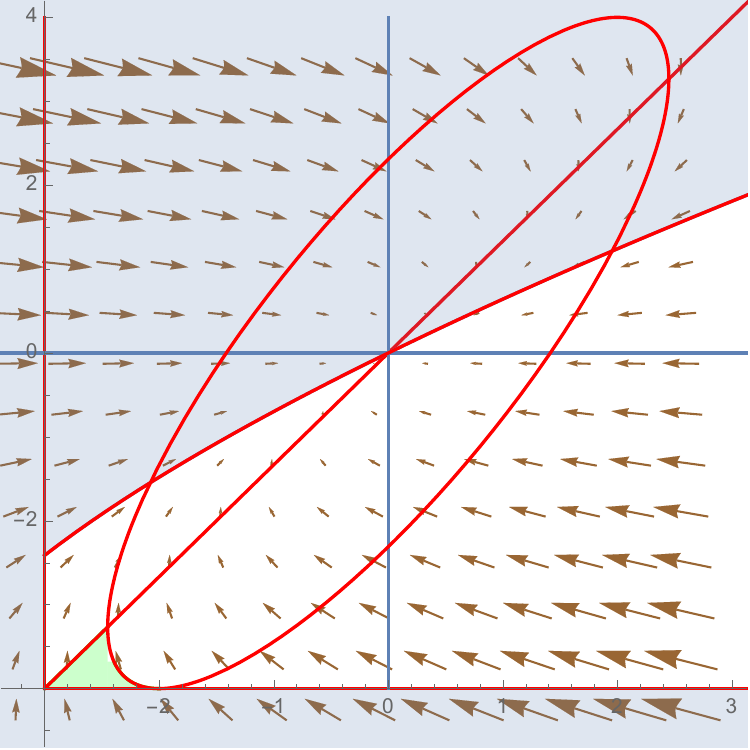}
\caption{\label{Regions} Regions A, B, C appear (from left to right) in green with $d=3$, $m=m_1=2/3$, $\aA=\frac{2\,d}m\,(1-m)$ and \hbox{$\bB=2\,d\,(m-m_c)$}.}
\end{figure}
%---------------------------------------------------------------------

%---------------------------------------------------------------------
\begin{corollary}\label{Cor:Regions} Let $d\ge1$, $m\in[m_1,1)$ and assume that $v_0\in\mathrm L^1_+(\R^d)$ is such that $\nrm{v_0}1=\Mstar$ and $(|x|^2\,v_0+v_0^m)\in\mathrm L^1(\R^d)$. Then the solution of~\eqref{FDr} with initial datum $v_0$ satisfies $\mathcal K[v(t,\cdot)]>-\,\J$ for all $t\ge 0$, and more precisely:
\begin{enumerate}
\item[(i)] If $(\mathcal K[v_0],\mathcal S[v_0])$ is in Region A, then
\[
\mathcal K[v(t,\cdot)]\ge\mathcal K[v_0]\quad\forall\,t\ge0\,.
\]
\item[(ii)] If $(\mathcal K[v_0],\mathcal S[v_0])$ is in Region B, then
\[
\mathcal K[v(t,\cdot)]\ge X_\star=-\frac\aA{2\sqrt{4+\bB}}\,\E\quad\forall\,t\ge0\,.
\]
\item[(iii)] If $(\mathcal K[v_0],\mathcal S[v_0])$ is in Region C, then
\[
\mathcal K[v(t,\cdot)]\ge-\frac1{\sqrt{\bB}}\,\sqrt{(4+\bB)\,\mathcal K[v_0]^2-2\,\aA\,\mathcal K[v_0]\,\mathcal S[v_0]+\aA^2\,\mathcal S[v_0]^2/4}\quad\forall\,t\ge0\,.
\]
\end{enumerate}
As a consequence, we have that for any $t \ge 0$
\be{Estim:Exact}
\mathcal K[v(t,\cdot)]\ge X(t)=\frac{1-m}m\,\mathcal F(0)\,e^{-4t}+\frac1{4\,m}\,\mathcal S(0)\(3\,e^{-4t}+e^{-2\,\alpha\,t}\)>-\,\J\,. %\quad\forall\,t\ge0\,.
\ee
\end{corollary}
%---------------------------------------------------------------------
\begin{proof} The proof follows from~\eqref{Energy}, the properties of Lemma~\ref{Lem:Ellipse} and Lem\-ma~\ref{Lem:BasicXY}, with elementary but tedious computations. The lower bound in Region C is obtained by writing
\[
\big(\aA\,Y(t)-4\,X(t)\big)^2+4\,\bB\,X(t)^2\le\big(\aA\,Y(0)-4\,X(0)\big)^2+4\,\bB\,X(0)^2
\]
for any $t\ge0$ and looking for the value of a vertical tangent to the ellipse corresponding to the equality case. Inequality~\eqref{Estim:Exact} is obtained by an explicit integration of~\eqref{XY}
\be{Estim:Exact-special}
X(t)=X(s)\,e^{-4(t-s)}+\frac{\aA}{4-\bB}\(e^{-\bB(t-s)}-e^{-4(t-s)}\)Y(s)\quad\forall\,t\ge s\ge0\,,
\ee
using $\aA/(4-\bB)=1/(4\,m)$ and $m\,\mathcal K-\mathcal S=(1-m)\,\mathcal F$, $X(0)=\mathcal K(0)$, $Y(0)=\mathcal S(0)$. \end{proof}
Depending whether $(\mathcal K[v_0],\mathcal S[v_0])$ is in Region A, B and C respectively, we define
\be{Kbullet}
\mathcal K_\bullet:=\left\{\begin{array}{ll}
\mathcal K[v_0]\quad&\mbox{in Region A}\,,\\[4pt]
X_\star=-\frac\aA{2\sqrt{4+\bB}}\,\E>-\,\J\quad&\mbox{in Region B}\,,\\[4pt]
-\frac1{\sqrt{\bB}}\,\sqrt{(4+\bB)\,\mathcal K[v_0]^2-2\,\aA\,\mathcal K[v_0]\,\mathcal S[v_0]+\aA^2\,\mathcal S[v_0]^2/4}\quad&\mbox{in Region C}\,.
\end{array}
\right.\ee
We learn from Corollary~\ref{Cor:Regions} that $\tau$ defined according to~\eqref{ODE-6} by
\[
\frac{\rd\tau}{\dt}=\(1+\tfrac{\mathcal K}\J\)^{-\alpha/2}-1\,,\quad\tau(0)=0\quad\forall\,t\ge0\,,
\]
is well defined on $\R^+$ and
\be{tau1st}
-\,t<\tau(t)\le\left[\(1+\tfrac{\mathcal K_\bullet}\J\)^{-\alpha/2}-1\right]t\quad\forall\,t>0\,.
\ee
As a special case, if $\mathcal K[v_0]=0$, we have $-\frac{\mathcal K(t)}{\J}\le\frac{ X_\star}{\J}=\frac2{\sqrt{4+\bB}}=\frac1{\sqrt{1+\alpha}}$ and
\[\label{tau1st-special}
-\,t<\tau(t)\le\left[\(1-\frac1{\sqrt{1+\alpha}}\)^{-\alpha/2}-1\right]\,t\quad\forall\,t>0\,.
\]

%%%%%%%%%%%%%%%%%%%%%%%%%%%%%%%%%%%%%%%%%%%%%%%%%%%%%%%%%%%%%%%%%%%%%%
\subsection{A bound on the \idx{delay} \texorpdfstring{$\tau$}{tau}}
%---------------------------------------------------------------------
\begin{theorem}\label{Thm:tau} Let $d\ge3$ and $m\in[m_1,1)$. Assume that $v_0$ is a nonnegative function in $\mathrm L^1\big(\R^d,(1+|x|^2)\dx\big)$ such that $v_0^m\in\mathrm L^1(\R^d,\dx)$ and $\ird{v_0}=\Mstar$. We consider the solution $v$ of~\eqref{FDr} with initial datum $v_0$ and $\tau$ defined by~\eqref{ODE-6}. Then we have
\[
|\tau(t)|\le\max\left\{1,\left[\(1+\tfrac{\mathcal K_\bullet}\J\)^{-\alpha/2}-1\right]\right\}t_1+\frac{\J}8\quad\forall\,t\ge0\,,
\]
where $\mathcal K_\bullet$ is defined by~\eqref{Kbullet}, 
\[\label{t1}
t_1:=\frac1{2\,\alpha}\,\log\(\max\left\{1,\frac{|\max\{0,\mathcal S(0)\}|}{2\,m\,\J}\right\}\)\,.
\]
If additionally $\mathcal K[v_0]=0$, then $\tau$ is uniformly bounded by a constant $\tau_\bullet$ which depends only on $m$ and $d$. \end{theorem}
%---------------------------------------------------------------------
\begin{proof} We deduce from~\eqref{Estim:Exact} that $X(t)\ge-\,\J/2$ if
\[
\frac1m\,\mathcal S(0)\,e^{-2\,\alpha\,t}>-\,\frac12\,\J\,,
\]
that is, for any $t\ge t_1$. Using~\eqref{tau1st}, we know that
\[
|\tau(t_1)|\le\max\left\{1,\left[\(1+\tfrac{\mathcal K_\bullet}\J\)^{-\alpha/2}-1\right]\right\}t_1\,.
\]

For any $t\ge t_1$, we deduce from~\eqref{Estim:Exact-special}
\[
\mathcal K(t)\ge X(t)=\frac{1-m}m\,\mathcal F(t_1)\,e^{-4\,(t-t_1)}+\frac1m\,\mathcal S(t_1)\,e^{-2\,\alpha\,(t-t_1)}>-\,\frac12\,\J\,e^{-2\,\alpha\,(t-t_1)}
\]
because $1\le\alpha\le2$ and $\max\left\{e^{-4\,(t-t_1)},e^{-2\,\alpha\,(t-t_1)}\right\}=e^{-2\,\alpha\,(t-t_1)}$. After taking into account~\eqref{ODE-6}, we find that
\[
\frac{\rd\tau}{\dt}\le\(1-\frac12\,e^{-2\,\alpha\,(t-t_1)}\)^{-\alpha/2}-1\quad\forall\,t\ge t_1\,,
\]
which gives the estimate
\[
|\tau(t)|\le|\tau(t_1)|+\frac{\mathsf{c}_\alpha}2\,\J\quad\forall\,t\ge0
\]
with 
\[
\mathsf{c}_\alpha:=\int_0^{+\infty}\(\(1-\tfrac12\,e^{-2\,\alpha\,s}\)^{-\alpha/2}-1\)\,\mathrm d s<\frac\alpha2\int_0^{+\infty}e^{-2\,\alpha\,s}\,\mathrm d s=\frac14<+\infty
\]
because $\alpha \le 2$. Also notice that $\mathcal S(0)\ge-\,\E=4\,\J/\aA$ so that $t_1$ is bounded uniformly for given $m$ and~$d$ by
\[
t_1\le\frac1{2\,\alpha}\,\log\(\max\left\{1,\frac4{d\,(1-m)}\right\}\)\,.
\]
Hence, if $\mathcal K[v_0]=0$, we have that for all $t\ge0$
\be{taubullet}
|\tau(t)|\le\max\left\{1,\left[\(1-\frac1{\sqrt{1+\alpha}}\)^{-\alpha/2}-1\right]\right\}t_1+ \frac{\J}8=:\tau_\bullet\,,
\ee
which concludes the proof.\end{proof}

%%%%%%%%%%%%%%%%%%%%%%%%%%%%%%%%%%%%%%%%%%%%%%%%%%%%%%%%%%%%%%%%%%%%%%
%%%%%%%%%%%%%%%%%%%%%%%%%%%%%%%%%%%%%%%%%%%%%%%%%%%%%%%%%%%%%%%%%%%%%%
\section{Entropy-entropy production inequalities and refinements}\label{sec:proofs}

We revisit the results of Chapter~\ref{Chapter-2} in the framework of the solutions of~\eqref{FDR} instead of the solutions of~\eqref{FDr}.

%%%%%%%%%%%%%%%%%%%%%%%%%%%%%%%%%%%%%%%%%%%%%%%%%%%%%%%%%%%%%%%%%%%%%%
\subsection{The quotient reformulation}\label{Sec:EEP}

For any given $\lambda>0$, let us define
\[
\mathcal Q_\lambda[w]:=m\,\lambda^{\frac d2\,(m-m_c)}\,\frac{\ird{w\,\left|\nabla w^{m-1}-\nabla\mB_\lambda^{m-1}\right|^2}}{\ird{\(\mB_\lambda^m-w^m\)}}\,.
\]
Changing variables with
\be{wlambda}
w_\lambda(x):=\lambda^{d/2}\,w\big(\sqrt\lambda\,x\big)\,,
\ee
we notice that
\[
\mathcal Q_\lambda[w]=\mathcal Q_1[w_\lambda]=\frac{\mathcal I[w_\lambda]}{\mathcal F[w_\lambda]}\ge4
\]
as a consequence of~\eqref{entropy.eep}. This has already been noticed in~\cite{Dolbeault2011a}. In particular, the \idx{carr\'e du champ} method has been used in~\cite{Dolbeault2013917} and shows, among other results, that
\be{EqQ-6}
\frac{\rd\mathcal Q_\lambda}{\dt}\le\mathcal Q_\lambda\,(\mathcal Q_\lambda-4)\,.
\ee
where, for brevity, we write $\mathcal Q_\lambda(t)=\mathcal Q_\lambda[w(t,\cdot)]$ and $w$ solves~\eqref{FDR}.

%%%%%%%%%%%%%%%%%%%%%%%%%%%%%%%%%%%%%%%%%%%%%%%%%%%%%%%%%%%%%%%%%%%%%%
\subsection{Initial time layer}

With almost no change except that $m\in[m_1,1)$, Equation~\eqref{FDr} being replaced by~\eqref{FDR} and $\mathcal Q_1$ by $\mathcal Q_\lambda$, the \idx{initial time layer} estimate of Lemma~\ref{prop.backward} applies: on the interval $(0,T)$, we have a uniform positive lower bound on \hbox{$\mathcal Q_\lambda[w(t,\cdot)]-4$} if we know that $\mathcal Q_\lambda[w(T,\cdot)]-4>0$.
%---------------------------------------------------------------------
\begin{lemma}\label{prop.backward-6} Let $d\ge3$ and $m\in[m_1,1)$. Assume that $w$ is a solution to~\eqref{FDR} with nonnegative initial datum $v_0\in\mathrm L^1(\R^d)$ such that $\mathcal F[v_0]<+\infty$ and $\ird{v_0}=\Mstar$. With the notation of Section~\ref{Sec:EEP}, if for some $\eta>0$ and $T>0$, we have $\mathcal Q_\lambda(T)\ge4+\eta$, then we also have
\[\label{backward.EEP-6}
\mathcal Q_\lambda(t)\ge\,4+\frac{4\,\eta\,e^{-4\,(T-t)}}{4+\eta-\eta\,e^{-4\,(T-t)}}\quad\forall\,t\in[0,T]\,.
\]\end{lemma}
%---------------------------------------------------------------------
\begin{proof} Exactly as for Lemma~\ref{prop.backward}, it is an easy consequence of a backward estimate based on~\eqref{EqQ-6}. \end{proof}
As a consequence, we obtain a uniform estimate on the \idx{initial time layer}. Under the assumptions of Lemma~\ref{prop.backward-6}, with $\zeta=\frac{4\,\eta\,e^{-4\,T}}{4+\eta-\eta\,e^{-4\,T}}$, we obtain
\[
\mathcal Q_\lambda(t)\ge4+\zeta\quad\forall\,t\in[0,T]\,.
\]

%%%%%%%%%%%%%%%%%%%%%%%%%%%%%%%%%%%%%%%%%%%%%%%%%%%%%%%%%%%%%%%%%%%%%%
\subsection{Asymptotic time layer}

As in Section~\ref{Sec:FDE-Asymptotic}, let us consider the \emph{\idx{linearized free energy}} $h\mapsto\mathsf F[h]$ and the \emph{\idx{linearized Fisher information}} $h\mapsto\mathsf I[h]$. By the \emph{\idx{Hardy-Poincar\'e inequality}}~\eqref{Hardy-Poincare}, for any $h\in\mathrm L^2(\R^d,\mB^{2-m}\dx)$ such that $\nabla h\in\mathrm L^2(\R^d,\mB\dx)$ and $\ird{h\,\mB^{2-m}}=0$, we have
\be{HP-PNAS-6}
\mathsf I[h]\ge4\,\alphaa\,\mathsf F[h]
\ee
with $\alphaa=1$. The \emph{improved \idx{Hardy-Poincar\'e inequality}} holds with $\alphaa=2-d\,(1-m)$ under the additional constraint $\ird{x\,h\,\mB^{2-m}}=0$ if $m\in(m_1,1)$ but this constraint provides no improvement on the spectral gap in the limit case $m=m_1$, as we still have $\alphaa=1$. See~\eqref{HP-PNAS} for more details. Finally, if $m=m_1$ and $h\in\mathrm L^2(\R^d,\mB^{2-m}\dx)$ is such that $\nabla h\in\mathrm L^2(\R^d,\mB\dx)$ and $\ird{\(1,x,|x|^2\)h\,\mB^{2-m}}=0$, according to~\cite[Corollary~2]{Dolbeault2011a}, Inequality~\eqref{HP-PNAS-6} holds with
\be{Imp:SG}
\alphaa=\frac{(d+2)^2}{8\,d}\quad\mbox{if}\quad3\le d\le6\quad\mbox{and}\quad\alphaa=2\,\frac{d-2}d\quad\mbox{if}\quad d\ge6\,.
\ee
As in Proposition~\ref{Prop:Gap}, the improved spectral gap in~\eqref{HP-PNAS-6} can be used to establish an improved lower bound for $\mathcal Q_{\lambda(t)}[w(t,\cdot)]$ in the \idx{asymptotic time layer}, as $t\to+\infty$. 
%---------------------------------------------------------------------
\begin{proposition}\label{Prop:Gap-6} Let $d\ge3$, $m=m_1$ and let us define
\[\label{eta-6}
\eta:= \frac{(d-2)^2}{8\,d}\quad\mbox{if}\quad3\le d\le6\quad\mbox{and}\quad\eta :=2\,\frac{d-4}{d}\quad\mbox{if}\quad d\ge6\,,
\]
and $\chi=1/580$. If $w$ is a nonnegative solution to~\eqref{FDR} with initial datum as in~Lemma~\ref{Lem:BestMatching} such that $\ird{v_0}=\Mstar$, $\ird{x\,v_0}=0$ and
\be{Uniform-6}
(1-\varepsilon)\,\mB_{\lambda(t)}(x)\le w(t,x)\le(1+\varepsilon)\,\mB_{\lambda(t)}(x)\quad\forall\,x\in\R^d\,,\quad\forall\,t\ge T
\ee
for some $\varepsilon\in(0,\chi\,\eta)$ and $T>0$, then we have
\be{Prop.Rayleigh.Ineq-6}
\mathcal Q_{\lambda(t)}[w(t,\cdot)]\ge4+\eta\quad\forall\,t\ge T\,.
\ee
\end{proposition}
%---------------------------------------------------------------------
\begin{proof} Since~\eqref{Uniform-6} and~\eqref{Prop.Rayleigh.Ineq-6} are pointwise in $t$, we can write $\lambda=\lambda(t)$ and forget about the dependence in $t$. Up to the rescaling~\eqref{wlambda}, we follow the proof of Proposition~\ref{Prop:Gap}. Let $\mathcal F_\lambda[w]:=\mathcal F[w_\lambda]$, $\mathcal I_\lambda[w]:=\mathcal I[w_\lambda]$ and consider the linearized functionals
\[
\mathsf F_\lambda[h]:=\frac{m}2 \ird{|h|^2\,\mB_\lambda^{2-m}}\quad\mbox{and}\quad \mathsf I_\lambda[h]:=m \,(1-m)\,\lambda^\frac\alpha2\,\ird{|\nabla h|^2\,\mB_\lambda}\,.
\]
Changing variables in~\eqref{HP-PNAS-6}, we obtain as in~\cite[Corollary~3]{Dolbeault2011a} the inequality
\[
\lambda^\frac\alpha2\,\mathsf I_\lambda[h]\ge4\,\alphaa\,\mathsf F_\lambda[h]\quad\forall\,h\in\mathrm L^2(\R^d,\mB_\lambda^{2-m}\dx)\;\mbox{s.t.}\ird{\(1,x,|x|^2\)h\,\mB_\lambda^{2-m}}=0\,.
\]
Let us consider $\h:=v\,\mB_\lambda^{m-2}-\mB^{m-1}_\lambda$. As in Proposition~\ref{Prop:Gap}, we use~\cite[Lemma~3]{Blanchet2009} to get that under Assumption~\eqref{Uniform-6}, we have
\[
(1+\varepsilon)^{-b}\,\mathsf F_\lambda[h]\le\mathcal F_\lambda[w]\le(1-\varepsilon)^{-b}\,\mathsf F_\lambda [h]
\]
where $b=2-m$ and, with $s_1$ and $s_2$ as in~\eqref{s1-s2},
\[
\mathsf I_\lambda[h] \le s_1(\varepsilon)\,\mathcal I_\lambda[w] + s_2(\varepsilon)\,\mathsf F_\lambda[h]\,.
\]
The remainder of the proof is identical to the one of Proposition~\ref{Prop:Gap}, up to the replacement of $\alphaa=1$ by $\alphaa$ given by~\eqref{Imp:SG}. \end{proof}

%%%%%%%%%%%%%%%%%%%%%%%%%%%%%%%%%%%%%%%%%%%%%%%%%%%%%%%%%%%%%%%%%%%%%%
%%%%%%%%%%%%%%%%%%%%%%%%%%%%%%%%%%%%%%%%%%%%%%%%%%%%%%%%%%%%%%%%%%%%%%
\section{Uniform convergence in relative error}\label{Sec:6-Harnack}

\index{uniform convergence in relative error}{After} incorporating the additional rescaling corresponding to the \idx{delay} $\tau(t)$ in our computations, we have the \index{uniform convergence in relative error}{convergence in relative error} of the solution $w$. Let us define $q:=(d+2)^{-1}\,2^{1-d/2}$ and
\be{Tplustau}
T(\varepsilon, A):= \frac1{2\,\alpha}\log\(1+\alpha\,\mathfrak{c}_\star\,\frac{1+A^{1-m}}{\varepsilon^{\mathsf a}}\)
\ee
where
\be{mathfrakcstar}
\mathfrak{c}_\star:=\alpha\,\taustar\,q^{-\mathsf a}\,\(1+\frac{\E}{1-m}\)^\frac\alpha{2}\,\(1+e^{2\,\alpha\,\tau_\bullet}\),
\ee
with $\taustar$ as in~\eqref{Tstar}, $\mathsf a$ as in Theorem~\ref{Thm:RelativeUniform}, $\E$ as in~\eqref{sec:defandcon} and $\tau_\bullet$ as in~\ref{taubullet}. We observe that $\mathfrak{c}_\star$ is such that $T(\varepsilon, A)> T_\star(q\,\varepsilon, A, \E/(1-m))+\tau_\bullet$ for any $\varepsilon$, $A >0$.
%---------------------------------------------------------------------
\begin{theorem}\label{Thm:MainRefinedHarnack} 
Let $d\ge3$, $m\in[m_1,1)$ and $\varepsilon\in(0, \chi\,\eta)$. Assume that $v_0$ is a nonnegative function in $\mathrm L^1\big(\R^d,(1+|x|^2)\dx\big)$ such that
\be{Cdt:InitialDatum}
\ird{(1,x, |x|^2)\,v_0}=\ird{(1,x, |x|^2)\,\mB}\quad\mbox{and}\quad
\sup_{r>0}r^d\int_{|x|>r}v_0\dx\le A\,.
\ee
If $v$ solves~\eqref{FDr} with initial datum $v_0$ and $w$ is obtained by~\eqref{sigma-6},~\eqref{ODE-6} and~\eqref{ChVar}, then 
\be{CdtVarepsilon1}
(1-\varepsilon)\,\mB_\star(s,x)\le w(s,x)\le(1+\varepsilon)\,\mB_\star(s,x)\quad\forall\,x\in\R^d\,,\quad\forall\,s\ge T\,,
\ee
where $T$ is as in~\eqref{Tplustau}.
\end{theorem}
%---------------------------------------------------------------------
\begin{proof} Using
\[
1-\frac w{\mB_\star}=\(1-\frac w{\mB_{\Rr(t)}}\)\frac{\mB_{\Rr(t)}}{\mB_\star}+\(1-\frac{\mB_{\Rr(t)}}{\mB_\star}\)
\]
and $\nrm{1-v/\mB}\infty=\nrm{1-w/\mB_{\Rr(t)}}\infty$ because of the change of variables~\eqref{ChVar}, we have the estimate
\be{finq}
\nrm{\frac{w(s)-\mB_\star(s)}{\mB_\star(s)}}\infty\le\nrm{\frac{v(t)-\mB}\mB}\infty\,\nrm{\frac{\mB_{\Rr(t)}}{\mB_\star(s)}}\infty+\nrm{\frac{\mB_{\Rr(t)}-\mB_\star(s)}{\mB_\star(s)}}\infty\,,
\ee
where $\mB_{\Rr(t)}$ is as in~\eqref{rescaled-second-moment}. Notice that $\nrm{{\mB_{\Rr(t)}}/{\mB_\star}-c}\infty=\nrm{{\mB_{\sqrt\lambda(t)\,\Rr(t)}}/{\mB}-c}\infty$ for $c=0$ and $c=1$. A simple computation shows that the supremum and the infimum in $\mB_{\sqrt\lambda(t)\,\mathfrak R(t)}/{\mB}$ are achieved either at the origin or as a limit for $|y|\rightarrow\infty$. As a consequence, we deduce that
\begin{multline*}
\nrm{\frac{\mB_{\Rr(t)}}{\mB_\star(s)}}\infty=\max\Big\{\sqrt\lambda(t)\,\Rr(t), \frac1{\sqrt\lambda(t)\,\Rr(t)}\Big\}^{d}\quad\mbox{and}\\
\nrm{\frac{\mB_{\Rr(t)}}{\mB_\star(s)}-1}\infty=1-\min\Big\{\(\sqrt\lambda(t)\,\Rr(t)\)^d, \(\sqrt\lambda(t)\,\Rr(t)\)^{-d}\Big\}\,.
\end{multline*}
Recall that $q=(d+2)^{-1}\,2^{1-d/2}$. The inequality $\mathcal F[v_0]= S[v_0]/(m-1) \le \E/(1-m)$ holds since $\mathcal K[v_0]=0$. As a consequence, we deduce from inequality~\eqref{uniformFDr} and identity~\eqref{sigma-6} that,
\be{inq.1}
1-q\,\varepsilon \le \Rr^2(t)\,\lambda(t) \le 1+q\,\varepsilon\quad\forall\,t\ge T_\star=T_\star\(q\,\varepsilon, A, \frac{\E}{1-m}\)\,,
\ee
where $T_\star$ is as in~\eqref{Tstar}. Since the functions $x\mapsto1-\left(1\mp x\right)^{\pm d/2}$ are concave, non-decreasing functions, by a Taylor expansion around $0$ and from~\eqref{inq.1}, we deduce
\[
\nrm{\frac{\mB_{\Rr(t)}}{\mB_\star(s)}-1}\infty\le\frac12\,d\,\varepsilon\,q\quad\mbox{and}\quad
\nrm{\frac{\mB_{\Rr(t)}}{\mB_\star(s)}}\infty\le2^\frac{d}2\,.
\]
Combining the above two inequalities with~\eqref{finq}, we obtain~\eqref{CdtVarepsilon1} for any $s\ge T_\star(q\,\varepsilon, A,\E/(1-m))+\tau(T_\star)$. Since $K[v_0]=0$, we have the estimate $\tau(T_\star)\le \tau_\bullet$, as a consequence we have that $T> T_\star(q\,\varepsilon, A, \E/(1-m))+\tau_\bullet$ and the proof is completed.
\end{proof}

%%%%%%%%%%%%%%%%%%%%%%%%%%%%%%%%%%%%%%%%%%%%%%%%%%%%%%%%%%%%%%%%%%%%%%
%%%%%%%%%%%%%%%%%%%%%%%%%%%%%%%%%%%%%%%%%%%%%%%%%%%%%%%%%%%%%%%%%%%%%%
\section{Computation of the stability constant}\label{Sec:6-Constant}

\begin{proof}[Proof of Theorem~\ref{Thm:Main}] Let us assume first that 
$\mathsf g_f=\mathsf g$. Exactly as in Section~\ref{Sec:Stability.improvedEEP2}, we obtain the improved \idx{entropy - entropy production inequality}~\eqref{improved.entropy.eep} with~$\eta$ as in Proposition~\ref{Prop:Gap-6}, $\zeta=\frac{4\,\eta\,e^{-4\,T}}{4+\eta-\eta\,e^{-4\,T}}$ as in Lemma~\ref{prop.backward-6} and $T$ given by~\eqref{Tplustau}. This proves that $\mathcal I[v]\ge(4+\zeta)\,\mathcal F[v]$ as in~\eqref{improved.entropy.eep}. As in the proof of Theorem~\ref{Thm:MainCh5}, from the expression on $\zeta$ we deduce that
\[
\zeta \ge\big(1\!+\!A^{1/(2\,d)}\big)^{-1}\frac{4\,\eta}{4+\eta}\(\frac{\eta^{\mathsf a}\,\chi^{\mathsf a}}{2^{1+\mathsf a}\,\alpha\,\mathfrak{c}_\star}\)^{2/\alpha}\,c_\alpha\,,
\]
where $\varepsilon_\star=\eta\,\chi/2$ where $\eta$ and $\chi$ are as in Proposition~\ref{Prop:Gap-6}, $\mathsf a$ as in Theorem~\ref{Thm:RelativeUniform}, $\mathfrak{c}_\star$ as in~\eqref{Tplustau}, $c_\alpha$ as in~\eqref{calpha}. From the identity~\eqref{deficit-entropy-production}, we deduce that inequality~\eqref{Stability-entropy} holds with 
\[
\mathfrak C_\star=(d-1)^{-1}\,\frac{4\,\eta}{4+\eta}\(\frac{\eta^{\mathsf a}\,\chi^{\mathsf a}}{2^{1+\mathsf a}\,\alpha\,\mathfrak{c}_\star}\)^{2/\alpha}\,c_\alpha\,.
\]
Inequality~\eqref{stability-fisher} can be deduced from~\eqref{Stability-entropy} as in Section~\ref{Sec:Stability.improvedDeficit}. 
\end{proof}

\begin{proof}[Proof of Corollary~\ref{cor:main1}] We proceed as in Section~\ref{Stab:RelFisher}. Let us define the normalized function $\mathsf N f$ as
\[
\mathsf N f(x):=\lambda[f]^\frac{2-d}2\,\mu[f]^\frac{2-d}{2d}\,f\big(x_f+\,x/\lambda[f]\big)\quad\forall\,x\in\R^d
\]
so that $\mathsf g_{\mathsf N f}=\mathsf g$. From inequalities~\eqref{Stability-entropy} and~\eqref{stability-fisher}, we deduce that
\[
\delta[\mathsf N f]\ge \mathcal C_\star(A)\,\mathcal E[\mathsf N f | \mathsf g]\quad\mbox{and}\quad \delta[\mathsf N f] \ge \mathcal C_\star(A)/\(4+\mathcal C_\star(A)\)\mathcal J[\mathsf N f | \mathsf g]\,.
\]
Inequalities~\eqref{stability-general-entropy} and~\eqref{stability-general-fisher} follow from 
\begin{multline*}
\mathcal E[\mathsf N f | \mathsf g]=\frac{\lambda[f]^\frac12}{\mu[f]^\frac1d}\,\mathcal E[f | \mathsf g_f]\,,\quad\mathcal J[\mathsf N f| \mathsf g] = \mu[f]^\frac{2-d}{d}\,\mathcal J[f | \mathsf g_f]\\
\quad\mbox{and}\quad A[\mathsf N f]=\mu[f]^{-d}\,\lambda[f]^{-d}\,A[f]\,,\quad\delta[f]=\mu[f]^\frac{d-2}d\,\delta[\mathsf N f]\,.
\end{multline*}
\end{proof}

\begin{proof}[Proof of Proposition~\ref{Prop:bestFisher}]
Let $(\lambda,\mu,y)\in(0,+\infty)\times(0,+\infty)\times\R^d$ and recall that $g_{\lambda,\mu,y}(x)=\lambda^\frac d{2p}\,\mu^\frac1{2p}\,\mathsf g\big(\lambda\,(x-y)\big)$, so that $\nabla g_{\lambda, \mu, y}^{1-p^\star}=2\,\lambda\,\mu^{-\frac1d}\,(x-y)$. By expanding the square in~$\mathcal J\big[|f|\,|\,\mathsf g_{\lambda, \mu, y}\big]$ and integrating by parts $(x-y)\,|f|^{p^\star}\,\nabla f$ we get
\begin{multline*}
\frac{(d-2)^2}{4\,(d-1)}\mathcal J\big[|f|\,|\,\mathsf g_{\lambda, \mu, y}\big] = \nrm{\nabla f}{2}^2\\
+(d-2)^2\,\lambda^2\,\mu^{-\frac2d}\,\ird{|f|^{2\,p^\star}\,|x-y|^2}-\frac{d\,(d-2)^2}{d-1}\,\lambda\,\mu^{-\frac1d}\,\ird{|f|^{p^\star+1}}\,.
\end{multline*}
A minimization in $y$ reveals that, for any $\lambda$, $\mu>0$ the minimum is attained at $y=y[f]=x_f$ where $y[f]=x_f$ are as in~\eqref{lambdamuy}. Let us call $f(z)=z^2\,(d-2)^2\,\ird{|f|^{2\,p^\star}\,|x-x_f|^2}-z\,d\,(d-2)^2\,(d-1)^{-1}\,\ird{|f|^{p^\star+1}} +\nrm{\nabla f}{2}^2$. A minimization in $z$ reveals that $f(x)$ attains its minimum at 
\[
z=\frac{d}{2\,(d-1)}\,\frac{\ird{|f|^{p^\star+1}}}{\ird{|f|^{2\,p^\star}\,|x-x_f|^2}}\,.
\]
Since $\delta[f]=\delta\big[|f|\big]$ and $\mathcal J\big[|f|\,|\,\mathsf g_{\lambda, \mu, y}\big]=\mathcal J\big[f\,|\,\mathsf g_{\lambda, \mu, y}\big]$, the result follows.
\end{proof}

%%%%%%%%%%%%%%%%%%%%%%%%%%%%%%%%%%%%%%%%%%%%%%%%%%%%%%%%%%%%%%%%%%%%%%
%%%%%%%%%%%%%%%%%%%%%%%%%%%%%%%%%%%%%%%%%%%%%%%%%%%%%%%%%%%%%%%%%%%%%%
\section{Concluding comments}\label{Sec:6-Comments}

In the setting of Chapter~\ref{Chapter-6}, we can handle not only the critical case $m=m_1$, but also any $m\in[m_1,1)$. We do not insist so much on this point as $m\in(m_1,1)$ is already covered in Chapter~\ref{Chapter-5}, with simpler estimate. The estimates of Chapter~\ref{Chapter-5} fail in the limit as $m\to m_1$, as a consequence of the expression of $\zeta_\star$ defined by~\eqref{zetastar}. The bigger advantage of the setting of Chapter~\ref{Chapter-6}, is that we obtain for the constant an expression which has a finite, positive limit as $m\to m_1$, to the price of a more complicated setting.

The estimate of $t_1$ in Theorem~\ref{Thm:tau} is rather rough. Actually, from Lemma~\ref{Lem:Gronwall}, one can easily notice that a solution of the \idx{fast diffusion equation}~\eqref{FD} has a second moment which grows at least like $t^{2/\alpha}$ as $t\to+\infty$, which in turns provides upper and lower estimates of $t_1$. This is exactly the rate of growth of the second moment of the self-similar solutions of~\eqref{FDr}, certainly not a surprise. In original variables, one can analogously introduce a \idx{delay} for the \idx{best matching} self-similar solution and prove a similar monotonicity property, cf.~\cite{Dolbeault2013917,1751-8121-48-6-065206,Dolbeault_2016}. This latter \idx{delay} is related with the function $t\mapsto\tau(t)$. However, the \index{self-similar variables}{self-similar change of variables} makes things rather involved and the analysis requires some care, as seen in Section~\ref{sec:new}.

In any case the overall message is somewhat simple. When studying the intermediate asymptotic of~\eqref{FD}, we are performing an asymptotic expansion of the solutions around the self-similar profiles. By fixing mass and center of mass, quantities preserved along the \index{fast diffusion equation}{fast diffusion flow}, we are able to use improved weighted Poincar\'e inequalities by killing the lowest modes. In this way, we obtain faster rates of convergence towards \index{Barenblatt function}{Barenblatt profiles}, which are equivalent to the \idx{stability} estimates. This is the strategy adopted in Chapter~\ref{Chapter-5}, and it breaks down at $m=m_1$ since the energy levels of the linearized problem associated to translations and to scalings coincide. In order to get an improved Poincar\'e inequality if $m=m_1$, we need the second moments to be preserved along the nonlinear flow. Unfortunately, this is not the case for equation~\eqref{FDr}. We need to find another flow, adjusting the time scale with a \idx{delay} as in~\eqref{ODE-6}. This fits the time scale to preserve the second moments and kill the next term in the linearized problem. The key observation here is that a time-shift in a \index{Barenblatt function}{Barenblatt} self-similar solution amounts to a rescaling. This is the reason why the rescaling~\eqref{ODE-6} results in a non-autonomous equation~\eqref{FDR} which involves also a non-local term, $\lambda_\star$, given by the second moment appropriately normalized. Only the presence of this time dependent nonlocal term $\lambda_\star$ can guarantee that the second moment is the same as the natural \idx{Barenblatt function} associated to the equation, along the nonlinear flow. This allows to get rid of the next mode and exploit the improved spectral gap to get our constructive and quantitative \idx{stability} estimates also in the critical case.

%%%%%%%%%%%%%%%%%%%%%%%%%%%%%%%%%%%%%%%%%%%%%%%%%%%%%%%%%%%%%%%%%%%%%%%%
%%%%%%%%%%%%%%%%%%%%%%%%%%%%%%%%%%%%%%%%%%%%%%%%%%%%%%%%%%%%%%%%%%%%%%%%
\chapter{Discussion of the method}\label{Chapter-7}
%%%%%%%%%%%%%%%%%%%%%%%%%%%%%%%%%%%%%%%%%%%%%%%%%%%%%%%%%%%%%%%%%%%%%%%%
%%%%%%%%%%%%%%%%%%%%%%%%%%%%%%%%%%%%%%%%%%%%%%%%%%%%%%%%%%%%%%%%%%%%%%%%

The \idx{stability} results of Chapters~\ref{Chapter-5} and~\ref{Chapter-6} are obtained under a tail decay condition which is not present in Theorem~\ref{Thm:StabSubCriticalNormalized}. So far, this is the price we have to pay for constructing an explicit constant. Here we discuss why such a tail decay condition is natural in our method.

%%%%%%%%%%%%%%%%%%%%%%%%%%%%%%%%%%%%%%%%%%%%%%%%%%%%%%%%%%%%%%%%%%%%%%%%
%%%%%%%%%%%%%%%%%%%%%%%%%%%%%%%%%%%%%%%%%%%%%%%%%%%%%%%%%%%%%%%%%%%%%%%%
\section{Decay properties}

In this section, we consider decay properties of the solutions of~\eqref{FD} and~\eqref{FDr} as measured by $
\|u\|_{\mathcal{X}_m}:=\sup_{r>0}\,r^{\alpha/(1-m)}\int_{|x|>r}|u|\,\dx$ defined in~\eqref{X.norm}.

%%%%%%%%%%%%%%%%%%%%%%%%%%%%%%%%%%%%%%%%%%%%%%%%%%%%%%%%%%%%%%%%%%%%%%%%
\subsection{Decay properties in original and in self-similar variables}\label{Sec:Original.Decay}~
According to~\cite[Proposition~5.3]{bonforte2020fine} and~\cite[Chapter 4]{Simonov2020}, the solution of the \idx{fast diffusion equation} has a growth property in $\mathcal X_m$, which can be stated as follows.
%-----------------------------------------------------------------------
\begin{proposition}\label{Prop:BS} Let $d\ge1$ and $m\in(m_c,1)$. Any nonnegative solution $u$ to~\eqref{FD} with initial datum $ u_0\in\mathrm L^1_+(\R^d)\cap\X$ satisfies
\[\label{inq.Prop:BS}
\|u(t,\cdot)\|_{\X}\le \,2^\frac{2\,\alpha}{1-m}\,\max\left\{1,\cc\,\alpha^{-1/\alpha}\right\}\,\big(1+\|u_0\|_{\X}\big)\,R(t)^{\frac{\alpha }{1-m}}\quad\forall\,t\ge0
\]
with $R(t)$ given by~\eqref{R} and $\cc$ given by~\eqref{C3.constant}.\end{proposition}
%-----------------------------------------------------------------------
\begin{proof} Inequality~\eqref{Herrero.Pierre.opposite} applied with $\rho_0=1$, $r=2\,R$ and $\tau=0$ gives
\[
R^\frac{\alpha}{1-m}\int_{|x|>4R} u(t, x)\dx\le\,2^\frac{m}{1-m}\,R^\frac{\alpha}{1-m}\,\int_{|x|>2R}{u_0(x)\dx}+\cc\,t^\frac1{1-m}
\]
for any $t\ge0$. By taking the supremum in $R$ in both sides of the above inequality, we find that
\[\label{inq7.1}
\|u(t,\cdot)\|_{\X}\le\,2^\frac{\alpha+m}{1-m}\,\|u_0\|_{\X}+2^\frac{2\alpha}{1-m}\,\cc\,t^\frac{1}{1-m}
\]
and the conclusion follows.
\end{proof}
The above result turns out to be a \idx{stability} result of $\X$ under the \index{fast diffusion equation}{fast diffusion flow} in \idx{self-similar variables}.
%-----------------------------------------------------------------------
\begin{corollary}\label{Cor:BS} Let $d\ge1$ and $m\in(m_c,1)$. With the same constant $\mathsf h>0$ as in Proposition~\ref{Prop:BS}, if $v_0\in\mathrm L^1_+(\R^d)\cap\X$, then the solution~$v$ to~\eqref{FDr} with initial datum $v_0$ is such that
\[
\|v(t)\|_{\X}\le\,2^\frac{2\,\alpha}{1-m}\,\max\left\{1,\cc\,\alpha^{-1/\alpha}\right\}\,(1+\|v_0\|_{\X})\quad\forall\,t\ge0\,.
\]
\end{corollary}
%-----------------------------------------------------------------------

%%%%%%%%%%%%%%%%%%%%%%%%%%%%%%%%%%%%%%%%%%%%%%%%%%%%%%%%%%%%%%%%%%%%%%%%
\subsection{On the spaces \texorpdfstring{$\X$, $\mathcal H_p(\R^d)$ and $\mathcal W_p(\R^d)$}{X, Hp and Wp}}\label{Sec:Spaces}~

The space $\X$ gives the right framework to investigate the \index{uniform convergence in relative error}{convergence in relative error} for the flows~\eqref{FD} and~\eqref{FDr}. We remark that the \idx{relative entropy} $\mathcal F[v]$ is well defined for initial data in $\mathrm L^1_+(\R^d)$ whose second moment is finite. This is the case for data in $\X$. Indeed, we have the following
%-----------------------------------------------------------------------
\begin{proposition}\label{Prop.Second.moments}
Let $d\ge1$ and $m\in\big(\widetilde m_1,1\big)$. If $u\in\X$, then
\[
\int_{\R^d}|x|^2\,|u|\dx\le \|u\|_{\mathrm L^1(\R^d)}+4\(1-2^{2-\frac\alpha{1-m}}\)^{-1}\|u\|_{\X}\,.
\]
\end{proposition}
%-----------------------------------------------------------------------
\begin{proof}
Let us split the integral of $|x|^2\,|u|$ as
\[
\int_{\R^d}|x|^2\,|u|\dx = \int_{|x|\le 1}|x|^2\,|u|\dx+ \sum_{j=0}^\infty \int_{2^j\,<|x|\le 2^{j+1}}|x|^2\,|u|\dx\,.
\]
The first term in the right-hand-side can be bounded by the $\mathrm L^1(\R^d)$ norm of $|u|$. For the second term we proceed as follows. Let $j\ge0$ be an integer, we have then, by definition of~\eqref{X.norm}
\[
 \int_{2^j\,<|x|\le 2^{j+1}}|x|^2\,|u|\dx \le 2^{2(j+1)}\,\int_{|x|>2^j}|u|\dx \le\, 2^{2-\left(\frac{\alpha}{1-m}-2\right)j}\,\|u\|_{\X}\,.
\]
Summing up on $j\ge0$, we find
\[
 \sum_{j=0}^\infty \int_{2^j\,<|x|\le 2^{j+1}}|x|^2\,|u|\dx \le 4\, \|u\|_{\X}\, \sum_{j=0}^\infty 2^{-\left(\frac{\alpha}{1-m}-2\right)j}\,,
\]
the series converges since $\frac{\alpha}{1-m}>2$. The proof is concluded. 
\end{proof}
In order to relate the \idx{Gagliardo-Nirenberg-Sobolev inequalities} to the \index{fast diffusion equation}{fast diffusion flow}~\eqref{FD}, we consider $u=|f|^{2p}$ where $p$ and $m$ are related through~\eqref{pm}. In this way, the hypothesis $u\in\X$ is satisfied if and only if $f\in\mathrm L^{2p}(\R^d)$ and 
\be{tail.condition.f}
\sup_{r>0}\,r^\frac{d-p(d-4)}{p-1}\int_{|x|>r}|f|^{2p}\dx<\infty\,.
\ee
As a consequence, functions $f\in\mathcal H_p(\R^d)$ that satisfy~\eqref{tail.condition.f}, are in $\mathcal W_p(\R^d)$.
%-----------------------------------------------------------------------
\begin{corollary}\label{Cor:FA} Let $p\in\left(1, p^\star\right]$ if $d\ge3$ and $p\in\left(1, p^\star\right)$ if $d=1$ or $d=2$. If $f\in\mathcal H_p(\R^d)$ and $f$ satisfies~\eqref{tail.condition.f}, then $f\in\mathcal W_p(\R^d)$.
\end{corollary}
%-----------------------------------------------------------------------
\begin{proof}
Since $f\in\mathcal H_p(\R^d)$ implies that $f\in\mathrm L^{2p}(\R^d)$ for any $p\in\left(1, p^\star\right]$ if $d\ge3$ and $p\in\left(1, p^\star\right)$ if $d=1$ or $d=2$, we only need to prove that $|x||f|^{p}\in\mathrm L^2(\R^d)$. This follows from Proposition~\eqref{Prop.Second.moments} and the relation~\eqref{pm}. 
\end{proof}

%%%%%%%%%%%%%%%%%%%%%%%%%%%%%%%%%%%%%%%%%%%%%%%%%%%%%%%%%%%%%%%%%%%%%%%%
%%%%%%%%%%%%%%%%%%%%%%%%%%%%%%%%%%%%%%%%%%%%%%%%%%%%%%%%%%%%%%%%%%%%%%%%
\section{A limitation of the method} 

The \idx{stability} results of Chapter~\ref{Chapter-5} and~\ref{Chapter-6} are based on Theorem~\ref{Thm:RelativeUniform}, on the \index{uniform convergence in relative error}{convergence in relative error} for the \index{fast diffusion equation}{fast diffusion flow}~\eqref{FD}. According to~\cite{bonforte2020fine} and~\cite{Simonov2020}, the assumption $u=|f|^{2p}\in\X$ is not only sufficient but also necessary. For completeness, let us state a result and give a short proof.
%-----------------------------------------------------------------------
\begin{proposition}\label{Prop:Limitation} Let $m\in[m_1, 1)$ if $d\ge2$ and $m\in(1/3,1)$ if $d=1$. If $u$ is a solution of~\eqref{FD} with initial datum $u_0\in\mathrm L^1_+(\R^d)$ such that $\ird{u_0}=\ird{\mB}$ and if
\be{conv.rel.error}
\lim_{t\to+\infty}\sup_{x\in\R^d} \Big|\frac{u(t,x)}{B(t,x)}-1\Big|=0
\ee
then $u_0\in\X$.\end{proposition}
%-----------------------------------------------------------------------
\begin{proof} Assume that the limit~\eqref{conv.rel.error} holds. As a consequence of~\eqref{conv.rel.error}, there exists a time $T>0$ such that
\[
|u(t,x)|\le 2\, B(t,x)\quad\forall\,x \in\R^d\,,\quad\forall\,t>T\,.
\]
By applying inequality~\eqref{Herrero.Pierre.opposite} with $\rho_0=1$, $r=2R>0$, $\tau=T$ and $t=0$, we find
\begin{equation*}\begin{split}
R^\frac{\alpha}{1-m}\int_{|x|>4R}u_0(x)\dx & \le 2^\frac{m}{1-m}\,R^\frac{\alpha}{1-m} \int_{|x|>2R}|u(T, x)|\dx + \cc\,T^\frac{1}{1-m} \\
& \le 2^\frac{1}{1-m}\,R^\frac{\alpha}{1-m} \int_{|x|>2R} B(T, x)\dx + \cc\,T^\frac{1}{1-m}\,.
\end{split}\end{equation*}
Taking the supremum in both sides of the inequality, we are left with
\[
\|u_0\|_{\X}\le c_1\,\|B(T)\|_{\X} + c_2\,T<\infty
\]
for some finite constants $c_1$, $c_2>0$. The proof is concluded. 
\end{proof}
As a consequence of Theorem~\ref{Thm:RelativeUniform} and Proposition~\ref{Prop:Limitation}, we have the following side observation.
%-----------------------------------------------------------------------
\begin{corollary}\label{Cor:Limitation} Let $m\in[m_1,1)$ if $d\ge2$ and $m\in(1/3,1)$ if $d=1$. If $u$ is a solution of~\eqref{FD} with initial datum $u_0\in\mathrm L^1_+(\R^d)$ and if $u_0\not\in\X$, then $u(t,\cdot)\not\in\X$ for any $t>0$. \end{corollary}
%-----------------------------------------------------------------------

%%%%%%%%%%%%%%%%%%%%%%%%%%%%%%%%%%%%%%%%%%%%%%%%%%%%%%%%%%%%%%%%%%%%%%%%
%%%%%%%%%%%%%%%%%%%%%%%%%%%%%%%%%%%%%%%%%%%%%%%%%%%%%%%%%%%%%%%%%%%%%%%%
\section{Boundedness of the second moment}\label{ssec:secondmomentdiscussion}

In Theorem~\ref{Thm:StabSubCriticalNormalized} the assumptions about the mass and center of mass are rather natural since they contribute to chose the right profile $\mathsf g$, while the assumption
\[
\ird{|x|^2\,f^{2p}}=\ird{|x|^2\,\mathsf g^{2p}}
\]
is a subtle one. Contrary to Theorem~\ref{Thm:MainCh5} where the space $\mathcal X_m$, $m=(p+1)/(2\,p)$, is clearly a restriction, the boundedness of the second moment in Theorem~\ref{Thm:StabSubCriticalNormalized} is \emph{necessary} for obtaining a stability result where the entropy measures the distance to $\mathfrak M$. This necessity comes from the fact that both the quantities $\mathcal E[f|g]$ and $\mathcal I[f|g]$ involve the second moment of $f$ while the deficit functional $\delta$ does not. Playing on this, one can construct sequences whose second moment diverges (and so does the relative entropy and the Fisher information) while the deficit converges to zero (this also happens for the Log-Sobolev inequality as previously noted in \cite{eldan2020stability}). Such sequences may even lie in $\mathcal X_m$ but this does \emph{not} provide a counterexample to Theorem~\ref{Thm:MainCh5}: indeed the dependence in $\|f^{2p}\|_{\mathcal{X}_m}$ of the constant $\mathcal{C}$ appearing in~\eqref{stability.entropy.1.cpt6} makes the right-hand-side of inequality go to zero, suggesting that such a dependence may be optimal. We recall that $
\|u\|_{\mathcal{X}_m}:=\sup_{r>0}\,r^{\alpha/(1-m)}\int_{|x|>r}|u|\,\dx$ is defined in~\eqref{X.norm}; the exponent $\alpha/(1-m)$ may be expressed as a function 
of $p$, see~\eqref{tail.condition.f}.
%-----------------------------------------------------------------------
\begin{proposition}
Let $d\le2$ or $d\ge 3$ and $1<p<p^\star$ and set $m=(p+1)/(2\,p)$. There exists a sequence $\{f_k\}\in\mathcal W_p(\R^d)\cap\mathcal X_m$ with
\[
\ird{(1,x)\,f_k^{2p}}=\ird{(1,x)\,\mathsf{g}^{2p}}\quad\mbox{and}\quad\lim_{k\to\infty}\ird{|x|^2\,f_k^{2p}}=\infty
\]
such that
\[
\lim_{k\to\infty}\delta[f_k]=0\,,\quad\lim_{k\to\infty}\mathcal{E}[f_k|g]=\infty\quad\mbox{and}\quad\lim_{k\to\infty}\frac{\mathcal{E}[f_k|g] }{\|f_k^{2p}\|_{\mathcal{X}_m}^\frac{2(1-m)}{\alpha}}=0
\]
where $\alpha=2-d\,(1-m)$ is as in~\eqref{R}.
\end{proposition}
%-----------------------------------------------------------------------
\noindent In the subcritical range, this means that the condition on the second moment in Theorems~\ref{Thm:StabSubCriticalNormalized} and~\ref{Thm:stabilityDraft2} is not a technical condition. However, by Proposition~\ref{Prop.Second.moments}, the bound on $\|f^{2p}\|_{\mathcal{X}_m}$ is a slightly stronger condition.
\begin{proof}
Let $x_k$ be a sequence of points in $\R^d$ such that $\lim_{k\to\infty}|x_k|^2/k=\infty$. Let us define 
\[
f_k(x)^{2p}:=\left(1-\tfrac{2}{k}\right)\, \mathsf{g}^{2p}(x) + \tfrac{1}{k}\, \mathsf{g}^{2p}(x+x_k) + \tfrac{1}{k}\, \mathsf{g}^{2p}(x-x_k)\,.
\]
We will first prove that the relative entropy diverges as $k$ goes to infinity. A simple computation proves that $\ird{(1,x)\,f_k^{2p}}=\ird{(1,x)\,\mathsf g^{2p}}$ and
\begin{multline*}
\ird{|x|^2\,f_k^{2p}} = \left(1-\tfrac{2}{k}\right)\,\ird{|x|^2\,\mathsf{g}^{2p}(x)}\\
+ \tfrac{2}{k}\,\left(\ird{|x|^2\,\mathsf{g}^{2p}(x)}\,+\,|x_k|^2\,\ird{\mathsf g^{2p}(x)}\right)\,.
\end{multline*}

The relative entropy can be written as
\[
\mathcal E[f_k|\mathsf{g}]=S + \frac{2\,p}{1-p}\nrm{f_k}{p+1}^{p+1} + \frac{p+1}{p-1}\,\ird{|x|^2\,f_k^{2p}}\,,
\]
where $S$ does not depend on $k$. The norm $\nrm{f_k}{p+1}^{p+1}$ is bounded in $k$ since
\[
\left|f_k(x)\right|^{p+1}\le h_k(x)
\]
with
\begin{equation}\label{h1}
h_k(x):=3^\frac{p+1}{2p}\left[\left(1-\tfrac{2}{k}\right)^\frac{p+1}{2p}\, \mathsf{g}^{p+1}(x) + \frac{\mathsf{g}^{p+1}(x+x_k)}{k^{\frac{p+1}{2p}}}\,+\frac{\mathsf{g}^{p+1}(x-x_k)}{k^{\frac{p+1}{2p}}}\,\right],
\end{equation}
implying that
\[
\nrm{f}{p+1}^{p+1}\le 3^\frac{p+1}{2p}\left[\left(1-\tfrac{2}{k}\right)^\frac{p+1}{2p} + 2\,k^{-\frac{p+1}{2p}}\right]\,\nrm{\mathsf{g}}{p+1}^{p+1}\,.
\]
As a consequence, since $|x_k|^2/k\to\infty$ we obtain, for $k$ large enough, that
\[
\mathcal E[f_k|\mathsf{g}] \ge c\, \frac{|x_k|^2}{k}\,\ird{\mathsf{g}^{2p}}\to \infty\quad\mbox{for}\quad k\to\infty\,, 
\]
where the constant $c$ does not depend on $k$. 

Let us now consider the deficit, since $\nrm{f_k}{2p}=\nrm{\mathsf{g}}{2p}$ we can write
\[
\delta[f_k]=(p-1)^2\,\nrm{\nabla f_k}2^2+4\,\frac{d-p\,(d-2)}{p+1}\,\nrm{f_k}{p+1}^{p+1}-\mathcal K_{\mathrm{GNS}}\,\mathcal{M}^{2p\gamma}\,.
\]
To prove that $\delta[f_k]$ converges to zero we shall prove that $\limsup_{k\to\infty}\nrm{f_k}{2}\le \nrm{\nabla \mathsf{g}}{2}$ and that $\nrm{f_k}{p+1}\to\nrm{\mathsf{g}}{p+1}$. We will start with the last assertion: we have that $f_k^{p+1}\to \mathsf{g}^{p+1}$ a.e. and $f_k^{p+1} \le h_k$ ($h_k$ being defined in~\eqref{h1}). Since $\ird{h_k}\to \ird{h_\infty}= 3^\frac{p+1}{2p}  \ird{\mathsf{g}^{p+1}}$ we are in the position of using a generalized version of the \emph{dominated convergence theorem} (see~\cite[Proposition~18, p.~270]{Royden1988}), so that $\nrm{f_k}{p+1}\to \nrm{\mathsf{g}}{p+1}$ as $k\to \infty$. A simple computation show that the gradient of $f_k$ is given by
\begin{multline*}
\nabla f_k=f_k^{1-2p}\,\Big(\left(1-\tfrac2k\right)\,\mathsf{g}^{2p-1}(x)\,\nabla \mathsf{g}(x)+\tfrac1k\,\mathsf{g}^{2p-1}(x+x_k)\,\nabla \mathsf{g}(x+x_k)\\
+\tfrac1k\,\mathsf\mathsf{g}^{2p-1}(x-x_k)\,\nabla \mathsf{g}(x-x_k)\Big)\,.
\end{multline*}
Therefore we have that
\begin{align*}
\left|\nabla f_k\right|&\,\le |f_k|^{2p}\,\Big(\left(1-\tfrac2k\right)\,\mathsf g^{2p-1}(x)\,\left|\nabla \mathsf g(x)\right|+\tfrac1k\,\mathsf g^{2p-1}(x+x_k)\,\left|\nabla \mathsf g(x+x_k)\right|\\
&\hspace*{2cm}+ \tfrac1k\,\mathsf g^{2p-1}(x-x_k)\,\left|\nabla \mathsf g(x-x_k)\right|\Big)\\
&\,\le \left(1-\tfrac2k\right)^\frac{1}{2p}\left|\nabla \mathsf g(x)\right| + \frac{1}{k^\frac{1}{2p}}\,\left|\nabla \mathsf g(x+x_k)\right|+\frac{1}{k^\frac{1}{2p}}\,\left|\nabla \mathsf g(x-x_k)\right|
\end{align*}
where we have used the fact that
\[
f_k(x)\ge \max\{\left(1-\tfrac2k\right)^\frac{1}{2p}\,\mathsf g^{2p-1}(x)\,,k^{-2p}\,\mathsf g(x+x_k)\,, k^{-2p}\,\mathsf g(x-x_k)\}\,.
\]
By taking the square in the above inequality and using the Cauchy-Schwarz inequality we obtain
\begin{align*}
\nrm{\nabla f_k}{2}^2 \le \left(1-\tfrac{2}{k}\right)^\frac{1}{p}\,\nrm{\nabla \mathsf g}{2}^2
\,&+ k^{-\frac 1p}\,\nrm{\nabla \mathsf g(x+x_k)}{2}^2+ k^{-\frac 1p}\,\nrm{\nabla \mathsf g(x-x_k)}{2}^2\\
&+\,2\left(1-\tfrac{2}{k}\right)^\frac{1}{p}\,k^{-\frac 1p} \nrm{\nabla \mathsf g}{2}\,\nrm{\nabla \mathsf g(x+x_k)}{2}\\
&\,+2\left(1-\tfrac{2}{k}\right)^\frac{1}{p}k^{-\frac 1p} \nrm{\nabla \mathsf g}{2}\,\nrm{\nabla \mathsf g(x-x_k)}{2}\,.
\end{align*}
Since $\nrm{\nabla \mathsf g(x\pm x_k)}{2}=\nrm{\nabla \mathsf g}{2}$, we have that $\limsup_{k\to\infty}\nrm{\nabla f_k}{2}^2\,\le\,\nrm{\nabla \mathsf g}{2}^2$ and as a consequence
\[
0\le\limsup_{k\to\infty}\delta[f_k]\le \delta[\mathsf g]=0\,,
\]
which proves ones of the assertions. 

It remains to prove that $\mathcal E[f_k|\mathsf g]/\|f_k^{2p}\|_{\mathcal{X}_m}^\frac{2(1-m)}{\alpha}\to 0$ as $k\to\infty$. We only need to bound from below the quantity in $\|f_k^{2p}\|_{\mathcal{X}_m}$. Take $k$ large enough such that $|x|_k>>1$, we have that
\begin{align*}
\|f_k^{2p}\|\ge&\,\mathcal{C}\,|x_k|^\frac{\alpha}{1-m} \,\int_{|x|\ge|x_k|/2}\mathsf g^{2p}(x+x_k)\dx\\
\ge&\,\mathcal{C}\,|x_k|^\frac{\alpha}{1-m} \,\int_{|x+x_k|\le 1 }\mathsf g^{2p}(x+x_k)\dx=\mathcal{C}\,|x_k|^\frac{\alpha}{1-m} \,\int_{|x|\le 1 } \mathsf g^{2p}(x)\dx\,.
\end{align*}
By previous computations we know there exists a positive constant $ \mathsf C$ such that $ \mathcal{E}[f_k|g]\le\mathsf{C}\,|x_k|^2/k$. Therefore, we obtain 
\[
0\le\frac{\mathcal{E}[f_k|g]}{\|f_k^{2p}\|_{\mathcal{X}_m}^\frac{2(1-m)}{\alpha}}\le \mathrm{C}\,\frac{|x_k|^2}{k}\,\frac{k^\frac{2\,(1-m)}{\alpha}}{|x_k|^2} = \mathrm{C}\, k^\frac{(d+2)(1-m)-2}{\alpha}\to 0\quad\mbox{as}\quad k\to\infty\,,
\]
since $2-(d+2)\,(1-m)>0$ if $m>d/(d+2)$ which holds in our range of parameters because, for $d\ge3$, $m>(d-1)/d$ while for $d=1,2$, $m>1/2$.
\end{proof}

%%%%%%%%%%%%%%%%%%%%%%%%%%%%%%%%%%%%%%%%%%%%%%%%%%%%%%%%%%%%%%%%%%%%%%%%
%%%%%%%%%%%%%%%%%%%%%%%%%%%%%%%%%%%%%%%%%%%%%%%%%%%%%%%%%%%%%%%%%%%%%%%%
\section{Conclusion} 

The space $\mathcal X_m$ is the space of \emph{well-behaved} functions in which our method applies, and Proposition~\ref{Prop:Limitation} clearly shows that it cannot be extended to $\mathcal W_p(\R^d)$. From Theorem~\ref{Thm:StabSubCriticalNormalized}, we know that there is a constant of \idx{stability} even for functions of $\mathcal H_p(\R^d)$ and $\mathcal W_p(\R^d)$ which are not in $\mathcal X_m$, $m=(p+1)/(2\,p)$, with \idx{stability} measured in, respectively, $\mathcal H_p(\R^d)$ or using the \idx{relative Fisher information}. It is therefore an intriguing and challenging question to find an alternative strategy based on fast diffusion flows and entropy estimates for functions which are not \emph{well-behaved}.

%%%%%%%%%%%%%%%%%%%%%%%%%%%%%%%%%%%%%%%%%%%%%%%%%%%%%%%%%%%%%%%%%%%%%%%%
%%%%%%%%%%%%%%%%%%%%%%%%%%%%%%%%%%%%%%%%%%%%%%%%%%%%%%%%%%%%%%%%%%%%%%%% 
%%%%%%%%%%%%%%%%%%%%%%%%%%%%%%%%%%%%%%%%%%%%%%%%%%%%%%%%%%%%%%%%%%%%%%%%
%%%%%%%%%%%%%%%%%%%%%%%%%%%%%%%%%%%%%%%%%%%%%%%%%%%%%%%%%%%%%%%%%%%%%%%%
\backmatter
\chapter*{Notations}\label{Chapter-Notations}
\appendix

Throughout the memoir, $d$ is the dimension of the Euclidean space $\R^d$. $\mathrm L^q$ stands for Lebesgue's spaces, $\mathrm L^q_+$ for nonnegative functions in $\mathrm L^q$, while $\mathrm H^1$ and $\mathrm W^{1,1}$ are standard notations for Sobolev spaces: we refer to~\cite{MR2424078,MR697382,MR1817225,zbMATH00048985} and \cite[page~xix]{MR2597943} for definitions.

On $\mathrm L^q(\R^d)$, we use the norm $\nrm fq$ for any $q\in[1,+\infty]$. By default, we use Lebesgue's measure $\dx$ on $\R^d$. In case of subdomains in $\R^d$, the domain of integration will be specified. Other measures, for instance in presence of a weight $b$, are used and we denote for instance by $\mathrm L^q(\Omega,b(x)\dx)$ such a space. For convenience, we use the notation $\mathrm L^q$ even if $q<1$ for any $f$ such that $\ird{|f|^q}<+\infty$.

\medskip\noindent\textbf{Sets}
\begin{itemize}
\item[$\rhd$] $B_R(x_0)$ ball of radius $R>0$ centered at $x_0 \in \R^d$
\item[$\rhd$] $B_R$ ball of radius $R>0$ centered at the origin
\item[$\rhd$] $B$ unit ball centered at the origin, $|B|={\omega_d}/{d}$
\item[$\rhd$] $\mathbb S^d$ unit sphere in $\R^{d+1}$
\item[$\rhd$] $\Omega$ open bounded domain in $\R^d$
\end{itemize}

\medskip\noindent\textbf{Miscellaneous notations}
\begin{itemize}
\item[$\rhd$] If $x=(x_1,x_2,\ldots x_d)\in\R^d$, $|x|^2=\sum_{i=1}^dx_i^2$ and $\lrangle x=\sqrt{1+|x|^2}$
\item[$\rhd$] $\omega_d=|\mathbb S^{d-1}|=2\,\pi^{d/2}/\Gamma(d/2)$
\item[$\rhd$] $\nabla\cdot W$ stands for the divergence of the vector field $W$
\item[$\rhd$] ${\rm diam}(\Omega)=\sup_{x,y\in\Omega}|x-y|$ 
\item[$\rhd$] $\|\mathsf m\|^2$ denotes the sum of the square of the elements of the matrix $\mathsf m$
\end{itemize}

\medskip\noindent\textbf{Convention.}
Chapters or identities below mostly refer for the first occurrence after the introduction. Some of the symbols already appear in the introduction, but their definitions are then repeated in the other chapters.

%%%%%%%%%%%%%%%%%%%%%%%%%%%%%%%%%%%%%%%%%%%%%%%%%%%%%%%%%%%%%%%%%%%%%%%%
%%%%%%%%%%%%%%%%%%%%%%%%%%%%%%%%%%%%%%%%%%%%%%%%%%%%%%%%%%%%%%%%%%%%%%%%
\bigskip\begin{center}\sc{Chapter~\ref{Chapter-1}}\end{center}

\medskip\noindent\textbf{Exponents and constants in functional inequalities.}
$\mathcal C_{\mathrm{GNS}}(p)$ denotes the optimal constant in the Gagliardo-Nirenberg-Sobolev inequality~\eqref{GNS} and $\mathcal K_{\mathrm{GNS}}$ is the optimal constant in the non-scale invariant  inequality~\eqref{GNS-Intro}. Both inequalities are related in Lemma~\ref{Lem:GNscaling} through~\eqref{CGN-KGN} which involves the constant $C(p,d)$ given by~\eqref{Ch1:Cpd}. The optimal constant in Sobolev's inequality~\eqref{SobolevRd} is $\mathsf S_d=\mathcal C_{\mathrm{GNS}}(p^\star)$. Parameters and exponents obey the following conditions:
\begin{itemize}
\item[$\rhd$] $p^\star=d/(d-2)$ and $p\in(1,p^\star]$ if $d\ge3$, $p\in(1,+\infty)$ if $d=1$ or $2$
\item[$\rhd$] $2^*=2\,p^\star=2\,d/(d-2)$
\item[$\rhd$] $\theta=\big(d\,(p-1)\big)/\big(\big(d+2-p\,(d-2)\big)\,p\big)$: the exponent in~\eqref{GNS}
\item[$\rhd$] $\gamma=\big(d+2-p\,(d-2)\big)/\big(d-p\,(d-4)\big)$: the exponent in~\eqref{GNS-Intro}
\item[$\rhd$] $p=1/(2\,m-1)$ or $m=(p+1)/(2\,p)$ as, \emph{e.g.}, in~\eqref{pm}
\end{itemize}

\noindent\textbf{Function spaces}
\begin{itemize}
\item[$\rhd$] $\mathcal H_p(\R^d)=\left\{f\in\mathrm L^{p+1}(\R^d)\cap\mathrm L^{2\,p}(\R^d)\,:\,|\nabla f|\in\mathrm L^2(\R^d)\right\}$
\item[$\rhd$] $\mathcal H_{p^\star}(\R^d)=\left\{f\in\mathrm L^{2\,p^\star}(\R^d)\,:\,|\nabla f|\in\mathrm L^2(\R^d)\right\}$
\item[$\rhd$] $\mathcal W_p(\R^d)=\left\{f\in\mathcal H_p(\R^d)\,:\,\lrangle x\,|f|^p\in\mathrm L^2(\R^d)\right\}$
\end{itemize}

\medskip\noindent\textbf{Aubin-Talenti functions}
\begin{itemize}
\item[$\rhd$] $\mathsf g(x)=\(1+|x|^2\)^{-\frac1{p-1}}$
\item[$\rhd$] $g_{\lambda,\mu,y}(x)=\lambda^\frac d{2p}\,\mu^\frac1{2p}\,\mathsf g\big(\lambda\,(x-y)\big)$ with the convention $\mu^q=|\mu|^{q-1}\,\mu$ if $\mu<0$
\item[$\rhd$] $\mathfrak M=\left\{g_{\lambda,\mu,y}\,:\,(\lambda,\mu,y)\in(0,+\infty)\times\R\times\R^d\right\}$
\end{itemize}

\medskip\noindent\textbf{Functionals}
\begin{itemize}
\item[$\rhd$] $\delta[f]$ Deficit functional defined by~\eqref{Deficit}
\item[$\rhd$] Free energy or relative entropy functional:\\ $\mathcal E[f|g]=2\,p\,(1-p)^{-1}\ird{\(f^{p+1}-g^{p+1}-\tfrac{1+p}{2\,p}\,g^{1-p}\(f^{2p}-g^{2p}\)\)}$
\item[$\rhd$] Relative Fisher information:\\ $\mathcal J[f|g]=(p+1)\,(p-1)^{-1}\ird{\left|(p-1)\,\nabla f+f^p\,\nabla g^{1-p}\right|^2}$
\end{itemize}

\medskip\noindent\textbf{Linear operator and spectrum}
\begin{itemize}
\item[$\rhd$] $\mathrm d\mu_\alphaa=\mu_\alphaa\dx$, $\mu_\alphaa(x)= (1+|x|^2)^{\alphaa}$
\item[$\rhd$] $\mathcal L_{\alphaa,d}\,u=-\,\mu_{1-\alphaa}\,\mathrm{div}\left[\,\mu_\alphaa\,\nabla u\,\right]$
\item[$\rhd$] $\Lambda_{\mathrm{ess}}=\(\alphaa+(d-2)/2\)^2$: bottom of the essential spectrum of $\mathcal L_{\alphaa,d}$, $\alphaa<0$
\item[$\rhd$] $\lambda_{\ell k}$: discrete eigenvalues of $\mathcal L_{\alphaa,d}$, $\alphaa<0$, given in Proposition~\ref{Prop:Spectrum}
\item[$\rhd$] $\Lambda$: spectral gap in the Hardy-Poincar\'e inequality~\eqref{Hardy-Poincare}
\item[$\rhd$] $\Lambda_{\mathrm{ess}}$: spectral gap in the improved Hardy-Poincar\'e~\eqref{Improved-Hardy-Poincare}
\end{itemize}

\medskip\noindent\textbf{Miscellaneous notations}
\begin{itemize}
\item[$\rhd$] $\varphi(s)=s^m/(m-1)$ for any $s\ge0$, $m\in(0,1)$
\item[$\rhd$] $\mu[f]$, $y[f]=x_f$ and $\lambda[f]$: defined in~\eqref{lambdamuy}, such that $\mu[g_{\lambda,\mu,y}]$, $y[g_{\lambda,\mu,y}]$ and $\lambda[g_{\lambda,\mu,y}]$ coincide with $\lambda$, $\mu$ and $y$ for any $g_{\lambda,\mu,y}\in\mathfrak M$
\item[$\rhd$] $g_f=g_{\lambda[f],\mu[f],y[f]}$
\item[$\rhd$] Pressure variable $\P$: defined in~\eqref{Eqn:h-BE}, consistently reused in Section~\ref{Sec:Pressure-Fisher}
\item[$\rhd$] $\mathrm D^2\P$ denotes the Hessian matrix of $\P$
\item[$\rhd$] $\DD w=\(\nabla w,\tfrac{\partial w}{\partial z}\)$ and $N=d+2\,\nu$
\item[$\rhd$] $\mathsf Q$: quadratic form given by~\eqref{Q}
\end{itemize}

%%%%%%%%%%%%%%%%%%%%%%%%%%%%%%%%%%%%%%%%%%%%%%%%%%%%%%%%%%%%%%%%%%%%%%%%
%%%%%%%%%%%%%%%%%%%%%%%%%%%%%%%%%%%%%%%%%%%%%%%%%%%%%%%%%%%%%%%%%%%%%%%%
\bigskip\begin{center}\sc{Chapter~\ref{Chapter-2}}\end{center}

\medskip\noindent\textbf{Parameters}
\begin{itemize}
\item[$\rhd$] $m$ denotes the exponent in the fast-diffusion equation~\eqref{FD} and $M$ is the mass of the solution $u$, that is, $M=\ird u$.
\item[$\rhd$] $m_c=0$ if $d=1$ (by convention), and $m_c=(d-2)/d$ if $d\ge2$ is the threshold exponent for $\mB\in\mathrm L^1(\R^d)$ for mass conservation in~\eqref{FD}
\item[$\rhd$] $m_1=(d-1)/d$ is the threshold exponent for inequality~\eqref{entropy.eep} and $m=m_1$ corresponds to $p=p^\star$ according to~\eqref{pm}, if $d\ge3$
\item[$\rhd$] $\widetilde m_1=d/(d+2)$: threshold exponent for the integrability of $|x|^2\,\mB$ and~$\mB^m$
\item[$\rhd$] $\alpha=d\,(m-m_c)$
\item[$\rhd$] $\muscal=\big((1-m)/(2\,m)\big){}^{1/\alpha}$ is a scaling parameter given by~\eqref{mu} which appears in the change of variables~\eqref{SelfSimilarChangeOfVariables} that transforms the solution $u$ of~\eqref{FD} into the solution $v$ of~\eqref{FDr}
\item[$\rhd$] $\Mstar=\ird{\mB}=\nrm{\mathsf g}{2p}^{2p}=\pi^{d/2}\,\Gamma\big((1-m)^{-1}-d/2\big)/\Gamma\big((1-m)^{-1}\big)$, see~\eqref{Mstar}
\end{itemize}

\medskip\noindent\textbf{Barenblatt solutions and self-similar scale}
\begin{itemize}
\item[$\rhd$] $\mB(x)=\(1+|x|^2\)^{1/(m-1)}=\mathsf g^{2p}$ is the Barenblatt function (or profile).
\item[$\rhd$] $R(t)=(1+\alpha\,t)^{1/\alpha}$
\item[$\rhd$] $B\big(t\,,\,x\,;\,M\big)=\(M/\Mstar\)^{2/\alpha}\,\muscal^d\,R(t)^{-d}\,\mB\big(\( M/\Mstar\)^{(1-m)/\alpha}\,\muscal\,x/R(t)\big)$
\item[$\rhd$] $B(t,x)=\muscal^d\,R(t)^{-d}\,\mB\(\muscal\,x/R(t)\)=B(t,x;\Mstar)$
\end{itemize}

\medskip\noindent\textbf{Entropy and entropy production functionals}
\begin{itemize}
\item[$\rhd$] The R\'enyi entropy functional $\mathsf E[u]=\ird{u^m}$ and the R\'enyi Fisher information functional $\mathsf I[u]=\frac{m^2}{(1-m)^2}\ird{u\,|\nabla u^{m-1}|^2}$ are associated with the solution $u$ of~\eqref{FD}. With the pressure variable $\P=m\,u^{m-1}/(1-m)$, we have $\mathsf I[u]=\ird{u\,|\nabla\P|^2}$.
\item[$\rhd$] Attached to the solution $v$ of\eqref{FDr}, the free energy or relative entropy and the Fisher information or relative entropy production are defined respectively by $\mathcal F[v]=(m-1)^{-1}\ird{\(v^m-\mB^m-m\,\mB^{m-1}\,(v-\mB)\)}$ and $\mathcal I[v]=m\,(1-m)^{-1}\ird{v\,\left|\nabla v^{m-1}-\nabla\mB^{m-1}\right|^2}$. The relative pressure variable $\relativePressure(t,x)=v^{m-1}(x)-\,|x|^2$ is such that $\mathcal I[v]=\ird{v\,|\nabla\mathsf Q|^2}$.
\item[$\rhd$] Linearized free energy: $\mathsf F[\h]=\frac m2\ird{|\h|^2\,\mB^{2-m}}$ and linearized Fisher information: $\mathsf I[\h]=m\,(1-m)\ird{|\nabla \h|^2\,\mB}$
\end{itemize}

%%%%%%%%%%%%%%%%%%%%%%%%%%%%%%%%%%%%%%%%%%%%%%%%%%%%%%%%%%%%%%%%%%%%%%%%
%%%%%%%%%%%%%%%%%%%%%%%%%%%%%%%%%%%%%%%%%%%%%%%%%%%%%%%%%%%%%%%%%%%%%%%%
\bigskip\begin{center}\sc{Chapter~\ref{Chapter-3}}\end{center}

\medskip\noindent\textbf{Sets}
\begin{itemize}
\item[$\rhd$] $\Omega_T=\left(0, T\right)\times\Omega$
\item[$\rhd$] $D_R^+(t_0,x_0),D_R^-(t_0,x_0)$, cylinders of the Harnack inequality~\eqref{harnack} defined in~\eqref{cylinder.harnack}
\item[$\rhd$] $Q_\varrho, Q^+_\varrho, Q^-_\varrho$ cylinders in the Moser's iteration, defines in~\eqref{Parab.Cylinders}
\end{itemize}

\medskip\noindent\textbf{Functions}
\begin{itemize}
\item[$\rhd$] $v$ is a solution to~\eqref{HE.coeff}.
\item[$\rhd$] $w=v^{p/2}$ is a power of a solution to~\eqref{HE.coeff} used in Section~\ref{Sec:MoserIteration}.
\item[$\rhd$] $w=-\log v$ is the logarithm of a solution to~\eqref{HE.coeff} used in Section~\eqref{Sec:LogarithmicEstimates}
\end{itemize}

\medskip\noindent\textbf{Constants}
\begin{itemize}
\item[$\rhd$] $\mathcal K$, constant of inequality~\eqref{sob.step2}
\item[$\rhd$] $\mathsf S_B$, constant of inequality~\eqref{SobolevH1}
\item[$\rhd$] $\lambda_b$, constant of the weighted Poincar\'e inequality~\eqref{Lem.Log.Est.4b}
\item[$\rhd$] $\mu=\lambda_0^{-1}+\lambda_1$, where $\lambda_0>0$ and $\lambda_1>0$ are the lower and upper bound of the uniformly ellipticity condition~\eqref{HE.coeff.lambdas}
\item[$\rhd$] $\mathsf h$ is given by~\eqref{h}
\item[$\rhd$] $\overline{\mathsf h}=\mathsf h^{\lambda_1+1/\lambda_0}$ is given by~\eqref{h-bar}
\item[$\rhd$] $\nu$ exponent of the $C^\nu$ continuity of solutions to~\eqref{HE.coeff} defined in~\eqref{nu}
\end{itemize}

\medskip\noindent\textbf{Miscellaneous notations}
\begin{itemize}
\item[$\rhd$] $A\,v$ stand for the standard multiplication of a matrix $A$ for a vector $v$.
\item[$\rhd$] $v\cdot(A\,w)=\sum_{i,j=1}^{d}A_{i,j}\,v_i\,w_j$ where $A$ is a matrix and $v, w$ are vectors
\item[$\rhd$] $\lfloor u\rfloor_{C^\nu\left(\Omega\right)}$, $C^\nu$ semi-norm defined in~\eqref{C-alpha-norms} in a domain $\Omega$
\item[$\rhd$] $dist(\cdot, \cdot)$, distance between cylinders, defined in~\eqref{parabolic-distance}
\end{itemize}

%%%%%%%%%%%%%%%%%%%%%%%%%%%%%%%%%%%%%%%%%%%%%%%%%%%%%%%%%%%%%%%%%%%%%%%%
%%%%%%%%%%%%%%%%%%%%%%%%%%%%%%%%%%%%%%%%%%%%%%%%%%%%%%%%%%%%%%%%%%%%%%%%
\bigskip\begin{center}\sc{Chapter~\ref{Chapter-4}}\end{center}

\medskip\noindent\textbf{Miscellaneous notations}
\begin{itemize}
\item[$\rhd$] $B(t,x)=B\big(t\,,\,x\,;\,\Mstar\big)=\muscal^d\,R(t)^{-d}\,\mB\({\muscal\,x}/{R(t)}\)$
\item[$\rhd$] $\lambdaBarenblatt=\(\tfrac{1-m}{2\,m\,\alpha}\)^{\!1/\alpha}$ appears in the definition of the self-similar profiles in~\eqref{BarenblattM-delta}.
\end{itemize}

\medskip\noindent\textbf{Function spaces}
\begin{itemize}
\item[$\rhd$] $\X=\{u\in \mathrm L^1(\R^d):\|u\|_{\mathcal{X}_m}<\infty \}$
\item[$\rhd$] $\|u\|_{\mathcal{X}_m}=\sup_{r>0}\, r^\frac{\alpha}{(1-m)}\int_{|x|>r}|u|\dx$
\end{itemize}

\medskip\noindent\textbf{Constants}
\begin{itemize}
\item[$\rhd$] $A$ and $G$ are the two main bounds on the initial data which appear in Theorem~\ref{Thm:RelativeUniform} in~\eqref{hyp:Harnack} and play a key role in the computation of the threshold time $t_\star$ in~\eqref{inq.RelativeUniform} and subsequent stability estimates. It turns out that $G$ can be estimated by $A$ and $M$: see~\cite[Section~1.3]{BDNS-CKN}.
\item[$\rhd$] $\overline\varepsilon:=\(\overline M/\Mstar\)^{2/\alpha}-1$
\item[$\rhd$] $\underline\varepsilon:=1-\(\underline M\,/\Mstar\)^{2/\alpha}$
\item[$\rhd$] $\varepsilon_{m,d}:=\min\big\{\overline\varepsilon,\,\underline\varepsilon,\,1/2\big\}$
\item[$\rhd$] $\lambda_0$ and $\lambda_1$ are used in the estimates of the $C^\nu$ continuity of solutions to~\eqref{FD} and are defined in~\eqref{lambdas}.
\item[$\rhd$] $\nu$ is the exponent of the H\"older continuity of solutions to~\eqref{FD}.
\item[$\rhd$] $\vartheta=\frac{d}{d+\nu}$ is defined in~\eqref{theta}.
\item[$\rhd$] $\mathsf a$ and $\taustarA$ appear in the computation of $t_\star$ and whose expression is given in~\eqref{a} and~\eqref{taustarabstract}.
\item[$\rhd$] $t_\star:=\taustarA\big(1+A^{1-m}+G^\frac\alpha2\big)/\varepsilon^{\mathsf a}$
\item[$\rhd$] $T_\star:= \frac1{2\,\alpha}\,\log\(1+\alpha\,\taustar\,\frac{1+A^{1-m}+G^\frac\alpha2}{\varepsilon^\mathsf{a}}\)$ and $\taustar$ is defined in~\eqref{Tstar}.
\end{itemize}

%%%%%%%%%%%%%%%%%%%%%%%%%%%%%%%%%%%%%%%%%%%%%%%%%%%%%%%%%%%%%%%%%%%%%%%%
%%%%%%%%%%%%%%%%%%%%%%%%%%%%%%%%%%%%%%%%%%%%%%%%%%%%%%%%%%%%%%%%%%%%%%%%
\bigskip\begin{center}\sc{Chapter~\ref{Chapter-5}}\end{center}

\medskip\noindent\textbf{Constants}
\begin{itemize}
\item[$\rhd$] $\zeta_\star=\frac{4\,\eta}{4+\eta}\big(\frac{\varepsilon_\star^{\mathsf a}}{2\,\alpha\,\taustar}\big)^{2/\alpha}c_\alpha$
\item[$\rhd$] $c_\alpha=\inf_{x,\,y>0}{\(1+x^{2/\alpha}+y\)}{\(1+x+y^{\alpha/2}\)^{-2/\alpha}}$
\item[$\rhd$] $\mathsf Z(A,G)={\zeta_\star}/\({1+A^{2\,(1-m)/\alpha}+G}\)$
\end{itemize}

\medskip\noindent\textbf{Functionals}
\begin{itemize}
\item[$\rhd$] $\kappa[f]=\mu[f]^{-1/(2p)}=\Mstar^{1/(2p)}\,\nrm f{2\,p}$
\item[$\rhd$] $\lambdasigma[f]=\(\frac{2\,d\,\kappa[f]^{p-1}}{p^2-1}\,\frac{\nrm f{p+1}^{p+1}}{\nrm{\nabla f}2^2}\)^\frac{2\,p}{d-p\,(d-4)}$
\item[$\rhd$] $\mathsf A_p[f]=\Mstar\,\lambdasigma[f]^{-\frac{d-p\,(d-4)}{p-1}}\,\nrm f{2\,p}^{-2p}\,\sup_{r>0}r^\frac{d-p\,(d-4)}{p-1}\int_{|x|>r}|f(x+x_f)|^{2p}\dx$
\item[$\rhd$] $\mathsf E_p[f]=\frac{2\,p}{1-p}\ird{\Big(\tfrac{\kappa[f]^{p+1}}{\lambdasigma[f]^{d\,\frac{p-1}{2\,p}}}\,|f|^{p+1}\!-\!\mathsf g^{p+1}\!-\!\frac{1+p}{2\,p}\,\mathsf g^{1-p}\,\Big(\tfrac{\kappa[f]^{2p}}{\lambdasigma[f]^2}\,|f|^{2p}\!-\!\mathsf g^{2p}\Big)\Big)}$
\item[$\rhd$] $\mathsf N f(x)=\lambdasigma[f]^\frac d{2\,p}\,\kappa[f]\,f\big(\lambdasigma[f]\,x+x_f\big)$
\item[$\rhd$] $\mathfrak S[f]=\frac{\Mstar^\frac{p-1}{2\,p}}{p^2-1}\,\frac1{C(p,d)}\,\mathsf Z\(\mathsf A[f],\,\mathsf E[f]\)$
\item[$\rhd$] $\mathcal C[f]=\tfrac{p-1}{p+1}\,\mathsf Z\big(A,\mathcal E[\mathsf Nf|\mathsf g]\big)$
\item[$\rhd$] $\mathsf J[f]=\ird{\left|\,\lambdasigma[f]^\frac{d-p\,(d-2)}{2\,p}\,\nabla f+\kappa[f]^{p-1}\,\lambdasigma[f]^{-1}\,x\,|f|^p\,\right|^2}$
\item[$\rhd$] $\mathfrak S_\star[f]=\frac{\mathcal C_{\mathrm{GNS}}^{2\,p\,\gamma-2}}{2\,p\,\gamma\,\Mstar^{1/p}}\,\frac{\(p^2-1\)}{C(p,d)}\,\frac{\mathsf Z\(\mathsf A[f],\,\mathsf E\big[|f|\big]\)}{4+\mathsf Z\(\mathsf A[f],\,\mathsf E\big[|f|\big]\)}$
\end{itemize}

%%%%%%%%%%%%%%%%%%%%%%%%%%%%%%%%%%%%%%%%%%%%%%%%%%%%%%%%%%%%%%%%%%%%%%%%
%%%%%%%%%%%%%%%%%%%%%%%%%%%%%%%%%%%%%%%%%%%%%%%%%%%%%%%%%%%%%%%%%%%%%%%%
\bigskip\begin{center}\sc{Chapter~\ref{Chapter-6}}\end{center}

\medskip\noindent\textbf{Constant and Parameters}
\begin{itemize}
\item[$\rhd$]  $\mathfrak{C}_\star=(d-1)^{-1}\,\frac{4\,\eta}{4+\eta}\(\frac{\eta^{\mathsf a}\,\chi^{\mathsf a}}{2^{1+\mathsf a}\,\alpha\,\mathfrak{c}_\star}\)^{2/\alpha}\,c_\alpha\,$, constant defined in the proof of Theorem~\ref{Thm:Main}, where $\varepsilon_\star=\eta\,\chi/2$ and $\eta$, $\chi$ are as in Proposition~\ref{Prop:Gap-6}, $\mathsf a$ as in Theorem~\ref{Thm:RelativeUniform}, $\mathfrak{c}_\star$ as in~\eqref{Tplustau}, $c_\alpha$ as in~\eqref{calpha} 
\item[$\rhd$] $\mathcal C_\star(A)=\mathfrak C_\star\,\big(1\!+\!A^{1/(2\,d)}\big)^{-1}$ where $A$ is in Theorem~\ref{Thm:Main}
\item[$\rhd$] $\J=\ird{|x|^2\,\mathcal B}$  defined in~\eqref{sigma-6}
\item[$\rhd$] $\E=\ird{\mB^m}$ defined in Section~\ref{sec:defandcon}.
\item[$\rhd$] $\aA=2\,d\,\frac{1-m}m$
\item[$\rhd$] $\bB=2\,d\,(m-m_c)=2\,\alpha$
\item[$\rhd$] $\tau_\bullet$ defined in~\eqref{taubullet} 
\item[$\rhd$] $\alphaa=\frac{(d+2)^2}{8\,d}$ if $3\le d\le6$ and $\alphaa=2\,\frac{d-2}d$ if $d\ge6$ defined in~\eqref{Imp:SG}
\item[$\rhd$] $\eta= \frac{(d-2)^2}{8\,d}$ if $3\le d\le6$ and $\eta =2\,\frac{d-4}{d}$ if $d\ge6$
\item[$\rhd$] $\chi=1/580$
\item[$\rhd$] $q=(d+2)^{-1}\,2^{1-d/2}$
\item[$\rhd$] $T(\varepsilon, A):= \log\(1+\alpha\,\mathfrak{c}_\star\,\(1+A^{1-m}\)/\varepsilon^{\mathsf a}\)$ defined in~\eqref{Tplustau} where $\mathfrak{c}_\star$ is defined in~\eqref{mathfrakcstar} and $\mathsf a$ is as in~\eqref{a}
\end{itemize}

\medskip\noindent\textbf{Functionals}
\begin{itemize}
\item[$\rhd$] $A[f]=\sup_{r>0}\,r^d\,\int_{r>0} |f|^{2^*}(x+x_f)$ where $x_f$ is as in~\eqref{lambdamuy}
\item[$\rhd$] $Z[f]:=\Big(1+\mu[f]^{-d}\,\lambda[f]^d\,A[f]\Big)$ where $\mu[f]$, $\lambda[f]$ and $x_f=y[f]$ are as in~\eqref{lambdamuy}
\end{itemize}

\medskip\noindent\textbf{Miscellaneous notations}
\begin{itemize}
\item[$\rhd$] $\mathcal B_\lambda(x)=\lambda^{-\frac d2}\,\mathcal B\(x/\sqrt\lambda\)$ is defined in~\eqref{rescaled-second-moment}.
\item[$\rhd$] $\lambda(t)$, $\Rr(t)$ and $\tau(t)$ are defined in~\eqref{ODE-6}.
\item[$\rhd$] $\lambda_\star$ is defined in~\eqref{Lem:BestMatching}.
\item[$\rhd$] $\mathcal Q_\lambda[w]=m\,\lambda^{\frac d2\,(m-m_c)}\,\ird{w\,\left|\nabla w^{m-1}-\nabla\mB_\lambda^{m-1}\right|^2}/\ird{\(\mB_\lambda^m-w^m\)}$ defined in Section~\ref{Sec:EEP}
\end{itemize}

%%%%%%%%%%%%%%%%%%%%%%%%%%%%%%%%%%%%%%%%%%%%%%%%%%%%%%%%%%%%%%%%%%%%%%%%
%%%%%%%%%%%%%%%%%%%%%%%%%%%%%%%%%%%%%%%%%%%%%%%%%%%%%%%%%%%%%%%%%%%%%%%%
\printindex
\bibliographystyle{amsalpha}
\bibliography{BDNS}

\end{document}